\documentclass[preprint]{article}

\makeatletter
\renewcommand\paragraph{\@startsection{paragraph}{4}{\z@}%
                                    {1ex \@plus1ex \@minus.2ex}%
                                    {-1em}%
                                    {\normalfont\normalsize\bfseries}}
\makeatother

\usepackage{amsmath}
\usepackage{amssymb}
\usepackage{amsthm}
\usepackage[margin=1in]{geometry}
\usepackage[sort&compress,numbers]{natbib} 
\usepackage{color}
\usepackage{enumitem}
\usepackage{hyperref}
 \hypersetup{
     colorlinks=true,
     linkcolor=blue,
     filecolor=blue,
     citecolor = blue,      
     urlcolor=black,
     }
\usepackage{bm}
\usepackage{mathrsfs}
\usepackage[scr=boondox]{mathalpha}

\usepackage{extarrows}
\usepackage{graphicx}
\usepackage{subfig}
\usepackage{algorithmic,algorithm}
\usepackage{mathrsfs}
\usepackage{cmll}
\usepackage{mathtools}

\usepackage[cmyk]{xcolor} 
\usepackage[title]{appendix}

\providecommand{\sref}[1]{Section~\ref{#1}}
\providecommand{\appref}[1]{Appendix~\ref{#1}}
\providecommand{\fref}[1]{Figure~\ref{#1}}

\providecommand{\thref}[1]{Theorem~\ref{#1}}
\providecommand{\defref}[1]{Definition~\ref{#1}}
\providecommand{\lemref}[1]{Lemma~\ref{#1}}

\providecommand{\assumpref}[1]{Assumption~\ref{#1}}

\providecommand{\factref}[1]{Fact~\ref{#1}}
\providecommand{\propref}[1]{Proposition~\ref{#1}}
\providecommand{\corref}[1]{Corollary~\ref{#1}}

\providecommand{\clref}[1]{Claim~\ref{#1}}

\providecommand{\R}{\ensuremath{\mathbb{R}}}
\providecommand{\C}{\ensuremath{\mathbb{C}}}
\providecommand{\N}{\ensuremath{\mathbb{N}}}
\newcommand{\ortho}{\mathbb{O}} 

\renewcommand{\vec}[1]{\ensuremath{\boldsymbol{#1}}}
\providecommand{\mat}[1]{\ensuremath{\boldsymbol{#1}}}

\providecommand{\mF}{\mat{F}}

\providecommand{\mI}{\mat{I}}

\providecommand{\mP}{\mat{P}} 
\providecommand{\mQ}{\mat{Q}} 

\providecommand{\mS}{\mat{S}} 
\providecommand{\mU}{\mat{U}} 

\providecommand{\mW}{\mat{W}}

\providecommand{\mPsi}{\mat{\Psi}}
\providecommand{\mPsi}{\mat{\Psi}}

\providecommand{\mO}{\mat{O}}
\providecommand{\mY}{\mat{Y}}

\providecommand{\va}{\vec{a}}

\providecommand{\vr}{\vec{r}}

\providecommand{\vu}{\vec{u}} 
\providecommand{\vw}{\vec{w}}
\providecommand{\vx}{\vec{x}} 

\providecommand{\vz}{\vec{z}}

\providecommand{\vzero}{\vec{0}}

\providecommand{\vv}{\vec{v}}

\newcommand{\ip}[2]{\left\langle {#1}, {#2} \right\rangle} 
\DeclareMathOperator{\diag}{diag}
\DeclareMathOperator{\Tr}{Tr}

\DeclareMathOperator{\op}{op}
\newcommand*{\tran}{^{\mkern-1.5mu\mathsf{T}}}

\newcommand{\explain}[2]{\overset{\text{\tiny{#1}}}{#2}} 
\providecommand{\bydef}{\explain{def}{=}}

\DeclareMathOperator*{\argmin}{\arg\min}

\newcommand*\diff{\mathop{}\!\mathrm{d}}

\newcommand{\unif}[1]{\mathsf{Unif}(#1)} 
\newcommand{\E}{\mathbb{E}} 
\renewcommand{\P}{\mathbb{P}} 
\newcommand{\Var}{\mathrm{Var}}
\newcommand{\Cov}{\mathrm{Cov}}
\newcommand{\gauss}[2]{\mathcal{N}\left( #1,#2 \right)} 
\newcommand{\pc}{\overset{\P}{\longrightarrow}} 
\newcommand{\wc}{\overset{W_2}{\longrightarrow}}
\newcommand{\dc}{\overset{d}{\longrightarrow}}

\renewcommand{\tilde}{\widetilde}
\renewcommand{\hat}{\widehat}
\DeclareFontFamily{U}{mathx}{\hyphenchar\font45}
\DeclareFontShape{U}{mathx}{m}{n}{
      <5> <6> <7> <8> <9> <10>
      <10.95> <12> <14.4> <17.28> <20.74> <24.88>
      mathx10
      }{}
\DeclareSymbolFont{mathx}{U}{mathx}{m}{n}
\DeclareFontSubstitution{U}{mathx}{m}{n}
\DeclareMathAccent{\widecheck}{0}{mathx}{"71}
\DeclareMathAccent{\wideparen}{0}{mathx}{"75}

\renewcommand{\dim}{N} 
\newcommand{\prior}{\pi} 
\newcommand{\bdot}{\bullet} 
\providecommand{\mmse}{\mathsf{mmse}_{\pi}}
\providecommand{\dmmse}{\mathsf{dmmse}_{\pi}}
\newcommand{\iter}[2]{{#1}_{#2}} 
\newcommand{\serv}[1]{\mathsf{#1}} 

\newcommand{\mfunc}{\Psi}
\newcommand{\fnonlin}{f}
\newcommand{\gnonlin}{g}
\newcommand{\hnonlin}{h}
\newcommand{\proc}{H}

\DeclareMathOperator*{\plim}{\mathrm{plim}}
\newcommand{\degree}{D}
\newcommand{\Fnonlin}{F}
\DeclareMathOperator*{\plimsup}{{\mathrm{\plim \sup}}}
\DeclareMathOperator*{\pliminf}{{\plim \inf}}

\newcommand{\Gnonlin}{G}
\newcommand{\aux}{a}
\newcommand{\auxdim}{k}

\newcommand{\dfnew}[1]{\bm\bar{#1}}
\newcommand{\bmmse}{\mathrm{MMSE}}
\newcommand{\bdmmse}{\mathrm{DMMSE}}
\newcommand{\bdnsr}{\varphi}
\newcommand{\dfbdnsr}{\dfnew{\bdnsr}}
\newcommand{\optlin}{v_{\mathrm{opt}}}
\newcommand{\effsnr}{\omega_{\mathrm{eff}}}
\newcommand{\scrM}{\mathscr{M}}
\newcommand{\scrm}{\mathscr{m}}

\newcommand{\optlintd}[1]{\dup{\iter{v}{#1}}{\degree}}
\newcommand{\effsnrt}[1]{\iter{\omega}{#1}}
\newcommand{\effsnrtd}[1]{\dup{\iter{\omega}{#1}}{\degree}}
\newcommand{\dmmset}[1]{\iter{\mathsf{d}}{#1}}
\newcommand{\dmmsetd}[1]{\dup{\iter{\mathsf{d}}{#1}}{\degree}}
\newcommand{\rhotd}[1]{\dup{\iter{\mathsf{\rho}}{#1}}{\degree}}
\newcommand{\rhot}[1]{{\iter{\mathsf{\rho}}{#1}}}

\newcommand{\scrL}[1]{\mathscr{L}_{#1}}
\newcommand{\polynom}[1]{\mathcal{P}_{#1}}
\providecommand{\clmref}[1]{Claim~\ref{#1}}
\newcommand*{\I}{\imath}
\newcommand*{\hlb}{\mathscr{H}}
\newcommand*{\stl}{\mathscr{S}}

\newcommand{\omd}{\Psi_{\star}}
\newcommand{\omdd}[1]{{\Psi}^{(#1)}_{\star}}

\newcommand{\dup}[2]{#1^{(#2)}}

\newcommand{\snrit}{\omega_{\mathrm{IT}}}
\newcommand{\rhoit}{\rho_{\mathrm{IT}}}

\newcommand{\BE}{\begin{equation}}
\newcommand{\EE}{\end{equation}}
\newcommand{\BS}{\begin{subequations}}
\newcommand{\ES}{\end{subequations}}
\newcommand{\UT}{\mathsf{T}}   

\providecommand{\sfLambda}{\mathsf{\Lambda}}

\theoremstyle{plain}

\newtheorem{theorem}{Theorem}
\newtheorem{proposition}{Proposition}
\newtheorem{lemma}{Lemma}
\newtheorem{corollary}{Corollary}
\newtheorem{conjecture}{Conjecture}
\theoremstyle{definition}
\newtheorem{claim}{Claim}
\newtheorem{definition}{Definition}
\newtheorem{assumption}{Assumption}

\theoremstyle{remark}
\newtheorem{remark}{Remark}
\newtheorem{fact}{Fact}

\graphicspath{{figures/}}

\begin{document}

\title{Optimality of Approximate Message Passing Algorithms for Spiked Matrix Models with Rotationally Invariant Noise}
\author{Rishabh Dudeja\thanks{Department of Statistics, University of Wisconsin–Madison. Email: \tt{rdudeja@wisc.edu}} \and Songbin Liu\thanks{Academy of Mathematics and Systems Science, Chinese Academy of Sciences. Email: \tt{liusongbin@lsec.cc.ac.cn}} \and Junjie Ma\thanks{Academy of Mathematics and Systems Science, Chinese Academy of Sciences. Email: \tt{majunjie@lsec.cc.ac.cn} } \footnote{The authors are alphabetically ordered.}}


\maketitle

\begin{abstract} We study the problem of estimating a rank-one signal matrix from a noisy observed matrix corrupted by additive rotationally invariant noise. We develop a new class of approximate message passing algorithms for this problem and provide a simple and concise characterization of their dynamics in the high-dimensional limit. At each iteration, these algorithms leverage prior knowledge about the noise structure by applying a non-linear matrix denoiser to the eigenvalues of the observed matrix, and utilize prior information regarding the signal structure by applying a non-linear iterate denoiser to the previous iterates generated by the algorithm. We derive the optimal choices for both the matrix and iterate denoisers and demonstrate that the resulting algorithm achieves the lowest possible asymptotic estimation error among a broad class of iterative algorithms under a fixed iteration budget.
\end{abstract}


\tableofcontents
\section{Introduction}
We consider the problem of estimating a symmetric rank-one matrix from a noisy $\dim \times \dim$ observed matrix $\mY$ generated from the \emph{spiked matrix model} \citep{johnstone2001distribution}:
\begin{align} \label{eq:model}
    \mY  & = \frac{\theta}{\dim}   \vx_\star \vx_\star \tran + \mW,
    \end{align}
where $\bm{x}_\star\in\mathbb{R}^N$ is the $\dim$-dimensional unknown signal of interest, $\theta\ge0$ is the signal-to-noise ratio (SNR) parameter, and $\bm{W}$ is a symmetric noise matrix. 
This model and its variants have been used to study a broad range of statistical inference problems, including sparse PCA \citep{deshpande2014information,johnstone2009consistency,zou2006sparse}, community detection \citep{deshpande2017asymptotic,abbe2018community}, and group synchronization \citep{singer2011angular,perry2018message,fan2021tap,celentano2023local}.
\paragraph*{Wigner Noise Model} The most well-studied variant of the spiked model is the \emph{Spiked Wigner Model} \citep{deshpande2014information,perry2018optimality,el2018estimation,el2020fundamental,barbier2019adaptive,montanari2021estimation}, which assumes that the noise matrix has i.i.d. Gaussian entries. A rich line of work in high-dimensional statistics and random matrix theory has studied the problem from various perspectives.
\begin{itemize}
\item \emph{Design and Analysis of Estimators.} A natural estimator for the signal is the PCA estimator or the leading eigenvector of the observed matrix. A line of work \citep{baik2005phase,baik2006eigenvalues,peche2006largest,feral2007largest,paul2007asymptotics,benaych2011eigenvalues,knowles2013isotropic} initiated by \citet{baik2005phase} has obtained a sharp asymptotic characterization of the performance of this estimator in the high-dimensional limit and uncovered surprising properties of this estimator. One way to improve the performance of PCA is to exploit prior structural information about the signal (\emph{e.g.}, sparsity). Approximate message passing (AMP) algorithms are a popular class of computationally efficient iterative algorithms designed to exploit such structural information. These algorithms were first discovered in the context of compressed sensing \citep{donoho2009message,kabashima2003cdma,bayati2011dynamics,bolthausen2014iterative} and have been adapted for the spiked Wigner model \cite{rangan2012iterative,parker2014bilinear,kabashima2016phase,montanari2021estimation,li2022non,li2023approximate}. A particularly attractive feature of these algorithms is that their performance in the high-dimensional limit is characterized by a simple deterministic recursion known as \textit{state evolution}, which makes it possible to assess their information-theoretic optimality (or sub-optimality).
\item \emph{Information-theoretic Limits.} Under the assumption that the signal is drawn from a known prior distribution, the optimal estimator is the Bayes estimator $\E[\vx_{\star} | \mY]$. Using the powerful cavity method from statistical physics, \citet{lesieur2015mmse} derived a conjecture for the Bayes risk (or the asymptotic performance of the Bayes estimator) for the Spiked Wigner model. This conjecture has been proved rigorously in increasing generality in a series of influential works \citep{deshpande2014information,dia2016mutual,krzakala2016mutual,lelarge2017fundamental,miolane2017fundamental,el2018estimation,barbier2019adaptive,behne2022fundamental}. This formula for asymptotic Bayes risk provides a fundamental information-theoretic limit on the performance of any estimator for the problem. 
\item \emph{Computational Limits.} In general, computing the Bayes-optimal estimator is computationally intractable in high dimensions. However, a suitably designed AMP algorithm called Bayes-optimal AMP can attain the Bayes risk for the problem under {sufficiently high signal-to-noise ratios. For lower signal-to-noise ratios, the Bayes-optimal AMP algorithm fails to do so}, and no computationally efficient algorithm is known to achieve the Bayes risk \citep{lesieur2015phase,krzakala2016mutual,dia2016mutual,miolane2017fundamental}. This has led to a popular conjecture that the Bayes-optimal AMP algorithm is the optimal polynomial time algorithm for this problem \citep{lesieur2015phase,miolane2017fundamental}. \citet{celentano2020estimation}, \citet{montanari2022statistically}, and \citet{montanari2022equivalence} (see also the earlier work of \citet{schramm2022computational}) have provided evidence for this conjecture by showing that the Bayes-optimal AMP algorithm achieves the lowest possible estimation error among a broad class of iterative algorithms and low-degree polynomial estimators.  \looseness=-1 
\end{itemize}
Collectively,  these works have enriched our understanding of the fundamental trade-offs between statistical optimality and computational efficiency in high-dimensional statistics.
\paragraph*{Rotationally Invariant Noise Model} In this paper, we study a natural generalization of the i.i.d. Gaussian noise model known as the rotationally invariant noise model where one posits that the eigenvectors of the noise matrix are given by a uniformly random orthogonal matrix independent of the eigenvalues. The rotationally invariant noise model is intended to model noise matrices with strong statistical dependence whose eigenvectors are generic \citep{fan2022approximate}. Several works (\emph{see e.g.,} \citep{donoho2009observed,oymak2014case,li2023spectrum}) have observed that this can be a good model in some applications. \citet{benaych2011eigenvalues} analyzed the performance of PCA (or spectral estimators) for this noise model. A line of work \citep{opper2016theory,fan2022approximate,mondelli2021pca,zhong2021approximate} initiated by \citet{opper2016theory} and \citet{fan2022approximate} has developed AMP algorithms for this problem, which can improve the performance of PCA by exploiting signal structure. Unlike AMP algorithms for i.i.d. Gaussian noise, the dynamics of these algorithms are characterized by a significantly more intricate state evolution, suggesting that analyzing the rotationally invariant noise model requires new ideas beyond the i.i.d. Gaussian case. Our current understanding of the fundamental information-theoretic and computational limits for the rotationally invariant noise model is extremely limited. Recent work by \citet{barbier2023fundamental} studies the spiked matrix model with rotationally invariant noise under the assumption that the noise matrix is drawn from the \emph{trace ensemble}. Under this model, the density of $\mW$ is given by:
\[
p(\mW)  \propto \exp\left( - \frac{\dim}{2}\sum_{i=1}^\dim V(\lambda_i(\mW)) \right), \text{ where $\left(\lambda_{i}(\mW)\right)_{i\in[N]}$ denote the eigenvalues of $\mW$,}
\]
and the \emph{potential function} $V:\R \mapsto \R$ is a functional parameter for the noise model. This assumption ensures that the problem has a well-defined likelihood, enabling the study of information-theoretic limits. Moreover, an appropriate choice of $V$ can capture a wide range of noise eigenvalue spectrums in different applications.
\citet{barbier2023fundamental} make progress towards understanding the information-theoretic limits of the problem and provide important insights regarding the structure of the optimal computationally efficient algorithms.
\begin{itemize}
\item \emph{Information-theoretic Limits.} \citet{barbier2023fundamental} develop a general (although non-rigorous) recipe based on the replica method to derive conjectured formulas for Bayes risk under the assumption that the potential $V$ is a polynomial function, providing explicit conjectured formulas for the asymptotic Bayes risk when $V$ is a quartic (degree four) or a sestic (degree six) polynomial. As the degree of $V$ grows, the derivation and the resulting formulas become increasingly complex and a general conjecture for the Bayes risk is unavailable.  \looseness=-1 
\item \emph{Optimal Computationally Efficient Algorithms.} Surprisingly, \citet{barbier2023fundamental} demonstrate that a natural generalization of the Bayes-optimal AMP algorithm (this is the optimal iterative algorithm for i.i.d. Gaussian noise and is conjectured to be the optimal polynomial-time algorithm) to the rotationally invariant noise model is sub-optimal. The authors develop AMP algorithms that achieve improved performance by applying a non-linear matrix denoiser to the eigenvalues of the observed matrix $\mY$ and characterize their state evolution. This is in sharp contrast to the i.i.d. Gaussian noise model, where such a matrix denoising step is unnecessary. Based on non-rigorous statistical physics techniques \citep{opper2001adaptive}, the authors provide a procedure to derive good matrix denoisers for the problem. The authors propose matrix denoisers with explicit formulas when $V$ is a quartic or sestic polynomial (a general formula is not available). However, the state evolution of the resulting AMP algorithm is quite complicated; hence, the resulting algorithm's optimality (or sub-optimality) properties are not understood.
\end{itemize}
\paragraph*{Our Contributions} We take inspiration from the insights of \citet{barbier2023fundamental} and study the spiked matrix model with rotationally invariant noise from an algorithmic point of view. Our results are not restricted to the trace ensemble with a  polynomial potential $V$. Rather, they apply to all rotationally invariant noise matrices that meet mild regularity conditions. Our main contributions are:  \looseness=-1 
 \begin{itemize}
     \item We develop a new class of AMP algorithms for this problem and provide a state evolution result that characterizes their dynamics in the high-dimensional limit (\thref{thm:SE} in \sref{sec:main-results-OAMP}). At each iteration, these algorithms exploit the noise structure by applying a non-linear matrix denoiser to the eigenvalues of the observed matrix and the signal structure by applying a non-linear iterate denoiser to the previous iterates. These algorithms can be viewed as natural analogs of orthogonal \citep{ma2017orthogonal} or vector AMP algorithms \citep{rangan2019vector} developed for compressed sensing.  \looseness=-1 
     \item A key feature of the AMP algorithms proposed in this work is that their state evolution is significantly simpler than that of existing AMP algorithms for this problem. Consequently, we are able to exploit our result to derive the optimal choices for the matrix and iterate denoisers (\sref{sec:OAMP_opt}). Interestingly, we find that the matrix denoisers that optimize the performance of the AMP algorithm are closely related to the eigenvalue shrinkage estimators discovered by \citet{bun2016rotational} in a separate line of work on denoising high-rank and unstructured signal matrices corrupted with rotationally invariant noise \citep{ledoit2011eigenvectors,ledoit2012nonlinear,bun2016rotational,lolas2021shrinkage,pourkamali2023rectangular,semerjian2024matrix}. \looseness=-1 
    \item Building on the techniques developed by \citet{celentano2020estimation} and \citet{montanari2022statistically} (in the context of i.i.d. Gaussian noise), we show that the AMP algorithm with the optimal choices for the matrix and iterate denoisers achieves the smallest possible asymptotic estimation error among a broad class of iterative algorithms under a fixed iteration budget (\thref{thm:optimality} in \sref{sec:optimality}). This suggests that this algorithm might be the natural candidate for the optimal polynomial-time estimator for this problem. 
    \item Finally, our results also suggest a general and concise conjecture for a set of fixed point equations which characterize the Bayes risk (\sref{sec:IT-optimality}) for the problem, which we hope can be proved rigorously in the future. 
 \end{itemize} 
\paragraph*{Related Concurrent Work} {Independent work by \citet{barbier2024information}, which appeared shortly after our arXiv submission, also studies this problem from a different and complementary perspective. Using the replica method, the authors derive conjectures for the asymptotic Bayes risk and mutual information for this problem, generalizing their prior work \citep{barbier2023fundamental}. The fixed point equations characterizing the Bayes risk coincide with those derived in our paper, providing further evidence supporting the conjecture. The authors also derive the Thouless-Anderson-Palmer (TAP) equations for this problem. These are high-dimensional non-linear fixed point equations that characterize the Bayes estimator $\E[\vx_{\star} | \mY]$. Using the TAP equations, \citet{barbier2024information} propose a natural algorithm to compute the Bayes estimator. While a state evolution result for this algorithm is not available, based on simulations, the authors conjecture that this algorithm achieves the Bayes-optimal performance whenever it is computationally feasible to do so.}  
\paragraph*{Organization} This paper is organized as follows. We start with some preliminary results in Section \ref{Sec:pre}. The main results of this paper are presented in Section \ref{Sec:main-results}. 
Section \ref{Sec:proof_idea} highlights the key ideas behind our results. Numerical experiments are presented in Section \ref{Sec:simulations}. The complete proofs of our main results and additional numerical results not presented in the main paper are provided in the appendix.
\paragraph*{Notation} We conclude the introduction by defining the notations used in this paper. 
\begin{description}[font=\normalfont\emph,leftmargin=0cm,itemsep=1ex]
\item [\textit{Some common sets:}] The sets $\N, \R, \C$ represent the set of positive integers, real numbers, and complex numbers, respectively. For $\dim \in \N$, $[\dim]$ is the set $\{1, 2, 3, \dotsc, \dim\}$ and $\ortho(\dim)$ denotes the set of $\dim \times \dim$ orthogonal matrices. 
\item[\textit{Linear Algebra:}] For vectors $u,v \in \R^{k}$,  $\|u\|$ is the $\ell_2$ norm of $u$, $\ip{u}{v} = \sum_{i=1}^k u_i v_i$ denotes the standard inner product on $\R^k$, and $\diag(u)$ represents the $k \times k$ diagonal matrix constructed by placing the entries of $u$ along the diagonal. For a matrix $M \in \R^{k \times k}$, $\Tr[M], \|M\|_{\op}, \|M\|$ represent the trace, operator (spectral) norm, and Frobenius norm of $M$ respectively. If $M$ is symmetric, we denote the sorted eigenvalues of $M$ by $\lambda_{1}(M) \geq \dotsb \geq \lambda_{k}(M)$. We use $1_k$ to denote the vector $(1, 1, \dotsc, 1)$ in $\R^k$,  $0_k$ to denote the zero vector $(0, 0, \dotsc, 0)$ in $\R^k$,  $I_{k}$ denotes the $k \times k$ identity matrix, and $e_1, e_2, \dotsc, e_{k}$ to denote the standard basis vectors in $\R^k$. When the dimension is clear from the context, we will abbreviate $1_k, 0_k, I_k$ as $1, 0, I$.  We use the bold-face font for vectors and matrices whose dimensions diverge as $\dim$ (the dimension of the signal vector) grows to $\infty$. For example, the signal $\vx_\star \in \R^{\dim}$, and the noise matrix $\mW \in \R^{\dim \times \dim}$ are bold-faced.
\item[\textit{Probability:}] We use $\E[\cdot], \Var[\cdot], \Cov[\cdot]$ to denote expectations, variances, and covariances of random variables. The 
Gaussian distribution on $\R^{k}$ with mean vector $\mu \in \R^k$ and covariance matrix $\Sigma \in \R^{k \times k}$ is denoted by $\gauss{{\mu}}{\Sigma}$.  For a finite set $A$, $\unif{A}$ represents the uniform distribution on $A$. We will use $\unif{\ortho(\dim)}$ to denote the Haar measure on the orthogonal group $\ortho(\dim)$. For any $x \in \R$, the probability measure $\delta_x$ on $\R$ denotes the point mass at $x$. We use $\pc$ and $\dc$ to denote convergence in probability and distribution, respectively. For a sequence of real-valued random variables $(Y_{\dim})_{\dim \in \N}$, we say that $\plim Y_{\dim} = y$ if $Y_{\dim} \pc y$, $\plimsup Y_{\dim} \leq y$ if for any $\epsilon > 0$, $\lim_{\dim \rightarrow \infty} \P( Y_{\dim} \geq y + \epsilon) = 0$, and $\pliminf Y_{\dim} \geq y$ if for any $\epsilon > 0$, $\lim_{\dim \rightarrow \infty} \P( Y_{\dim} \leq y - \epsilon) = 0$.    
\end{description}

\section{Preliminaries}\label{Sec:pre}
We begin by introducing our assumptions along with some concepts that play an important role in this paper. 
\subsection{Convergence and Asymptotic Equivalence of High-Dimensional Vectors} We will rely on the following notions of convergence and equivalence of high-dimensional vectors.
\begin{definition} \label{def:W2} Let $(\vv_1, \dotsc, \vv_{\ell})$ be a collection of random vectors in $\R^\dim$. We say that the empirical distribution of the entries of the vectors $(\vv_1, \dotsc, \vv_{\ell})$ converges to random variables $(\serv{V}_1, \dotsc, \serv{V}_{\ell})$ as $\dim \rightarrow \infty$ if, for any test function $h: \R^\ell \mapsto \R$ which satisfies: 
\begin{align} \label{eq:PL2-testfunc}
|h(v) - h(v^\prime)|  \leq L \|v - v^\prime\| \cdot (1 + \|v\| + \|v^\prime\|) \quad \forall \; v,v^\prime \; \in \; \R^{\ell},
\end{align}
for some $L < \infty$, we have:
\begin{align*}
\frac{1}{\dim} \sum_{i=1}^\dim h({v}_{1}[{i}], \dotsc, {v}_{\ell}[i]) \pc \E[h(\serv{V}_1, \dotsc, \serv{V}_{\ell})] \text{ as } \dim \rightarrow \infty.
\end{align*}
We denote convergence in this sense using the notation: $({\vv}_{1}, \dotsc, {\vv}_{\ell}) \wc (\serv{V}_1, \dotsc, \serv{V}_{\ell})$. We refer the reader to \citep{bayati2011dynamics,feng2022unifying,fan2022approximate} for more information on this type of convergence. We say that two $\dim$-dimensional random vectors $\vu$ and $\vv$ are asymptotically equivalent if:
\begin{align*}
\frac{\|\vu - \vv\|^2}{\dim} \pc 0 \quad \text{ as } \quad  \dim \rightarrow \infty.
\end{align*}
We denote equivalence in this sense using the notation: $\vu  \explain{$\dim \rightarrow \infty$}{\simeq} \vv$.
\end{definition}
\subsection{Signal and Noise Models} Recall that our goal is to recover the signal vector $\vx_{\star} \in \R^{\dim}$ from the noisy observation $\mY = (\theta/\dim) \cdot \vx_{\star} \vx_{\star}^\top + \mW$. We will allow our estimators the flexibility to exploit any side information $\va \in \R^{\dim \times k}$ available regarding the signal. We posit the following assumptions on the signal $\vx_{\star}$, the side information $\va$, and the noise $\mW$.
\begin{assumption}[Signal and Noise Models] \label{assump:signal-noise} The signal vector and side information satisfy $(\vx_{\star}; \va) \wc (\serv{X}_{\star}; \serv{A})$  for some limiting random variables $(\serv{X}_{\star}, \serv{A})$ with joint distribution $\prior$ which satisfies  $\E[\serv{X}_{\star}^2] = 1$ and $\E[\|\serv{A}\|^2] < \infty$. We model the noise matrix $\mW$ as a random matrix sampled independently of $(\vx_{\star} ; \va)$, with eigen-decomposition:
\begin{align*}
\mW & = \mU \cdot \diag( \lambda_1(\mW), \dotsc, \lambda_{\dim}(\mW)) \cdot \mU^\top,
\end{align*}
where the matrix of eigenvectors $\mU \sim \unif{\ortho(\dim)}$ is a Haar-distributed random orthogonal matrix and the eigenvalues $ \lambda_1(\mW), \dotsc, \lambda_{\dim}(\mW)$ are deterministic. We assume that $\|\mW\|_{\op}$ is bounded by a $\dim$-independent constant $C$ and the spectral measure $\mu_{\dim}$ of $\mW$:
\begin{align*}
\mu_{\dim} & \bydef \frac{1}{\dim} \sum_{i=1}^\dim \delta_{\lambda_i(\mW)}
\end{align*}
converges weakly to a compactly supported distribution $\mu$ on $\R$. We require the limiting spectral measure $\mu$ to be absolutely continuous with respect to the Lebesgue measure. Furthermore, the density of $\mu$, which we denote using the same symbol $\mu: \R \mapsto \R$, is assumed to be Holder continuous in the following sense:
\begin{align} \label{eq:holder}
|\mu(\lambda) - \mu(\lambda^\prime)| & \leq L \cdot |\lambda - \lambda^\prime|^{\alpha} \quad \forall \; \lambda, \lambda^\prime \; \in \; \R,
\end{align}
where $L< \infty$ and $\alpha >0$ are some constants.
\end{assumption}

\subsection{Spectral Measure in the Signal Direction} \label{sec:spec-sig} Let $\nu_{\dim}$ denote the spectral measure of the observed matrix $\mY$ in the direction of the signal:
\begin{align} \label{eq:signal-spec}
\nu_{\dim} & \bydef \frac{1}{\dim} \sum_{i=1}^\dim  \ip{\vu_i(\mY)}{\vx_{\star}}^2 \cdot \delta_{\lambda_i(\mY)},
\end{align}
where $\lambda_{1}(\mY), \dotsc, \lambda_{\dim}(\mY)$ denote the eigenvalues of $\mY$ and $\vu_1(\mY), \dotsc, \vu_{\dim}(\mY)$ denote the corresponding eigenvectors. 
The measure $\nu_{\dim}$ and its weak limit $\nu$ play a key role in our analysis. Expressing the state evolution of the proposed AMP algorithm in terms of $\nu$ leads to a concise formula. The following lemma (proved in Section \ref{App:pre-RMT} in the appendix) provides a formula for the density of $\nu$ in terms of the function $\phi: \R \mapsto \R$:
\begin{align} \label{eq:phi}
\phi(\lambda) & \explain{def}{=} {(1- \pi \theta \hlb_{\mu}(\lambda))^2 + \pi^2 \theta^2  \mu^2(\lambda)} \quad \lambda \; \in \; \R.
\end{align}
In the above display, $\hlb_{\mu}:\R \mapsto \R$ denotes the Hilbert transform of $\mu$,  which is defined as the Cauchy principal value of the following singular integral\footnote{Under the Holder continuity requirement \eqref{eq:holder} on $\mu$, the limit in \eqref{Eqn:Hilbert_def_pre} exists and the Hilbert transform $\hlb_{\mu}: \R \mapsto \R$ is also Holder continuous in the sense of \eqref{eq:holder}; see \citep[Section 2.1]{pastur2011eigenvalue} and \citep[Theorem 14.11a]{henrici1993applied}.
}:
\begin{equation}\label{Eqn:Hilbert_def_pre}
\hlb_{\mu}(z) \explain{def}{=}  \lim_{\epsilon \rightarrow 0} \frac{1}{\pi} \int_{|z-\lambda| \geq \epsilon} \frac{\mu(\lambda)}{z- \lambda} \; \diff 
 \lambda \quad \forall \; z \; \in \;  \R.
\end{equation}
The measure $\nu$ and the function $\phi$ will be frequently referenced in this paper. 
\begin{lemma} \label{lem:RMT} We have,
\begin{enumerate}
\item The measure $\nu_{\dim}$ converges weakly in probability to a compactly supported probability measure $\nu$ on $\R$ as $\dim \rightarrow \infty$.
\end{enumerate}
Let $\nu = \nu_{\parallel} + \nu_{\perp}$ denote the Lebesgue decomposition of $\nu$ into the absolutely continuous part $\nu_{\parallel}$ and the singular part $\nu_{\perp}$. Then,
\begin{enumerate}
\setcounter{enumi}{1}
\item For Lebesgue-almost every $\lambda$, $\phi(\lambda) \neq 0$ and for $\nu_{\perp}$-almost every $\lambda$, $\phi(\lambda) = 0$.   
\item The density of the absolutely continuous part of $\nu$ is given by $\mu(\cdot)/\phi(\cdot)$ where $\mu(\cdot)$ denotes the density of $\mu$ and the function $\phi(\cdot)$ is as defined in \eqref{eq:phi}. 
\end{enumerate}
\end{lemma}

\begin{remark} {Consider a typical situation when $\mu$ is supported on a single interval $[\lambda_{-}, \lambda_{+}]$ with a positive density on $(\lambda_-, \lambda_+)$. The results of \citet{benaych2011eigenvalues} show that when the SNR $\theta$ exceeds the critical threshold $\theta_c \explain{def}{=} 1/(\pi \hlb_\mu(\lambda_+))$, $\lambda_1(\mY)$, the largest eigenvalue of $\mY$,  separates from the support of $\mu$ and converges to a limit $\lambda_c \in (\lambda_+, \infty)$ which is the unique root of the equation (in $\lambda$):
\begin{align} \label{eq:outlier}
    \hlb_{\mu}(\lambda) = \frac{1}{\pi \theta} \; \explain{\eqref{eq:phi}}{\Leftrightarrow} \; \phi(\lambda) = 0 \quad \text{on the domain } \lambda \in (-\infty,\lambda_-] \cup [\lambda_+, \infty).
\end{align}
Moreover, the corresponding eigenvector satisfies:
\begin{align*}
    \frac{\ip{\vu_1(\mY)}{\vx_{\star}}^2}{\dim} \pc  -\frac{1}{\theta^2 \pi \hlb_\mu^\prime(\lambda_c)}.
\end{align*}
In light of the definition of $\nu_\dim$ \eqref{eq:signal-spec},  we expect that $\nu$ should have a point mass at $\lambda_c$ with weight $-\frac{1}{\theta^2 \pi \hlb_\mu^\prime(\lambda_c)}$. This is consistent with \lemref{lem:RMT}, which shows that the singular part of $\nu$ concentrates on the set $\{\lambda \in \R: \phi(\lambda) = 0\}$ (cf. \eqref{eq:outlier}), and hence consists of a point mass at $\lambda_c$.}
\end{remark}

\subsection{Gaussian Channels} We will also use some basic notions regarding Gaussian channels, introduced in the definition below.
\begin{definition}[Gaussian Channel]\label{def:gauss-channel} A Gaussian channel is  a collection of real-valued random variables $(\serv{X}_\star, \serv{X}_1, \dotsc, \serv{X}_t; \serv{A})$ where the signal $\serv{X}_{\star}$ and the side information $\serv{A}$ are drawn from the prior $\pi$ from  \assumpref{assump:signal-noise}. The observations $\serv{X}_1, \dotsc, \serv{X}_t$ are given by: $\serv{X}_i =  \alpha_i \serv{X}_{\star} + \serv{Z}_i \quad \forall \; i \in [t]$, where $\alpha_1, \dotsc, \alpha_t$ are real numbers and $(\serv{Z}_1, \dotsc, \serv{Z}_t) \sim \gauss{0}{\Sigma}$ are zero mean and jointly Gaussian random variables, which are independent of $(\serv{X}_{\star}, \serv{A})$. 
We introduce some important notions related to Gaussian channels.
\begin{description}[font=\normalfont\emph,leftmargin=0cm,itemsep=1ex]
\item[\textit{MMSE and MMSE Estimator.}] The minimum mean squared error (MMSE) for estimating the signal $\serv{X}_{\star}$ based on the observations $(\serv{X}_1, \dotsc, \serv{X}_t)$ and the side information $\serv{A}$, denoted by $\bmmse(\serv{X}_\star | \serv{X}_1, \dotsc, \serv{X}_t; \serv{A})$, is defined as:
\begin{align} \label{eq:MMSE-def}
\bmmse(\serv{X}_\star | \serv{X}_1, \dotsc, \serv{X}_t; \serv{A}) \explain{def}{=}  \min_{f \in L^2(\serv{X}_1, \dotsc, \serv{X}_t; \serv{A})} \E[\{\serv{X}_\star - f(\serv{X}_1, \dotsc, \serv{X}_t; \serv{A})\}^2],
\end{align}
where the minimum is over $L^2(\serv{X}_1, \dotsc, \serv{X}_t; \serv{A})$, which denotes the set of all measurable functions $f: \R^{t+\auxdim} \mapsto \R$ satisfying $\E[f^2(\serv{X}_1, \dotsc, \serv{X}_t; \serv{A})] < \infty$. The function $f \in L^2(\serv{X}_1, \dotsc, \serv{X}_t; \serv{A})$ that  minimizes the RHS in \eqref{eq:MMSE-def} is called the MMSE estimator for $\serv{X}_{\star}$. 
\item[\textit{DMMSE and DMMSE Estimator.}] The divergence-free minimum mean squared error (DMMSE) for estimating the signal $\serv{X}_{\star}$ based on  $(\serv{X}_1, \dotsc, \serv{X}_t; \serv{A})$, denoted by $\bdmmse(\serv{X}_\star | \serv{X}_1, \dotsc, \serv{X}_t; \serv{A})$, is defined as:
\begin{align} 
&\bdmmse(\serv{X}_\star | \serv{X}_1, \dotsc, \serv{X}_t; \serv{A}) \explain{def}{=}  \label{eq:DMMSE-def} \\ &\min_{f \in L^2(\serv{X}_{1:t}; \serv{A})} \E[\{\serv{X}_\star - f(\serv{X}_1, \dotsc, \serv{X}_t; \serv{A})\}^2] \quad \text{subject to } \E[\serv{Z}_i f(\serv{X}_{1:t}; \serv{A})] = 0, \; \forall \; i \; \in \; [t], \nonumber
\end{align}
where $\serv{X}_{1:t}$ is a shorthand for the collection of random variables $(\serv{X}_1,\ldots,\serv{X}_t)$. The constraints $\E[\serv{Z}_i f(\serv{X}_1, \dotsc, \serv{X}_t; \serv{A})] =0,\; \forall \; i \in [t]$ in \eqref{eq:DMMSE-def} require the estimator to be uncorrelated with the noise and are called \emph{divergence-free} constraints. The function $f$ which minimizes the RHS in \eqref{eq:DMMSE-def} is called the DMMSE estimator for $\serv{X}_{\star}$.
\item [\textit{Scalar Gaussian Channels.}] A important role is played by \emph{scalar} Gaussian channels, which refers to a  collection of real-valued random variables $(\serv{X}_\star, \serv{X}; \serv{A})$ generated as follows:
\begin{align*}
(\serv{X}_{\star}; \serv{A}) \sim \pi, \quad \serv{X}| \serv{X}_{\star}; \serv{A} \sim \gauss{\sqrt{\omega} \cdot \serv{X}_{\star}}{1-\omega},
\end{align*}
where $\omega \in [0,1]$ is called the signal-to-noise ratio (SNR) of the channel. Notice that a scalar Gaussian channel is a Gaussian channel with a single observation $\serv{X}$ normalized to satisfy $\E[\serv{X}^2] = 1$. For a scalar Gaussian channel $(\serv{X}_\star, \serv{X}; \serv{A})$, we define several important functions which will play a key role in this paper. 
\begin{description}[font=\normalfont\emph,leftmargin=0cm,itemsep=0ex]
\item[\textit{MMSE Function and MMSE Estimator for Scalar Gaussian Channels.}] The function $\mmse: [0,1] \mapsto [0,1]$ represents MMSE of a scalar Gaussian channel as a function of the SNR $\omega$:
\begin{align*}
\mmse(\omega) \explain{def}{=} \bmmse(\serv{X}_\star | \serv{X}, \serv{A}).
\end{align*}The function $\bdnsr(\cdot | \omega) : \R \times \R^{\auxdim} \mapsto \R$ denotes the MMSE estimator for the scalar Gaussian channel at SNR $\omega$:
\begin{align} \label{eq:mmse-scalar}
    \bdnsr(x; \aux | \omega) \explain{def}{=} \E[\serv{X}_{\star} | \serv{X} = x, \serv{A} = a] \quad \forall \; x \in \R, \; \aux \in \R^{\auxdim}.
\end{align}
\item[\textit{DMMSE Function and DMMSE Estimator for Scalar Gaussian Channels.}] The function $\dmmse: [0,1] \mapsto [0,1]$ represents DMMSE of a scalar Gaussian channel as a function of the SNR $\omega$:
\begin{align*}
\dmmse(\omega) \explain{def}{=} \bdmmse(\serv{X}_\star | \serv{X}, \serv{A}).
\end{align*}
\item The function $\dfbdnsr(\cdot | \omega) : \R \times \R^{\auxdim} \mapsto \R$ denotes the DMMSE estimator for the scalar Gaussian channel at SNR $\omega$. \citet{ma2017orthogonal} have shown that the DMMSE estimator of a scalar Gaussian channel at SNR $\omega$ is given by (see Lemma \ref{lem:scalarization-mmse} in the appendix for a self-contained proof):
\begin{align} \label{eq:dmmse-scalar}
\dfbdnsr (x; \aux| \omega) &\explain{def}{=}  \begin{cases} 
 \left(1 - \frac{\sqrt{\omega}}{\sqrt{1-\omega}} \E[\serv{Z}\bdnsr({\serv{X};\serv{A} | \omega)]}   \right)^{-1} \left(\bdnsr(x;a | \omega) - \frac{\E[\serv{Z}\bdnsr({\serv{X};\serv{A}} | \omega)]}{\sqrt{1-\omega}} x \right), &  \omega < 1, \\ x, & \omega = 1. \end{cases}
\end{align} 
\end{description}
\end{description}
\end{definition}
Finally, we will impose the following regularity condition in our analysis.
\begin{assumption}\label{assump:regularity} For any $\omega \in [0,1]$, the MMSE estimator $\bdnsr( \cdot | \omega) : \R \times \R^{\auxdim} \mapsto \R$ for the scalar Gaussian channel $(\serv{X}_{\star}; \serv{A}) \sim \pi, \quad \serv{X} | \serv{A}, \serv{X}_{\star} \sim  \gauss{\sqrt{\omega}  \cdot \serv{X}_{\star}}{1-\omega}$ is continuously differentiable and Lipschitz\footnote{In the absence of any side information, a sufficient condition for \assumpref{assump:regularity} to hold is that the signal random variable $\serv{X}_{\star}$  is compactly supported (see \citep[Remark 2.3]{montanari2021estimation} and \citep[Lemma 3.8]{feng2022unifying}).}.
\end{assumption}


\section{Main results}\label{Sec:main-results}
We now present the main results obtained in this paper.
\subsection{Orthogonal Approximate Message Passing Algorithms} \label{sec:main-results-OAMP}
We introduce a class of iterative methods to estimate the signal $\vx_{\star}$ using observation $\mY$ generated from the spiked matrix model \eqref{eq:model}, called \emph{Orthogonal Approximate Message Passing} (OAMP) algorithms.
\begin{definition}[OAMP algorithms]\label{Def:OAMP_main} An OAMP algorithm generates iterates $\vx_1, \vx_2, \dotsc$ in $\R^{\dim}$ according to the update rule:
\BS\label{Eqn:OAMP}
\begin{align} 
\vx_t&=\Psi_t(\bm{Y}) \cdot f_t(\vx_1, \dotsc, \vx_{t-1};\va) \quad \forall \;  t\in\mathbb{N},
\end{align}
where $\mY \in \R^{\dim \times \dim}$ is the observed noisy matrix and  $\va\in\mathbb{R}^{\dim \times k}$ denotes the side information available for estimating $\vx_\star$. The estimate of $\vx_\star$ at iteration $t$ is obtained by applying a post-processing function $\psi_t$ to the iterates $\vx_1, \dotsc, \vx_t$ and the side-information $\va$:
\BE
\hat{\vx}_t=\psi_t\left(\vx_{1}, \dotsc, \vx_t;\va \right).
\EE
\ES
In the above equations, 
for each $t \in \N$, the matrix denoiser $\Psi_t:\mathbb{R}\mapsto\mathbb{R}$ acts on the matrix $\mY$ in the following way: if $\mY = \mO \diag(\lambda_1, \dotsc, \lambda_{\dim}) \mO^\top $ is the eigen-decomposition of $\mY$, then, $\mfunc_t(\mY) = \mO \diag(\mfunc_t(\lambda_1),  \dotsc, \mfunc_t(\lambda_\dim)) \mO^\top$. Likewise, the iterate denoisers $\fnonlin_t: \R^{t-1} \times \R^{\auxdim} \mapsto \R$ and the post-processing function function $\psi_t:\mathbb{R}^{t} \times \R^{\auxdim} \mapsto\mathbb{R}$ act entry-wise on the $\dim$ components of its vector inputs. {We also require that the matrix denoisers $(\mfunc_t)_{t \in \N}$ and iterate denoisers $(\fnonlin_t)_{t \in \N}$ satisfy certain constraints, which we will introduce in equations \eqref{eq:trace-free} and \eqref{eq:div-free-constraint} below, after defining the notion of state evolution random variables associated with an OAMP algorithm.}
\begin{description}[font=\normalfont\emph,leftmargin=0cm,itemsep=0ex]
\item[\textit{State Evolution Random Variables.}] Each OAMP algorithm is associated with a collection of {state evolution random variables} $(\serv{X}_{\star}, (\serv{X}_{t})_{t \in \N}; \serv{A})$, which describes the joint asymptotic behavior of the signal $\vx_{\star}$, the iterates $(\vx_t)_{t \in \N}$, and the side information $\va$. The distribution of these random variables is given by:
\begin{subequations} \label{Eqn:SE_variables_def}
\begin{align}
(\serv{X}_\star, \serv{A}) \sim \pi, \quad \serv{X}_{t} = \beta_t \serv{X}_\star + \serv{Z}_{t} \quad \forall \; t \; \in \; \N,
\end{align}
where $(\beta_t)_{t \in \N}$ is defined via the recursion:
\begin{align} 
\beta_t & \explain{def}{=} \E[\serv{X}_\star \fnonlin_t(\serv{X}_1, \dotsc, \serv{X}_{t-1}; \serv{A})] \cdot \E_{\serv{\Lambda}_{\nu} \sim \nu}\big[\mfunc_t(\serv{\Lambda}_\nu)],
\end{align}
and $(\serv{Z}_{t})_{t \in \N}$ are zero mean jointly Gaussian random variables, independent of $(\serv{X}_\star;\serv{A})$, whose covariance matrix is given by the recursion:
\begin{eqnarray}
\E[\serv{Z}_{s}\serv{Z}_{t}]  &=& \E[\serv{X}_\star {\serv{F}}_{s}] \E[\serv{X}_\star {\serv{F}}_{t}] \cdot \Cov_{\substack{}{\serv{\Lambda}_\nu} \sim \nu}[{\Psi}_s(\serv{\Lambda}_\nu),{\Psi}_t(\serv{\Lambda}_\nu)] \\
&&\nonumber+ (\E[{\serv{F}}_{s} {\serv{F}}_{t}] - \E[\serv{X}_\star {\serv{F}}_{s}]\E[\serv{X}_\star {\serv{F}}_{t}] )  \cdot \Cov_{\substack{}{\serv{\Lambda}} \sim \mu}[{\Psi}_s(\serv{\Lambda}),{\Psi}_t(\serv{\Lambda})].
\end{eqnarray}
\end{subequations}
In the above display, $\serv{F}_s \explain{def}{=} \fnonlin_s(\serv{X}_1, \dotsc, \serv{X}_{s-1}; \serv{A})$ and $\serv{F}_t \explain{def}{=} \fnonlin_t(\serv{X}_1, \dotsc, \serv{X}_{t-1}; \serv{A})$.
\item[\textit{Requirements on Matrix Denoisers.}] The matrix denoisers $(\mfunc_t)_{t \in \N}$ used in the OAMP algorithm should be continuous functions which do not change with $\dim$, and are required to satisfy the \emph{trace-free constraint}:
\begin{align} \label{eq:trace-free} \E[\mfunc_t(\serv{\Lambda})] = 0, \quad \serv{\Lambda} \sim \mu.
\end{align}
 \item [\textit{Requirements on Iterate Denoisers and Post-processing Functions.}] For each $t \in \N$, the iterate denoiser $\fnonlin_t: \R^{t-1} \times \R^{\auxdim} \mapsto \R$ and the post-processing function $\psi_t:\mathbb{R}^{t} \times \R^{\auxdim} \mapsto\mathbb{R}$ used in the OAMP algorithm should be continuously differentiable and Lipschitz functions which do not change with $\dim$. In addition, the iterate denoisers $(\fnonlin_t)_{t \in \N}$ are required to satisfy the \emph{divergence-free} constraint:
    \begin{align} \label{eq:div-free-constraint}
    \E[\partial_{s} \fnonlin_t(\serv{X}_1, \dotsc, \serv{X}_{t-1}; \serv{A})] & = 0 \quad \forall \; s \; \in \; [t-1], \; t \; \in \; \N,
    \end{align}
where $\partial_{s} \fnonlin_t$ denotes the partial derivative of $ \fnonlin_t(x_1, \dotsc, x_s, \dotsc, x_{t-1}; \aux)$ with respect to $x_s$.    
\end{description}
\end{definition}
Our first main result is the following theorem, which provides a characterization of the dynamics of an OAMP algorithm, in the high-dimensional limit, in terms of the associated state evolution random variables.
\begin{theorem}[State evolution of OAMP] \label{thm:SE} Consider a general OAMP algorithm of the form \eqref{Eqn:OAMP} {that satisfies the requirements stated in \defref{Def:OAMP_main}}, and let $\{\serv{X}_{\star}, (\serv{X}_t)_{t \in \N}, \serv{A}\}$ be the associated state evolution random variables. Then for any $t \in \N$,
\begin{align} \label{eq:oamp-asymptotics}
\left(\vx_\star, \vx_1, \vx_2, \dotsc, \vx_t; \bm{a}\right) \wc (\serv{X}_\star, \serv{X}_{1}, \dotsc, \serv{X}_{t}; \serv{A}).
\end{align}
\end{theorem}
We present a proof sketch of this result highlighting the key ideas in Section \ref{Sec:SE-heuristics}, and defer the complete proof to Appendix \ref{Sec:proof_SE}.

\begin{remark}[Designing divergence-free denoisers]One can construct iterate denoisers that satisfy the divergence-free constraint \eqref{eq:div-free-constraint} by designing them adaptively: once divergence-free iterate denoisers $\fnonlin_{1}, \dotsc, \fnonlin_t$ have been specified, any candidate denoiser $\tilde{f}_{t+1}:\R^t \times \R^{\auxdim} \mapsto \R$ can be corrected to obtain a divergence-free denoiser: $$\fnonlin_{t+1}(x;\aux) \explain{def}{=} \tilde{\fnonlin}_{t+1}(x;\aux) - \sum_{i=1}^t \E[\partial_i\tilde{\fnonlin}_{t+1}(\serv{X}_{1}, \dotsc, \serv{X}_t; \serv{A})] \cdot x_i \quad \forall \; x \in \R^t, \aux \in \R^k,$$ which satisfies the divergence-free constraint \eqref{eq:div-free-constraint} by construction. Here, $(\serv{X}_{\star}, \serv{X}_1, \dotsc, \serv{X}_t; \serv{A})$ are the state evolution random variables associated with the first $t$ iterations of the algorithm.
\end{remark}
\begin{remark}[Connections to AMP Algorithms for Compressed Sensing]\label{Rem:compare_CS}
The OAMP algorithm introduced in \defref{Def:OAMP_main} can be viewed as a natural analog of the orthogonal AMP (OAMP) \citep{ma2017orthogonal} or vector AMP (VAMP) algorithm \citep{rangan2019vector} developed for compressed sensing (or regularized linear regression). The key feature of these algorithms, which is shared by the algorithm in Definition 1, is the use of trace-free matrix denoisers (cf. \eqref{eq:trace-free}) and divergence-free iterate denoisers (cf. \eqref{eq:div-free-constraint}). These features significantly simplify the algorithm's state evolution. However, unlike the OAMP algorithm introduced in this work, the compressed sensing algorithms compute a matrix-vector multiplication involving a rotationally invariant matrix at each iteration. In contrast, only the observed matrix $\mY$ is available in the spiked matrix model, not the rotationally invariant noise matrix $\mW$. Although $\mY$ is a rank-1 perturbation of $\mW$, $\Psi_t(\mY)$, the matrix used by the OAMP algorithm at iteration $t$, cannot be expressed as a simple perturbation of $\mW$ for a general matrix denoiser $\mfunc_t$. This complicates the analysis of these algorithms. We refer the reader to \citet[Remark 3.3]{fan2022approximate} for related discussions. \looseness=-1\end{remark}

\subsection{The Optimal OAMP Algorithm}\label{sec:OAMP_opt}
\thref{thm:SE} provides a characterization of the asymptotic mean squared error (MSE) of the estimator $\hat{\vx}_t$ computed by a general OAMP algorithm \eqref{Eqn:OAMP} in terms of the associated state evolution random variables:
\begin{align*}
\plim_{\dim \rightarrow \infty} \frac{\|\vx_{\star} - \hat{\vx}_t\|^2}{\dim} & = \E |\serv{X}_{\star} - \psi_t(\serv{X}_1, \dotsc, \serv{X}_t ; \serv{A}) |^2.
\end{align*}
The limiting value of the MSE depends implicitly on the functions $(\Psi_t, \fnonlin_t, \psi_t)_{t\in\N}$ used in the OAMP algorithm since these functions determine the joint distribution of $(\serv{X}_{\star}, (\serv{X}_t)_{t \in \N}; \serv{A})$. As our second contribution, we use the above characterization to derive the optimal choice for these functions which minimizes the MSE. We call the resulting algorithm the \emph{optimal OAMP algorithm}, and introduce it below.  
\paragraph*{The Optimal OAMP Algorithm} We first introduce the optimal OAMP algorithm in some simple corner cases, and then consider the typical case. 
\begin{description}[font=\normalfont\emph,leftmargin=0cm,itemsep=0ex]
\item [\emph{Corner Cases.}] If $\mmse(0) = \E \Var[\serv{X}_{\star} | \serv{A}] = 0$, then it is possible to reconstruct the signal perfectly from the side information alone. Hence, the optimal OAMP algorithm outputs the estimator:
    \begin{subequations} \label{eq:optimal-OAMP}
    \begin{align}
    \hat{\vx}_t & = \bdnsr(\va|0) \quad \forall \; t \in \N,
    \end{align}
    where $\bdnsr(\cdot|0)$ is the MMSE estimator for the Gaussian channel $(\serv{X}_{\star} ; \serv{A})$ (which operates at SNR $\omega = 0$). In this case, the optimal OAMP algorithm achieves zero asymptotic mean squared error:
    \begin{align*}
    \plim_{\dim \rightarrow \infty} \frac{\|\vx_{\star} - \hat{\vx}_t\|^2}{\dim} & \explain{Thm. \ref{thm:SE}}{=} \E |\serv{X}_{\star} - \bdnsr(\serv{A}|0)|^2 \explain{\eqref{eq:mmse-scalar}}{=}  \E |\serv{X}_{\star} - \E[\serv{X}_{\star} | \serv{A}]|^2 = \E \Var[\serv{X}_{\star} | \serv{A}] = 0 \quad \forall \; t\; \in \; \N.
    \end{align*}
 In the other extreme, if $\mmse(0) = \E \Var[\serv{X}_{\star} | \serv{A}] =  1$ (recall $\E[\serv{X}_{\star}^2] = 1$ from \assumpref{assump:signal-noise}), the optimal OAMP algorithm returns the trivial estimator:
    \begin{align}
    \hat{\vx}_t & = \bm{0} \quad \forall \; t \in \N.
    \end{align}
    In this case, the asymptotic MSE of the optimal OAMP algorithm is:
    \begin{align*}
    \plim_{\dim \rightarrow \infty} \frac{\|\vx_{\star} - \hat{\vx}_t\|^2}{\dim} &= \plim_{\dim \rightarrow \infty} \frac{\|\vx_{\star}\|^2}{\dim} \explain{}{=} \E \serv{X}_{\star}^2 = 1  \quad \text{ by \assumpref{assump:signal-noise}}.
    \end{align*}
    \item[\emph{Typical Case.}] In the typical situation when $\mmse(0) = \E \Var[\serv{X}_{\star} | \serv{A}] \in (0,1)$, the optimal OAMP algorithm generates a sequence of iterates $\vx_1, \vx_2, \dotsc$ using the update rule:
    \begin{align}
    {\vx}_{t} & = \frac{1}{\sqrt{{\omega}_{t}}}\left( 1+  \frac{1}{{\rho}_{t}}  \right) \cdot \omd(\mY; {\rho}_{t}) \cdot \dfbdnsr({\vx}_{t-1}; \va | {\omega}_{t-1}).\label{Eqn:psi_bar_normalization}
    \end{align}
    The estimator returned by the optimal OAMP algorithm at iteration $t$ is:
    \begin{align}
    \hat{\vx}_t \explain{def}{=} \bdnsr(\vx_t; \va | \omega_t).
    \end{align}
    In the above equations,
    \begin{itemize} 
    \item The matrix denoiser $\omd$ used by the optimal OAMP algorithm is given by:
    \begin{align} \label{eq:optimal-OAMP-denoiser}
     \omd(\lambda; \rho) & = 1 - \left( \E \left[ \frac{ \phi(\serv{\Lambda})}{ \phi(\serv{\Lambda}) + \rho} \right] \right)^{-1} \cdot \frac{\phi(\lambda)}{\phi(\lambda)+\rho} \quad \forall \; \lambda \; \in \; \R, \; \rho \; \in \;  (0,\infty),
    \end{align}
    where $\serv{\Lambda} \sim \mu$ and the function $\phi: \R \mapsto \R$ was introduced in \eqref{eq:phi}.
    \item ${\omega}_{t}$ and ${\rho}_{t}$ are computed using the recursion:
    \begin{align}  \label{eq:rho-omega}
    {\rho}_{t} = \frac{1}{\dmmse({\omega}_{t-1})} - 1, \quad {\omega}_{t}  = 1 -  \left( \E \left[ \frac{\phi(\serv{\Lambda})}{\phi(\serv{\Lambda}) + {\rho}_{t}} \right] \right)^{-1} \cdot \E \left[\frac{1}{\phi(\serv{\Lambda})+{\rho}_{t}} \right]
    \end{align}
    initialized with $\omega_0 \explain{def}{=} 0$.
    \item $\dmmse(\omega)$ denotes the DMMSE for a scalar Gaussian channel with SNR $\omega$, and $\bdnsr(\cdot | \omega), \dfbdnsr(\cdot | \omega)$ denote the MMSE and DMMSE estimators (\defref{def:gauss-channel}).
    \end{itemize}   
    \end{subequations}    
\end{description}

\begin{remark}
The additional scaling factor $1/\sqrt{\omega_t}\cdot(1+1/\rho_t)$ in \eqref{Eqn:psi_bar_normalization} is introduced to normalize the iterates so that $\|\vx_{t}\|^2/\dim \pc 1$.
\end{remark}
The following result characterizes the asymptotic MSE of the optimal OAMP algorithm in this case. Its proof can be found in Appendix \ref{App:proof_OAMP_opt_SE}.

\begin{proposition}\label{prop:optimal-OAMP-SE} Assume that $\mmse(0) = \E\Var[\serv{X}_{\star} | \serv{A}] \in (0,1)$. Let $(\serv{X}_{\star}, (\serv{X}_t)_{t \in \N}; \serv{A})$ denote the state evolution random variables associated with the optimal OAMP algorithm \eqref{eq:optimal-OAMP}. Then,
\begin{enumerate}
\item For each $t \in \N$, $\omega_t \in [0,1)$ and $\rho_t \in (0,\infty)$. Moreover, the state evolution random variables $(\serv{X}_{\star}, \serv{X}_t ; \serv{A})$ form a scalar Gaussian channel with SNR $\omega_t$ and, the asymptotic MSE of the estimator $\hat{\vx}_t$ returned by the optimal OAMP algorithm in \eqref{eq:optimal-OAMP} is given by:
\begin{align*}
\plim_{\dim \rightarrow \infty} \frac{\|\hat{\vx}_t - \vx_{\star} \|^2}{\dim} & = \mmse(\omega_t).
\end{align*}
\item The sequences $(\omega_t)_{t \in \N}$ and $(\rho_t)_{t\in \N}$ from \eqref{eq:rho-omega} are non-decreasing and converge to limit points $\omega_{\ast} \in [0,1)$ and $\rho_{\ast} \in (0,\infty)$ as $t \rightarrow \infty$. The limit points $(\omega_{\ast}, \rho_{\ast})$ solve the following fixed point equation in $(\omega,\rho)$: 
\begin{align} \label{eq:OAMP-fp}
\omega = 1- \left( \E \left[ \frac{\phi(\serv{\Lambda})}{\phi(\serv{\Lambda}) + {\rho}} \right] \right)^{-1} \cdot {\mathbb{E}\left[\frac{1}{\rho+\phi(\sfLambda)}\right]}, \quad \rho=\frac{1}{\dmmse(\omega)}-1, \quad \serv{\Lambda} \sim \mu.
\end{align}
Consequently as $t \rightarrow \infty$, the asymptotic MSE of $\hat{\vx}_t$ converges to: 
\begin{align*}
\lim_{t \rightarrow \infty} \plim_{\dim \rightarrow \infty} \frac{\|\hat{\vx}_t - \vx_{\star} \|^2}{\dim} & = \mmse(\omega_{\ast}).
\end{align*}
\end{enumerate}
\end{proposition}

\begin{remark} {In the absence of side information, if the signal is drawn from a zero-mean prior, $\mmse(0) = 1$, and the optimal OAMP algorithm in \eqref{eq:optimal-OAMP} returns the trivial estimator $\hat{\vx}_t = \vzero$. Our optimality result (introduced as \thref{thm:optimality} in the following section) shows that no iterative algorithm that runs for a constant ($\dim$-independent) number of iterations can achieve a better performance. An interesting direction for future work is to analyze OAMP algorithms with spectral initialization \citep{venkataramanan2022estimation,mondelli2021pca,zhong2021approximate} or randomly initialized iterative algorithms that run for $T \gtrsim \ln(\dim)$ iterations \citep{rush2018finite,li2022non,li2023approximate} in this situation.}
\end{remark}

\subsection{An Optimality Result}\label{sec:optimality} Next, to discuss the optimality properties of the OAMP algorithm introduced in \eqref{eq:optimal-OAMP}, we introduce the following broad class of iterative algorithms. 
\begin{definition}[Iterative Algorithms]\label{def:itr}An iterative algorithm is any procedure which generates iterates $\vr_1, \vr_2, \dotsc$ in $\R^{\dim}$ according to an update rule of the form:
\begin{subequations} \label{eq:GFOM}
\begin{align} \label{eq:GFOM-algo}
\iter{\vr}{t} & = \mfunc_t(\mY) \cdot \fnonlin_t(\iter{\vr}{1}, \dotsc, \iter{\vr}{t-1}; \va) + \gnonlin_t(\iter{\vr}{1}, \dotsc, \iter{\vr}{t-1}; \va) \quad \forall \; t\; \in \; \N.
\end{align}
At iteration $t$, an estimate $\hat{\vr}_t$ of $\vx_\star$  is obtained by applying a post-processing function $\psi_t$ to the iterates $\vr_1, \dotsc, \vr_t$ and the side-information $\va$:
\begin{align}
\hat{\vr}_t=\psi_t\left(\vr_{1}, \dotsc, \vr_t;\va \right).
\end{align}
For each $t\in \N$, the matrix denoiser $\mfunc_t: \R \mapsto \R$ is required to be continuous and the functions  $\fnonlin_t: \R^{t-1} \times \R^{\auxdim} \mapsto \R$, $\gnonlin_t: \R^{t-1} \times \R^{\auxdim} \mapsto \R$, and $\psi_t: \R^{t} \times \R^{\auxdim} \mapsto \R$ are required to be continuously differentiable and Lipschitz. Furthermore, $(\mfunc_{t})_{t \in \N}$, $(\fnonlin_t)_{t\in \N}$, $(\gnonlin_t)_{t \in \N}$, and $(\psi_t)_{t \in \N}$ do not change with the dimension $\dim$. 
\end{subequations}
\end{definition}
The definition above extends the notion of general first-order methods (GFOMs) introduced by \citet{celentano2020estimation} (for the i.i.d. Gaussian noise model) to include a matrix denoising step at each iteration. The class of iterative algorithms defined above includes many commonly used estimators, such as those computed using gradient descent or power method, and their proximal and projected generalizations. The following theorem (proved in Section \ref{sec:optimality-proof}) shows that the optimal OAMP algorithm achieves the minimum possible estimation error among all algorithms in this class under a given iteration budget.

\begin{theorem}\label{thm:optimality} Let $\hat{\vr}_t$ be the estimator returned by any iterative algorithm of the form \eqref{eq:GFOM} after $t \in \N$ iterations. Let $\hat{\vx}_t$ be the estimator returned by the optimal OAMP algorithm in \eqref{eq:optimal-OAMP} after $t$ iterations. Then,
\begin{align*}
 \pliminf_{\dim \rightarrow \infty} \frac{\|\hat{\vr}_t - \vx_{\star} \|^2}{\dim} & \geq \plim_{\dim \rightarrow \infty} \frac{\|\hat{\vx}_t - \vx_{\star} \|^2}{\dim}.
\end{align*}
\end{theorem}

\subsection{Information-Theoretic v.s. Computational Limits}
\label{sec:IT-optimality}
A natural question is whether the optimal OAMP algorithm introduced in \eqref{eq:optimal-OAMP} achieves the smallest asymptotic estimation error among \emph{all estimators}, not just estimators computable using efficient iterative algorithms. To address this question, we begin by recalling the conjecture of \citet{barbier2023fundamental} regarding the fundamental information-theoretic limits of this problem. 

\paragraph*{Replica conjecture for the Bayes risk}  Under the assumption that the entries of $(\vx_{\star}, \va)$ are drawn i.i.d. from a prior $\pi$, the information-theoretically optimal estimator is the Bayes estimator $\E[\vx_{\star}| \mY, \va]$. \citet{barbier2023fundamental} provide a conjecture for the asymptotic MSE of this estimator (also called the Bayes risk) assuming the noise matrix drawn from the \emph{trace ensemble} \citep{barbier2023fundamental,pastur2011eigenvalue} with density:
\begin{align} \label{eq:trace-ensemble-recall}
p(\mW) & \propto \exp\left( - \frac{\dim}{2}\sum_{i=1}^\dim V(\lambda_i(\mW)) \right) \text{ where $\lambda_{1:\dim}(\mW)$ denote the eigenvalues of $\mW$,}
\end{align}
and the \emph{potential function} $V:\R \mapsto \R$ is a functional parameter for the noise model. The authors develop a recipe to derive conjectured formulas for the asymptotic Bayes risk based on the non-rigorous replica method for polynomial potentials $V$, providing explicit formulas for quartic and sestic polynomials. As the degree of $V$ increases, the complexity of the derivation and resulting formulas also increases, with no general formula available. We state the replica conjecture for the Bayes risk when the potential $V: \R \mapsto \R$ is a quartic polynomial:
\begin{align}\label{Eqn:quartic-main}
V(\lambda)=\frac{\gamma \lambda^2}{2}+\frac{\kappa \lambda^4}{4} \quad \forall \; \lambda \; \in \; \R \  \text{ where } \  \kappa=\kappa(\gamma)= \frac{8-9\gamma +\sqrt{64-144\gamma+108\gamma^2-27\gamma^3}}{27}, 
\end{align}
and $\gamma\in[0,1]$ is a parameter for the noise model. The choice of $\kappa$ in \eqref{Eqn:quartic-main} ensures that the limiting spectral measure ($\mu$) of $\mW$ drawn from the trace ensemble has unit variance \citep{barbier2023fundamental}.
\begin{conjecture}[Replica Conjecture for Bayes risk \cite{barbier2023fundamental}]\label{conj:RS} Suppose that the entries of $(\vx_{\star}, \va)$ are drawn i.i.d. from a prior $\pi$ with $\mmse(0) \in (0,1)$ and the noise matrix $\mW$ is drawn from the trace ensemble \eqref{eq:trace-ensemble-recall} with the quartic potential function $V$ from \eqref{Eqn:quartic-main}. Then,
\begin{align*}
\plim_{\dim \rightarrow \infty} \frac{\| \E[\vx_{\star} | \mY, \va] - \vx_{\star} \|^2}{\dim} & = \mmse(\snrit),
\end{align*}
for some $\snrit \in (0,1)$ and $\rhoit \in (0,\infty)$ which solve the following system of fixed point equations\footnote{When the fixed point equations in \eqref{Eqn:SE_FP_quartic-main} have multiple solutions, the correct fixed point is the one that minimizes a certain free energy function calculated in \citep{barbier2023fundamental}.} (in $\omega,\rho$):
\begin{subequations}\label{Eqn:SE_FP_quartic-main}
\begin{align}
m &= 1-\mmse(\omega),\quad \mathbb{E}[\mathsf{H}] = 1-m, \quad \chi = \mathbb{E}[\sfLambda \mathsf{Q H}], \\ \hat{m}\equiv \frac{\omega}{1-\omega} &= \kappa\theta^2 \left(\frac{m}{1-m}\E[\sfLambda^2\mathsf{H}]+\frac{\chi}{1-m}\E[\sfLambda \mathsf{H}]+\E[\sfLambda^2 \mathsf{QH}] \right)+ \gamma\theta^2 m 
\end{align}
where $m\in\mathbb{R}$, $\chi\in\mathbb{R}$ are two intermediate variables, $\serv{\Lambda} \sim \mu$, and the random variables $\mathsf{Q} = \mathsf{Q}(\sfLambda,m,\chi)$, $\mathsf{H}=\mathsf{H}(\sfLambda,\rho)$ are defined by
\begin{align}
\mathsf{Q} &\bydef \kappa\theta^2 m\sfLambda^2 + \kappa\theta^2\chi\sfLambda
- \frac{\kappa\theta^2}{1-m} \E\left[ m\sfLambda^2\mathsf{H} + \chi\sfLambda \mathsf{H}\right] + \frac{m}{1-m} ,
\\
\mathsf{H} &\bydef \left( \rho + \theta^2a^2(\gamma+2a^2\kappa)^2 + 1 -\theta\left(\gamma\serv{\Lambda}-\theta\kappa\serv{\Lambda}^2+\kappa\serv{\Lambda}^3\right) \right)^{-1}.
\end{align}
\end{subequations}

In the above equations, $(\gamma,\kappa)$ are the parameters for the quartic potential function \eqref{Eqn:quartic-main}, $\theta$ is the SNR for the spiked matrix model \eqref{eq:model}, and $a^2 \explain{def}{=} \big(\sqrt{\gamma^2+12\kappa}-\gamma\big)/(6\kappa)$.
\end{conjecture}
On the other hand, from \propref{prop:optimal-OAMP-SE} that the asymptotic MSE of the estimator $\hat{\vx}_t$ returned by the optimal OAMP algorithm in \eqref{eq:optimal-OAMP} satisfies:
\begin{align*}
\lim_{t \rightarrow \infty} \plim_{\dim \rightarrow \infty} \frac{\|\hat{\vx}_t - \vx_{\star} \|^2}{\dim} & = \mmse(\omega_{\ast}),
\end{align*}
where $(\omega_{\ast}, \rho_{\ast})$ denote the solution of the state evolution fixed point equations (in $\omega \in (0,1), \rho \in (0,\infty)$):
\begin{align} \label{eq:se-fp}
\omega & = \mathscr{F}_1(\rho), \; \rho = \mathscr{F}_2(\omega) \;   \text{ with }  \;   \mathscr{F}_1(\rho)\bydef 1 - \frac{\E_{\substack{}{\serv{\Lambda} \sim \mu}} \left[\frac{1}{\phi(\serv{\Lambda})+{\rho}} \right]}{ \E_{\substack{}{\serv{\Lambda} \sim \mu}} \left[ \frac{\phi(\serv{\Lambda})}{\phi(\serv{\Lambda}) + {\rho}} \right] } , \; \mathscr{F}_2(\omega)\bydef\frac{1}{\dmmse({\omega})} - 1,
\end{align}
found by the recursion $ \rho_{t} = \mathscr{F}_2(\omega_{t-1}), \; \omega_{t} = \mathscr{F}_1(\rho_t)$. The following proposition (proved in Appendix \ref{App:replica}) shows that for the quartic potential, the \emph{replica fixed point equations} \eqref{Eqn:SE_FP_quartic-main} and the \emph{state evolution fixed point equations} \eqref{eq:se-fp} are equivalent. \looseness=-1
\begin{proposition}\label{Lem:SE2_quartic} Assume that the prior $\pi$ satisfies $\mmse(0) \in (0,1)$. Any solution $(\omega,\rho)$ to the state evolution fixed point equations \eqref{eq:se-fp} with $\omega \in (0,1), \rho \in (0,\infty)$ is also a solution to the replica fixed point equations \eqref{Eqn:SE_FP_quartic-main}, and vice-versa. 
\end{proposition}
We remark that an analogous result holds for the sestic (degree $6$) potential. Since the replica equations for the sestic ensemble in \citet{barbier2023fundamental} are even more involved, we do not provide the details here. We anticipate that the fixed point equations in \eqref{eq:se-fp} are a unified and concise reformulation of the replica fixed point equations and characterize the asymptotic Bayes risk for general potential functions $V$. This reformulation of the replica conjecture for the asymptotic Bayes risk may be more amenable to rigorous proof than the significantly more complicated replica formulas derived using the approach of \citet{barbier2023fundamental}.

\paragraph*{Information-Theoretic Optimality and Sub-Optimality of OAMP}  We can simplify the fixed point equations in \eqref{eq:se-fp} (which are equivalent to the replica fixed point equations \eqref{Eqn:SE_FP_quartic-main}) by eliminating $\rho$ to obtain a single equation $\omega = \mathscr{F}_1(\mathscr{F}_2(\omega))$. In some cases, this equation has a unique solution, implying the optimal OAMP algorithm matches the {conjectured} asymptotic MSE of the Bayes estimator and is {expected to be} information-theoretically optimal ({see \fref{Fig:IT_Alg_gap}(a)). In other cases, the equation may have multiple solutions, as illustrated in \fref{Fig:IT_Alg_gap}(b), where it has two stable fixed points. Here, the optimal OAMP algorithm converges to the inferior fixed point $\omega_{\ast}$, while the Bayes risk corresponds to the superior fixed point $\snrit$. Since $\omega_{\ast} < \snrit$, the optimal OAMP or any iterative algorithm (in the form of \eqref{eq:GFOM-algo}) with a constant ($\dim$-independent) number of iterations fails to achieve the Bayes-optimal MSE. Such scenarios also occur in i.i.d. Gaussian noise models \citep{celentano2020estimation,montanari2022statistically,montanari2022equivalence}, and it is conjectured that no polynomial time algorithm can achieve the information-theoretically optimal MSE in these cases. 

 \begin{figure}[htbp]
\begin{center}
\subfloat[Three-points prior with $\epsilon_1=\frac{1}{3}$ and $\epsilon_2=\frac{1}{5}$.
\label{Fig:OAMP_1FP}]{
\includegraphics[width=0.47\linewidth]{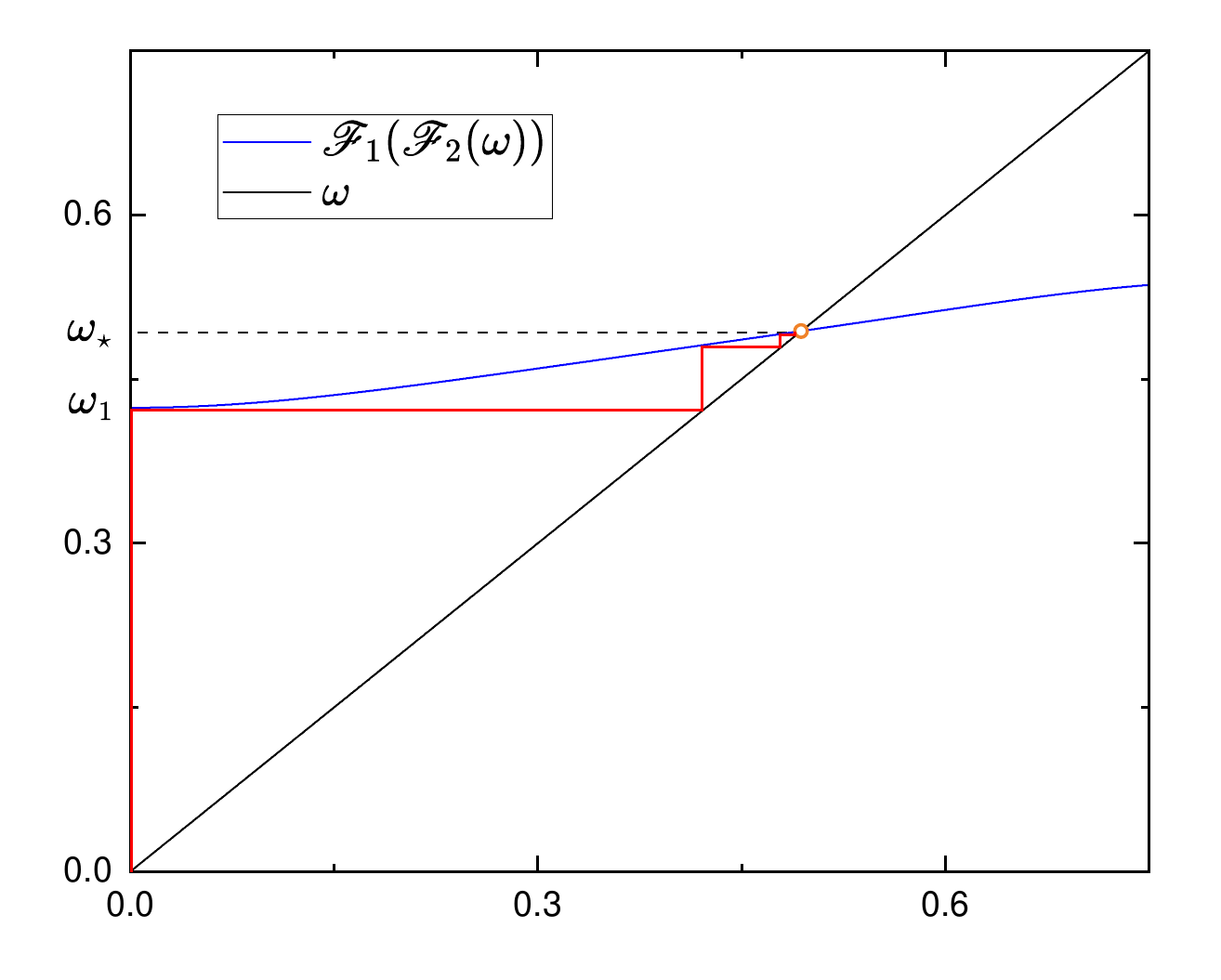}
                }
\subfloat[Three-points prior with $\epsilon_1=\frac{1}{6}$ and $\epsilon_2=\frac{1}{10}$.\label{Fig:OAMP_Multiple_FPs}]{
\includegraphics[width=0.47\linewidth]{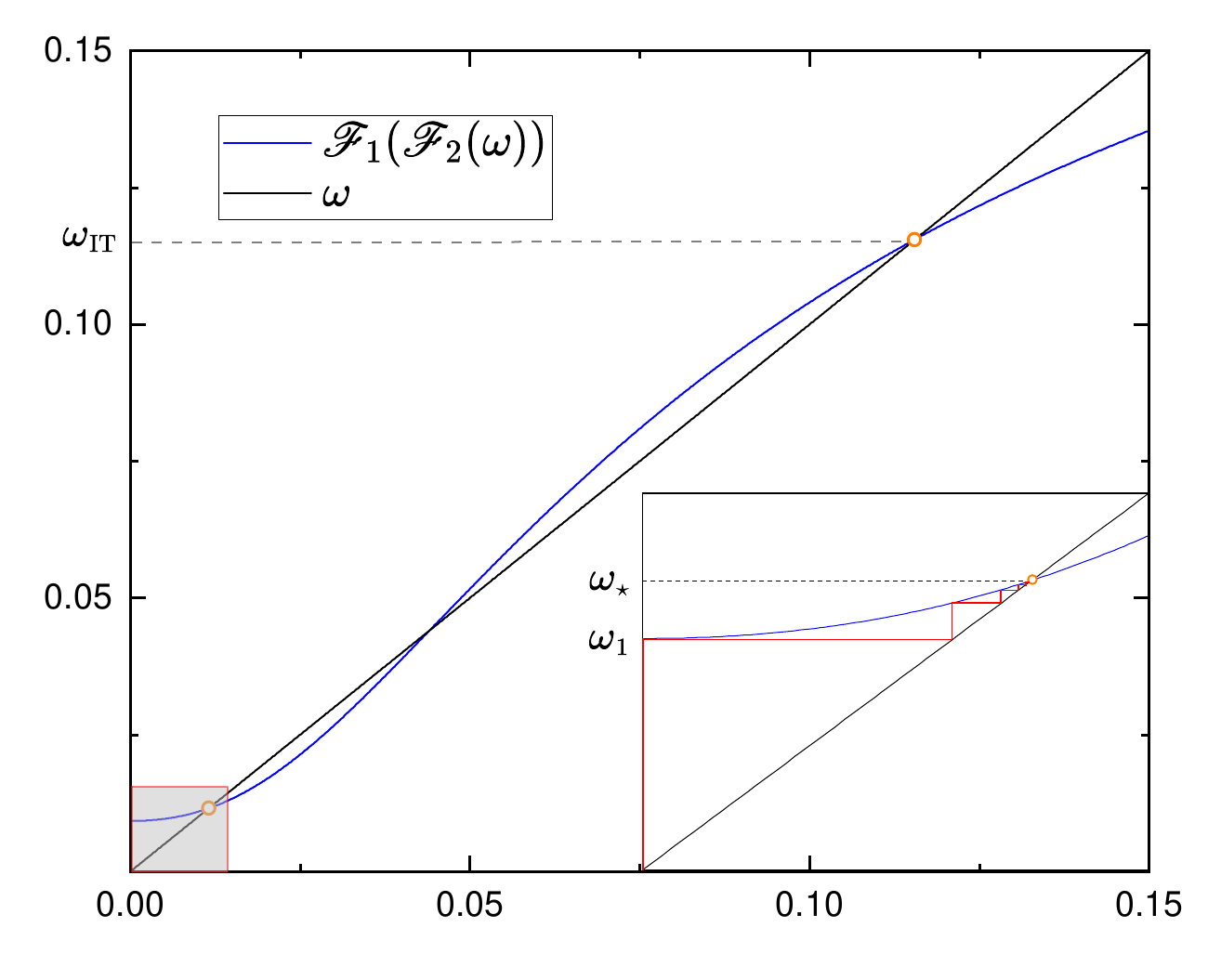} }
\end{center}
\caption{Plot of the fixed point equation $\omega = \mathscr{F}_1(\mathscr{F}_2(\omega))$ for SNR $\theta = 0.46$, quartic noise model \eqref{Eqn:quartic-main} with $\gamma=0$, and 
a three-point signal prior $\serv{X}_{\star} \sim \frac{\epsilon_1^2}{2} \delta_{\frac{1}{\epsilon_1}} 
+ \frac{\epsilon_2^2}{2} \delta_{\frac{1}{\epsilon_2}} 
+ (1-\frac{\epsilon_1^2}{2} - \frac{\epsilon_2^2}{2}  ) \delta_{0}$.}
\label{Fig:IT_Alg_gap}
\end{figure} 
\section{Proof Ideas}\label{Sec:proof_idea}
We now present some important ideas used to obtain our main results. \looseness=-1

\subsection{Heuristic Derivation of State Evolution (\thref{thm:SE})}\label{Sec:SE-heuristics}
We begin by an intuitive derivation of the state evolution result (Theorem \ref{thm:SE}) for OAMP algorithms, highlighting the key ideas involved in the proof. A formal proof of \thref{thm:SE} is presented in Appendix \ref{Sec:proof_SE}. Consider a general OAMP algorithm (\defref{Def:OAMP_main}):
\begin{equation}\label{Eqn:SE_derivation_1}
\bm{x}_t = {\Psi}_t(\bm{Y}) \cdot {f}_t(\bm{x}_1,\ldots,\bm{x}_{t-1};\bm{a}) \quad \forall \; t \; \in \; \N,
\end{equation}
where $({\Psi}_t)_{t\in \N}$ are trace-free (cf. \eqref{eq:trace-free}) and $(f_t)_{t\in \N}$ are divergence-free (cf. \eqref{eq:div-free-constraint}). For ease of exposition, we present the key ideas assuming that the matrix denoisers \emph{$({\Psi}_t)_{t \in \N}$ are polynomials} (the general case follows by a polynomial approximation argument). Let us decompose ${f}_t(\bm{x}_1,\ldots,\bm{x}_{t-1};\bm{a})$ into a component along the direction of $\bm{x}_\star$ and a component perpendicular to it:
\begin{align*}
{f}_t(\bm{x}_{1}, \dotsc, \vx_{t-1};\bm{a})\explain{}{=}\alpha_t \bm{x}_\star + {\bm{f}}_t^\perp \;  \text{ where }  \; \alpha_t \explain{def}{=} \frac{\ip{\bm{x}_\star}{{f}_t(\bm{x}_{<t};\bm{a})}}{\|\bm{x}_\star\|^2}, \; \; {\bm{f}}_t^\perp \explain{def}{=} {f}_t(\bm{x}_{<t};\bm{a}) - \alpha_t \vx_{\star}.
\end{align*}
Using this decomposition, we can write the new iterate $\bm{x}_t$ in \eqref{Eqn:SE_derivation_1} as
\begin{equation}\label{Eqn:term_I_II}
\bm{x}_t = \alpha_t\cdot \Psi_t(\bm{Y})  \bm{x}_\star +{\Psi}_t(\bm{Y}){\bm{f}}_t^\perp.
\end{equation}
As pointed out in Remark \ref{Rem:compare_CS}, the main difficulty in analyzing the algorithm above is that the matrix $\Psi_t(\mY)$ is not rotationally invariant. To address this, we express the update equation \eqref{Eqn:term_I_II} in terms of the rotationally invariant noise matrix $\mW$ by expanding the polynomial ${\Psi}_t(\bm{Y})$ in terms of $\mW$ using the fact that $\mY = (\theta/\dim) \cdot \bm{x}_\star\bm{x}_\star^\UT+\bm{W}$. Although this expansion can initially seem complicated, the key insight that makes it tractable is to show that the following approximations hold:
\begin{equation}\label{Eqn:SE_heuristic_Y2W}
{\Psi}_t(\bm{Y})\bm{x}_\star \explain{$\dim \rightarrow \infty$}{\simeq} \tilde{\Psi}_t(\bm{W})\bm{x}_\star,\quad {\Psi}_t(\bm{Y}){\bm{f}}_t^\perp \explain{$\dim \rightarrow \infty$}{\simeq} {\Psi}_t(\bm{W}){\bm{f}}_t^\perp,
\end{equation}
where $\tilde{\Psi}_t:\mathbb{R}\mapsto\mathbb{R}$ is polynomial obtained by appropriately transforming $\Psi_t$; we refer the reader to Appendix \ref{Sec:proof_SE} for additional details regarding these approximations. The transformed matrix denoiser $\tilde{\Psi}_t$ is determined by ${\Psi}_t$ via a complicated recursion and is not necessarily trace-free. Substituting \eqref{Eqn:SE_heuristic_Y2W} into \eqref{Eqn:term_I_II}, we get
\BS
\begin{eqnarray} 
\nonumber\bm{x}_t & \;  \explain{$\dim \rightarrow \infty$}{\simeq}  \;  &\alpha_t\cdot \tilde{\Psi}_t(\bm{W})\bm{x}_\star + {\Psi}_t(\bm{W}){\bm{f}}_t^\perp\\
&\overset{}{=} &\alpha_t\cdot \frac{\mathrm{Tr}[\tilde{\Psi}_t(\bm{W})]}{N}\cdot\bm{x}_\star + \alpha_t\cdot\bigg(\tilde{\Psi}_t(\bm{W}) -\frac{\mathrm{Tr}[\tilde{\Psi}_t(\bm{W})]}{N}\cdot \bm{I}_N\bigg)\bm{x}_\star+ {\Psi}_t(\bm{W}){\bm{f}}_t^\perp \nonumber\\
&\explain{def}{=}& \underbrace{\alpha_t\cdot \frac{\mathrm{Tr}[\tilde{\Psi}_t(\bm{W})]}{N}\cdot\bm{x}_\star}_{\text{Signal Component}} \; + \; \underbrace{\bm{z}_t}_{\text{Eff. Noise}}, \label{eq:sig+noise-decomp}
\end{eqnarray}
where the effective noise $\vz_t$ is defined as:
\begin{align} 
\bm{z}_t&\bydef
\alpha_t\cdot \hat{\Psi}_t(\bm{W})\bm{x}_\star+ {\Psi}_t(\bm{W}){\bm{f}}_t^\perp \quad \text{where} \quad \hat{\Psi}_t(\bm{W}) \explain{def}{=} \tilde{\Psi}_t(\bm{W}) -\frac{\mathrm{Tr}[\tilde{\Psi}_t(\bm{W})]}{N}\cdot \bm{I}_N. \label{eq:noise-update-rule}
\end{align}
\ES
We analyze the signal and noise components separately. 
\paragraph*{Signal Component} Notice that the signal component in \eqref{eq:sig+noise-decomp} involves the factor $\Tr[ \tilde{\Psi}_t(\bm{W})]/N$ . Computing the limiting value of this factor is challenging because the transformed matrix denoiser $\tilde{\Psi}_t$ does not have a convenient formula and is determined by $\Psi_t$  via a complicated recursion (see Appendix \ref{Sec:proof_SE} for further details). The key idea which yields a concise formula for this factor is to observe:
\begin{align} \label{eq:SE-keyidea}
\frac{\mathrm{Tr}[\tilde{\Psi}_t(\bm{W})]}{N}& \explain{(a)}{\approx} \frac{\bm{x}_\star^\UT \tilde{\Psi}_t(\bm{W})\bm{x}_\star}{N} \explain{(b)}{\approx} \frac{\bm{x}_\star^\UT {\Psi}_t(\bm{Y})\bm{x}_\star}{N} \explain{(c)}{=} \int_{\R} \Psi_t(\lambda) \; \nu_{\dim}(\diff \lambda) \explain{(d)}{\approx} \E_{\serv{\Lambda}_{\nu} \sim \nu}[\Psi_t(\serv{\Lambda}_{\nu})],
\end{align}
where the approximation in (a) follows from standard concentration results for quadratic forms of rotationally invariant matrices (see Fact \ref{fact:qf} in Appendix \ref{appendix:misc}), (b) follows from the defining property of $\tilde{\Psi}_t(\bm{W})$ in \eqref{Eqn:SE_heuristic_Y2W}, (c) follows by recalling the definition of the spectral measure in the signal direction ($\nu_{\dim}$) from \eqref{eq:signal-spec}, and (d) follows from the weak convergence of the compactly supported measure $\nu_{\dim}$ to $\nu$ (\lemref{lem:RMT}, item (1)).
\paragraph*{Noise Component} The update rule \eqref{eq:noise-update-rule} for the effective noise $\vz_t$ is written in a form for which existing state evolution results apply \citep{ma2017orthogonal,rangan2019vector,takeuchi2019rigorous,fan2022approximate,dudeja2023universality,wang2022universality,dudeja2022spectral}. Indeed, $\hat{\Psi}_t(\bm{W})$ and ${\Psi}_t(\bm{W})$ are trace-free functions of the rotationally-invariant matrix $\bm{W}$; and ${\bm{f}}_t^\perp$ is a divergence-free function of $\{\bm{x}_{i}\}_{i<t}$. By appealing to existing results on the dynamics of AMP algorithms driven by rotationally invariant matrices, we show that the effective noise component $\bm{z}_t$ converges to a centered Gaussian random variable $\mathsf{Z}_t$. Moreover, the covariance of $\mathsf{Z}_t$ and $\mathsf{Z}_s$ can be heuristically calculated based on a convenient property of OAMP algorithms. Specifically, for ${G}_1(\bm{W}), {G}_2 (\bm{W}) \in\big\{\hat{\Psi}_t(\bm{W}),{\Psi}_t(\bm{W}),{\Psi}_s(\bm{W})\big\}$ and $\vv_1, \vv_2 \in\{\bm{x}_\star,{\bm{f}}_s^\perp,{\bm{f}}_t^\perp\}$, we have
\begin{equation}
\frac{\vv_1^\UT{G}_1(\bm{W}){G}_2(\bm{W})\vv_2}{N} \approx \frac{\mathrm{Tr}[{G}_1(\bm{W}){G}_2(\bm{W})]}{N}\cdot\frac{\ip{\vv_1}{\vv_2}}{N}.
\end{equation}
The intuition for the above property is that the random vectors $\vv_1, \vv_2$ behave as if they are independent of the noise matrix $\bm{W}$, and the divergence-free and trace-free requirements imposed on OAMP algorithms are crucial for the validity of the above approximation. Using the above property, we immediately have that
\begin{align} \label{Eqn:SE_heuristic_cov}
\E[\serv{Z}_s \serv{Z}_t] &\approx \frac{\ip{\vz_s}{\vz_t}}{N}=\frac{\ip{\alpha_t\hat{\Psi}_t(\bm{W})\bm{x}_\star+ {\Psi}_t(\bm{W}){\bm{f}}_t^\perp}{\alpha_s\hat{\Psi}_s(\bm{W})\bm{x}_\star+ {\Psi}_s(\bm{W}){\bm{f}}_s^\perp}}{N}\\
&\approx\alpha_s\alpha_t\cdot  \frac{\mathrm{Tr}[\hat{\Psi}_t(\bm{W})\hat{\Psi}_s(\bm{W})]}{N} \cdot \frac{\|\bm{x}_\star\|^2}{N}+\frac{\mathrm{Tr}[{\Psi}_t(\bm{W}){\Psi}_s(\bm{W})]}{N}\cdot \frac{\ip{\bm{f}_t^\perp}{\bm{f}_s^\perp}}{N}, \nonumber 
\end{align}
where the cross terms vanish due to the orthogonality of $\bm{x}_\star$ and $\bm{f}_t^\perp$ (and $\bm{f}_s^\perp$). The limiting value of $\mathrm{Tr}[\hat{\Psi}_t(\bm{W}),\hat{\Psi}_s(\bm{W})]/{\dim}$ can be expressed in terms of $\nu$ (the spectral measure in the signal direction) using the argument from \eqref{eq:SE-keyidea}. Replacing the various normalized inner products by their limiting values eventually leads to the claimed state evolution  \eqref{Eqn:SE_variables_def}. 

\subsection{Derivation of the Optimal OAMP Algorithm}\label{Sec:proof-idea2}
Next, we present an intuitive derivation of the optimal OAMP algorithm introduced in \eqref{eq:optimal-OAMP}. While this derivation does not show that this algorithm attains the smallest estimation error among all iterative algorithms (as claimed in Theorem \ref{thm:optimality}), it provides a simple and natural approach to derive the matrix denoisers and the iterate denoisers used by the algorithm. 
For simplicity, we consider the class of simple memory-free OAMP algorithms in which the iterate denoiser $f_t$ only depends on the most recent iterate (cf.~\eqref{Eqn:OAMP}):
\begin{align}\label{Eqn:OAMP_memory_free}
\vx_t&=\Psi_t(\bm{Y}) \cdot  f_t(\vx_{t-1};\va),\quad \forall \;  t\in\mathbb{N},
\end{align} 
We will use a natural greedy heuristic to design the matrix and iterate denoisers. Specifically, we derive a good choice of $\Psi_t$ and $f_t$, assuming we have already specified the iterate and matrix denoisers for the first $t-1$ iterations. For any candidate $\Psi_t$ and $\fnonlin_t$, the distribution of the state evolution random variable $\serv{X}_t$ associated with the $\vx_t$ is given by (see \defref{Def:OAMP_main}):
\begin{subequations} \label{Eqn:SE_memoryfree_def}
\begin{align} \label{Eqn:AWGN}
(\serv{X}_\star, \serv{A}) \sim \pi, \quad  \serv{X}_{t} = \beta_t \serv{X}_\star + \serv{Z}_{t},
\end{align}
where $\serv{Z}_{t}$ is a centered Gaussian random variable  independent of $(\serv{X}_\star, \serv{A})$, and
\begin{align*} 
\beta_t & = \E[\serv{X}_\star \fnonlin_t(\serv{X}_{t-1}; \serv{A})] \cdot \mathbb{E}_{\serv{\Lambda}_\nu\sim\nu}\big[{\Psi}_t(\serv{\Lambda}_\nu)], \quad \alpha_t = \E[\serv{X}_\star \fnonlin_t(\serv{X}_{t-1}; \serv{A})],\\
\E[\serv{Z}_{t}^2] & =  \alpha_t^2 \cdot  \Var_{\serv{\Lambda}_\nu\sim\nu}[{\Psi}_t(\serv{\Lambda}_\nu)] + \left(\E[\fnonlin_t(\serv{X}_{t-1}; \serv{A})^2] -\alpha_t^2 \right) \cdot    \E_{\serv{\Lambda}\sim\mu}[{\Psi}_t^2(\serv{\Lambda})].
\end{align*}
\end{subequations}
A natural design principle is to choose $\fnonlin_t, \mfunc_t$ to maximize the SNR $\omega_t$ of the Gaussian channel $(\serv{X}_{\star}, \serv{X}_t = \beta_t \serv{X}_{\star} + \serv{Z}_t; \serv{A})$, which is given by the squared cosine similarity between $\mathsf{X}_t, \mathsf{X}_\star$:
\begin{align}\label{Eqn:OAMP_opt_omega_expression}
\omega_t &\explain{def}{=}\frac{\left(\mathbb{E}\left[\serv{X}_\star\serv{X}_t\right]\right)^2}{\mathbb{E}[\serv{X}_\star^2]\cdot\mathbb{E}[\serv{X}_t^2]}= \frac{\beta_t^2}{\beta_t^2 + \E[\serv{Z}_t^2]} = \frac{\left(\mathbb{E}\big[{\Psi}_t(\sfLambda_\nu)]\right)^2}{\mathbb{E}\big[{\Psi}_t^2(\sfLambda_\nu)] +\frac{\delta_t}{1-\delta_t}\cdot \mathbb{E}[{\Psi}_t^2(\sfLambda)] },
\end{align}
where $\delta_t\bydef 1-\left( \E[\serv{X}_\star \fnonlin_t(\serv{X}_{t-1}; \serv{A})] \right)^2/\E[\fnonlin_t(\serv{X}_{t-1}; \serv{A})^2]$. We consider the two maximization problems over $\Psi_t$ and $f_t$ one-by-one.
\paragraph*{Optimal choice of $\Psi_t$} Choosing $\mfunc_t$ to maximize $\omega_t$ is equivalent to solving the following optimization problem:
\begin{equation}\label{eq:orig-max}
\underset{\Psi}{\max}\ \frac{\left(\mathbb{E}\big[{\Psi}(\sfLambda_\nu)]\right)^2}{\mathbb{E}\big[{\Psi}^2(\sfLambda_\nu)] +\rho^{-1}\cdot \mathbb{E}[{\Psi}^2(\sfLambda)]  }  \quad \text{subject to}\quad \mathbb{E}\left[{\Psi}(\sfLambda)\right]=0.
\end{equation}
with the parameter $\rho\in(0,\infty)$ fixed at $\rho = \delta_t^{-1} - 1$ (notice $\delta_t$ does not depend on $\mfunc_t$). We transform the optimization problem in \eqref{eq:orig-max} into a simple quadratic optimization problem:
\BS\label{Eqn:OAMP_opt_F1_opt}
\begin{align}
&\underset{\Psi}{\min}\ 1-\frac{\left(\mathbb{E}\big[{\Psi}(\sfLambda_\nu)]\right)^2}{\mathbb{E}\big[{\Psi}^2(\sfLambda_\nu)] +\rho^{-1}\cdot \mathbb{E}[{\Psi}^2(\sfLambda)]  }  \quad \text{subject to}\quad \mathbb{E}\left[{\Psi}(\sfLambda)\right]=0 \label{Eqn:local_opt_Psi_MSE_a}\\
&\explain{(a)}{=} \underset{\Psi}{\min}\min_{c\in\mathbb{R}}\ \mathbb{E}\left[\left(1-c\Psi(\sfLambda_\nu)\right)^2\right]+\rho^{-1}\cdot \mathbb{E}[c^2{\Psi}^2(\sfLambda)] \quad \text{subject to}\quad \mathbb{E}\left[c{\Psi}(\sfLambda)\right]=0\\
&\explain{(b)}{=}  \underset{\Psi}{\min}\ \mathbb{E}\left[\left(1-\Psi(\sfLambda_\nu)\right)^2\right]+\rho^{-1}\cdot \mathbb{E}[{\Psi}^2(\sfLambda)] \quad \text{subject to}\quad \mathbb{E}\left[{\Psi}(\sfLambda)\right]=0,\label{Eqn:local_opt_Psi_MSE}
\end{align}
\ES
where step (a) can be verified by solving the inner quadratic minimization problem over $c$ explicitly, and step (b) is due to a change of variable $\Psi\mapsto c\Psi$; note that $c\Psi$ also satisfies the trace-free constraint. Note that any minimizer of  \eqref{Eqn:local_opt_Psi_MSE} is also a maximizer of \eqref{eq:orig-max} since the objective function of \eqref{eq:orig-max} is invariant to a rescaling of $\Psi$. To solve this optimization problem, we consider the Lebesgue decomposition $\nu = \nu_{\parallel} + \nu_{\perp}$, where $\nu_{\perp}$ is the singular part and $\nu_{\parallel}$ is the absolutely continuous part. Recall the definition of $\phi$ from \eqref{eq:phi} and let $S$ denote the set:
\begin{align} \label{eq:S-set-def-main}
    S \explain{def}{=} \{\lambda \in \R : \phi(\lambda) \neq 0 \} \quad \text{ where } \quad \phi(\lambda) \explain{def}{=} {(1- \pi \theta \hlb_{\mu}(\lambda))^2 + \pi^2 \theta^2  \mu^2(\lambda)}.
\end{align}
Notice that $\mu(S^c) = 0$. \lemref{lem:RMT} (items (2-3)) shows that $\nu_{\perp}(S) = 0$ and the density of $\nu_{\parallel}$ with respect to the Lebesgue measure is given by $\mu(\cdot)/\phi(\cdot)$ where $\mu(\cdot)$ is the density of $\mu$ (recall \assumpref{assump:signal-noise}). Hence, the Lagrangian of the minimization problem \eqref{Eqn:local_opt_Psi_MSE} is given by:
\begin{align*}
&\E[|\mfunc(\serv{\Lambda}_{\nu})-1|^2] + \frac{1}{\rho} \cdot \E[\mfunc^2(\serv{\Lambda})]  - 2\gamma \E[\mfunc(\serv{\Lambda})] \\
& \explain{}{=} \int_{\R} |\mfunc({\lambda})-1|^2 \; \nu(\diff \lambda) + \int_{\R}  \left( \frac{\mfunc^2({\lambda})}{\rho} - 2 \gamma \mfunc(\lambda) \right) \mu(\lambda) \diff \lambda  \\  
& \explain{}{=}  \int_{S^c} (\mfunc(\lambda) - 1)^2 \; \nu_{\perp}(\diff \lambda) + \int_{S} \left( (\mfunc(\lambda) - 1)^2 \cdot \frac{1}{\phi(\lambda)}  + \frac{\mfunc^2(\lambda)}{\rho} - 2\gamma \mfunc(\lambda) \right)\; \mu(\diff \lambda),
\end{align*}
where $\gamma \in \R$ is the Lagrange multiplier. The Lagrangian is minimized by minimizing the integrand pointwise, leading to the following formula for the minimizer $\mfunc_{\gamma}$:
\begin{align*}
\mfunc_{\gamma}(\lambda) & \explain{def}{=} \frac{\gamma + \frac{1}{\phi(\lambda)} }{\frac{1}{\phi(\lambda)}  + \frac{1}{\rho}} = 1+ \frac{(\rho\gamma - 1) \cdot \phi(\lambda)}{\rho + \phi(\lambda)} \quad \forall \; \lambda \; \in \;  S \text{ and} \quad \mfunc_{\gamma}(\lambda) = 1 \quad \forall \; \lambda \; \in \;  S^c.
\end{align*}
The solution of the constrained minimization problem \eqref{Eqn:local_opt_Psi_MSE} is obtained by choosing $\gamma$ so that the constraint $\E_{\serv{\Lambda} \sim \mu}[\mfunc_{\gamma}(\serv{\Lambda})] = 0$ is fulfilled. Hence, the optimizer of \eqref{Eqn:local_opt_Psi_MSE} and \eqref{eq:orig-max} is:
\begin{align} \label{eq:opt-denoiser-recall-main}
\mfunc(\lambda) & =  1 - \left( \E_{\serv{\Lambda} \sim \mu} \left[ \frac{ \phi(\serv{\Lambda})}{ \phi(\serv{\Lambda}) + \rho} \right] \right)^{-1} \cdot \frac{\phi(\lambda)}{\phi(\lambda)+\rho} \explain{def}{=} \omd(\lambda; \rho) \quad \forall \; \lambda \; \in \; \R.
\end{align}
This is precisely the matrix denoiser used by the optimal OAMP algorithm in \eqref{eq:optimal-OAMP}.
\paragraph*{Optimal choice of $f_t$} Notice from \eqref{Eqn:OAMP_opt_omega_expression} that the iterate denoiser $\fnonlin_t$ which maximizes the SNR $\omega_t$  can be derived by minimizing $\delta_t$ with respect to $\fnonlin_t$, leading to the minimization problem:
\begin{equation}
\underset{f}{\min}\  1-\frac{\left( \E[\serv{X}_\star f(\serv{X}_{t-1};\serv{A})] \right)^2}{ \E[f^2(\serv{X}_{t-1};\serv{A})] } \quad \text{subject to}\quad \mathbb{E}[f'(\serv{X}_{t-1};\serv{A})]=0.
\end{equation}
In the above definition, $\serv{X}_{t-1}=\beta_{t-1}\serv{X}_\star+\sigma_{t-1}\serv{Z}$ is the state evolution random variable for $\vx_{t-1}$, $\serv{Z}\sim\mathcal{N}(0,1)$ is independent of $(\serv{X}_{\star},\serv{A})\sim\pi$, and $\beta_{t-1}\in\mathbb{R}$ and $\sigma_{t-1}>0$ are fixed parameters (determined by the matrix and iterate denoisers used in the first $t-1$ iterations). Repeating the arguments used in  \eqref{Eqn:OAMP_opt_F1_opt}, we obtain:
\BS\label{Eqn:OAMP_opt_F2_opt}
\begin{align}
& \min_f\ 1-\frac{\left( \E[\serv{X}_\star f(\serv{X}_{t-1};\serv{A})] \right)^2}{ \E[f^2(\serv{X}_{t-1};\serv{A})] } \quad \text{subject to}\quad \mathbb{E}[f'(\serv{X}_{t-1};\serv{A})]=0\\
&= \min_f\, \min_{c\in\mathbb{R}}\ \mathbb{E}\left[\left(\serv{X}_\star-cf(\serv{X}_{t-1};\serv{A})\right)^2\right] \quad \text{subject to}\quad \mathbb{E}[cf'(\serv{X}_{t-1};\serv{A})]=0\\
&= \min_f\ \mathbb{E}\left[\left(\serv{X}_\star-f(\serv{X}_{t-1};\serv{A})\right)^2\right] \quad \text{subject to}\quad \mathbb{E}[f'(\serv{X}_{t-1};\serv{A})]=0\label{Eqn:OAMP_opt_F2_opt_c}
\end{align}
\ES
Recalling the definition of the DMMSE estimator from \defref{def:gauss-channel} (see also Lemma \ref{lem:scalarization-mmse} in Appendix \ref{appendix:gauss-channel}), the minimizer of \eqref{Eqn:OAMP_opt_F2_opt_c} is:
\begin{equation} \label{eq:opt-denoiser-sol}
f_t(x;a)=\dfnew{\varphi}\left( (\beta_{t-1}^2+\sigma_{t-1}^2)^{-\frac{1}{2}}\cdot {x}{};a \big| \omega_{t-1}\right)\quad\text{where}\quad \omega_{t-1}\bydef \frac{\beta_{t-1}^2}{\beta_{t-1}^2+\sigma_{t-1}^2}.
\end{equation}

\paragraph*{Optimal OAMP algorithm} Using the matrix and iterate denoisers derived above, 
we obtain the following memory-free OAMP algorithm (cf.~\eqref{Eqn:OAMP_memory_free}):
\begin{equation}\label{Eqn:OAMP_opt_derivation}
\bm{x}_t = c_t\cdot{\Psi}_\star(\bm{Y};\rho_t) \cdot \dfnew{\varphi}_t(\bm{x}_{t-1};\bm{a}|\omega_{t-1}),
\end{equation}
where we defined $c_t\bydef 1/\sqrt{\beta_{t}^2+\sigma_t^2}$ and  made a slight change of variable by absorbing the normalization factor (acting on the $x$ input of $\dfbdnsr$) in \eqref{eq:opt-denoiser-sol} into the definition of $\vx_t$. This scaling parameter $c_t$ ensures state evolution random variable $\serv{X}_t$ corresponding to $\bm{x}_t$ satisfies $\E[\serv{X}_t^2] = 1$. The resulting algorithm is precisely the optimal OAMP algorithm introduced in \eqref{eq:optimal-OAMP}. The recursions for $\omega_t, \rho_t$ and the alternative formula $c_t=1/\sqrt{w_t}\cdot(1+1/\rho_t)$ is derived by specializing the general state evolution equations in \eqref{Eqn:SE_memoryfree_def} to the optimal OAMP algorithm; see the proof of \propref{prop:optimal-OAMP-SE} in Appendix \ref{App:proof_OAMP_opt_SE} for details.

{

\subsection{Proof of the Optimality Result (\thref{thm:optimality})}\label{sec:optimality-proof}

We now present the proof of \thref{thm:optimality}. We will introduce the key ideas involved in the form of some intermediate results, whose proofs are deferred to \appref{appendix:optimality-proof}.

\paragraph*{Implementing Iterative Algorithms using OAMP Algorithms} Our proof builds on the techniques introduced  by \citet{celentano2020estimation} and \citet{montanari2022statistically}, who show that for the i.i.d. Gaussian noise model, any estimator computed using $t \in \N$ iterations of an iterative algorithm can also be computed using $t$ iterations of a suitably designed AMP algorithm. Consequently, we can restrict the search space for the optimal $t$-iteration iterative algorithm to $t$-iteration AMP algorithms, which are easier to analyze. Based on these results, it is natural to hope that any estimator computed using $t$ iterations of a general iterative algorithm (as defined in Definition~\ref{def:itr}) can be computed using $t$ iterations of a suitably designed OAMP algorithm (as defined in Definition~\ref{Def:OAMP_main}). Unfortunately, this strategy breaks down in our case because we require that the iterate denoisers used in the OAMP algorithm satisfy the divergence-free requirement (\ref{eq:div-free-constraint}) of Definition~\ref{Def:OAMP_main}, which is crucial to ensure that the OAMP algorithms have tractable state evolutions. We illustrate the difficulty caused by imposing the divergence-free requirement  and motivate our proof strategy by trying to construct an OAMP algorithm to implement the first two iterations of a simple iterative algorithm:
\begin{align} \label{eq:itr-problem}
    \vr_1 & = \mfunc_1(\mY) \cdot \fnonlin_1(\va), \quad \vr_2 = \mfunc_2(\mY) \fnonlin_2(\vr_1 ; \va),
\end{align} 
which we refer to as the \emph{target algorithm}. Our goal is to design an OAMP algorithm which can reconstruct the iterates of the target algorithm. For simplicity, we assume that the matrix denoisers used by the target algorithm satisfy the trace-free requirement $\E_{\serv{\Lambda} \sim \mu}[\mfunc_1(\serv{\Lambda})] = \E_{\serv{\Lambda} \sim \mu}[\mfunc_2(\serv{\Lambda})]  = 0$ (requirement (\ref{eq:trace-free}) in Definition~\ref{Def:OAMP_main}). Since the divergence-free constraint (\ref{eq:div-free-constraint}) is inactive for the first iteration, we can construct a valid OAMP algorithm which simply copies the first iteration of the target algorithm $\vx_1 = \mfunc_1(\mY) \fnonlin_1(\va)$. Let $(\serv{X}_{\star}, \serv{X}_1; \serv{A})$ be the state evolution random variables corresponding to the first iteration of this OAMP algorithm. Next, we consider the second iteration of the target algorithm \eqref{eq:itr-problem}. Notice that the denoiser $\fnonlin_2$ might not satisfy the divergence-free requirement $\E[\partial_1 \fnonlin_2(\serv{X}_1; \serv{A})] = 0$ imposed in the definition of OAMP algorithms (Definition~\ref{Def:OAMP_main}). To ensure this requirement, we design the second iteration of the OAMP algorithm as:
\begin{align*}
\iter{\vx}{2} & = \mfunc_2(\mY)  \cdot \dfnew{\fnonlin}_2(\iter{\vx}{1}; \va)\ \text{ where } \   \dfnew{\fnonlin}_2(x; \aux) \explain{def}{=} {\fnonlin}_2 \left(x; \aux \right) - \E[\partial_1 {\fnonlin}_2(\serv{X}_1; \serv{A})] \cdot x \quad \forall \; x  \in  \R,  \aux \in  \R^{\auxdim}.
\end{align*}
By construction, $\dfnew{f}_2$ satisfies the divergence-free constraint $\E[\partial_1 \dfnew{\fnonlin}_2(\serv{X}_1; \serv{A})] = 0$. We can express the second iterate of the target algorithm \eqref{eq:itr-problem} as:
\begin{align} \label{eq:itr-2-stuck}
\vr_2 & = \vx_2  + \E[\partial_1 {\fnonlin}_2(\serv{X}_1; \serv{A})] \cdot \mfunc_2(\mY)  \vx_1 \explain{(a)}{=} \vx_2  + \E[\partial_1 {\fnonlin}_2(\serv{X}_1; \serv{A})] \cdot \underbrace{\mfunc_2(\mY) \mfunc_1(\mY)  \fnonlin_1(\va)}_{\text{uncomputed}} ,
\end{align}
where in step (a) we recalled that $\vx_1 = \mfunc_1(\mY) \cdot \fnonlin_1(\va)$. In particular, we are unable to reconstruct the second iterate of the target algorithm in \eqref{eq:itr-problem} using the first two iterations of the OAMP algorithm we have designed since the matrix-vector product $\mfunc_2(\mY) \mfunc_1(\mY) \cdot \fnonlin_1(\va)$ in \eqref{eq:itr-2-stuck} has not been computed by the OAMP algorithm in the sense that it cannot be reconstructed from the OAMP iterates $\vx_1, \vx_2$.

\paragraph*{Lifted OAMP Algorithms} To address the issue highlighted above, we introduce a more powerful class of algorithms called \emph{lifted OAMP algorithms}, which are parameterized by a degree parameter $\degree \in \N$, which modulates the computational power of these algorithms. A degree-$\degree$ OAMP algorithm can compute $\degree$ matrix-vector products per iteration. For instance, in the first iteration, the algorithm computes the matrix-vector products (${\serv{\Lambda}\sim \mu}$):
\begin{align} \label{eq:matvec-prod}
(\mY - \E[\serv{\Lambda}] \cdot \mI_{\dim})  \fnonlin_1(\va), \; (\mY^2 - \E[\serv{\Lambda}^2] \cdot \mI_{\dim})  \fnonlin_1(\va), \;  \dotsc, \; (\mY^{\degree} - \E[\serv{\Lambda}^{\degree}] \cdot \mI_{\dim})  \fnonlin_1(\va). 
\end{align}
Since continuous functions can be approximated arbitrarily well by polynomial functions on compact sets, for any continuous function $\Phi: \R \mapsto \R$, we can approximate the matrix-vector product $\Phi(\mY) \cdot  \fnonlin_1(\va)$ using a linear combination of the matrix-vector products in \eqref{eq:matvec-prod}. In particular, the missing vector-matrix product in \eqref{eq:itr-2-stuck} needed to reconstruct the second iteration of the target algorithm \eqref{eq:itr-problem} can also be approximated using the matrix-vector products \eqref{eq:matvec-prod} computed by the lifted OAMP algorithm. The following definition formally introduces the class of lifted OAMP algorithms. 

\begin{definition}[Degree-$D$ Lifted OAMP Algorithm] \label{def:LOAMP} For any $D \in \N$ (independent of $\dim$), a degree-$D$ lifted OAMP algorithm maintains  a sequence of iterates $(\iter{\vw}{ti})_{t \in \N, i \in [D]}$ indexed by a time index $t \in \N$ and a degree index $i \in [\degree]$. At step $t$, the lifted OAMP computes the iterates $\iter{\vw}{t,1}, \dotsc, \iter{\vw}{t,\degree}$ as follows:
\begin{align} \label{eq:lifted-OAMP}
\iter{\vw}{t,i} & = \left( \mY^i - \E[\serv{\Lambda}^i] \cdot \mI_{\dim} \right) \cdot \fnonlin_{t}\big(\iter{\vw}{1,\bdot}, \dotsc, \iter{\vw}{t-1,\bdot}; \va \big)  \quad \forall \; i \in [\degree], \; t \in \N.
\end{align}
The estimate of $\vx_\star$ at iteration $t$ is obtained by applying a post-processing function $\psi_t$ to the iterates $\vw_{1,\bdot}, \dotsc, \vw_{t,\bdot}$ generated so far and the side-information $\va$:
\begin{align} \label{eq:LOAMP-postprocess} 
{\hat{\vw}_{t}} & = \psi_t(\iter{\vw}{1,\bdot}, \dotsc, \iter{\vw}{t-1,\bdot}; \va),
\end{align}
In the above equations, the notation $\iter{\vw}{t,\bdot}$ is a shorthand for the collection of vectors $\iter{\vw}{t,1}, \iter{\vw}{t,2}, \dotsc, \iter{\vw}{t,\degree}$. Furthermore for each $t \in \N$, the iterate denoiser $\fnonlin_t: \R^{(t-1)\degree} \times \R^{\auxdim} \mapsto \R$ and the postprocessing function $\psi_t:\mathbb{R}^{t\degree} \times \R^{\auxdim} \mapsto\mathbb{R}$ are continuously differentiable, Lipschitz, act entry-wise on their vector inputs, and do not change with $\dim$. Each lifted OAMP algorithm is associated with a collection of \emph{state evolution random variables} $(\serv{X}_{\star}, (\serv{W}_{t,i})_{t \in \N, i \in [\degree]}; \serv{A})$, whose distribution is given by:
\begin{subequations} \label{eq:SE-LOAMP}
\begin{align}
(\serv{X}_\star, \serv{A}) \sim \pi, \quad \serv{W}_{t,i} = \alpha_t \cdot (\E_{{\serv{\Lambda}}_{\nu} \sim \nu}[{\serv{\Lambda}}_{\nu}^i] - \E_{\serv{\Lambda}\sim \mu}[{\serv{\Lambda}}^i]) \cdot \serv{X}_\star + \serv{Z}_{ti} \quad \forall \; t \; \in \; \N, i \; \in \; [\degree],
\end{align}
where $\{\serv{Z}_{ti}\}_{t \in \N, i \in \N}$ are centered Gaussian random variables, sampled independently of $(\serv{X}_\star, \serv{A}$) with covariance:
\begin{align}
\E[\serv{Z}_{si}\serv{Z}_{tj}] & = \alpha_s \alpha_t \cdot \Cov_{\stackrel{}{\serv{\Lambda}_{\nu} \sim \nu}}[\serv{\Lambda}_{\nu}^i, \serv{\Lambda}_{\nu}^j] + (\Sigma_{st} - \alpha_s \alpha_t)  \cdot \Cov_{\stackrel{}{\serv{\Lambda} \sim \mu}}[\serv{\Lambda}_{}^i, \serv{\Lambda}^j].
\end{align}
In the above equations, $(\alpha_t)_{t \in \N}$ and  $(\Sigma_{st})_{s \in \N, t \in \N}$ are given by the recursion:
\begin{align}
\alpha_t & = \E\left[\serv{X}_\star {\fnonlin}_t\left(\serv{W}_{1,\bdot}, \dotsc, \serv{W}_{t-1,\bdot}; \serv{A}\right)\right], \  \Sigma_{st}  = \E\left[{\fnonlin}_s\left(\serv{W}_{1,\bdot}, \dotsc, \serv{W}_{s-1,\bdot}; \serv{A}\right) {\fnonlin}_t\left(\serv{W}_{1,\bdot}, \dotsc, \serv{W}_{t-1,\bdot}; \serv{A}\right)\right].
\end{align}
\end{subequations}
Finally, we require the iterate denoisers $(\fnonlin_t)_{t \in \N}$ to satisfy the \emph{divergence-free} constraint:
    \begin{align}
    \E[\partial_{s,i} \fnonlin_t(\serv{W}_{1,\bdot}, \dotsc, \serv{W}_{t-1,\bdot}; \serv{A})] & = 0 \quad \forall \; s \; \in \; [t-1], \; i \; \in \; [\degree], \; t \; \in \; \N,
    \end{align}
where $\partial_{s,i} \fnonlin_t(w_{1,\bdot},\dotsc, w_{t-1,\bdot}; \aux)$ denotes the partial derivative of $\fnonlin_t$ with respect to $w_{s,i}$.    
\end{definition}
A degree-$\degree$ lifted OAMP algorithm which runs for $t$ iterations can be viewed as an instance of an OAMP algorithm (Definition~\ref{Def:OAMP_main}) which runs for $t \degree$ iterations. Hence, the following result as an immediate corollary of Theorem~\ref{thm:SE}.

\begin{corollary}[State evolution of lifted OAMP algorithms] \label{cor:SE-LOAMP} Consider a degree-$\degree$ lifted OAMP algorithm of the form \eqref{eq:lifted-OAMP} and let $(\serv{X}_{\star}, (\serv{W}_{t,i})_{t \in \N, i \in [\degree]} ; \serv{A})$ be the associated state evolution random variables. Then for any $t \in \N$,
\begin{align*}
\left(\vx_\star, \vw_{1,\bdot}, \vw_{2,\bdot}, \dotsc, \vw_{t,\bdot}; \bm{a}\right) \wc (\serv{X}_\star, \serv{W}_{1,\bdot}, \dotsc, \serv{W}_{t,\bdot}; \serv{A}).
\end{align*}
\end{corollary}
The following proposition (proved in \appref{appendix:LOAMP-reduction}), shows that any estimator that is computed by an iterative algorithm (Definition~\ref{def:itr}) can be approximated arbitrarily well by a suitably designed degree-$\degree$ lifted OAMP algorithm in the limit $\degree \rightarrow \infty$. 

\begin{proposition} \label{prop:LOAMP-reduction} Let $\hat{\vr}_t$ be the estimator returned by any iterative algorithm (Definition~\ref{def:itr}) after $t \in \N$ iterations. Then, for each $\degree \in \N$, there is a degree-$\degree$ lifted OAMP algorithm which returns an estimator $\dup{\tilde{\vw}_{t}}{\degree}$ after $t$ iterations, which satisfies:
\begin{align*}
\lim_{\degree \rightarrow \infty} \plimsup_{\dim \rightarrow \infty} \frac{\|\hat{\vr}_t - \dup{\tilde{\vw}_{t}}{\degree} \|^2}{\dim} & = 0.
\end{align*}
\end{proposition}
\paragraph*{The Optimal Lifted OAMP Algorithm} Since \propref{prop:LOAMP-reduction} guarantees that any estimator computed using an iterative algorithm can be approximated by a suitably designed lifted OAMP algorithm, we focus on characterizing the optimal lifted OAMP algorithms. The optimal degree-$\degree$ lifted OAMP algorithm takes the form:
\begin{subequations} \label{eq:optimal-LOAMP}
  \begin{align}
    \iter{\vw}{t,i} & = \left( \mY^i - \E[\serv{\Lambda}^i] \cdot \mI_{\dim} \right) \cdot \fnonlin_{t}^{\star}\big(\iter{\vw}{1,\bdot}, \dotsc, \iter{\vw}{t-1,\bdot}; \va \big)  \quad \forall \; i \in [\degree], \; t \in \N,
\end{align}
and returns the following estimator after $t \in \N$ iterations:
\begin{align}
    \dup{\hat{\vw}_{t}}{\degree} & = \hnonlin_t^{\star}(\iter{\vw}{1,\bdot}, \dotsc, \iter{\vw}{t-1,\bdot}; \va).
\end{align}  
\end{subequations}
The description of the iterate denoisers $(\fnonlin_t^{\star})_{t \in \N}$ and the post-processing functions $(\hnonlin_t^{\star})_{t \in \N}$ used by the optimal lifted OAMP algorithm is recursive. Suppose that the functions $\fnonlin^{\star}_{1}, \dotsc, \fnonlin_{t}^\star$ have been specified. Let $(\serv{X}_\star, \serv{W}_{1,\bdot}, \dotsc,  \serv{W}_{t,\bdot} ; \serv{A})$ denote the state evolution random variables for the first $t$ iterations of the resulting lifted OAMP algorithm. Then:
\begin{enumerate}
\item The post-processing function  $\hnonlin_{t}^\star$ used at iteration $t$ is the MMSE estimator (recall Definition~\ref{def:gauss-channel})  for the Gaussian channel $(\serv{X}_\star, \serv{W}_{1,\bdot}, \dotsc,  \serv{W}_{t,\bdot} ; \serv{A})$.
\item $\fnonlin_{t+1}^\star$, the iterate denoiser used by the lifted OAMP algorithm in the next iteration, is the DMMSE estimator for the Gaussian channel  $(\serv{X}_\star, \serv{W}_{1,\bdot}, \dotsc,  \serv{W}_{t,\bdot} ; \serv{A})$.
\end{enumerate}
The following proposition, shows that the estimator \eqref{eq:optimal-LOAMP} achieves the lowest mean squared error among all estimators that can be computed using $t$ iterations of a degree-$\degree$ lifted OAMP algorithm. 
\begin{proposition} \label{prop:greedy-optimality} Let $\dup{\hat{\vw}_{t}}{\degree}$ be the estimator described in \eqref{eq:optimal-LOAMP}. Let $\dup{\tilde{\vw}_t}{\degree}$ be any other estimator that can be computed using $t$ iterations of some degree-$\degree$ lifted OAMP algorithm. Then, we have:
\begin{align*}
\plim_{\dim \rightarrow \infty} \frac{\|\dup{\hat{\vw}_{t}}{\degree} - \vx_\star\|^2}{\dim} & \leq \plim_{\dim \rightarrow \infty} \frac{\|\dup{\tilde{\vw}_t}{\degree} - \vx_\star\|^2}{\dim}.
\end{align*}
\end{proposition}

\begin{proof}[Proof Sketch] We outline the main ideas here and defer the complete proof to \appref{appendix:optimal-LOAMP}. Consider a general lifted OAMP algorithm:
\begin{align*}
\iter{\vw}{t,i} & = \left( \mY^i - \E[\serv{\Lambda}^i] \cdot \mI_{\dim} \right) \cdot \fnonlin_{t}\big(\iter{\vw}{1,\bdot}, \dotsc, \iter{\vw}{t-1,\bdot}; \va \big)  \quad \forall \; i \in [\degree], \; t \in \N,
\end{align*}
which returns the estimator $\dup{\hat{\vw}_t}{\degree} = \hnonlin_t(\vw_{1,\bdot}, \dotsc, \vw_{t,\bdot}; \va)$ after $t$ iterations. By Corollary~\ref{cor:SE-LOAMP}, the asymptotic mean squared error for this estimator is $\E|\serv{X}_{\star} - h_t(\serv{W}_{1, \bdot}, \dotsc, \serv{W}_{t,\bdot}; \serv{A})|^2$, where $(\serv{X}_{\star}, \serv{W}_{1, \bdot}, \dotsc, \serv{W}_{t,\bdot}; \serv{A})$ denote the state evolution random variables associated with the lifted OAMP algorithm. Recalling the definition of the MMSE denoiser for a Gaussian channel from Definition~\ref{def:gauss-channel}, we find that the optimal choice for the post-processing function is $\hnonlin_t = \hnonlin_t^{\star}$, where $\hnonlin_t^{\star}$ is the MMSE estimator for the Gaussian channel $(\serv{X}_{\star}, \serv{W}_{1, \bdot}, \dotsc, \serv{W}_{t,\bdot}; \serv{A})$. With this optimal choice of the post-processing function, the limiting MSE of the resulting estimator is $\bmmse(\serv{X}_{\star}| \serv{W}_{1, \bdot}, \dotsc, \serv{W}_{t,\bdot}; \serv{A})$, which depends implicitly on the iterate denoisers $\fnonlin_{1:t}$ used in the lifted OAMP algorithm since $\fnonlin_{1:t}$ determine the joint distribution of the state evolution random variables. To make this dependence explicit, we define $\scrM_t(\fnonlin_1, \dotsc, \fnonlin_t) \explain{def}{=} \bmmse(\serv{X}_{\star}| \serv{W}_{1, \bdot}, \dotsc, \serv{W}_{t,\bdot}; \serv{A})$ for each $t \in \N$. Deriving the optimal choice of the iterate denoisers is equivalent to minimizing the functional  $\scrM_t(\fnonlin_1, \dotsc, \fnonlin_t)$ with respect to $\fnonlin_{1:t}$. The crux of the argument is to show that the iterate denoisers $(\fnonlin_{t}^{\star})_{t \in \N}$ used by the optimal lifted OAMP algorithm \eqref{eq:optimal-LOAMP} correspond to a \emph{greedy approach} to minimizing $\scrM_t$:
\begin{align} \label{eq:greedy-approach-crux}
\fnonlin_1^{\star} \in \argmin_{f_1} \scrM_1(\fnonlin_1), \  \fnonlin_2^{\star} \in \argmin_{f_2} \scrM_2(\fnonlin_1^{\star}, 
 \fnonlin_2), \  \dotsc \quad  \fnonlin_t^{\star} \in \argmin_{f_t} \scrM_t(\fnonlin_{1:t-1}^{\star},\fnonlin_t) .
\end{align}
Moreover, we show that the functional $\scrM_t$ has the property that for any choice of iterate denoisers $\fnonlin_{1:t-1}$:
\begin{align} \label{eq:key-prop}
\min_{\fnonlin_t} \scrM_{t}(\fnonlin_1, \dotsc, \fnonlin_t) & = \mathscr{A}(\scrM_{t-1}(\fnonlin_1, \dotsc, \fnonlin_{t-1})),
\end{align}
for some explicit \emph{non-decreasing} function $\mathscr{A}:[0,1] \mapsto [0,1]$. This property guarantees that the greedy approach \eqref{eq:greedy-approach-crux} finds the global minimizer of $\scrM_t$. Indeed, if we assume as our induction hypothesis that $\scrM_{t-1}(\fnonlin_{1:t-1}^{\star}) \leq \scrM_{t-1}(\fnonlin_{1:t-1})$ for any choice of $\fnonlin_{1:t-1}$, we can conclude that:
\begin{align*}
 \scrM_{t}(\fnonlin_{1}, \dotsc, \fnonlin_{t}) &\geq \min_{\fnonlin_t}  \scrM_{t}(\fnonlin_{1}, \dotsc, \fnonlin_{t})  \explain{\eqref{eq:key-prop}}{=} \mathscr{A}(\scrM_{t-1}(\fnonlin_{1}, \dotsc, \fnonlin_{t-1}))\\
 & \explain{(a)}{\geq} \mathscr{A}(\scrM_{t-1}(\fnonlin_{1}^{\star}, \dotsc, \fnonlin_{t-1}^{\star}))  \explain{\eqref{eq:key-prop}}{=}  \min_{\fnonlin_t}  \scrM_{t}(\fnonlin_{1}^{\star}, \dotsc, \fnonlin_{t-1}^{\star},\fnonlin_t) \explain{\eqref{eq:greedy-approach-crux}}{=}  \scrM_{t}(\fnonlin_{1}^{\star}, \dotsc, \fnonlin_{t}^{\star}),
\end{align*}
where the inequality (a) follows from the induction hypothesis and the monotonicity of $\mathscr{A}$. 
\end{proof}
\paragraph*{Simplifying the optimal lifted OAMP algorithm} Observe that the optimal degree-$\degree$ lifted OAMP algorithm introduced in \eqref{eq:optimal-LOAMP} is not an iterative algorithm in the sense of Definition~\ref{def:itr}. Indeed, an iterative algorithm is allowed a single matrix-vector multiplication per iteration, whereas the optimal lifted OAMP computes $\degree$ matrix-vector products per iteration.  Fortunately, the optimal lifted OAMP algorithm \eqref{eq:optimal-LOAMP} can be simplified to eliminate these extra matrix-vector multiplications. This simplification leads to the following result, which shows that the optimal OAMP algorithm introduced in (\ref{eq:optimal-OAMP}) can match the infinite degree ($\degree \rightarrow \infty$) limit of the performance of the optimal lifted OAMP algorithm in \eqref{eq:optimal-LOAMP}. 
\begin{proposition} \label{prop:simplification} Let ${\hat{\vx}}_t$ be the estimator returned by the optimal OAMP algorithm in (\ref{eq:optimal-OAMP}) after $t$ iterations. Let $\dup{\hat{\vw}_t}{\degree}$ be the estimator returned by optimal degree-$\degree$ lifted OAMP algorithm in \eqref{eq:optimal-LOAMP} after $t$ iterations. Then, for any $t \in \N$:
\begin{align*}
\plim_{\dim \rightarrow \infty} \frac{\|{\hat{\vx}_{t}} - \vx_\star\|^2}{\dim} & = \lim_{\degree \rightarrow \infty} \;  \plim_{\dim \rightarrow \infty} \frac{\|\dup{\hat{\vw}_t}{\degree} - \vx_\star\|^2}{\dim}.
\end{align*}
\end{proposition}
\begin{proof}[Proof Sketch] We outline the high-level idea here, a complete proof is given in \appref{App:proof_prop_simplications}. We show that the estimator $\dup{\hat{\vw}_t}{\degree}$  returned by the optimal lifted OAMP algorithm \eqref{eq:optimal-LOAMP} after $t$ iterations depends \emph{only} on a linear combination $\dup{\vx_t}{\degree}  = \sum_{i=1}^{\degree} v_{i} \cdot \vw_{t,i}$ of the iterates $\vw_{t,\bdot}$ generated at the last step of the algorithm. Thanks to this property, the $\dup{\vx_t}{\degree}$ (and hence, the estimator $\hat{\vw}_t$) can be computed using a single (instead of $\degree$) matrix-vector multiplication(s):
\begin{align*}
\dup{\vx_t}{\degree} & = \sum_{i=1}^{\degree} v_i  \dup{\vw_{t,i}}{\degree} \explain{\eqref{eq:optimal-LOAMP}}{=} \bigg( \underbrace{\sum_{i=1}^{\degree} v_i  (\mY^i - \E_{\serv{\Lambda} \sim \mu}[\serv{\Lambda^i}] \cdot \mI_{\dim}) }_{\explain{def}{=} \dup{\Psi_t}{\degree}(\mY)}\bigg)  \fnonlin_t^{\star}(\vw_{<t,\bdot}; \va) = \dup{\Psi_t}{\degree}(\mY)   \fnonlin_t^{\star}(\vw_{<t,\bdot}; \va).
\end{align*}
We argue that as $\degree \rightarrow \infty$ the matrix denoiser $\dup{\Psi_t}{\degree}$ that appears in the above equation converges to the optimal matrix denoiser used by the optimal OAMP algorithm (\ref{eq:optimal-OAMP}) in an $L^2$ sense, and hence the claimed result follows.  
\end{proof}
The optimality result stated as Theorem~\ref{thm:optimality} is an immediate consequence of the above intermediate propositions.
\begin{proof}[Proof of Theorem~\ref{thm:optimality}] Let $\hat{\vx}_t$ be the estimator returned by the optimal OAMP algorithm from (\ref{eq:optimal-OAMP}) and let $\hat{\vr}_t$ be any estimator that can be computed using $t$ iterations of some iterative algorithm. \propref{prop:LOAMP-reduction} guarantees the existence of a sequence of estimators $(\dup{\tilde{\vw}_t}{\degree})_{\degree \in \N}$  that approximate $\hat{\vr}_t$ and can be computed using $t$ iterations of a degree-$\degree$ lifted OAMP algorithm. Finally, let $\dup{\hat{\vw}_t}{\degree}$ be the estimator returned by the optimal degree-$\degree$ lifted OAMP algorithm from \eqref{eq:optimal-LOAMP}. By the reverse triangle inequality, for any $\degree \in \N$, $\|\hat{\vr}_t- \vx_{\star}\| \geq \|\dup{\tilde{\vw}_t}{\degree}- \vx_{\star}\| - \|\dup{\tilde{\vw}_t}{\degree}- \hat{\vr}_t\|$. We let $\dim \rightarrow \infty$ and then $\degree \rightarrow \infty$ to obtain:
\begin{align*}
&\pliminf_{\dim \rightarrow \infty} \frac{\|\hat{\vr}_t- \vx_{\star}\|}{\sqrt{\dim}}  \explain{}{\geq}  \liminf_{\degree \rightarrow \infty} \pliminf_{\dim \rightarrow \infty} \frac{\|\dup{\tilde{\vw}_t}{\degree}- \vx_{\star}\|}{\sqrt{\dim}}  - \limsup_{\degree \rightarrow \infty} \plimsup_{\dim \rightarrow \infty} \frac{\|\dup{\tilde{\vw}_t}{\degree}- \hat{\vr}_t\|}{\sqrt{\dim}} \\ &  \explain{Prop. \ref{prop:LOAMP-reduction}}{=} \; \;   \liminf_{\degree \rightarrow \infty} \plim_{\dim \rightarrow \infty} \frac{\|\dup{\tilde{\vw}_t}{\degree}- \vx_{\star}\|}{\sqrt{\dim}} \; \; \explain{Prop. \ref{prop:greedy-optimality}}{\geq} \; \; \liminf_{\degree \rightarrow \infty} \plim_{\dim \rightarrow \infty} \frac{\|\dup{\hat{\vw}_t}{\degree}- \vx_{\star}\|}{\sqrt{\dim}}  \; \; \explain{Prop. \ref{prop:simplification}}{=}\; \;  \plim_{\dim \rightarrow \infty} \frac{\|{\hat{\vx}_t}- \vx_{\star}\|}{\sqrt{\dim}}.
\end{align*} 
Comparing the first and final expressions in the above display proves the theorem.
\end{proof}
} 

\section{Numerical Experiments}\label{Sec:simulations}

We conclude the paper with some numerical experiments. Appendix \ref{App:numerical_additional} provides additional numerical results. 

Following \citep{ma2017orthogonal,rangan2019vector}, we generate the noise matrix $\mW$ in a structured manner, making it feasible to efficiently simulate the dynamics of iterative algorithms on very high-dimensional matrices ($\dim = 2\times 10^6$). Specifically, we use the noise matrix $\bm{W}=\bm{O}\text{diag}(\lambda_1,\ldots,\lambda_N)\bm{O}^\UT$ with $\mO = \mS_1 \mF \mS_2 \mF^\UT \mS_3$, where $\mS_1,\mS_2,\mS_3$ are independent random signed diagonal matrices, and $\mF$ is the discrete cosine transform matrix. The eigenvalues $\{\lambda_i\}_{i\in[N]}$ are sampled independently from a spectral $\mu$. We remark that known universality results \citep{dudeja2023universality,wang2022universality,dudeja2022universality} guarantee that using this structured noise matrix (instead of random rotationally invariant noise drawn from the quartic model) does not change the asymptotic dynamics of AMP algorithms. {The signal $\vx_{\star}$ is sampled from an i.i.d. two-point prior $\epsilon^2\delta_{1/\epsilon} +
(1-\epsilon^2)\delta_0$, where $0<\epsilon<1$.}

\begin{figure}[htbp]
\begin{center}
\subfloat[Statistical-computational gap exists.
\label{Fig:quartic_u0_vs_snr}]{
\includegraphics[width=0.48\linewidth]{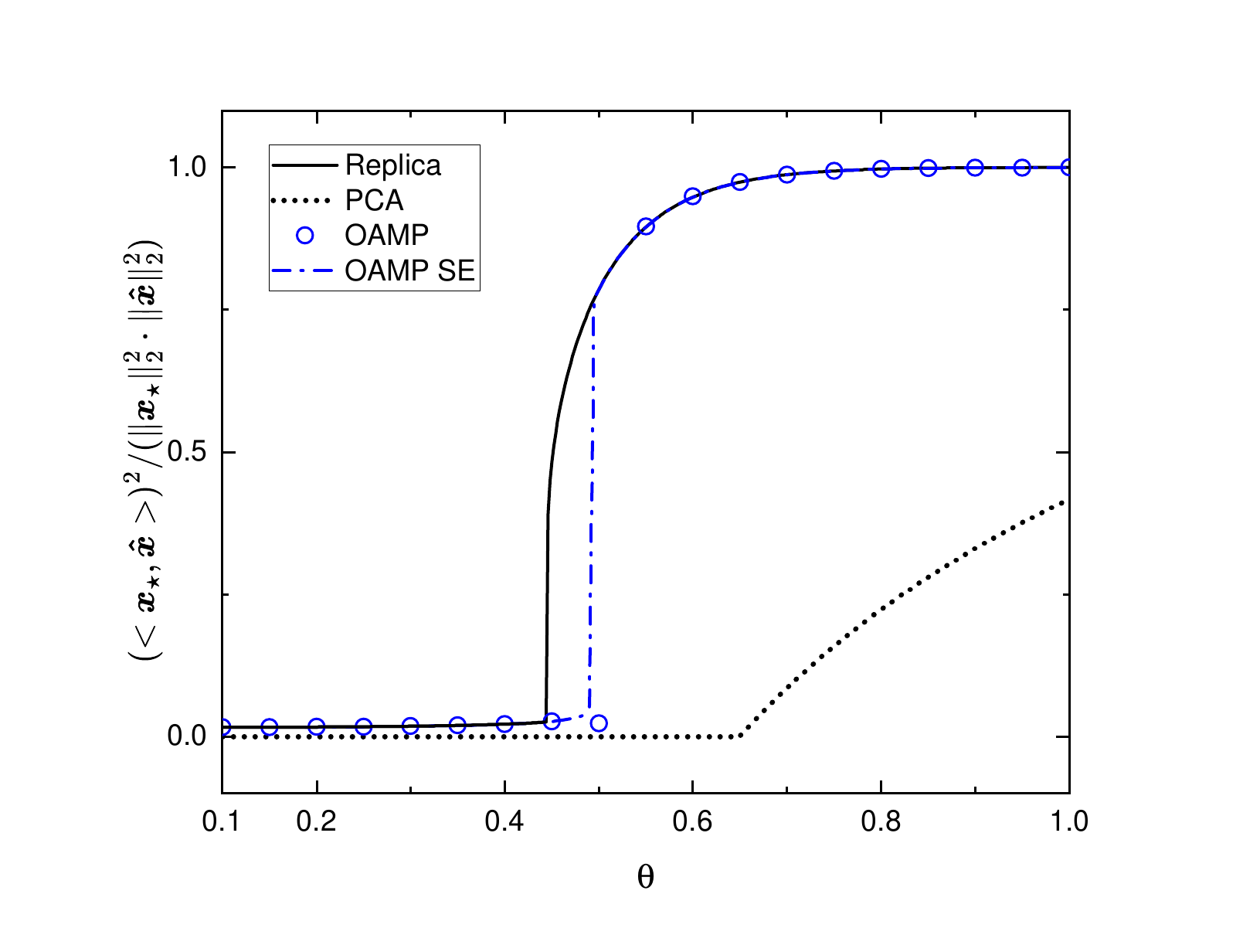} }
\subfloat[No statistical-computational gap.\label{Fig:sestic_vs_snr}]{
\includegraphics[width=0.48\linewidth]{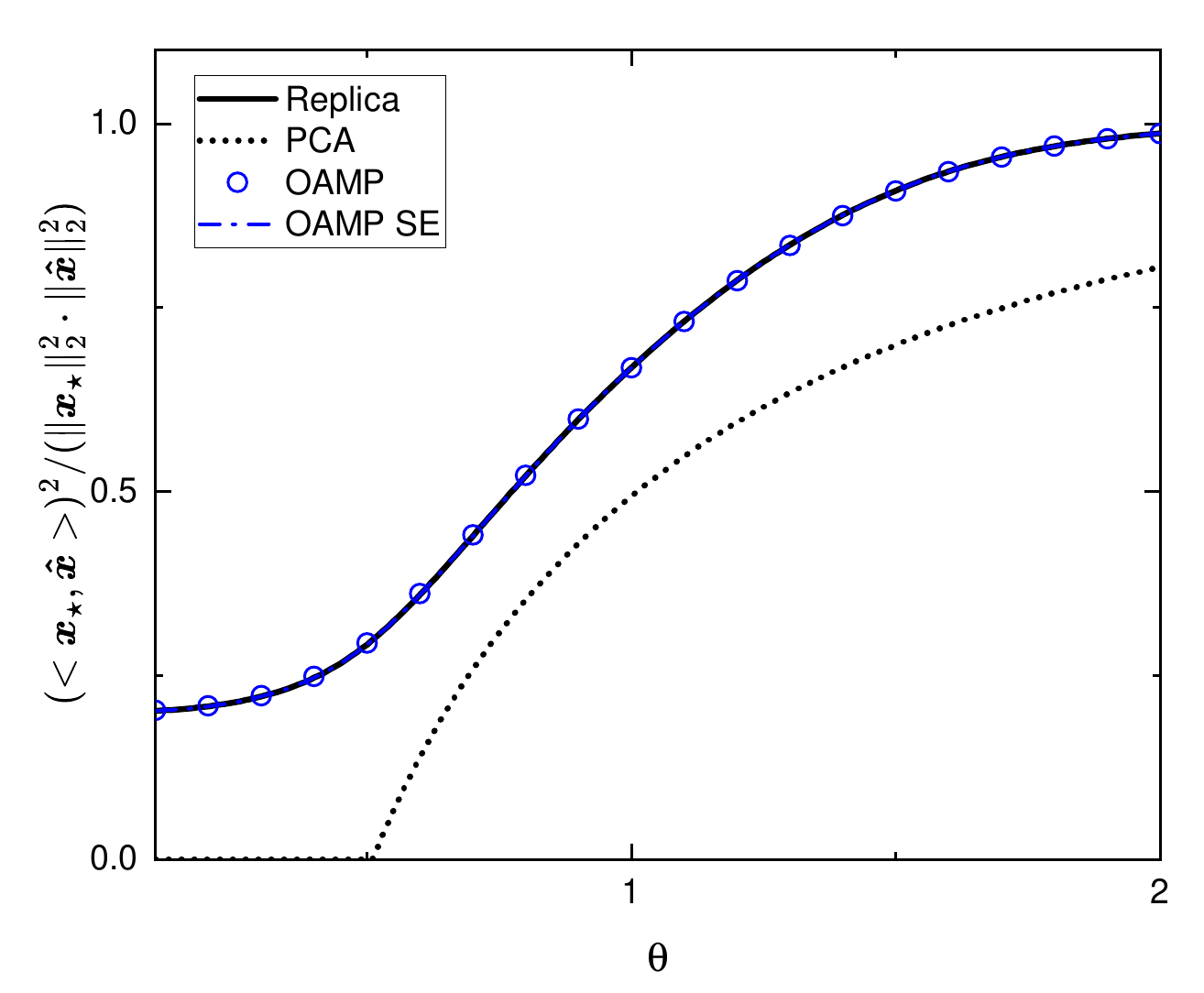} }
\end{center}
\caption{Performance (measured by the normalized overlap with signal) of PCA, the Bayes estimator, and the optimal OAMP algorithm in \eqref{eq:optimal-OAMP}. (a) The spectral $\mu$ corresponds to the quartic noise model in \eqref{Eqn:quartic-main} with parameter $\gamma = 0$. The signal $\vx_{\star}$ is sampled from an i.i.d. two-point prior $\epsilon^2\delta_{1/\epsilon} +
(1-\epsilon^2)\delta_0$, where $\epsilon=1/8$. (b) The spectral $\mu$ corresponds to the purely sestic noise model (see \eqref{Eqn:sestic} in the appendix) and the the signal is sampled from the two-point prior with $\epsilon=1/\sqrt{5}$.}
\label{Fig:Fixed_point_vs_snr}
\end{figure} 

Figs.~\ref{Fig:quartic_u0_vs_snr} and \ref{Fig:sestic_vs_snr} compare the performance of the optimal OAMP algorithm with the state evolution prediction from \thref{thm:SE}. {In our simulations, the OAMP algorithm is initialized by $f_1(\bm{x}_0;\bm{a})=\epsilon\cdot\bm{1}$ (which is the mean of the two-point prior) and no side information is provided.} We also plot the performance of PCA and the Bayes-optimal estimator. In both Fig.~\ref{Fig:quartic_u0_vs_snr} and Fig.~\ref{Fig:sestic_vs_snr}, OAMP achieves much better estimation accuracy than the PCA estimator and matches the Bayes-optimal performance at high SNR. Moreover, the performance of OAMP closely matches its state evolution. Notice that, for Fig.~\ref{Fig:quartic_u0_vs_snr}, there exists a performance gap between the {conjectured} Bayes-optimal performance (as predicted by the replica method) and the performance achieved by optimal OAMP (which has been shown to be optimal among a broad class of iterative algorithms). On the other hand, no such statistical-computational gap exists for Fig.~\ref{Fig:sestic_vs_snr} where the signal is relatively ``dense'' (namely, the fraction of the nonzero components is relatively large). Similar phenomena have been observed in the context of spiked Wigner models \citep{deshpande2014information,montanari2021estimation}.

{The theoretical prediction of the asymptotic performance of OAMP requires the noise matrix to be rotationally-invaraint. The noise matrix used in the experiments for Fig.~\ref{Fig:Fixed_point_vs_snr} is structured but still synthetic. We now consider a more realistic model where the noise matrix is obtained from real datasets (while the signal is still randomly generated). Taking inspiration from \citep{zhong2022empirical},  we generate the noise matrix from the covariance matrix of real datasets in bioinformatics, namely, the 1000 Genomes Project (1000G) \cite{10002015global} and International HapMap
Project (Hapmap3) \cite{international2010integrating}. Both datasets undergo the preprocessing steps, producing a $2054\times 2054$ symmetric matrix for 1000G and a $1397\times 1397$ symmetric matrix for Hapmap3; see Appendix \ref{App:numerical_additional} for details of the data processing procedure.

\begin{figure}[htbp]
\begin{center}
\subfloat[Noise matrix derived from 1000G \cite{10002015global}. \label{Fig:Universality_1000G2}]{
\includegraphics[width=0.45\linewidth]{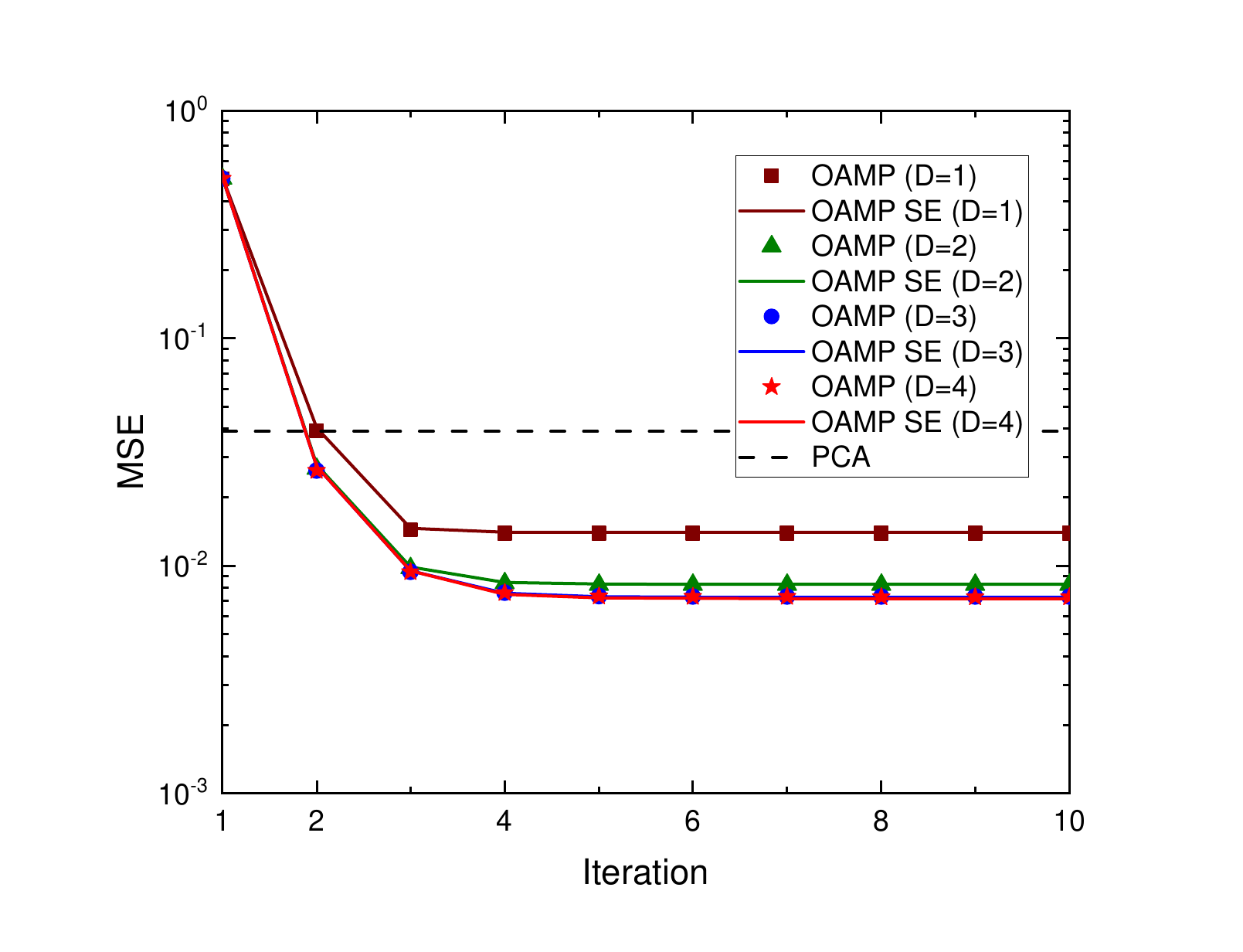}}
\subfloat[Noise matrix derived from Hapmap3 \cite{international2010integrating}. \label{Fig:Universality_Hapmap32}]{
\includegraphics[width=0.45\linewidth]{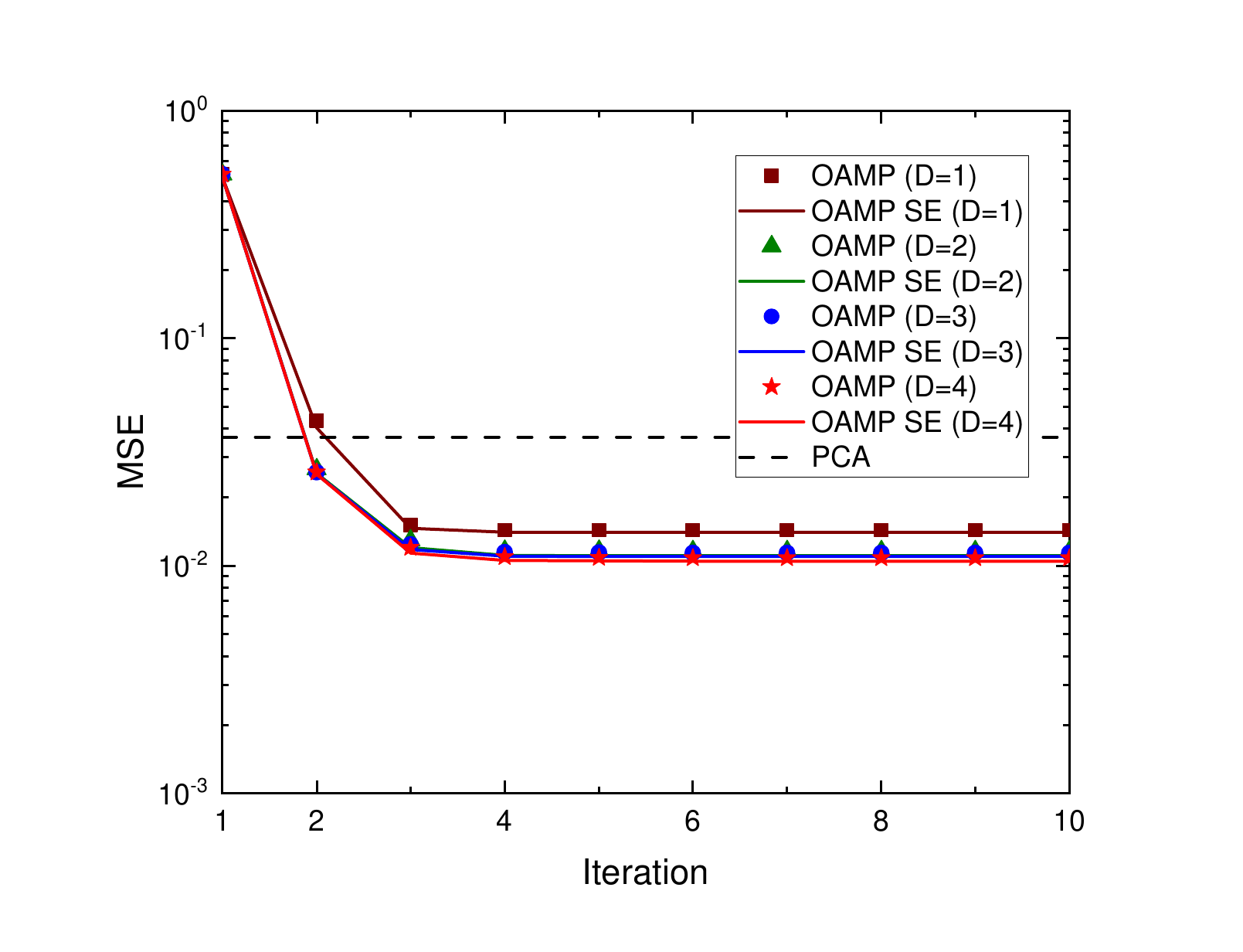} }
\end{center}
\caption{Performance of a data-driven implementation of the optimal degree-$\degree$ lifted OAMP algorithm on noise matrices derived from real datasets and signals randomly drawn from $\mathsf{X}_\star\sim 0.1\mathcal{N}(0,10)+0.9\delta_0$. The results are averaged over 30 random realizations.}
\label{Fig:Universality2}
\end{figure} 

Running the optimal OAMP algorithm requires estimating the spectral measure $\mu$ and the associated optimal matrix denoiser \eqref{eq:optimal-OAMP-denoiser} from the observed data. 
 While feasible, we found it more convenient to design a data-driven implementation of the optimal degree-$\degree$ lifted OAMP algorithm introduced in \eqref{eq:optimal-LOAMP}. Fig.~\ref{Fig:Universality2} also displays the asymptotic MSE predicted by the state evolution result for the corresponding rotationally invariant noise model with a matching spectrum. The performance of the lifted OAMP algorithm closely aligns with the state evolution prediction, suggesting an underlying universality phenomenon.}

\section{Conclusion and Future Work}
{In this paper, we introduced an OAMP algorithm for spiked matrix models with rotationally-invariant noise and demonstrated that it obeys a simple state evolution characterization in the high-dimensional limit. We showed that the optimal OAMP algorithm achieves the best performance among a broad class of iterative algorithms. 

Several promising directions for future research remain. As a first step, our analysis focused on a stylized symmetric rank-one model, and it would be valuable to extend these results to more practical settings, including asymmetric and multi-rank models. As discussed in Remark 1, in the absence of side information, if the signal is drawn from a zero-mean prior, no iterative algorithm that runs for a constant ($\dim$-independent) number of iterations can achieve non-trivial estimation error. An interesting direction for future work is to analyze OAMP algorithms with spectral initialization \citep{venkataramanan2022estimation,mondelli2021pca,zhong2021approximate}, or randomly initialized iterative algorithms that run for a diverging ($\dim$-dependent) number of iterations \citep{rush2018finite,li2022non,li2023approximate}. Finally, since the optimal OAMP algorithm relies on the knowledge of the signal prior and noise spectrum, it would be interesting to develop and analyze a practical procedure to estimate these parameters from the data \citep{zhong2022empirical}.}

\bibliographystyle{plainnat}

\bibliography{refs}

\appendix
\section{Preliminary Results}\label{App:pre}
\subsection{Random Matrix Theory}
This appendix collects some useful random matrix theory results. We present the proof of Lemma~\ref{lem:RMT} and a useful result on the concentration of quadratic forms involving rotationally invariant matrices (\factref{fact:qf}).  
\subsubsection{Proof of Lemma~\ref{lem:RMT}}\label{App:pre-RMT}
\begin{proof}[Proof of Lemma~\ref{lem:RMT}] Before presenting the proof of Lemma~\ref{lem:RMT}, we recall the definition of the Stieltjes transform of a probability measure. For any probability measure (or more generally, finite measure) $\chi$ on $\R$, the Stieltjes Transform $\stl_{\chi}: \C \backslash \R \mapsto \C$ of $\chi$ is defined as:
\begin{align*}
\stl_{\chi}(z) \explain{def}{=}  \int_{\R} \frac{\chi(\diff \lambda)}{z- \lambda} \quad \forall \; z \; \in \;  \C \backslash \R.
\end{align*}
It is well-known that a probability measure is uniquely determined by its Stieltjes transform (see for e.g., \citep[Theorem 2.4.3]{anderson2010introduction}). We prove each of the claims made in Lemma~\ref{lem:RMT} one by one.
\paragraph*{Weak convergence of $\nu_{\dim}$ (Claim (1))} Let $\tilde{\nu}_N$ be a slightly modified version of $\nu_N$:
\[
\tilde{\nu}_N = \frac{1}{N}\sum_{i=1}^N \left(\tilde{\bm{x}}_\star^\UT\bm{u}_i\right)^2\cdot \delta_{\lambda_i(\bm{Y})}
\]
where $\tilde{\bm{x}}_\star\bydef \frac{\sqrt{N}\bm{x}_\star}{\|\bm{x}_\star\|}$. This modification makes $\tilde{\nu}_N$ a probability measure, i.e., $\tilde{\nu}_N(\mathbb{R})=1$. Notice that for any bounded, continuous function $f: \R \mapsto \R$, we have:
\begin{align*}
\int_{\R} f(\lambda) \; \tilde{\nu}_{\dim}(\lambda) & = \frac{\dim}{\|\vx_{\star}\|^2} \cdot \int_{\R} f(\lambda) \; {\nu}_{\dim}(\lambda).
\end{align*}
By Assumption~\ref{assump:signal-noise}, $\|\vx_\star\|^2/N \pc 1$. Hence, it suffices to show that $\tilde{\nu}_N$ converges weakly to a probability measure $\nu$ in probability. To do so, it suffices to verify that the Stieltjes transform $\stl_{\tilde{\nu}_{\dim}}(z)$ of $\tilde{\nu}_{\dim}$ converges pointwise to the Stieltjes transform of a probability measure $\nu$ for any $z \in \C \backslash \R$ (see for e.g., \citep[Theorem 2.4.4]{anderson2010introduction}). To this end, we note that:
\begin{align*}
\stl_{\tilde{\nu}_N}(z) & = \frac{1}{\|\vx_{\star}\|^2} \sum_{i=1}^\dim \frac{\ip{\vu_i(\mY)}{\vx_{\star}}^2}{z- \lambda_i(\mY)} =  \frac{1}{\|\bm{x}_\star\|^2} \cdot \bm{x}_\star^\UT\left(z\bm{I}_N-\bm{Y}\right)^{-1}\bm{x}_\star,
\end{align*}
where $\lambda_{1:\dim}(\mY)$ and $\vu_{1:\dim}(\mY)$ denote the eigenvalues and eigenvectors of $\mY$. Hence,
\begin{align} \label{eq:hlb-nr-dr}
\stl_{\tilde{\nu}_N}(z)&= \frac{1}{\|\bm{x}_\star\|^2} \bm{x}_\star^\UT\left(z\bm{I}_N-\bm{Y}\right)^{-1}\bm{x}_\star =\frac{1}{\|\bm{x}_\star\|^2} \bm{x}_\star^\UT\left(z\bm{I}_N-\frac{\theta\bm{ x}_\star\bm{x}_\star^\UT}{N}-\bm{W}\right)^{-1}\bm{x}_\star\\
& \explain{(a)}{=}\frac{\frac{1}{\|\bm{x}_\star\|^2}\bm{x}_\star^\UT\left(z\bm{I}_N-\bm{W}\right)^{-1}\bm{x}_\star}{1-\frac{\theta}{N}\bm{x}_\star^\UT\left(z\bm{I}_N-\bm{W}\right)^{-1}\bm{x}_\star},
\end{align}
where step (a) uses the Sherman-Morrison formula for the rank-one update to the matrix inverse.  Using standard concentration results regarding quadratic forms of rotationally invariant matrices (stated as \factref{fact:qf} in \appref{appendix:misc} for convenience), we find that:
\begin{align} \label{eq:hlb-nr}
\frac{\bm{x}_\star^\UT\left(z\bm{I}_N-\bm{W}\right)^{-1}\bm{x}_\star}{\|\bm{x}_\star\|^2} & \pc  \int_{\R} \frac{\mu(\diff \lambda)}{z-\lambda} =  \stl_\mu(z)\quad \forall \;  z\in\mathbb{C}\backslash\mathbb{R}.
\end{align}
By Assumption~\ref{assump:signal-noise}, $\|\vx_\star\|^2/N \pc 1$. Hence, 
\begin{align} \label{eq:hlb-dr}
    1-\frac{\theta}{N}\bm{x}_\star^\UT\left(z\bm{I}_N-\bm{W}\right)^{-1}\bm{x}_\star& = 1- \theta \cdot \frac{\|\vx_{\star}\|^2}{\dim} \cdot \frac{\bm{x}_\star^\UT\left(z\bm{I}_N-\bm{W}\right)^{-1}\bm{x}_\star}{\|\vx_{\star}\|^2} \\
    &   \pc 1-\theta \stl_{\mu}(z) \explain{(a)}{\neq} 0 \quad \forall \;  z\in\mathbb{C}\backslash\mathbb{R},
\end{align} 
where the non-equality marked (a) follows by observing that $\Im(\stl_\mu(z))\neq0$ whenever $\mathrm{Im}(z)\neq0$, and so $1-\theta \stl_{\mu}(z) \neq0$. Plugging \eqref{eq:hlb-nr} and \eqref{eq:hlb-dr} into \eqref{eq:hlb-nr-dr}, we obtain:
\begin{equation}\label{Eqn:m_nu_final}
\stl_{\tilde{\nu}_N}(z)\pc \frac{\stl_{\mu}(z)}{1-\theta \stl_{\mu}(z)} \quad\forall \; z\in\mathbb{C}\backslash\mathbb{R}.
\end{equation}
To prove the weak convergence of $\tilde{\nu}_N$ to a \emph{probability measure} $\nu$, we need to show additionally that the RHS of the above equation is the Stieltjes transform of some probability measure. By \citep[Theorem 10, Chapter 3]{mingo2017free}, it is sufficient to verify that:
\begin{align} \label{eq:hlb-prob}
\limsup_{y\to\infty}y\left|\frac{\stl_{\mu}(\I y)}{1-\theta \stl_{\mu}(\I y)}\right|=1.
\end{align}
Since $\mu$ is a probability measure, we have $ y|\stl_\mu (\I y)|\to1$ and $\stl_{\mu}(\I y)\to 0$ as $y\to\infty$ (see \citep[Lemma 3, Chapter 3]{mingo2017free}). Hence, \eqref{eq:hlb-prob} holds. Consequently, the RHS of \eqref{Eqn:m_nu_final} is a Stieltjes transform of a probability measure $\nu$ and $\tilde{\nu}_{\dim}$ and $\nu_{\dim}$ converge weakly to $\nu$ in probability.

\paragraph*{Lebesgue Decomposition of $\nu$ (Claims (2) and (3))} \citet[Lemma 2.17]{belinschi2008lebesgue} has shown that for any probability measure $\chi$ on $\R$ with Lebesgue decomposition $\chi = \chi_{\parallel} + \chi_{\perp}$, the Stieltjes transform of $\nu$ satisfies:
\begin{align}
\lim_{\epsilon \downarrow 0} \Im[\stl_{\chi}(\lambda+\I \epsilon)] &= -\infty \quad \text{for $\chi_{\perp}$-almost every $\lambda \in \R$}, \label{eq:pdf-from-hlb0}\\
\lim_{\epsilon \downarrow 0} \Im[\stl_{\chi}(x+\I \epsilon)] & = - \pi \frac{\diff \chi_{\parallel}}{ \diff \lambda}(\lambda) \quad \text{for Lebesgue-almost every $\lambda \in \R$}. \label{eq:pdf-from-hlb}
\end{align}
In the display above, $\frac{\diff \chi_{\parallel}}{ \diff \lambda}$ denotes the density of $\chi_{\parallel}$ with respect to the Lebesgue measure and the limit on the LHS of \eqref{eq:pdf-from-hlb} is guaranteed to exist and is finite for Lebesgue-almost every $\lambda$. In light of this result, we study the limit:
\begin{align*}
\lim_{\epsilon \downarrow 0} \Im[\stl_{\nu}(\lambda+\I \epsilon)]
\end{align*}
for an arbitrary $\lambda \in \R$. For any $\epsilon > 0$, we have:
\begin{align*}
\Im[\stl_{\nu}(\lambda+\I \epsilon)] & = \Im \left[ \frac{\stl_{\mu}(\lambda + \I \epsilon)}{1-\theta \stl_{\mu}(\lambda+ \I \epsilon)} \right] = \frac{\Im[\stl_{\mu}(\lambda + \I \epsilon)]}{(1-\theta \Re[\stl_{\mu}(\lambda + \I \epsilon)])^2 + \theta^2\Im^2[\stl_{\mu}(\lambda + \I \epsilon)]}.
\end{align*}
Since $\mu$ is absolutely continuous with respect to the Lebesgue measure (Assumption~\ref{assump:signal-noise}), \eqref{eq:pdf-from-hlb} implies:
\begin{align*}
\Im[\stl_{\mu}(\lambda + \I \epsilon)] \rightarrow - \pi \mu(\lambda) \quad \text{ as } \epsilon   \downarrow 0.
\end{align*}
where $\mu(\cdot)$ denotes the density of $\mu$. Moreover, by the Sohotsky-Plemelj formula (see for instance, \citep[Section 2.1]{pastur2011eigenvalue}),
\begin{align*}
 \Re[\stl_{\mu}(\lambda + \I \epsilon)] \rightarrow \pi \hlb_{\mu}(\lambda) \quad \text{ as } \epsilon \downarrow 0.
\end{align*}
Recall that:
\begin{align*}
\phi(\lambda) \bydef  (1-\theta \pi \hlb_{\mu}(\lambda))^2 + \theta^2 \pi^2 \cdot \mu^2(\lambda).
\end{align*}
Hence, 
\begin{align*}
\lim_{\epsilon \downarrow 0} \Im[\stl_{\nu}(\lambda+\I \epsilon)] & = -\infty \quad \text{ is possible only when } \phi(\lambda) = 0, \\
\lim_{\epsilon \downarrow 0} \Im[\stl_{\nu}(\lambda+\I \epsilon)] & = -\frac{\pi \mu(\lambda)}{\phi(\lambda)} \in (-\infty, 0] \quad \text{ iff } \phi(\lambda) \neq 0.
\end{align*}
Combining the above results with \eqref{eq:pdf-from-hlb0} and \eqref{eq:pdf-from-hlb}, we conclude that:
\begin{align*}
\phi(\lambda) &\neq 0 \quad \text{ for Lebesgue-almost every $\lambda \in \R$.}, \\
\phi(\lambda) & = 0 \quad \text{ for $\nu_{\perp}$-almost every $\lambda \in \R$}, \\
\frac{\diff \nu_{\parallel}}{\diff \lambda}(\lambda) &= \frac{\mu(\lambda)}{\phi(\lambda)} \quad \text{ for Lebesgue-almost every $\lambda \in \R$,}
\end{align*}
as claimed. This concludes the proof of the lemma.
\end{proof}

\subsubsection{Concentration of Quadratic Forms} \label{appendix:misc}

\begin{fact}[{\citealt[Proposition 9.3]{benaych2011eigenvalues}}] \label{fact:qf} Let $\mW$ be a $\dim \times \dim$ random noise matrix which satisfies Assumption~\ref{assump:signal-noise} and let $\vu$ be a $\dim$-dimensional random vector which is independent of $\mW$ and satisfies: $\|\vu\|^2/\dim \pc r$ as $\dim \rightarrow \infty$. Then, for any continuous function $f: \R \mapsto \R$,
\begin{align*}
\frac{\vu^\top f(\mW) \vu }{\dim} \pc r \cdot \E[f(\serv{\Lambda})] \quad \serv{\Lambda} \sim \mu,
\end{align*}
where $\mu$ is the limiting spectral distribution of $\mW$.
\end{fact}

\begin{proof}
Recall from Assumption~\ref{assump:signal-noise} that the eigen-decomposition of $\mW$ is given by:
\begin{align*}
\mW & = \mU \cdot \diag( \lambda_1(\mW), \dotsc, \lambda_{\dim}(\mW)) \cdot \mU^\top,
\end{align*}
where the matrix of eigenvectors $\mU \sim \unif{\ortho(\dim)}$ is a Haar-distributed random orthogonal matrix independent of the eigenvalues $ \lambda_1(\mW), \dotsc, \lambda_{\dim}(\mW)$.
Defining  $\bm{v}\explain{def}{=}\mU^\UT\vu/\|\vu\|\sim\text{Unif}(\mathbb{S}^{N-1})$, we find that:
\begin{align*}
    \frac{\vu^\top f(\mW) \vu }{\dim}  &= \frac{\|\vu\|^2}{\dim} \cdot \sum_{i=1}^N  v_i^2 \cdot f(\lambda_i(\mW)).
\end{align*}
Since $\bm{v}\sim\text{Unif}(\mathbb{S}^{N-1})$, by \citep[Proposition 9.3]{benaych2011eigenvalues} we have that:
\begin{align*}
\sum_{i=1}^N  v_i^2 \cdot f(\lambda_i(\mW))  \pc \E[f(\serv{\Lambda})], \quad \serv{\Lambda} \sim \mu.
\end{align*}
The claim follows by combining the above conclusion with the hypothesis $\|\vu\|^2/\dim \pc r$ using Slutsky's theorem. 
\end{proof}

\subsection{Gaussian Channels} \label{appendix:gauss-channel}
We state and prove some important properties of Gaussian channels (Definition~\ref{def:gauss-channel}) in this appendix. 

\paragraph*{The DMMSE Estimator for Scalar Gaussian Channels} The following lemma, due to \citet{ma2017orthogonal} provides a formula for the DMMSE estimator for a scalar Gaussian channel with SNR $\omega$. We provide a proof of this result in \sref{sec:dmmse-scalar} for completeness.  
\begin{lemma}[\citet{ma2017orthogonal}] \label{lem:dmmse-scalar} Let $(\serv{X}_{\star}, \serv{X} ; \serv{A})$ be a scalar Gaussian channel with SNR $\omega \in [0,1]$. Then,
\begin{enumerate}
    \item The function $\dfbdnsr(\cdot | \omega) : \R^{1+\auxdim} \mapsto \R$ ($\forall \; x \in \R, \; \aux \in \R^{\auxdim}$):
\begin{align} \label{eq:dmmse-scalar-recall}
\dfbdnsr (x; \aux| \omega) &\explain{def}{=}  \begin{cases} 
 \left(1 - \frac{\sqrt{\omega}}{\sqrt{1-\omega}} \cdot \E[\serv{Z}\bdnsr({\serv{X};\serv{A} | \omega)]}   \right)^{-1} \cdot \left(\bdnsr(x;a | \omega) - \frac{\E[\serv{Z}\bdnsr({\serv{X};\serv{A}} | \omega)]}{\sqrt{1-\omega}} \cdot x \right) & : \omega < 1 \\ x &: \omega = 1 \end{cases}
\end{align}
is the DMMSE estimator for $(\serv{X}_{\star}, \serv{X}; \serv{A})$.
\item The DMMSE estimator satisfies the identities:
\begin{align*}
\E[\serv{X}_{\star} \dfbdnsr(\serv{X}; \serv{A} | \omega)] & = \E[\bdnsr(\serv{X}; \serv{A} | \omega) \dfbdnsr(\serv{X}; \serv{A} | \omega)] =   \E[\dfbdnsr(\serv{X}; \serv{A} | \omega)^2] = 1 - \dmmse(\omega). 
\end{align*}
\item When $\omega < 1$, the $\mmse$ and $\dmmse$ functions are related by the identity:
\begin{align*}
\frac{1}{\dmmse(\omega)} = \frac{1}{\mmse(\omega)} - \frac{\omega}{1-\omega}.
\end{align*}
\end{enumerate}
\end{lemma}

\paragraph*{MMSE and DMMSE Estimators for General Gaussian Channels} Next, we provide formulas for the MMSE and DMMSE estimators for general (multivariate) Gaussian channels. Consider a Gaussian channel $(\serv{X}_{\star}, \serv{X}_1, \dotsc, \serv{X}_t ; \serv{A})$:
\begin{align*}
(\serv{X}_{\star}, \serv{A}) \sim \pi, \quad (Z_1, \dotsc, Z_t) \sim \gauss{0}{\Sigma}, \quad \serv{X}_i = \alpha_i \cdot \serv{X}_{\star} + Z_i \quad \forall\;  i \; \in \; [t],
\end{align*}
where the Gaussian noise $(Z_1, \dotsc, Z_t)$ is sampled independently of the signal and side information $(\serv{X}_{\star}; \serv{A})$. Observe that for any vector $v \in \R^t$ which satisfies:
\begin{align} \label{eq:variance-normalization}
\ip{v}{\alpha}^2 + v^\top \Sigma v & = 1,
\end{align}
$(\serv{X}_{\star}, \ip{v}{\serv{X}_{\leq t}} ; \serv{A})$ forms a scalar Gaussian channel with SNR $\omega = \ip{v}{\alpha}^2$. Among all vectors $v \in \R^t$ which satisfy \eqref{eq:variance-normalization}, the maximum SNR for the corresponding scalar channel is achieved by the vector:
\begin{subequations} \label{eq:scalarization}
    \begin{align} \label{eq:optimal-weights}
\optlin(\serv{X}_{\star}| \serv{X}_1, \dotsc, \serv{X}_t; \serv{A}) \explain{def}{=} \begin{cases} \left( \alpha^\top \Sigma^{\dagger} \alpha + (\alpha^\top \Sigma^{\dagger} \alpha)^2  \right)^{-\frac{1}{2}} \cdot \Sigma^{\dagger} \alpha & \text{ if } \alpha \in \mathrm{Range}(\Sigma), \\
\ip{P_{\perp}[\alpha]}{\alpha}^{-1} \cdot P_{\perp}[\alpha] & \text{ if } \alpha \notin \mathrm{Range}(\Sigma),
\end{cases}
\end{align}
where $\Sigma^{\dagger}$ denotes the pseudo-inverse of $\Sigma$, $\mathrm{Range}(\Sigma)$ is the range space of $\Sigma$, and $P_{\perp}[\cdot]$ denotes the projector on the orthogonal complement of $\mathrm{Range}(\Sigma)$. We will refer to $\optlin(\serv{X}_{\star}| \serv{X}_1, \dotsc, \serv{X}_t; \serv{A})$ as the \emph{optimal linear combination} for the Gaussian channel $(\serv{X}_{\star}, \serv{X}_1, \dotsc, \serv{X}_t ; \serv{A})$. The maximum SNR is given by:
\begin{align} \label{eq:effective-snr}
\effsnr(\serv{X}_{\star}| \serv{X}_1, \dotsc, \serv{X}_t; \serv{A}) &\explain{def}{=} \ip{\alpha}{\optlin(\serv{X}_{\star}| \serv{X}_1, \dotsc, \serv{X}_t; \serv{A})}^2 \\
\nonumber& =  \begin{cases} \left( 1 + \alpha^\top \Sigma^{\dagger} \alpha \right)^{-1} \cdot \alpha^\top \Sigma^{\dagger} \alpha & \text{ if } \alpha \in \mathrm{Range}(\Sigma), \\
1 & \text{ if } \alpha \notin \mathrm{Range}(\Sigma),
\end{cases}
\end{align}
\end{subequations}
which we call the \emph{effective SNR} of the Gaussian channel. The following lemma shows that the MMSE and DMMSE of the Gaussian channel $(\serv{X}_{\star}, \serv{X}_1, \dotsc, \serv{X}_t ; \serv{A})$  is same as the MMSE and DMMSE of the scalar Gaussian channel $(\serv{X}_{\star}, \ip{\serv{X}_{\leq t}}{\optlin}; \serv{A})$, which operates at the SNR $\effsnr$.

\begin{lemma} \label{lem:scalarization-mmse} Let $(\serv{X}_{\star}, \serv{X}_1, \dotsc, \serv{X}_t ; \serv{A})$ be a Gaussian channel. Let $\optlin$ and $\effsnr$ denote the optimal linear combination and effective SNR for the Gaussian channel, as defined in \eqref{eq:scalarization}. Then:
\begin{enumerate}
\item $\bmmse(\serv{X}_{\star} | \serv{X}_1, \dotsc, \serv{X}_t, \serv{A}) = \mmse(\effsnr)$ and the function:
\begin{align*}
f_{\star}(x ; \aux) \explain{def}{=} \bdnsr(\ip{x}{\optlin}; \aux |\effsnr) \quad \forall \; x \in \R^t, \quad \aux \in \R^{\auxdim}, 
\end{align*}
is the MMSE estimator:
\begin{align*}
f_{\star} & \in \arg  \min_{f \in L^2(\serv{X}_1, \dotsc, \serv{X}_t; \serv{A})} \E[\{\serv{X}_\star - f(\serv{X}_1, \dotsc, \serv{X}_t; \serv{A})\}^2].
\end{align*}
\item $\bdmmse(\serv{X}_{\star} | \serv{X}_1, \dotsc, \serv{X}_t, \serv{A}) = \dmmse(\effsnr)$ and the function:
\begin{align*}
\dfnew{f}_{\star}(x ; \aux) \explain{def}{=} \dfbdnsr(\ip{x}{\optlin}; \aux |\effsnr) \quad \forall \; x \in \R^t, \quad \aux \in \R^{\auxdim}, 
\end{align*}
is the DMMSE estimator:
\begin{align*}
\dfnew{f}_{\star} \in &\arg\min_{f \in L^2(\serv{X}_1, \dotsc, \serv{X}_t; \serv{A})} \E[\{\serv{X}_\star - f(\serv{X}_1, \dotsc, \serv{X}_t; \serv{A})\}^2] \\
& \text{subject to } \E[\serv{Z}_i f(\serv{X}_1, \dotsc, \serv{X}_t; \serv{A})] = 0 \; \forall \; i \; \in \; [t].
\end{align*}
\end{enumerate}
\end{lemma}
\begin{proof}
The proof is deferred to \sref{sec:scalarization-mmse}.
\end{proof}

\paragraph*{Monotonicity of MMSE and DMMSE} Lastly, we will rely on the following natural monotonicity property of the $\mmse$ function.
\begin{fact}[{\citealt[Proposition 9]{guo2011estimation}}] \label{fact:mmse-monotonic} The function $\mmse:[0,1] \mapsto [0,1]$ is a continuous, non-increasing function. If $\mmse(0) = \E[\Var[\serv{X}_{\star}|\serv{A}]] \neq 0$, $\mmse:[0,1] \mapsto [0,1]$ is strictly decreasing. 
\end{fact}
The following lemma, whose proof is provided in \sref{sec:dmmse-monotonicity}, shows that the $\dmmse$ function is also non-decreasing.
\begin{lemma}\label{lem:dmmse-monotonicity} The function $\dmmse: [0,1] \mapsto [0,1]$ is non-increasing. Moreover, $\dmmse(\cdot)$ is continuous on the domain $[0,1)$.
\end{lemma}

The remainder of this section is devoted to the proofs of \lemref{lem:dmmse-scalar}, \lemref{lem:scalarization-mmse}, and \lemref{lem:dmmse-monotonicity} introduced above. 

\subsubsection{Proof of \lemref{lem:dmmse-scalar}} \label{sec:dmmse-scalar}
\begin{proof}[Proof of \lemref{lem:dmmse-scalar}]
This result is due to \citet{ma2017orthogonal}, we provide a proof here for completeness. We consider a scalar Gaussian channel at SNR $\omega$:
\begin{align} \label{eq:AWGN-model}
(\serv{X}_{\star}; \serv{A}) \sim \pi, \quad  \serv{Z} \sim \gauss{0}{1}, \quad  \serv{X} = \sqrt{\omega} \cdot \serv{X}_{\star} + \sqrt{1-\omega} \cdot \serv{Z}, 
\end{align}
where the Gaussian noise $\serv{Z}$ is sampled independently of the signal and side information $(\serv{X}_{\star}; \serv{A})$. Recall from Definition~\ref{def:gauss-channel} that:
\begin{align} \label{eq:DMMSE-recall}
\dmmse(\omega)& \explain{def}{=} \bdmmse(\serv{X}_\star | \serv{X}_1, \dotsc, \serv{X}_t; \serv{A}) \\
\nonumber&\explain{def}{=} \min_{f \in L^2(\serv{X}; \serv{A})} \E[\{\serv{X}_\star - f(\serv{X}; \serv{A})\}^2] \quad \text{subject to } \E[\serv{Z} f(\serv{X}; \serv{A})] = 0.
\end{align}
Moreover, the function that achieves the lowest MSE in \eqref{eq:DMMSE-recall} is the DMMSE estimator for the channel. Let $f$ be any function in $L^2(\serv{X}; \serv{A})$ which satisfies the divergence-free requirement $\E[\serv{Z} f(\serv{X}; \serv{A})] = 0$. We can expand the mean squared error (MSE) of $f$ as follows:
\begin{align} \label{eq:bias-variance}
\E[(\serv{X}_{\star} - f(\serv{X}; \serv{A}))^2] & = 1 + \E[ f^2(\serv{X}; \serv{A})] - 2\E[\serv{X}_\star  f(\serv{X}; \serv{A})] \\ & =1 + \E[ f^2(\serv{X}; \serv{A})] - 2\E[\bdnsr(\serv{X}; \serv{A})  f(\serv{X}; \serv{A})], \nonumber
\end{align}
where we used the tower property in the last step. For convenience, we define:
\begin{align} \label{eq:hat-psi-def}
\beta \explain{def}{=} \frac{\E[\serv{Z} \bdnsr(\serv{X};\serv{A} | \omega)]}{\sqrt{1-\omega}}, \quad \hat{\bdnsr}(x;\aux| \omega) \explain{def}{=} \bdnsr(x;\aux| \omega) - \beta x.
\end{align}
Observe that $\hat{\bdnsr}$ satisfies the divergence-free requirement:
\begin{align} \label{eq:hat-psi-div-free}
\E[\serv{Z}\hat{\bdnsr}(\serv{X}; \serv{A}| \omega)] & = 0. 
\end{align}
We consider the cross-term in \eqref{eq:bias-variance}:
\begin{align}
\E[\bdnsr(\serv{X}; \serv{A}| \omega)  f(\serv{X}; \serv{A})] & = \E[\hat{\bdnsr}(\serv{X}; \serv{A}| \omega)  f(\serv{X}; \serv{A})] + \beta \E[\serv{X} f(\serv{X}; \serv{A})] \nonumber\\
& \explain{(a)}{=} \E[\hat{\bdnsr}(\serv{X}; \serv{A}| \omega)  f(\serv{X}; \serv{A})] + \beta \sqrt{\omega} \cdot \E[\serv{X}_{\star} f(\serv{X}; \serv{A})] \nonumber\\
& \explain{(b)}{=}  \E[\hat{\bdnsr}(\serv{X}; \serv{A}| \omega)  f(\serv{X}; \serv{A})] + \beta \sqrt{\omega} \cdot \E[\bdnsr(\serv{X}; \serv{A}| \omega) f(\serv{X}; \serv{A})]. \label{eq:cross-term-simplification}
\end{align}
In the above display step (a) follows from the fact that $\E[\serv{Z} f(\serv{X}; \serv{A})] = 0$ and step (b) uses the tower property. We claim that:
\begin{align} \label{eq:non-degen-claim}
    \omega < 1 \implies \beta {\sqrt{\omega}} \neq 1. 
\end{align}
Assuming the above claim, we can split our analysis into two cases. 
\paragraph*{Case 1: $\omega < 1$} We first prove each of the claims made in the lemma assuming $\omega < 1$.
\begin{description}
    \item[Proof of Claim 1.] Notice that \eqref{eq:non-degen-claim} allows us to rearrange \eqref{eq:cross-term-simplification} to obtain:
    \begin{align} \label{eq:dmmse-est-sub1}
    \E[\bdnsr(\serv{X}; \serv{A}| \omega)  f(\serv{X}; \serv{A})] & = \frac{\E[\hat{\bdnsr}(\serv{X}; \serv{A}| \omega)  f(\serv{X}; \serv{A})]}{1-\beta \sqrt{\omega}} = \E[{\dfbdnsr}(\serv{X}; \serv{A}| \omega)  f(\serv{X}; \serv{A})],
    \end{align}
    where we recalled from the statement of the lemma that $\dfbdnsr$ was defined as: 
    \begin{align}
    \dfbdnsr (x; \aux| \omega) &\explain{def}{=}  
 \left(1 - \frac{\sqrt{\omega}}{\sqrt{1-\omega}} \cdot \E[\serv{Z}\bdnsr({\serv{X};\serv{A} | \omega)]}   \right)^{-1} \cdot \left(\bdnsr(x;a | \omega) - \frac{\E[\serv{Z}\bdnsr({\serv{X};\serv{A}} | \omega)]}{\sqrt{1-\omega}} \cdot x \right) \nonumber   \\ &\explain{\eqref{eq:hat-psi-def}}{=} \frac{\bdnsr(x;a | \omega) - \beta x}{1-\beta \sqrt{\omega}} \explain{\eqref{eq:hat-psi-def}}{=}  \frac{\hat{\bdnsr}(x;a | \omega) }{1-\beta \sqrt{\omega}} .\label{eq:psi-bar-recall}
    \end{align}
    Substituting the formula \eqref{eq:dmmse-est-sub1} in \eqref{eq:bias-variance} and completing the square yields the following MSE decomposition:
    \begin{align} \label{eq:dmmse-est-sub2}
    \E[(\serv{X}_{\star} - f(\serv{X}; \serv{A}))^2] & = 1 - \E[\dfbdnsr(\serv{X}; \serv{A})^2] + \E[\{f(\serv{X}; \serv{A}) - \dfbdnsr(\serv{X};\serv{A})\}^2]
    \end{align}
    Since $\dfbdnsr$ satisfies the divergence-free requirement $\E[\serv{Z}\dfbdnsr(\serv{X};\serv{A})] = 0$ (cf. \eqref{eq:hat-psi-div-free}), the above decomposition shows that $\dfbdnsr$ is the DMMSE estimator. 
    \item[Proof of Claim 2.] Setting $f = \dfbdnsr(\cdot \; ; \cdot | \omega)$ in \eqref{eq:dmmse-est-sub1} and \eqref{eq:dmmse-est-sub2} yields:
    \begin{align*}
    \E[\bdnsr(\serv{X}; \serv{A}| \omega) \dfbdnsr(\serv{X}; \serv{A}| \omega)] = \E[\dfbdnsr^2(\serv{X}; \serv{A}| \omega)] = 1 - \E[(\serv{X}_{\star} - \dfbdnsr(\serv{X}; \serv{A}| \omega))^2] =1 - \dmmse(\omega).
    \end{align*}
    By the Tower property, $\E[\serv{X}_{\star}\dfbdnsr(\serv{X}; \serv{A}| \omega)]= \E[\bdnsr(\serv{X}; \serv{A}| \omega) \dfbdnsr(\serv{X}; \serv{A}| \omega)]$. Hence,
    \begin{align} \label{eq:scalar-dmmse-claim2}
    \E[\serv{X}_{\star}\dfbdnsr(\serv{X}; \serv{A}| \omega)] = \E[\bdnsr(\serv{X}; \serv{A}| \omega) \dfbdnsr(\serv{X}; \serv{A}| \omega)] = \E[\dfbdnsr^2(\serv{X}; \serv{A}| \omega)]  =1 - \dmmse(\omega),
    \end{align}
    which is the second claim made in the lemma. 
    \item[Proof of Claim 3.] Notice from \eqref{eq:psi-bar-recall} that the MMSE and DMMSE estimators are related by:
    \begin{align} \label{eq:mmse-dmmse-denoiser-connection}
    \bdnsr(x; \aux | \omega) = (1-\beta \sqrt{\omega}) \cdot \dfbdnsr(x; \aux | \omega) + \beta x. 
    \end{align}
    Hence, we can relate $\mmse(\omega)$ and $\dmmse(\omega)$ as follows:
    \begin{align}
    \mmse(\omega) = \E[(\serv{X}_{\star} - \bdnsr(\serv{X}; \serv{A} | \omega))^2] &= \E[\serv{X}_{\star}^2] - 2 \E[\serv{X}_{\star}\bdnsr(\serv{X}; \serv{A} | \omega)] + \E[\bdnsr(\serv{X}; \serv{A} | \omega)^2] \nonumber \\
    & \explain{(a)}{=}1 - \E[\serv{X}_{\star}\bdnsr(\serv{X}; \serv{A} | \omega)] \label{eq:mmse-pythagoras} \\
    & \explain{\eqref{eq:mmse-dmmse-denoiser-connection}}{=} 1 - (1-\beta \sqrt{\omega}) \cdot \E[\serv{X}_{\star}\dfbdnsr(\serv{X}; \serv{A} | \omega)] - \beta \E[\serv{X}_{\star} \serv{X}] \nonumber \\
    & = (1-\beta \sqrt{\omega}) \cdot (1- \E[\serv{X}_{\star}\dfbdnsr(\serv{X}; \serv{A} | \omega)]) \nonumber \\
    & \explain{\eqref{eq:scalar-dmmse-claim2}}{=} (1-\beta \sqrt{\omega}) \cdot \dmmse(\omega). \label{eq:sub-beta-expresssion-here}
    \end{align}
    In the above display, step (a) relies on the assumption $\E[\serv{X}_{\star}^2] = 1$ (cf. Assumption~\ref{assump:signal-noise}) and the tower property $\E[\serv{X}_{\star}\bdnsr(\serv{X}; \serv{A} | \omega)] = \E[\bdnsr(\serv{X}; \serv{A} | \omega)^2]$.
    Moreover, notice that:
    \begin{align*}
    \beta &\explain{def}{=} \frac{\E[\serv{Z} \bdnsr(\serv{X};\serv{A} | \omega)]}{\sqrt{1-\omega}}\explain{\eqref{eq:AWGN-model}}{=} \frac{\E[\serv{X} \bdnsr(\serv{X};\serv{A} | \omega)]- \sqrt{\omega}\E[\serv{X}_{\star} \bdnsr(\serv{X};\serv{A} | \omega)]}{{1-\omega}}  \\ &\explain{(a)}{=} \frac{\E[\serv{X} \serv{X}_{\star}]- \sqrt{\omega}\E[\serv{X}_{\star} \bdnsr(\serv{X};\serv{A} | \omega)]}{{1-\omega}} = \frac{\sqrt{\omega}(1-\E[\serv{X}_{\star} \bdnsr(\serv{X};\serv{A} | \omega)])}{1-\omega}  \explain{\eqref{eq:mmse-pythagoras}}{=} \frac{\sqrt{\omega} \cdot \mmse(\omega)}{1-\omega}.
    \end{align*}
    In the above display, step (a) again uses the tower property $\E[\serv{X}_{\star}\bdnsr(\serv{X}; \serv{A} | \omega)] = \E[\bdnsr(\serv{X}; \serv{A} | \omega)^2]$. Finally, we substitute the formula for $\beta$ obtained in the above display in \eqref{eq:sub-beta-expresssion-here} to obtain:
    \begin{align*}
    \frac{1}{\dmmse(\omega)} & = \frac{1}{\mmse(\omega)} - \frac{\omega}{1-\omega}
    \end{align*}
    after rearrangement. This is exactly the third claim made in the lemma. 
\end{description} 
    
\paragraph*{Case 2: $\omega = 1$} When $\omega = 1$, $\serv{X} = \serv{X}_{\star}$. The estimator $f(\serv{X}; \serv{A}) = \serv{X}$ has zero MSE and satisfies the divergence-free requirement $\E[\serv{Z} \serv{X}] = \E[\serv{Z}] \E[\serv{X}_\star] = 0$. Hence, $\dfbdnsr(x;\aux) \explain{def}{=} x$ is the DMMSE estimator. Moreover, the identity (which is the second claim made in the lemma):
\begin{align*}
 \E[\serv{X}_{\star}\dfbdnsr(\serv{X}; \serv{A}| \omega)] = \E[\bdnsr(\serv{X}; \serv{A}| \omega) \dfbdnsr(\serv{X}; \serv{A}| \omega)] = \E[\dfbdnsr^2(\serv{X}; \serv{A}| \omega)]  =1 - \dmmse(\omega)
\end{align*}
holds trivially in this case since $\serv{X}_{\star} = \dfbdnsr(\serv{X}; \serv{A}| \omega) = \serv{X}$, $\dmmse(1) = 0$, and $\E[\serv{X}_{\star}^2] = 1$.

To finish the proof of the lemma, we need to prove the claim in \eqref{eq:non-degen-claim}. 
\paragraph*{Proof of \eqref{eq:non-degen-claim}} We prove the contrapositive of \eqref{eq:non-degen-claim}: $\beta \sqrt{\omega} = 1 \implies \omega = 1$. When $\beta \sqrt{\omega} = 1$, \eqref{eq:cross-term-simplification} shows that for any $f \in L^2(\serv{X}; \serv{A})$ which satisfies the divergence-free requirement $\E[\serv{Z} f(\serv{X}; \serv{A})] = 0$, we have $\E[\hat{\bdnsr}(\serv{X}; \serv{A})  f(\serv{X}; \serv{A})] = 0$. Taking $f = \hat{\bdnsr}$ (cf. \eqref{eq:hat-psi-div-free}), we conclude that $\E[\hat{\bdnsr}^2] = 0$. Recalling the definition of $\hat{\bdnsr}$, we obtain the following formula for the MMSE estimator for the scalar Gaussian channel:
    \begin{align*}
    {\bdnsr}(\serv{X}; \serv{A}) = \beta \serv{X} = \frac{\serv{X}}{\sqrt{\omega}},
    \end{align*}
    which gives us the following formula for the MMSE:
    \begin{align*}
    \bmmse(\serv{X}_{\star} | \serv{X} ; \serv{A}) & = \frac{1-\omega}{\omega}.
    \end{align*}
    Since $\bmmse(\serv{X}_{\star} | \serv{X} ; \serv{A})$ is bounded by the MSE of the linear estimator $\sqrt{\omega} \cdot \serv{X}$ (this is the linear estimator with the lowest MSE), 
    \begin{align*}
    \bmmse(\serv{X}_{\star} | \serv{X} ; \serv{A})  = \frac{1-\omega}{\omega} \leq \E[(\sqrt{\omega} \serv{X} - \serv{X}_{\star})^2] = 1- \omega.
    \end{align*}
    Since $\omega \in [0,1]$, by rearranging the above inequality 
    we conclude that $\omega = 1$.  
\end{proof}

\subsubsection{Proof of \lemref{lem:scalarization-mmse}} \label{sec:scalarization-mmse}
\begin{proof}[Proof of \lemref{lem:scalarization-mmse}]
It will be convenient to define the random vector $\serv{X}_{\leq t} \explain{def}{=} (\serv{X}_1, \dotsc, \serv{X}_t)^\top$ and the random variable $\serv{S}  \explain{def}{=} \ip{\serv{X}_{\leq t}}{\optlin}$. 
Observe that the conditional joint distribution of $(\serv{X}_{\leq t}, \serv{S})$ given $(\serv{X}_{\star}; \serv{A})$ is:
\begin{align*}
\begin{bmatrix}
\serv{X}_{\leq t}\\ \serv{S}
\end{bmatrix} \; \Big| \; (\serv{X}_{\star}; \serv{A}) \sim \gauss{\serv{X}_{\star}\cdot \begin{bmatrix} \alpha \\ \sqrt{\effsnr} \end{bmatrix}}{\begin{bmatrix} \Sigma  &\Sigma \optlin \\  \optlin^\top \Sigma &\, {1-\effsnr} \end{bmatrix}}.
\end{align*}
Using Gaussian conditioning, we find that the conditional distribution of $\serv{X}_{\leq t}$ given $(\serv{S}, \serv{X}_\star; \serv{A})$ is given by:
\begin{align*}
\serv{X}_{\leq t} \; | \; (\serv{S}, \serv{X}_\star; \serv{A}) \sim \gauss{\frac{\serv{ S\alpha}}{\sqrt{\effsnr}} }{\Sigma - \frac{1-\effsnr}{\effsnr} \cdot \alpha \alpha^\top}.
\end{align*}
Since the RHS does not have any dependence on $(\serv{X}_\star; \serv{A})$, this implies that $\serv{X}_{\leq t}$ and $(\serv{X}_{\star}; \serv{A})$ are conditionally independent given $\serv{S}$:
\begin{align} \label{eq:cond-indp}
    \serv{X}_{\leq t} \perp (\serv{X}_{\star}; \serv{A}) \; | \; \serv{S}.
\end{align}
\paragraph*{Analysis of MMSE} We first prove the MMSE formula and identify the MMSE estimator. To do so, we begin by showing the lower bound:
\begin{align*}
\bmmse(\serv{X}_{\star} | \serv{X}_{\leq t}; \serv{A}) = \min_{f \in L^2(\serv{X}_{\leq t}; \serv{A})} \E[\{\serv{X}_\star - f(\serv{X}_{\leq t}; \serv{A})\}^2] \geq \mmse(\effsnr).
\end{align*}
To do so, we consider any $f \in L^2(\serv{X}_{\leq t}; \serv{A})$:
\begin{align*}
\E[\{\serv{X}_\star - f(\serv{X}_{\leq t}; \serv{A})\}^2] & \explain{(a)}{\geq}  \E[\{\serv{X}_\star - \E[f(\serv{X}_{\leq t}; \serv{A})| \serv{S}, \serv{X}_{\star},\serv{A}]\}^2] \\
& \explain{\eqref{eq:cond-indp}}{=} \E[\{\serv{X}_\star - \E[f(\serv{X}_{\leq t}; \serv{A})| \serv{S}, \serv{A}]\}^2] \\
& \explain{(b)}{\geq} \min_{g \in L^2(\serv{S};\serv{A})} \E[\{\serv{X}_\star - g(\serv{S}; \serv{A})\}^2] \\
& = \bmmse(\serv{X}_{\star}|\serv{S}; \serv{A}) \\
& \explain{(c)}{=} \mmse(\effsnr),
\end{align*}
where step (a) used Jensen's inequality, step (b) follows by observing $ \E[f(\serv{X}_{\leq t}; \serv{A})| \serv{S}, \serv{A}] \in L^2(\serv{S};\serv{A})$, and step (c) is justified by noticing that $(\serv{X}_{\star}, \serv{S}; \serv{A})$ forms a scalar Gaussian channel with SNR $\effsnr$. Since $f\in L^2(\serv{X}_{\leq t}; \serv{A})$ was arbitrary, we can minimize the LHS of the above display with respect to $f$, which yields:
\begin{align} \label{eq:mmse-lb}
    \bmmse(\serv{X}_{\star} | \serv{X}_{\leq t}; \serv{A}) \geq \mmse(\effsnr).
\end{align}
To obtain the upper bound, we observe that:
\begin{align} 
\mmse(\effsnr) & \explain{(a)}{=} \E[\{\serv{X}_{\star}- \bdnsr(\serv{S}; \serv{A} | \effsnr)\}^2] \nonumber \\ &\explain{(b)}{\geq} \min_{f \in L^2(\serv{X}_{\leq t}; \serv{A})} \E[\{\serv{X}_\star - f(\serv{X}_{\leq t}; \serv{A})\}^2] \explain{def}{=} \bmmse(\serv{X}_{\star} | \serv{X}_{\leq t}; \serv{A}), \label{eq:mmse-ub}
\end{align}
where the equality in (a) follows by recalling the formula for the MMSE estimator in a scalar Gaussian channel and the inequality in (b) follows by observing that: $$\bdnsr(\serv{S}; \serv{A} | \effsnr) = \bdnsr(\ip{\serv{X}_{\leq t}}{\optlin}; \serv{A} | \effsnr) \explain{def}{=} f_{\star}(\serv{X}_{\leq t}; \serv{A}) \in L^2(\serv{X}_{\leq t}; \serv{A}).$$ Combining the conclusions of \eqref{eq:mmse-lb} and \eqref{eq:mmse-ub} yields:
\begin{align*}
\bmmse(\serv{X}_{\star} | \serv{X}_{\leq t}; \serv{A}) & = \E[\{\serv{X}_\star - f_{\star}(\serv{X}_{\leq t}; \serv{A})\}^2] =  \mmse(\effsnr), 
\end{align*}
which gives us the claimed formula for the MMSE and the MMSE estimator. 
\paragraph*{Analysis of DMMSE} We begin by observing that since $(\serv{X}_{\star}, \serv{S} \explain{def}{=} \ip{\optlin}{\serv{X}_{\leq t}}; \serv{A})$ form a scalar Gaussian channel with SNR $\effsnr$, we can write $\serv{S}$ as:
\begin{align} \label{eq:scalarized-channel}
    \serv{S} & \explain{def}{=} \ip{\optlin}{\serv{X}_{\leq t}} = \sqrt{\effsnr} \serv{X}_{\star} + \sqrt{1-\effsnr} 
    \serv{W}, \quad \serv{W} \explain{def}{=} \frac{\ip{\optlin}{\serv{Z}_{\leq t}}}{\sqrt{1-\effsnr}} \sim \gauss{0}{1}.
\end{align}As before, we begin by showing the lower bound:
\begin{align*}
&\bdmmse(\serv{X}_{\star} | \serv{X}_{\leq t}; \serv{A}) \\
&\explain{def}{=} \left( \min_{f \in L^2(\serv{X}_{\leq t}; \serv{A})} \E[\{\serv{X}_\star - f(\serv{X}_{\leq t}; \serv{A})\}^2] \; \text{subject to } \E[\serv{Z}_i f(\serv{X}_1, \dotsc, \serv{X}_t; \serv{A})] = 0 \; \forall \; i \; \in \; [t]\right) \\&\geq \dmmse(\effsnr).
\end{align*}
For any $f \in L^2(\serv{X}_{\leq t}; \serv{A})$ which satisfies the divergence-free constraints $\E[\serv{Z}_i f(\serv{X}_1, \dotsc, \serv{X}_t; \serv{A})] = 0 \; \forall \; i \; \in \; [t]$, we have:
\begin{align*}
\E[\{\serv{X}_\star - f(\serv{X}_{\leq t}; \serv{A})\}^2] & \explain{(a)}{\geq}  \E[\{\serv{X}_\star - \E[f(\serv{X}_{\leq t}; \serv{A})| \serv{S}, \serv{X}_{\star},\serv{A}]\}^2] \\
& \explain{\eqref{eq:cond-indp}}{=} \E[\{\serv{X}_\star - \E[f(\serv{X}_{\leq t}; \serv{A})| \serv{S}, \serv{A}]\}^2] \\
& \explain{(b)}{\geq} \min_{g \in L^2(\serv{S};\serv{A})} \left( \E[\{\serv{X}_\star - g(\serv{S}; \serv{A})\}^2] \text{ subject to } \E[\serv{W} g(\serv{S}; \serv{A})] = 0 \right)\\
& = \bdmmse(\serv{X}_{\star}|\serv{S}; \serv{A}) \\
& \explain{}{=} \dmmse(\effsnr).
\end{align*}
In the above display:
\begin{enumerate}
\item Step (a) uses Jensen's inequality. 
\item In step (b) we observed that $\E[f(\serv{X}_{\leq t}; \serv{A})| \serv{S}, \serv{A}] \in L^2(\serv{S}; \serv{A})$. Moreover, the estimator $\E[f(\serv{X}_{\leq t}; \serv{A})| \serv{S}, \serv{A}]$ satisfies the divergence-free constraint $\E[\serv{W}\E[f(\serv{X}_{\leq t}; \serv{A})| \serv{S}, \serv{A}]] = 0$. This can be verified as follows:
\begin{align*}
\E[\serv{W}\E[f(\serv{X}_{\leq t}; \serv{A})| \serv{S}, \serv{A}]] &\explain{\eqref{eq:cond-indp}}{=} \E[\serv{W}\E[f(\serv{X}_{\leq t}; \serv{A})| \serv{S}, \serv{X}_{\star} \serv{A}]] \\
& \explain{(i)}{=} \E[\E[\serv{W}f(\serv{X}_{\leq t}; \serv{A})| \serv{S}, \serv{X}_{\star} \serv{A}]] = \E[\serv{W}f(\serv{X}_{\leq t}; \serv{A})]   \explain{(ii)}{=} 0,
\end{align*}
where we used the fact that $\serv{W}$ is measurable with respect to $\serv{X}_{\star},\serv{S}$ (recall \eqref{eq:scalarized-channel}) in step (i). To obtain the equality marked (ii), we expressed $\serv{W}$ as a linear combination of $\serv{Z}_{\leq t}$ (recall \eqref{eq:scalarized-channel}) and used the fact that $f$ satisfies the divergence-free constraints. 
\end{enumerate}
Since $f\in L^2(\serv{X}_{\leq t}; \serv{A})$ was an arbitrary function that satisfies the divergence-free constraints, we can minimize the LHS of the above display with respect to $f$ and obtain:
\begin{align} \label{eq:dmmse-lb}
    \bdmmse(\serv{X}_{\star} | \serv{X}_{\leq t}; \serv{A}) \geq \dmmse(\effsnr).
\end{align}
To obtain the upper bound, we observe that:
\begin{align} 
&\dmmse(\effsnr)  \explain{(a)}{=} \E[\{\serv{X}_{\star}- \dfbdnsr(\serv{S}; \serv{A} | \effsnr)\}^2] \nonumber \\&\explain{(b)}{\geq}  \left( \min_{f \in L^2(\serv{X}_{\leq t}; \serv{A})}  \E[\{\serv{X}_\star - f(\serv{X}_{\leq t}; \serv{A})\}^2] \text{ subject to } \E[\serv{Z}_i f(\serv{X}_1, \dotsc, \serv{X}_t; \serv{A})] = 0 \; \forall \; i \; \in \; [t]\right) \nonumber \\ &\explain{def}{=} \bdmmse(\serv{X}_{\star} | \serv{X}_{\leq t}; \serv{A}). \label{eq:dmmse-ub}
\end{align}
In the above display:
\begin{enumerate}
\item The equality in (a) follows by recalling the formula for the DMMSE estimator in a scalar Gaussian channel (\lemref{lem:dmmse-scalar}).
\item To obtain the inequality in (b), we observed that:
\begin{align*}
    \dfbdnsr(\serv{S}; \serv{A} | \effsnr) = \dfbdnsr(\ip{\serv{X}_{\leq t}}{\optlin}; \serv{A} | \effsnr) \explain{def}{=} \dfnew{f}_{\star}(\serv{X}_{\leq t}; \serv{A}) \in L^2(\serv{X}_{\leq t}; \serv{A}).
\end{align*}
Moreover, $\dfnew{f}_{\star}(\serv{X}_{\leq t}; \serv{A}) \explain{def}{=} \dfbdnsr(\ip{\serv{X}_{\leq t}}{\optlin}; \serv{A} | \effsnr)$ satisfies the divergence-free requirements: $$\E[\serv{Z}_i \dfbdnsr(\ip{\serv{X}_{\leq t}}{\optlin}; \serv{A} | \effsnr)] = 0 \quad  \forall \; i \  \in [t].$$ This can be verified as follows:
\begin{align*}
\E[\serv{Z}_i \dfbdnsr(\ip{\serv{X}_{\leq t}}{\optlin}; \serv{A} | \effsnr)] &\explain{\eqref{eq:scalarized-channel}}{=} \E\left[\serv{Z}_i \dfbdnsr\left(\sqrt{\effsnr}\serv{X}_{\star} + \sqrt{1-\effsnr}\serv{W} ; \serv{A} | \effsnr\right)\right] \\
& = \E[\dfbdnsr\left(\sqrt{\effsnr}\serv{X}_{\star} + \sqrt{1-\effsnr}\serv{W} ; \serv{A} | \effsnr\right) \cdot \E[\serv{Z}_i | \serv{X}_{\star}, \serv{A}, \serv{W}]] \\
& \explain{(i)}{=} c \E[\dfbdnsr\left(\sqrt{\effsnr}\serv{X}_{\star} + \sqrt{1-\effsnr}\serv{W} ; \serv{A} | \effsnr\right) \cdot \serv{W}] \\
& \explain{(ii)}{=} 0,
\end{align*}
where the equality marked (i) follows by observing that $(\serv{Z}_i, \serv{W})$ are centered Gaussian random variables independent of $(\serv{X}_{\star}, \serv{A})$ and hence $\E[\serv{Z}_i | \serv{X}_{\star}, \serv{A}, \serv{W}] = \E[\serv{Z}_i | \serv{W}] = c \serv{W}$ for some constant $c$ (the exact formula for $c$ is not important for the argument). The equality (ii) follows because $\dfbdnsr(\serv{S};\serv{A})$ satisfies the divergence-free requirement $\E[\serv{W} \dfbdnsr(\serv{S};\serv{A})] = 0$.  
\end{enumerate}
Combining the conclusions of \eqref{eq:dmmse-lb} and \eqref{eq:dmmse-ub} yields:
\begin{align*}
\bdmmse(\serv{X}_{\star} | \serv{X}_{\leq t}; \serv{A}) & = \E[\{\serv{X}_\star - \dfnew{f}_{\star}(\serv{X}_{\leq t}; \serv{A})\}^2] =  \dmmse(\effsnr), 
\end{align*}
which gives us the claimed formula for the DMMSE and the DMMSE estimator. This concludes the proof of this lemma.
\end{proof}

\subsubsection{Proof of \lemref{lem:dmmse-monotonicity}} \label{sec:dmmse-monotonicity}
\begin{proof}[Proof of \lemref{lem:dmmse-monotonicity}]
It is sufficient to show that for any $\rho \in [0,1]$ and any $\omega \in [0,1]$, $\dmmse(\omega) \leq \dmmse(\rho \omega)$. Consider a scalar Gaussian channel $(\serv{X}_{\star}, \serv{X}_0 ; \serv{A})$ at SNR $\omega$:
\begin{align*}
(\serv{X}_{\star}; \serv{A}) \sim \pi, \quad \serv{X}_0 = \sqrt{\omega} \cdot \serv{X}_{\star} + \sqrt{1-\omega} \cdot \serv{Z}_0, \quad \serv{Z}_0 \sim \gauss{0}{1},
\end{align*}
where $(\serv{X}_{\star}; \serv{A})$ and $\serv{Z}_0$ are sampled independently. Let $\serv{W}$ be a $\gauss{0}{1}$ random variable, sampled independently of $(\serv{X}_{\star}; \serv{A})$ and $\serv{Z}_0$. We introduce the random variable $\serv{X}_1 = \sqrt{\rho} \cdot \serv{X}_0 + \sqrt{1-\rho} \cdot \serv{W}$. This construction ensures that $(\serv{X}_{\star}, \serv{X}_1; \serv{A})$ is a scalar Gaussian channel with SNR $\rho \omega$:
\begin{align} \label{eq:degraded-channel}
    \serv{X}_1 & = \sqrt{\rho \omega} \cdot \serv{X}_{\star} + \sqrt{1- \rho \omega} \cdot \serv{Z}_1, \quad \serv{Z}_1 \explain{def}{=} \frac{\sqrt{\rho(1-\omega)} \cdot \serv{Z}_0 + \sqrt{1-\rho} \cdot \serv{W}}{\sqrt{1-\rho \omega}} \sim \gauss{0}{1}.
\end{align}To obtain the claim of the lemma, we observe that:
\begin{align*}
&\dmmse(\omega)  \explain{(a)}{=}\bdmmse(\serv{X}_{\star}| \serv{X}_0, \serv{X}_1 ; \serv{A}) \\ &\explain{def}{=} \left( \min_{f \in L^2(\serv{X}_{0},\serv{X}_{1}; \serv{A})} \E[\{\serv{X}_\star - f(\serv{X}_{0},\serv{X}_{1}; \serv{A})\}^2] \; \text{subject to } \E[\serv{Z}_i f(\serv{X}_0,\serv{X}_1; \serv{A})] = 0 \; \forall \; i \; \in \; \{0,1\}\right) \\
& \explain{(b)}{\leq}  \left( \min_{f \in L^2(\serv{X}_{1}; \serv{A})} \E[\{\serv{X}_\star - f(\serv{X}_{1}; \serv{A})\}^2] \; \text{subject to } \E[\serv{Z}_1 f(\serv{X}_1; \serv{A})] = 0\right) \\
& \explain{def}{=} \bdmmse(\serv{X}_{\star} | \serv{X}_1; \serv{A}) \\
& = \dmmse(\rho \omega).
\end{align*}
In the above display, the equality in step (a) follows from \lemref{lem:scalarization-mmse} (the effective SNR of the Gaussian channel $(\serv{X}_{\star}, \serv{X}_0, \serv{X}_1; \serv{A})$ is $\omega$) and the inequality in step (b) follows by observing that the feasible set in the definition of  $\bdmmse(\serv{X}_{\star} | \serv{X}_1; \serv{A})$ is a subset of the feasible set in the definition of  $\bdmmse(\serv{X}_{\star} | \serv{X}_0, \serv{X}_1; \serv{A})$. Indeed if $f \in L^2(\serv{X}_{1}; \serv{A})$ satisfies $\E[\serv{Z}_1 f(\serv{X}_1; \serv{A})] = 0$, then  $f \in L^2(\serv{X}_0, \serv{X}_{1}; \serv{A})$, and:
\begin{align*}
    \E[\serv{Z}_0 f(\serv{X}_1; \serv{A})] &= \E[f(\serv{X}_1; \serv{A}) \E[\serv{Z}_0 | \serv{X}_{\star},\serv{Z}_1,\serv{A}]]\\
    & \explain{(i)}{=} \E[f(\serv{X}_1; \serv{A}) \E[\serv{Z}_0 | \serv{Z}_1]] \explain{\eqref{eq:degraded-channel}}{=} \frac{\sqrt{1-\rho \omega}}{\sqrt{\rho(1-\omega)}} \cdot \E[f(\serv{X}_1; \serv{A}) \serv{Z}_1] = 0,
\end{align*}
where the equality marked (i) follows by observing that $(\serv{Z}_0,\serv{Z}_1)$ are centered Gaussian random variables independent of $(\serv{X}_\star; \serv{A})$ (cf. \eqref{eq:degraded-channel}). This shows that $\dmmse(\cdot)$ is non-increasing. The continuity of $\dmmse(\cdot)$ on $[0,1)$ follows by expressing $\dmmse(\cdot)$ in terms of $\mmse(\cdot)$ using \lemref{lem:dmmse-scalar} (item (3)) and the continuity of $\mmse(\cdot)$ (\factref{fact:mmse-monotonic}).  
\end{proof}

\section{State Evolution for OAMP Algorithms (Theorem~\ref{thm:SE})}\label{Sec:proof_SE}

In this appendix, we present the proof of Theorem~\ref{thm:SE}. We begin by introducing the key ideas involved in the proof in the form of some intermediate results. 

\paragraph*{Polynomial Approximation} Consider a general OAMP algorithm:
\begin{align} \label{eq:OAMP-SE}
\vx_t&=\Psi_t(\bm{Y}) \cdot f_t(\vx_1, \dotsc, \vx_{t-1};\va) \quad \forall \;  t\in\mathbb{N}.
\end{align}
By a polynomial approximation argument, we may assume that the matrix denoisers $\{\mfunc_t\}_{t \in \N}$ are polynomials. This argument is summarized in the following lemma, whose proof is deferred to \appref{appendix:poly-approx}.
\begin{lemma} \label{lem:poly-approx} It is sufficient to prove Theorem~\ref{thm:SE} under the additional assumption that for each $t \in \N$, the matrix denoiser $\mfunc_t: \R \mapsto \R$ is a polynomial. 
\end{lemma}

\paragraph*{Orthogonal Decomposition} Let $(\serv{X}_{\star}, (\serv{X}_t)_{t \in \N}; \serv{A})$ denote the state evolution random variables associated with the OAMP algorithm in \eqref{eq:OAMP-SE}. A key idea {involved} in the proof of Theorem~\ref{thm:SE} is to consider the following decomposition of the functions $(\fnonlin_{t})_{t \in \N}$:
\begin{align} \label{eq:f-perp}
\fnonlin_t(x_1, \dotsc, x_{t-1}; \aux) \explain{}{=} \alpha_t x_{\star} + \fnonlin_t^{\perp}(x_1, \dotsc, x_{t-1}; x_{\star},  \aux),
\end{align}
where:
\begin{align*}
\alpha_{t} \explain{def}{=} \E[\serv{X}_{\star} \fnonlin_t(\serv{X}_1, \dotsc, \serv{X}_{t-1}; \serv{A})], \quad  \fnonlin_t^{\perp}(x_1, \dotsc, x_{t-1}; x_{\star}, \aux) \explain{def}{=} \fnonlin_t(x_1, \dotsc, x_{t-1}; \aux) - \alpha_t x_{\star}.
\end{align*}
Notice that by construction, $\fnonlin_t^{\perp}$ is orthogonal to the signal state evolution random variable $\serv{X}_{\star}$: 
\begin{align} \label{eq:f-perp-ortho}
\E[\serv{X}_{\star} \fnonlin_t^{\perp}(\serv{X}_1, \dotsc, \serv{X}_{t-1}; \serv{X}_{\star}, \serv{A})] & = 0.
\end{align}
Using this decomposition, we have:
\begin{align} \label{eq:original-OAMP-decomp}
\vx_t & =  \alpha_t \cdot \mfunc_t(\mY) \cdot  \vx_{\star} + \mfunc_t(\mY) \cdot \fnonlin_t^{\perp}(\vx_1, \dotsc, \vx_{t-1}
; \vx_{\star}, \va). 
\end{align}
Recall that the observed matrix $\mY$ consists of a signal part and a noise part: $\mY = \frac{\theta}{\dim} \vx_{\star} \vx_{\star}^\top + \mW$. The decomposition above allows us to disentangle the action of the signal component and the noise component using the following lemma, whose proof is deferred to \appref{appendix:aux1-proof}
\begin{lemma}\label{lem:aux1} Let $\mfunc: \R \mapsto \R$ be a polynomial function.
\begin{enumerate}
    \item There exists a polynomial function $\tilde{\mfunc}: \R \mapsto \R$ associated with $\mfunc$ such that $ \mPsi(\mY) \cdot  \vx_{\star} \;   \explain{$\dim\rightarrow \infty$}{\simeq}  \;  \tilde{\mfunc}(\mW) \cdot \vx_{\star}$, where $\explain{$\dim \rightarrow \infty$}{\simeq}$ denotes asymptotic equivalence of random vectors (Definition~\ref{def:W2}). 
    \item Let $\vv$ be a $\dim$-dimensional random vector with the property that:
    \begin{align} \label{eq:ortho-prop}
            \frac{\ip{\mW^{i} \vv}{\vx_{\star}}}{\dim} \pc 0 \quad \forall \; i \;  \in \;  \N.
    \end{align}
    Then, $ \mPsi(\mY) \cdot  \vv \;   \explain{$\dim\rightarrow \infty$}{\simeq}  \;  {\mfunc}(\mW) \cdot \vv$.
\end{enumerate}
\end{lemma}
We will use the lemma above to show that $\fnonlin_t^{\perp}(\vx_1, \dotsc, \vx_t; \vx_{\star}, \va)$ does not interact with the signal component of $\Psi_t(\mY)$, thanks to the orthogonality property in \eqref{eq:f-perp-ortho}, which will guarantee that $\vv = \fnonlin_t^{\perp}(\vx_1, \dotsc, \vx_t; \vx_{\star}, \va)$ satisfies the orthogonality condition in \eqref{eq:ortho-prop}. Formally, we will show that:
\begin{subequations} \label{eq:OAMP-to-auxOAMP-approx}
\begin{align}
\mfunc_t(\mY) \cdot \fnonlin_t^{\perp}(\vx_1, \dotsc, \vx_{t-1}; \vx_{\star}, \va) \explain{$\dim \rightarrow \infty$}{\simeq} \mfunc_t(\mW)  \cdot \fnonlin_t^{\perp}(\vx_1, \dotsc, \vx_{t-1}; \vx_{\star}, \va),
\end{align}
On the other hand, $\vx_{\star}$ does interact with the signal component in $\mfunc_t(\mY)$ and this interaction is captured by the \emph{transformed polynomial} $\tilde{\mfunc}_t$ associated with $\mfunc_t$, whose existence is guaranteed by the lemma above:
\begin{align}
\mfunc_t(\mY) \cdot \vx_{\star} \explain{$\dim \rightarrow \infty$}{\simeq} \tilde{\mfunc}_t(\mW)  \cdot \vx_{\star}.
\end{align}
\end{subequations}
\paragraph*{Auxiliary OAMP Algorithm} Using the approximations from \eqref{eq:OAMP-to-auxOAMP-approx} in the update equation of the OAMP algorithm given in \eqref{eq:original-OAMP-decomp}, leads us to introduce an auxiliary OAMP algorithm which generates iterates $\tilde{\vx}_1, \tilde{\vx}_2, \dotsc $ via the update rule:
\begin{align} \label{eq:aux-OAMP}
\tilde{\vx}_t & = \alpha_t \cdot \tilde{\mfunc}_t(\mW) \cdot \vx_{\star} + \mfunc_t(\mW)  \cdot \fnonlin_t^{\perp}(\tilde{\vx}_1, \dotsc, \tilde{\vx}_{t-1}; \vx_{\star}, \va) \quad \forall \; t \; \in \; \N.
\end{align}
This auxiliary OAMP algorithm will serve as an easy-to-analyze approximation to the original OAMP algorithm in \eqref{eq:OAMP-SE}. Since the auxiliary OAMP algorithm uses the rotationally invariant noise matrix $\mW$ in its iterations (rather than the non-rotationally invariant matrix $\mY$ used by the original OAMP algorithm), its dynamics can be easily analyzed using known results on the state evolution of AMP algorithms for rotationally invariant matrices and the associated universality class \citep{ma2017orthogonal,rangan2019vector,takeuchi2019rigorous,fan2022approximate,dudeja2023universality,wang2022universality,dudeja2022spectral}. Applying these results requires us to compute:
\begin{align*}
\plim_{\dim \rightarrow \infty} \frac{\Tr[\tilde{\mfunc}_t(\mW)]}{\dim}, \quad \plim_{\dim \rightarrow \infty} \frac{\Tr[ \tilde{\mfunc}_s(\mW) \tilde{\mfunc}_t(\mW)]}{\dim} \quad \forall \; s, t\; \in \; \N.
\end{align*}
At first glance, computing the above limits appears challenging as the transformed polynomials $(\tilde{\mfunc}_t)_{t \in \N}$ constructed in \lemref{lem:aux1} have a complicated recursive characterization. Fortunately, the above limits have a simple formula in terms of $\nu$, the limiting spectral measure of $\mY$ in the direction of the signal (recall Lemma~\ref{lem:RMT}). 
\begin{lemma}\label{lem:nu-trick} For any $s,t \in \N$, we have:
\begin{align*}
\frac{\Tr[ \tilde{\mfunc}_t(\mW)]}{\dim} &\pc \E[\tilde{\mfunc}_t(\serv{\Lambda})] =  \E[\mfunc_t(\serv{\Lambda}_{\nu})], \\
 \frac{\Tr[ \tilde{\mfunc}_s(\mW)  \tilde{\mfunc}_t(\mW)]}{\dim} &\pc \E[\tilde{\mfunc}_s(\serv{\Lambda})  \tilde{\mfunc}_t(\serv{\Lambda})] =  \E[\mfunc_s(\serv{\Lambda}_{\nu})  \mfunc_t(\serv{\Lambda}_{\nu})], 
\end{align*}
where $\serv{\Lambda} \sim \mu$ and $\serv{\Lambda}_{\nu} \sim \nu$.
\end{lemma}
The proof of this lemma is presented in \appref{appendix:nu-trick-proof}. An immediate consequence of this lemma and an existing result on the dynamics of AMP algorithms driven by rotationally invariant matrices \citep[Theorem 2]{dudeja2022spectral} is the following characterization of the dynamics of the auxiliary OAMP algorithm. 
\begin{lemma}\label{lem:aux-OAMP-dyn} \begin{enumerate}
\item For any $t \in \N$, the iterates generated by the auxiliary OAMP algorithm in \eqref{eq:aux-OAMP} satisfy $$(\vx_{\star}, \tilde{\vx}_1, \dotsc, \tilde{\vx}_t; \va) \wc (\serv{X}_{\star}, \serv{X}_1, \dotsc, \serv{X}_t; \serv{A}),$$
where $ (\serv{X}_{\star}, \serv{X}_1, \dotsc, \serv{X}_t; \serv{A})$ are the state evolution random variables associated with the original OAMP algorithm in \eqref{eq:OAMP-SE}. 
\item Moreover, for any $t \in \N$ and any $i \in \N$,
\begin{align*}
\frac{\ip{\mW^{i} \fnonlin^\perp_t(\vx_1, \dotsc, \vx_{t-1}; \vx_{\star}, \va)}{\vx_{\star}}}{\dim} & \pc 0.
\end{align*}
\end{enumerate}
\end{lemma}
The proof of the above lemma is deferred to \appref{app:aux-OAMP-dyn}. We have now introduced all the key ideas involved in the proof of Theorem~\ref{thm:SE} and are in a position to present its proof. 
\begin{proof}[Proof of Theorem~\ref{thm:SE}] Our goal is to show that for any $t \in \N$,
\begin{align*}
(\vx_{\star}, {\vx}_1, \dotsc, {\vx}_t; \va) \wc (\serv{X}_{\star}, \serv{X}_1, \dotsc, \serv{X}_t; \serv{A}).
\end{align*}
In light of \lemref{lem:aux-OAMP-dyn}, it suffices to show that the auxiliary OAMP algorithm in \eqref{eq:aux-OAMP} approximates the given OAMP algorithm in the sense:
\begin{align} \label{eq:OAMP-to-auxOAMP-induction}
\vx_t  \; \explain{$\dim \rightarrow \infty$}{\simeq} \; \tilde{\vx}_t \quad \forall \; t \; \in \; \N.
\end{align}
We show the above claim by induction. As our induction hypothesis, we assume that $\vx_s \; \explain{$\dim \rightarrow \infty$}{\simeq} \; \tilde{\vx}_s$ for all $s < t$, and verify the claim at iteration $t$:
\begin{align*}
\vx_{t} & \;\;  \explain{\eqref{eq:original-OAMP-decomp}}{=} \; \;  \alpha_t \cdot \mfunc_t(\mY) \cdot  \vx_{\star} + \mfunc_t(\mY) \cdot \fnonlin_t^{\perp}(\vx_1, \dotsc, \vx_{t-1}; \vx_{\star}, \va) \\
  &\;\stackrel{(a)}{\explain{$\dim \rightarrow \infty$}{\simeq}} \; \alpha_t \cdot \mfunc_t(\mY) \cdot  \vx_{\star} + \mfunc_t(\mY) \cdot \fnonlin_t^{\perp}(\tilde{\vx}_1, \dotsc, \tilde{\vx}_{t-1}; \vx_{\star}, \va) \\
  & \; \stackrel{(b)}{\explain{$\dim \rightarrow \infty$}{\simeq}} \; \alpha_t \cdot \tilde{\mfunc}_t(\mW) \cdot  \vx_{\star} + \mfunc_t(\mW) \cdot \fnonlin_t^{\perp}(\tilde{\vx}_1, \dotsc, \tilde{\vx}_{t-1}; \vx_{\star}, \va) \\
  & \; \; \explain{\eqref{eq:aux-OAMP}}{=} \; \; \tilde{\vx}_t
\end{align*}
In the above display, step (a) follows from the induction hypothesis and the Lipschitz continuity of $\fnonlin_t^{\perp}$ (which is implied by the Lipschitz continuity of $\fnonlin_t$; see Definition~\ref{Def:OAMP_main} and \eqref{eq:f-perp}). In step (b), we appealed to \lemref{lem:aux1}. Indeed, \lemref{lem:aux-OAMP-dyn} (claim (2)) guarantees that the orthogonality requirement \eqref{eq:ortho-prop} imposed in \lemref{lem:aux1} is met. This proves the claim \eqref{eq:OAMP-to-auxOAMP-induction} by induction and concludes the proof of Theorem~\ref{thm:SE}. 
\end{proof}
\subsection{Proof of \lemref{lem:poly-approx}} \label{appendix:poly-approx}
\begin{proof}[Proof of \lemref{lem:poly-approx}]
To prove \lemref{lem:poly-approx}, we assume that the claim of Theorem~\ref{thm:SE} holds under the additional assumption that the matrix denoisers used in the OAMP algorithm are polynomial functions. We will show that Theorem~\ref{thm:SE} continues to hold even without this additional assumption. To this end, we consider a general OAMP algorithm (recall Definition~\ref{Def:OAMP_main}):
\begin{align} \label{eq:div-free-algo-recall}
\iter{\vx}{t} & = \mfunc_{t}(\mY) \cdot \fnonlin_t(\iter{\vx}{1}, \dotsc, \vx_{t-1}; \va) \quad  \forall \; t \; \in  \;  \N,
\end{align}
where the matrix denoisers $(\mfunc_t)_{t \in \N}$ are possibly non-polynomial functions. Let $(\serv{X}_{\star}, \{\serv{X}_{t}\}_{t \in \N}; \serv{A})$ denote the state evolution random variables associated with the OAMP algorithm above. Our goal is to show that:
\begin{align*}
(\vx_{\star}, \vx_1, \dotsc, \vx_t; \va) \wc (\serv{X}_{\star}, \serv{X}_1, \dotsc, \serv{X}_t; \serv{A}) \quad \forall \; t \; \in \;  \N.
\end{align*}
\paragraph*{Polynomial Approximation} We begin by approximating the matrix denoisers $\mfunc_{t}: \R \mapsto \R$ by polynomials. Recall from Assumption~\ref{assump:signal-noise}, that there is an $\dim$-independent finite constant $C$ such that $\|\mW\|_{\op} \leq C$. Furthermore, thanks to Assumption~\ref{assump:signal-noise},
\begin{align*}
\|\mY\|_{\op} & = \left\| \frac{\theta}{n} \vx_{\star} \vx_{\star}^\top + \mW \right\|_{\op} \leq C + \frac{\theta \cdot \|\vx_{\star}\|^2}{\dim} \pc C + \theta.
\end{align*}
Hence, the event:
\begin{align*}
\mathcal{E}_{\dim} \explain{def}{=} \{ \|\mY\|_{\op} \leq K\} \quad \text{ with } \quad K \explain{def}{=} C + \theta + 1
\end{align*}
occurs with probability tending to $1$. We restrict ourselves to this good event in the remainder of the proof. For each $t \in \N$, the Weierstrass Approximation Theorem guarantees the existence of a sequence of approximating polynomial functions $(\dup{\mfunc_t}{\degree})_{D \in \N}$ of increasing degree, where $\dup{\mfunc_t}{\degree}$ is a degree $\degree$ polynomial, with the approximation guarantee:
\begin{align} \label{eq:weierstrass}
\lim_{\degree \rightarrow \infty}  \sup_{|\lambda| \leq K} \left|\dup{\mfunc_{t}}{\degree}(\lambda) - \mfunc_{t}(\lambda) \right| & = 0,
\end{align}
Consider the OAMP algorithm which uses the degree-$\degree$ approximations $(\dup{\Psi_t}{\degree})_{t\in \N}$ as the matrix denoisers:
\begin{align} \label{eq:low-degree-approx-div-free-algo}
\dup{\vx_{t}}{\degree} & = \left( \dup{\mfunc_{t}}{\degree}(\mY) - \E[\dup{\mfunc_{t}}{\degree}(\serv{\Lambda})]  \mI_{\dim} \right) \cdot \left( \fnonlin_t(\dup{\vx_{<t}}{\degree}; \va) - \sum_{s=1}^{t-1}  \E[\partial_{s} \fnonlin_t(\dup{\serv{X}_{<t}}{\degree}; \serv{A})]  \dup{\vx_s}{\degree} \right)  \  \forall \; t \; \in \; \N,
\end{align}
where $(\serv{X}_{\star}, (\dup{\serv{X}_t}{\degree})_{t\in \N}; \serv{A})$ are the state evolution random variables associated with the OAMP algorithm above. Since we assume that Theorem~\ref{thm:SE} holds for OAMP algorithms with polynomial matrix denoisers, we have that:
\begin{align*}
(\vx_{\star}, \dup{\vx_1}{\degree}, \dotsc, \dup{\vx_t}{\degree}; \va) \wc (\serv{X}_{\star}, \dup{\serv{X}_1}{\degree}, \dotsc, \dup{\serv{X}_t}{\degree}; \serv{A}) \quad \forall \; t \; \in \;  \N, \; \degree \; \in \; \N.
\end{align*}
We make the following claims regarding the polynomial approximation to the original OAMP algorithm and its state evolution random variables.
\begin{claim}[Convergence of State Evolution Random Variables] \label{claim:SE-conv} For any $t \in \N$, $$(\serv{X}_{\star}, \dup{\serv{X}_1}{\degree}, \dotsc, \dup{\serv{X}_t}{\degree}; \serv{A})  \wc (\serv{X}_{\star}, \serv{X}_1, \dotsc, \serv{X}_t; \serv{A}) \quad \text{as $\degree \rightarrow \infty$},$$ where $\wc$ denotes Wasserstein-$2$ convergence for random variables. Formally, this means that for any test function $\hnonlin: \R^{t+1} \times \R^{\auxdim}$, independent of $\dim$, which satisfies the smoothness condition:
\begin{align} \label{eq:PL2}
|\hnonlin(x; \aux) - \hnonlin(x^\prime ; \aux^\prime)| &\leq L \cdot (\|x-x^\prime\| + \|\aux - \aux^\prime\|) \cdot (1 + \|x\| + \|x^\prime\| + \|\aux\| + \|\aux^\prime\|) \\
&\nonumber\qquad\qquad \forall \; x, x^\prime \; \in \; \R^{t+1}, \; \aux, \aux^\prime \; \in \; \R^k,
\end{align}
for some $L < \infty$,  we have:
\begin{align*}
\lim_{\degree \rightarrow \infty} \E[\hnonlin(\dup{\serv{X}_1}{\degree}, \dotsc, \dup{\serv{X}_t}{\degree}; \serv{A})] & = \E[\hnonlin (\serv{X}_{\star}, \serv{X}_{1}, \dotsc, \serv{X}_{t}; \serv{A})].
\end{align*}
\end{claim}
\begin{claim}[Convergence of Iterates] \label{claim:itr-conv} For any $t \in \N$,
\begin{align*}
\limsup_{\degree \rightarrow \infty} \plimsup_{\dim \rightarrow \infty} \frac{\|\vx_t- \dup{\vx_t}{\degree} \|^2}{\dim} & = 0.
\end{align*}
\end{claim}
We first prove \lemref{lem:poly-approx} assuming the above results. To do so, we consider a test function $h: \R^{t} \times \R^{\auxdim} \mapsto \R$, which satisfies the smoothness hypothesis in \eqref{eq:PL2} and show that:
\begin{align*}
\plim_{\dim \rightarrow \infty} \frac{1}{\dim} \sum_{i=1}^\dim \hnonlin(x_{\star}[i], x_1[i], \dotsc, x[t]; \aux[i]) & = \E \hnonlin(\serv{X}_{\star}, \serv{X}_{1}, \dotsc, \serv{X}_{t}; \serv{A}).
\end{align*}
Notice that for any $\degree \in \N$, we have the decomposition:
\begin{align} \label{eq:OAMP-SE-poly-approx-decomp}
\left|\frac{1}{\dim} \sum_{i=1}^\dim \hnonlin(x_{\star}[i], x_1[i], \dotsc, x[t]; \aux[i]) - \E \hnonlin(\serv{X}_{\star}, \serv{X}_{1}, \dotsc, \serv{X}_{t}; \serv{A}) \right|  \leq (i) + (ii) + (iii),
\end{align}
where,
\begin{align*}
(i) &\explain{def}{=} \left|\frac{1}{\dim} \sum_{i=1}^\dim \hnonlin(x_{\star}[i], x_1[i], \dotsc, x[t]; \aux[i]) - \frac{1}{\dim} \sum_{i=1}^\dim \hnonlin(x_{\star}[i], \dup{x_1}{\degree}[i], \dotsc, \dup{x}{\degree}_t[i]; \aux[i]) \right| \\ (ii) &\explain{def}{=} \left| \frac{1}{\dim} \sum_{i=1}^\dim \hnonlin(x_{\star}[i], \dup{x_1}{\degree}[i], \dotsc, \dup{x}{\degree}_t[i]; \aux[i]) - \E[\hnonlin(\serv{X}_{\star}, \dup{\serv{X}}{\degree}_{1}, \dotsc, \dup{\serv{X}}{\degree}_{t}; \serv{A})] \right| \\ (iii) & \explain{def}{=} \left|\E[\hnonlin(\serv{X}_{\star}, \dup{\serv{X}}{\degree}_{1}, \dotsc, \dup{\serv{X}}{\degree}_{t}; \serv{A})]  - \E[ \hnonlin(\serv{X}_{\star}, \serv{X}_{1}, \dotsc, \serv{X}_{t}; \serv{A})]\right|.
\end{align*}
We consider each of the terms above. From \clref{claim:itr-conv} and \eqref{eq:PL2}, we conclude that:
\begin{align*}
\lim_{\degree \rightarrow \infty} \plimsup_{\dim \rightarrow \infty } (i) & = 0.
\end{align*}
Since we assume that Theorem~\ref{thm:SE} holds for OAMP algorithms with polynomial matrix denoisers,
\begin{align*}
\plim_{\dim \rightarrow \infty } (ii) & = 0.
\end{align*}
Finally, \clref{claim:SE-conv} implies that:
\begin{align*}
\lim_{\degree \rightarrow \infty} (iii) & = 0.
\end{align*}
We let $\dim \rightarrow \infty$ and $\degree \rightarrow \infty$ in \eqref{eq:OAMP-SE-poly-approx-decomp} and use the results above to conclude that:
\begin{align*}
\plim_{\dim \rightarrow \infty} \frac{1}{\dim} \sum_{i=1}^\dim \hnonlin(x_{\star}[i], x_1[i], \dotsc, x[t]; \aux[i]) & = \E \hnonlin(\serv{X}_{\star}, \serv{X}_{1}, \dotsc, \serv{X}_{t}; \serv{A}),
\end{align*}
as desired. To finish the proof of the lemma, we need to prove \clref{claim:SE-conv} and \clref{claim:itr-conv}.

\paragraph*{Convergence of State Evolution Random Variables (Proof of \clref{claim:SE-conv})} Our goal is to show that:
\begin{align} \label{eq:SE-rv-conv}
(\serv{X}_{\star}, \dup{\serv{X}_1}{\degree}, \dotsc, \dup{\serv{X}_t}{\degree}; \serv{A})  \wc (\serv{X}_{\star}, \serv{X}_1, \dotsc, \serv{X}_t; \serv{A})  \quad \text{ as $\degree \rightarrow \infty$} \quad \forall \; t \; \in \;  \N,
\end{align}
We show this by induction. As our induction hypothesis, we assume that \eqref{eq:SE-rv-conv} holds for some $t \in \N$ and show that:
\begin{align*}
(\serv{X}_{\star}, \dup{\serv{X}_1}{\degree}, \dotsc, \dup{\serv{X}_{t+1}}{\degree}; \serv{A})  \wc (\serv{X}_{\star}, \serv{X}_1, \dotsc, \serv{X}_{t+1}; \serv{A})  \quad \text{ as $\degree \rightarrow \infty$}.
\end{align*}
To this end, we consider a test function $\hnonlin:\R^{t + 2} \times \R^{\auxdim} \mapsto \R$ which satisfies the smoothness requirement in \eqref{eq:PL2} and verify that:
\begin{align} \label{eq:SE-conv-induction-goal}
\lim_{\degree \rightarrow \infty} \E[\hnonlin(\serv{X}_{\star}, \dup{\serv{X}_1}{\degree}, \dotsc, \dup{\serv{X}_{t+1}}{\degree}; \serv{A})] & = \E[\hnonlin (\serv{X}_{\star}, \serv{X}_{1}, \dotsc, \serv{X}_{t+1}; \serv{A})].
\end{align}
Recall from Definition~\ref{Def:OAMP_main} that the joint distribution of the state evolution random variables is given by:
\begin{align*}
(\serv{X}_{\star}; \serv{A}) \sim \pi, \quad \serv{X}_i = \serv{\beta}_i \serv{X}_{\star} + \serv{Z}_i, \quad \dup{\serv{X}_i}{\degree} = \dup{\serv{\beta}_i}{\degree} \cdot \serv{X}_{\star} + \dup{\serv{Z}_i}{\degree} \quad \forall \; i \; \in \; [t+1], 
\end{align*}
where $(\serv{Z}_{1}, \dotsc, \serv{Z}_{t+1})$ and $(\dup{\serv{Z}_{1}}{\degree}, \dotsc, \dup{\serv{Z}_{t+1}}{\degree})$ are Gaussian random vectors, sampled independently of $(\serv{X}_{\star}; \serv{A})$:
\begin{align*}
(\serv{Z}_{1}, \dotsc, \serv{Z}_{t+1}) \sim \gauss{0}{\Sigma_{t+1}}, \quad (\dup{\serv{Z}_{1}}{\degree}, \dotsc, \dup{\serv{Z}_{t+1}}{\degree}) \sim \gauss{0}{\dup{\Sigma}{\degree}_{t+1}}.
\end{align*}
In the above equations, coefficients $(\beta_{i})_{i \in [t+1]}$ and  $(\dup{\beta_{i}}{\degree})_{i \in [t+1]}$ are given by:
\begin{align*} 
\beta_i & \explain{def}{=} \E[\serv{X}_\star \serv{F}_i] \mathbb{E}\big[\mfunc_i(\serv{\Lambda}_\nu)], \quad \dup{\beta_i}{\degree} \explain{def}{=} \E[\serv{X}_\star \dup{\serv{F}_i}{\degree}]  \left(\mathbb{E}\big[\dup{\mfunc_i}{\degree}(\serv{\Lambda}_\nu)] - \mathbb{E}\big[\dup{\mfunc_i}{\degree}(\serv{\Lambda})] \right) \; \; \forall \; i \; \in \; [t+1],
\end{align*}
and the entries of the covariance matrices ${\Sigma}{}_{t+1}, \dup{\Sigma}{\degree}_{t+1}$ are given by:
\begin{eqnarray*}
\left({\Sigma}{}_{t+1}\right)_{i,j} & =& \E[\serv{X}_\star {\serv{F}}_{i}] \cdot \E[\serv{X}_\star {\serv{F}}_{j}] \cdot \Cov[{\Psi}_i(\serv{\Lambda}_\nu),{\Psi}_j(\serv{\Lambda}_\nu)] \\ 
&& + (\E[{\serv{F}}_{i} {\serv{F}}_{j}] - \E[\serv{X}_\star {\serv{F}}_{i}]\E[\serv{X}_\star {\serv{F}}_{j}] )  \cdot \Cov[{\Psi}_i(\serv{\Lambda}),{\Psi}_j(\serv{\Lambda})] \quad \forall \; i,j \; \in \; [t+1], \\
\left(\dup{\Sigma}{\degree}_{t+1}\right)_{i,j} & =& \E[\serv{X}_\star \dup{\serv{F}}{\degree}_{i}] \cdot \E[\serv{X}_\star \dup{\serv{F}}{\degree}_{j}] \cdot \Cov[\dup{\Psi}{\degree}_i(\serv{\Lambda}_\nu),\dup{\Psi}{\degree}_j(\serv{\Lambda}_\nu)] \\ 
&&+ (\E[\dup{\serv{F}}{\degree}_{i} \dup{\serv{F}}{\degree}_{j}] - \E[\serv{X}_\star {\serv{F}}_{i}]\E[\serv{X}_\star {\serv{F}}_{j}] )  \cdot \Cov[\dup{\Psi}{\degree}_i(\serv{\Lambda}),\dup{\Psi}{\degree}_j(\serv{\Lambda})] \quad \forall \; i,j \; \in \; [t+1],
\end{eqnarray*}
where $\serv{\Lambda} \sim \mu$ and $\serv{\Lambda}_{\nu} \sim \nu$ and:
\begin{align*}
\serv{F}_i &\explain{def}{=} \fnonlin_i(\serv{X}_1, \dotsc, \serv{X}_{i-1}; \serv{A}), \\ \dup{\serv{F}}{\degree}_i &\explain{def}{=} \fnonlin_i(\dup{\serv{X}}{\degree}_1, \dotsc, \dup{\serv{X}}{\degree}_{i-1}; \serv{A}) - \sum_{j=1}^{i-1} \E[\partial_j \fnonlin_i(\dup{\serv{X}}{\degree}_1, \dotsc, \dup{\serv{X}}{\degree}_{i-1}; \serv{A})] \cdot \dup{\serv{X}}{\degree}_j.
\end{align*}
By a weak convergence and uniform integrability argument, \eqref{eq:SE-conv-induction-goal} follows if we can show that:
\begin{align} \label{eq:SE-conv-induction-goal-suff}
\lim_{\degree \rightarrow \infty} \dup{\beta_i}{\degree}  & = \beta_i, \quad \lim_{\degree \rightarrow \infty} \left(\dup{\Sigma}{\degree}_{t+1}\right)_{i,j} = \left({\Sigma}{}_{t+1}\right)_{i,j} \quad \forall \; i,j \; \in \; [t+1].
\end{align}
Indeed, from \eqref{eq:weierstrass} we deduce that:
\begin{align*}
\lim_{\degree \rightarrow \infty} \mathbb{E}\big[\dup{\mfunc_t}{\degree}(\serv{\Lambda})] &= \mathbb{E}\big[{\mfunc_t}(\serv{\Lambda})], \\
 \lim_{\degree \rightarrow \infty} \mathbb{E}\big[\dup{\mfunc_t}{\degree}(\serv{\Lambda}_{\nu})] &= \mathbb{E}\big[{\mfunc_t}(\serv{\Lambda}_{\nu})], \\ 
\lim_{\degree \rightarrow \infty} \Cov[\dup{\Psi}{\degree}_i(\serv{\Lambda}),\dup{\Psi}{\degree}_j(\serv{\Lambda})] &=  \Cov[{\Psi}_i(\serv{\Lambda}),{\Psi}_j(\serv{\Lambda})], \\
 \lim_{\degree \rightarrow \infty} \Cov[\dup{\Psi}{\degree}_i(\serv{\Lambda}_{\nu}),\dup{\Psi}{\degree}_j(\serv{\Lambda}_{\nu})] &=  \Cov[{\Psi}_i(\serv{\Lambda}_{\nu}),{\Psi}_j(\serv{\Lambda}_{\nu})].
\end{align*}
Furthermore, using the induction hypothesis $(\serv{X}_{\star}, \dup{\serv{X}_1}{\degree}, \dotsc, \dup{\serv{X}_t}{\degree}; \serv{A})  \wc (\serv{X}_{\star}, \serv{X}_1, \dotsc, \serv{X}_t; \serv{A})$ and the fact that the functions $\fnonlin_1, \dotsc, \fnonlin_t$ are continuously differentiable and Lipschitz, we conclude that:
\begin{align*}
\lim_{\degree \rightarrow \infty} \E[\serv{X}_\star \dup{\serv{F}}{\degree}_{i}] & = \E[\serv{X}_\star {\serv{F}}{}_{i}] , \quad \lim_{\degree \rightarrow \infty} \E[\dup{\serv{F}}{\degree}_{i} \dup{\serv{F}}{\degree}_{j}] = \E[{\serv{F}}{}_{i}  {\serv{F}}_{j}].
\end{align*}
Plugging in these limits into the formulae for $(\dup{\beta}{\degree}_{i})_{i \in [t+1]}$ and ${\Sigma}{}_{t+1}, \dup{\Sigma}{\degree}_{t+1}$, immediately yields the desired conclusion \eqref{eq:SE-conv-induction-goal-suff} and hence proves \eqref{eq:SE-conv-induction-goal}. This completes the proof of the claim \eqref{eq:SE-rv-conv} by induction. 
\paragraph*{Convergence of Iterates (Proof of \clref{claim:itr-conv})} We  need to show that the iterates generated by the algorithm in \eqref{eq:low-degree-approx-div-free-algo} approximate the iterates generated by the original OAMP algorithm \eqref{eq:div-free-algo-recall} in the sense that,
\begin{align*}
\limsup_{\degree \rightarrow \infty} \plimsup_{\dim \rightarrow \infty} \frac{\|\vx_t- \dup{\vx_t}{\degree} \|^2}{\dim} & = 0\quad \forall \; t \in  \N.
\end{align*}
We will shorthand asymptotic approximation statements like the above using the notation:
\begin{align} \label{eq:poly-approx-goal}
\dup{\vx_t}{\degree}  \quad  \explain{$\dim, \degree \rightarrow \infty$}{\simeq} \iter{\vx}{t} \quad \forall \; t \in  \N.
\end{align}
We will show the above claim using induction. As our induction hypothesis, we will assume that the claim holds for all the iterates generated before step $t$ of the OAMP algorithm:
\begin{align} \label{eq:lowdegree-approx-induction-hypo}
\dup{\vx_s}{\degree}  \quad  \explain{$\dim, \degree \rightarrow \infty$}{\simeq} \iter{\vx}{s} \quad \forall \; s< t.
\end{align}
To complete the inductive proof, we need to show that the claim applies to the iterate generated at step $t$:
\begin{align*}
\dup{\vx_t}{\degree}  \quad  &\explain{$\dim, \degree \rightarrow \infty$}{\simeq} \iter{\vx}{t}.
\end{align*}
Recalling the update equation for ${\vx_t}$ from \eqref{eq:div-free-algo-recall}, we have:
\begin{align*}
&\iter{\vx}{t}  = \mfunc_{t}(\mY) \cdot \fnonlin_t(\iter{\vx}{1}, \dotsc, \vx_{t-1}; \va) \\
&\stackrel{(a)}{\explain{$\dim, \degree \rightarrow \infty$}{\simeq}}  \left( \dup{\mfunc_{t}}{\degree}(\mY) - \E[\dup{\mfunc_{t}}{\degree}(\serv{\Lambda})] \cdot \mI_{\dim} \right)\cdot \fnonlin_t(\iter{\vx}{1}, \dotsc, \vx_{t-1}; \va) \\
&\stackrel{(b)}{\explain{$\dim, \degree \rightarrow \infty$}{\simeq}}  \left( \dup{\mfunc_{t}}{\degree}(\mY) - \E[\dup{\mfunc_{t}}{\degree}(\serv{\Lambda})]  \mI_{\dim} \right)  \left( \fnonlin_t(\iter{\vx}{<t}; \va) - \sum_{s=1}^{t-1} \E[\partial_{s} \fnonlin_t(\dup{\serv{X}_{<t}}{\degree}; \serv{A})] {\vx_s} \right) \\
& \stackrel{(c)}{\explain{$\dim, \degree \rightarrow \infty$}{\simeq}}  \left( \dup{\mfunc_{t}}{\degree}(\mY) - \E[\dup{\mfunc_{t}}{\degree}(\serv{\Lambda})]  \mI_{\dim} \right)  \left( \fnonlin_t(\dup{\vx_{<t}}{\degree}; \va) - \sum_{s=1}^{t-1}  \E[\partial_{s} \fnonlin_t(\dup{\serv{X}_1}{\degree}, \dotsc, \dup{\serv{X}_{t-1}}{\degree}; \serv{A})]  \dup{\vx_s}{\degree} \right) \\
& \explain{\eqref{eq:low-degree-approx-div-free-algo}}{=} \dup{\vx_t}{\degree}.
\end{align*}
In the above display:
\begin{enumerate}
\item The approximation in step (a) relies on the observations:
\begin{align*}
\limsup_{\degree \rightarrow \infty} \plimsup_{\dim \rightarrow \infty} \|\mfunc_{t}(\mY) - \dup{\mfunc_{t}}{\degree}(\mY)\|_{\op} &\leq \lim_{\degree \rightarrow \infty}  \sup_{|\lambda| \leq K} \left|\dup{\mfunc_{t}}{\degree}(\lambda) - \mfunc_{t}(\lambda) \right|  \explain{\eqref{eq:weierstrass}}{=} 0, \\
 \lim_{\degree \rightarrow \infty}\E[\dup{\mfunc_{t}}{\degree}(\serv{\Lambda})]  = \E[\mfunc_t(\serv{\Lambda})] &\explain{Def. \ref{Def:OAMP_main}}{=}  0.
\end{align*}
\item The approximation in step (b) uses \clref{claim:SE-conv} and the fact that $\fnonlin_t$ is Lipschitz to conclude that:
\begin{align*}
\lim_{\degree \rightarrow \infty} \E[\partial_{s} \fnonlin_t(\dup{\serv{X}_1}{\degree}, \dotsc, \dup{\serv{X}_{t-1}}{\degree}; \serv{A})] & = \E[\partial_{s} \fnonlin_t({\serv{X}_1}, \dotsc, {\serv{X}_{t-1}}; \serv{A})] = 0,
\end{align*}
where the last equality follows by recalling that the iterate denoisers satisfy the divergence-free requirement $\E[\partial_{s} \fnonlin_t({\serv{X}_1}, \dotsc, {\serv{X}_{t-1}}; \serv{A})] = 0$ (recall Definition~\ref{Def:OAMP_main}). 
\item The approximation in step (c) relies on the induction hypothesis \eqref{eq:lowdegree-approx-induction-hypo} and the fact that the functions $\fnonlin_t$ are Lipschitz. 
\end{enumerate}
Hence, \eqref{eq:poly-approx-goal} holds by induction. This concludes the proof of \lemref{lem:poly-approx}.
\end{proof}

\subsection{Proof of \lemref{lem:aux1}} \label{appendix:aux1-proof}
\begin{proof}[Proof of \lemref{lem:aux1}]. Since any polynomial can be written as a linear combination of monomials, it suffices to prove the claims of the lemma when $\Psi(\lambda) = \lambda^{d} \; \forall \; \lambda \; \in \R$ for some $d \in \N$. 
\paragraph*{Proof of Claim (1)} We need to show that there exists a sequence of polynomials $(Q_{d})_{d \in \N}: \R \mapsto \R$ such that:
\begin{align} \label{eq:aux-lemma-goal1}
\mY^{d} \cdot \bm{x}_\star \; & \explain{$\dim\rightarrow \infty$}{\simeq} \; {Q}_{d}(\bm{W}) \cdot \bm{x}_\star \quad \forall \; d \; \in \; \N.
\end{align}
We will show that \eqref{eq:aux-lemma-goal1} holds by induction on $d$. We assume that \eqref{eq:aux-lemma-goal1} holds for some  $d \in \N$ and show that the claim also holds for $d+1$. Recalling that $\bm{Y}=\theta \cdot \frac{\bm{x}_\star\bm{x}_\star^\UT}{N}+\bm{W}$, we have:
\begin{align*}
\bm{Y}^{d+1} \cdot \bm{x}_\star & =  \left(\theta \cdot \frac{\bm{x}_\star\bm{x}_\star^\UT}{N}+\bm{W} \right) \cdot \mY^d  \cdot  \vx_{\star} \\
& \stackrel{(a)}{\explain{$\dim \rightarrow \infty$}{\simeq}} \left(\theta \cdot \frac{\bm{x}_\star\bm{x}_\star^\UT}{N}+\bm{W} \right) \cdot Q_{d}(\mW) \cdot \vx_{\star} \\
& = \theta \cdot \frac{\vx_{\star}^\top Q_{d}(\mW) \vx_{\star}}{\dim} + \mW Q_{d}(\mW) \cdot \vx_{\star} \\
& \stackrel{(b)}{\explain{$\dim \rightarrow \infty$}{\simeq}}  \theta \cdot \E[Q_{d}(\serv{\Lambda})] +  \mW Q_{d}(\mW) \cdot \vx_{\star}.
\end{align*}
In the above display, step (a) follows from the induction hypothesis $\mY^d \cdot \vx_{\star} \explain{$\dim\rightarrow \infty$}{\simeq} \; {Q}_{d}(\bm{W}) \cdot \bm{x}_\star$. and step (b) uses a standard result on the concentration of quadratic forms of rotationally invariant matrices (see \factref{fact:qf} in \appref{appendix:misc}). We define the polynomial $Q_{d+1}: \R \mapsto \R$ as:
\begin{equation}\label{Eqn:Q_recursive_def}
Q_{d+1}(\lambda) = \theta \cdot \E[Q_{d}(\serv{\Lambda})] + \lambda Q_d(\lambda) \quad \forall \; \lambda \;  \in \;  \R.
\end{equation}
Hence, $\bm{Y}^{d+1} \cdot \bm{x}_\star \explain{$\dim\rightarrow \infty$}{\simeq} \; {Q}_{d+1}(\bm{W}) \cdot \bm{x}_\star$, as desired. This proves the first claim made in the lemma by induction.
\paragraph*{Proof of Claim (2)} The proof of the second claim follows from an analogous induction argument. Specifically, we show by induction that:
\begin{align*}
    \mY^{d} \cdot \vv \; & \explain{$\dim\rightarrow \infty$}{\simeq} \; \mW^d \cdot \vv \quad \forall \; d \; \in \; \N.
\end{align*}As before, we assume the claim holds for some $d \in \N$, and we verify it for $d+1$:
\begin{align*}
\bm{Y}^{d+1} \cdot \vv & =  \left(\theta \cdot \frac{\bm{x}_\star\bm{x}_\star^\UT}{N}+\bm{W} \right) \cdot \mY^d  \cdot  \vv \\
& \stackrel{(a)}{\explain{$\dim \rightarrow \infty$}{\simeq}} \left(\theta \cdot \frac{\bm{x}_\star\bm{x}_\star^\UT}{N}+\bm{W} \right) \cdot \mW^d \cdot \vv \\
& = \theta \cdot \frac{\vx_{\star}^\top \mW^d \vv}{\dim} + \mW^{d+1} \cdot \vx_{\star} \\
& \stackrel{(b)}{\explain{$\dim \rightarrow \infty$}{\simeq}}  \mW^{d+1} \cdot \vv,
\end{align*}
where step (a) follows from the induction hypothesis and step (b) uses the assumption $\vx_{\star}^\top \mW^d \vv/\dim  \pc 0$. This concludes the proof of this lemma. 
\end{proof}
\subsection{Proof of \lemref{lem:nu-trick}} \label{appendix:nu-trick-proof}
\begin{proof}[Proof of \lemref{lem:nu-trick}]
We prove the second claim, the proof of the first claim is analogous. Since the spectral measure of $\mW$ converges to $\mu$ (recall Assumption~\ref{assump:signal-noise}), 
\begin{align*}
\frac{\Tr[ \tilde{\mfunc}_s(\mW)  \tilde{\mfunc}_t(\mW)]}{\dim} \pc  \E_{\serv{\Lambda} \sim \mu}[\tilde{\mfunc}_s(\serv{\Lambda}) \tilde{\mfunc}_t(\serv{\Lambda})], \;\; \frac{\Tr[ \tilde{\mfunc}_s(\mW)  \tilde{\mfunc}_t(\mW)]}{\dim} \pc  \E_{\serv{\Lambda} \sim \mu}[\tilde{\mfunc}_s(\serv{\Lambda}) \tilde{\mfunc}_t(\serv{\Lambda})]. 
\end{align*}
To obtain a more convenient formula for the limit which does not involve the transformed polynomials $(\tilde{\mfunc}_t)_{t \in \N}$, we observe that:
\begin{align*}
\E[\tilde{\mfunc}_s(\serv{\Lambda}) \tilde{\mfunc}_t(\serv{\Lambda})] & \explain{(a)}{=} \plim_{\dim \rightarrow \infty} \frac{\vx_{\star}^\top \tilde{\mfunc}_s(\mW) \tilde{\mfunc}_t(\mW) \vx_{\star}}{\dim} \explain{Lem. \ref{lem:aux1}}{=} \plim_{\dim \rightarrow \infty} \frac{\vx_{\star}^\top {\mfunc}_s(\mY) {\mfunc}_t(\mY) \vx_{\star}}{\dim}. 
\end{align*}
where the equality in step (a) follows from standard concentration results for quadratic forms of rotationally invariant matrices (see \factref{fact:qf} in \appref{appendix:misc}). Notice that the RHS of the above display can be written as an expectation with respect to $\nu_{\dim}$, the spectral measure of $\mY$ in the direction of $\vx_{\star}$. Indeed, if $\lambda_{1:\dim}(\mY)$ and $\vu_{1:\dim}(\mY)$ denote the eigenvalues and corresponding eigenvectors of $\mY$:
\begin{align*}
\frac{\vx_{\star}^\top {\mfunc}_s(\mY) {\mfunc}_t(\mY) \vx_{\star}}{\dim} &= \frac{1}{\dim} \sum_{i=1}^{\dim} \ip{\vx_{\star}}{\vu_i(\mY)}^2 \cdot \mfunc_s(\lambda_i(\mY)) {\mfunc}_t(\lambda_i(\mY)) \\ &= \int_{\R} \mfunc_s(\lambda) \mfunc_t(\lambda) \;  \nu_{\dim}(\diff \lambda).
\end{align*}
By Lemma~\ref{lem:RMT} (item (1)), we know that $\nu_{\dim}$ converges weakly to $\nu$. Hence,
\begin{align*}
\plim_{\dim \rightarrow \infty} \frac{\Tr[ \tilde{\mfunc}_s(\mW)  \tilde{\mfunc}_t(\mW)]}{\dim} = \plim_{\dim \rightarrow \infty} \frac{\vx_{\star}^\top \tilde{\mfunc}_s(\mW) \tilde{\mfunc}_t(\mW) \vx_{\star}}{\dim} = \E[\mfunc_s(\serv{\Lambda}_{\nu})  \mfunc_t(\serv{\Lambda}_{\nu})], \quad \serv{\Lambda}_{\nu} \sim \nu,
\end{align*}
as claimed. 
\end{proof}

\subsection{Proof of \lemref{lem:aux-OAMP-dyn}} \label{app:aux-OAMP-dyn}
\begin{proof}
Recall that we started with a general OAMP algorithm (cf. Definition~\ref{Def:OAMP_main}) with polynomial matrix denoising functions:
\begin{align} \label{eq:div-free-algo-recall-2}
\iter{\vx}{t} & = \mfunc_{t}(\mY) \cdot \fnonlin_t(\iter{\vx}{1}, \dotsc, \vx_{t-1}; \va) \quad  \forall \; t \; \in  \;  \N,
\end{align}
with associated state evolution random variables $$(\serv{X}_{\star}, (\serv{X}_{t})_{t\in \N}; \serv{A}),$$ and constructed an auxiliary OAMP algorithm:
\begin{align} \label{eq:aux-OAMP-recall}
\tilde{\vx}_t & = \alpha_t \cdot \tilde{\mfunc}_t(\mW) \cdot \vx_{\star} + \mfunc_t(\mW)  \cdot \fnonlin_t^{\perp}(\tilde{\vx}_1, \dotsc, \tilde{\vx}_{t-1}; \vx_{\star}, \va) \quad \forall \; t \; \in \; \N,
\end{align}
which is intended as an easy-to-analyze approximation to the original OAMP algorithm in \eqref{eq:div-free-algo-recall-2}. In the above display:
\begin{align} \label{eq:perp-f-recall}
\alpha_{t} \explain{def}{=} \E[\serv{X}_{\star} \fnonlin_t(\serv{X}_1, \dotsc, \serv{X}_{t-1}; \serv{A})], \quad  \fnonlin_t^{\perp}(x_1, \dotsc, x_{t-1}; x_{\star}, \aux) \explain{def}{=} \fnonlin_t(x_1, \dotsc, x_{t-1}; \aux) - \alpha_t x_{\star},
\end{align}
and $(\tilde{\Psi}_{t})_{t \in \N}$ are the transformed polynomials associated with matrix denoisers $(\Psi_{t})_{t\in \N}$, which were constructed in \lemref{lem:aux1}. Our goal is to show that for any $t \in \N$ and any $i \in \N$,
\begin{align*}
(\vx_{\star}, \tilde{\vx}_1, \dotsc, \tilde{\vx}_{t}; \va) \wc (\serv{X}_{\star}, \serv{X}_1, \dotsc, \serv{X}_t ; \serv{A}), \quad
\frac{\ip{\mW^{i} \fnonlin^\perp_t(\vx_1, \dotsc, \vx_{t-1}; \vx_{\star}, \va)}{\vx_{\star}}}{\dim} \pc 0.
\end{align*}We will show that the auxiliary OAMP algorithm can be re-written as an algorithm whose dynamics have characterized in prior work \citep[Theorem 2] {dudeja2022spectral}. 
To this end, we fix a $i \in \N$ and a $t \in \N$.  Introduce the random variables:
\begin{align*}
\serv{\Lambda} \sim \mu, \quad \serv{\Lambda}_{\nu} \sim \nu.
\end{align*}
For any $s \leq t$, we define:
\begin{subequations}\label{eq:v-w-def}
\begin{align}
\vv_s &\explain{def}{=} \Big( \tilde{\Psi}_s(\mW) - \E[\tilde{\Psi}_s(\serv{\Lambda})] \cdot \mI_{\dim}\Big)\cdot\vx_{\star},  \quad \vw_s \explain{def}{=}\Psi_s(\mW) \cdot \fnonlin_s^\perp(\tilde{\vx}_1, \dotsc, \tilde{\vx}_{s-1} ;\bm{x}_\star,\bm{a}), \\
\vw_{t+1} & \explain{def}{=} (\mW^i - \E[\serv{\Lambda^i}] \cdot \mI_{\dim}) \cdot \fnonlin_t^\perp(\tilde{\vx}_1, \dotsc, \tilde{\vx}_{t-1} ;\bm{x}_\star,\bm{a}).
\end{align}
\end{subequations}
Comparing \eqref{eq:aux-OAMP-recall} and \eqref{eq:v-w-def}, we conclude that:
\begin{align} \label{eq:aux-OAMP-to-prior}
\tilde{\vx}_s = \beta_s \cdot \vx_{\star} + \alpha_s \cdot \vv_s + \vw_s \quad \forall \; s \; \leq \; t,
\end{align}
where we defined:
\begin{align*}
\beta_s \explain{def}{=} \alpha_s \cdot  \E[\tilde{\Psi}_s(\serv{\Lambda})] \explain{Lem. \ref{lem:nu-trick}}{=} \alpha_s \cdot \E[{\Psi}_s(\serv{\Lambda}_{\nu})]  \quad \forall \;  s \leq t.
\end{align*}
Hence, \eqref{eq:v-w-def} can be re-expressed as:
\begin{align*}
\vv_s &= \Big( \tilde{\Psi}_s(\mW) - \E[\tilde{\Psi}_s(\serv{\Lambda})] \cdot \mI_{\dim}\Big)\cdot\vx_{\star},  \\ \vw_s &\explain{\eqref{eq:aux-OAMP-to-prior}}{=}\Psi_s(\mW) \cdot \fnonlin_s^\perp(\beta_1 \vx_{\star} + \alpha_1 \vv_1 + \vw_1, \dotsc, \beta_{s-1} \vx_{\star} + \alpha_{s-1} \vv_{s-1} + \vw_{s-1} ;\bm{x}_\star,\bm{a}).
\end{align*}
The above algorithm is an instance of a Vector Approximate Message Passing (VAMP) algorithm, as defined in \citep[Section 4.1.2]{dudeja2022spectral}. Moreover, \citep[Theorem 2] {dudeja2022spectral} shows that:
\begin{align} \label{eq:prior-work-SE}
(\vx_{\star}, \vv_1, \dotsc, \vv_t, \vw_1, \dotsc, \vw_{t+1}; \va) \wc (\serv{X}_{\star}, \serv{V}_1, \dotsc, \serv{V}_t, \serv{W}_1, \dotsc, \serv{W}_{t+1}; \serv{A}),
\end{align}
where the collection of random variables:
\begin{align*}
    (\serv{X}_{\star}; \serv{A}), \quad (\serv{V}_1, \dotsc, \serv{V}_t, \serv{W}_1, \dotsc, \serv{W}_{t+1})
\end{align*}
are mutually independent.  The random variables $(\serv{V}_1, \dotsc, \serv{V}_t, \serv{W}_1, \dotsc, \serv{W}_{t+1})$ have zero mean and are jointly Gaussian. The covariance matrices of $(\serv{V}_1, \dotsc, \serv{V}_t)$ and  $(\serv{W}_1, \dotsc, \serv{W}_{t+1})$ are given by the recursion:
\begin{subequations} \label{eq:SE-cov-prior-work}
    \begin{align} 
\E[\serv{V}_s \serv{V}_{\tau}] & = \Cov[\tilde{\Psi}_s(\serv{\Lambda}), \tilde{\Psi}_{\tau}(\serv{\Lambda})] \explain{Lem. \ref{lem:nu-trick}}{=} \Cov[\Psi_s(\serv{\Lambda}_{\nu}), \Psi_{\tau}(\serv{\Lambda}_{\nu})] \quad \forall \; s, \tau \leq t, \\
\E[\serv{W}_s \serv{W}_{\tau}] & = \Cov[\Psi_s(\serv{\Lambda}), \Psi_{\tau}(\serv{\Lambda})] \cdot (\E[\serv{F}_s \serv{F}_t] - \alpha_s \alpha_t) \quad \forall \; s, \tau \leq t, \\
\E[\serv{V}_s \serv{W}_{\tau}] & = \Cov[\tilde{\Psi}_s(\serv{\Lambda}), \Psi_{\tau}(\serv{\Lambda})] \cdot (\E[\serv{F}_\tau\serv{X}_{\star}] - \alpha_s) \quad \forall \; s, \tau \leq t.
\end{align}
\end{subequations}
In the above display, for any $\tau \leq t$, $$\serv{F}_\tau \explain{def}{=} \fnonlin_s(\beta_1 \serv{X}_{\star} + \alpha_1 \serv{V}_1 + \serv{W}_1, \dotsc, \beta_{s\tau-1} \serv{X}_{\star} + \alpha_{\tau-1} \serv{V}_{\tau-1} + \serv{W}_{\tau-1}; \serv{A}).$$ Consequently,
\begin{align}
&(\vx_{\star}, \tilde{\vx}_1, \dotsc, \tilde{\vx}_{t}; \va) \explain{\eqref{eq:aux-OAMP-to-prior}}{=} (\vx_{\star}, \beta_1 \cdot \vx_{\star} + \alpha_1 \cdot \vv_1 + \vw_1, \dotsc, \beta_t \cdot \vx_{\star} + \alpha_t \cdot \vv_t + \vw_t; \va) \nonumber \\
& \hspace{0.2cm} \wc (\serv{X}_{\star}, \beta_1 \cdot \serv{X}_{\star} + \alpha_1 \cdot \serv{V}_1 + \serv{W}_1, \dotsc, \beta_t \cdot \serv{X}_{\star} + \alpha_t \cdot \serv{V}_t + \serv{W}_t; \serv{A}) \quad \text{[Using \eqref{eq:prior-work-SE}]}.  \label{eq:aux-SE-claim-extra}
\end{align}
{We claim that:
\begin{subequations} \label{eq:reviewer-details}
    \begin{align} 
    \E[\serv{V}_\tau \serv{W}_{\tau^\prime}] & = 0 \quad \forall \; \tau^\prime, \tau \leq t, \\
    (\serv{X}_{\star}, \beta_1 \cdot \serv{X}_{\star} + \alpha_1 \cdot \serv{V}_1 + \serv{W}_1, \dotsc, \beta_t \cdot \serv{X}_{\star} + \alpha_t \cdot \serv{V}_t + \serv{W}_t; \serv{A}) &\explain{d}{=} (\serv{X}_{\star}, \serv{X}_1, \dotsc, \serv{X}_t ; \serv{A}).
\end{align}
\end{subequations}
The above claims are easily proved by induction. The base case is immediate, and if we assume:
\begin{subequations}  \label{eq:reviewer-details-hypothesis}
\begin{align}
&\E[\serv{V}_\tau \serv{W}_{\tau^\prime}]  = 0 \quad \forall \; \tau \leq t, \; \tau^\prime < s \\
    &(\serv{X}_{\star}, \beta_1 \cdot \serv{X}_{\star} + \alpha_1 \cdot \serv{V}_1 + \serv{W}_1, \dotsc, \beta_{s-1} \cdot \serv{X}_{\star} + \alpha_{s-1} \cdot \serv{V}_{s-1} + \serv{W}_{s-1}; \serv{A}) \explain{d}{=} (\serv{X}_{\star}, \serv{X}_1, \dotsc, \serv{X}_{s-1} ; \serv{A}),
\end{align}
\end{subequations}
for some $s \leq t$ as our induction hypothesis, we can complete the induction step:
\begin{align*}
    &\E[\serv{V}_\tau \serv{W}_s]  \explain{\eqref{eq:SE-cov-prior-work}}{=} \Cov[\tilde{\Psi}_\tau(\serv{\Lambda}), \Psi_{s}(\serv{\Lambda})] \cdot (\E[\serv{F}_s\serv{X}_{\star}] - \alpha_s) \quad \text{ for any $\tau \leq t$} \\
    & = \Cov[\tilde{\Psi}_\tau(\serv{\Lambda}), \Psi_{s}(\serv{\Lambda})] \cdot (\E[\serv{F}_s(\serv{X}_1,\dotsc, \serv{X}_{s -1}; \serv{A}) \cdot \serv{X}_{\star}] - \alpha_s) & \text{[by induction hypothesis \eqref{eq:reviewer-details-hypothesis}]} \\
    & \explain{\eqref{eq:perp-f-recall}}{=} 0.
\end{align*}
On the other hand, by comparing \eqref{eq:SE-cov-prior-work} and the description of the joint distribution of the state evolution random variables $(\serv{X}_{\star}, \serv{X}_1, \dotsc, \serv{X}_s ; \serv{A})$ corresponding to the original OAMP algorithm \eqref{eq:div-free-algo-recall-2} provided in Definition~\ref{Def:OAMP_main} we conclude that:
\begin{align*}
     (\serv{X}_{\star}, \beta_1 \cdot \serv{X}_{\star} + \alpha_1 \cdot \serv{V}_1 + \serv{W}_1, \dotsc, \beta_{s} \cdot \serv{X}_{\star} + \alpha_{s} \cdot \serv{V}_{s} + \serv{W}_{s}; \serv{A}) &\explain{d}{=} (\serv{X}_{\star}, \serv{X}_1, \dotsc, \serv{X}_{s} ; \serv{A}),
\end{align*}
which proves the claims in \eqref{eq:reviewer-details} by induction. Combining \eqref{eq:aux-SE-claim-extra} with \eqref{eq:reviewer-details} we conclude that:
\begin{align}  \label{eq:aux-SE-claim1}
(\vx_{\star}, \tilde{\vx}_1, \dotsc, \tilde{\vx}_{t}; \va) &\wc (\serv{X}_{\star}, \serv{X}_1, \dotsc, \serv{X}_t ; \serv{A}),
\end{align}
which proves the first claim made in the lemma.} To prove the second claim, we observe that:
\begin{align*}
\frac{\ip{\mW^{i} \fnonlin^\perp_t(\vx_1, \dotsc, \vx_{t-1}; \vx_{\star}, \va)}{\vx_{\star}}}{\dim} & \explain{\eqref{eq:v-w-def}}{=} \frac{\ip{\vw_{t+1}}{\vx_{\star}}}{\dim} + \E[\serv{\Lambda}^i] \cdot \frac{\ip{\fnonlin_t^\perp(\tilde{\vx}_1, \dotsc \tilde{\vx}_{t-1}; \vx_{\star}, \va)}{\vx_{\star}}}{\dim} \\
&\pc \E[\serv{W}_{t+1}\serv{X}_{\star}] + \E[\serv{\Lambda}^i] \cdot \E[\serv{X}_{\star} \fnonlin_t^\perp(\serv{X}_1, \dotsc, \serv{X}_{t-1}; \serv{X}_{\star}, \serv{A})] \\
& = 0.
\end{align*}
In the above display, the convergence in the second step follows from \eqref{eq:prior-work-SE} and \eqref{eq:aux-SE-claim1}. The equality in the last step follows by recalling that $\serv{W}_{t+1}$ is a mean-zero Gaussian random variable independent of $\serv{X}_{\star}$ and the definition of $\fnonlin_t^\perp$ in \eqref{eq:perp-f-recall} ensures $\E[\serv{X}_{\star} \fnonlin_t^\perp(\serv{X}_1, \dotsc, \serv{X}_{t-1}; \serv{X}_{\star}, \serv{A})]$ = 0.
\end{proof}

\section{State Evolution for Optimal OAMP (Proposition~\ref{prop:optimal-OAMP-SE})}\label{App:proof_OAMP_opt_SE}

Under the assumption $\E\Var[\serv{X}_{\star} | \serv{A}] \in (0,1)$, for the readers convenience, we recall that the optimal OAMP algorithm is given by (cf. \eqref{eq:optimal-OAMP}):
\begin{align}
    {\vx}_{t} & = \frac{1}{\sqrt{{\omega}_{t}}}\left( 1+  \frac{1}{{\rho}_{t}}  \right) \cdot \omd(\mY; {\rho}_{t}) \cdot \dfbdnsr({\vx}_{t-1}; \va | {\omega}_{t-1}).
    \end{align}
    The estimator returned by the optimal OAMP algorithm at iteration $t$ is:
    \begin{align}
    \hat{\vx}_t \explain{def}{=} \bdnsr(\vx_t; \va | \omega_t).
    \end{align}
    In the above equations, the matrix denoiser $\omd$ used by the optimal OAMP algorithm is given by:
    \begin{align} \label{eq:optimal-OAMP-denoiser-2}
     \omd(\lambda; \rho) & = 1 - \left( \E \left[ \frac{ \phi(\serv{\Lambda})}{ \phi(\serv{\Lambda}) + \rho} \right] \right)^{-1} \cdot \frac{\phi(\lambda)}{\phi(\lambda)+\rho} \quad \forall \; \lambda \; \in \; \R, \; \rho \; \in \;  (0,\infty),
    \end{align}
    where $\serv{\Lambda} \sim \mu$ and the function $\phi: \R \mapsto \R$ was introduced in Lemma~\ref{lem:RMT} (recall \eqref{eq:phi}). The parameters ${\omega}_{t}$ and ${\rho}_{t}$ are computed using the following recursion, initialized with $\omega_0 \explain{def}{=} 0$:
    \begin{align}  \label{eq:rho-omega-recall}
    {\rho}_{t} = \mathscr{F}_2(\omega_{t-1}), \; \omega_t = \mathscr{F}_1(\rho_t) \quad \text{where} \quad \mathscr{F}_2(\omega) \explain{def}{=} \frac{1}{\dmmse({\omega})} - 1, \quad \mathscr{F}_1(\rho) \explain{def}{=}   1 - \frac{ \E \left[\frac{1}{\phi(\serv{\Lambda})+{\rho}} \right]}{ \E \left[ \frac{\phi(\serv{\Lambda})}{\phi(\serv{\Lambda}) + {\rho}} \right] }. 
    \end{align}
 Let $(\serv{X}_{\star}, (\serv{X}_t)_{t \in \N}; \serv{A})$ denote the state evolution random variables associated with the above algorithm.  We will rely on the following intermediate lemma.
\begin{lemma} \label{lem:scrF-prop} We have,
\begin{enumerate}
    \item The function $\mathscr{F}_1$ is a continuous non-decreasing function which maps $(0,\infty)$ to $[0,1)$ and satisfies $\lim_{\rho \rightarrow \infty} \mathscr{F}_1(\rho) < 1$. 
   \item The function $\mathscr{F}_2$ is a continuous non-decreasing function which maps $[0,1)$ to $(0,\infty)$. 
\end{enumerate} 
\end{lemma}
We defer the proof of this lemma to the end of this appendix (\appref{app:scrF-prop-proof}), and present the proof of Proposition~\ref{prop:optimal-OAMP-SE}.
\begin{proof}[Proof of Proposition~\ref{prop:optimal-OAMP-SE}] We prove each of the two claims below. 
\paragraph*{Proof of Claim (1)} We will show by induction that:
\begin{align} \label{eq:induction-SE}
\omega_t \in [0,1), \quad \rho_t \in (0,\infty), \quad (\serv{X}_{\star}, \serv{X}_t; \serv{A}) \text{ form a scalar Gaussian channel with SNR $\omega_t$.}
\end{align}
Assuming this claim, we can compute the asymptotic MSE of the estimator $\hat{\vx}_t = \bdnsr(\vx_t; \va | \effsnrt{t})$ returned by the OAMP algorithm:
\begin{align*}
\plim_{\dim \rightarrow \infty} \frac{\|\hat{\vx}_t - \vx_{\star} \|^2}{\dim} & \explain{Thm. \ref{thm:SE}}{=} \E| \bdnsr(\serv{X}_t; \serv{A} | \effsnrt{t}) - \serv{X}_{\star}|^2 \explain{(a)}{=}  \mmse(\omega_t).
\end{align*}
In the above display, step (a) follows by observing that $\bdnsr(\serv{X}_t; \serv{A} | \effsnrt{t})$ is the MMSE estimator for the scalar Gaussian channel $(\serv{X}_{\star}, \serv{X}_t; \serv{A})$ (which has SNR $\omega_t$). We now focus on proving \eqref{eq:induction-SE} by induction.  We assume that the claim \eqref{eq:induction-SE} holds at step $t$ as the induction hypothesis and show that it also holds at step $t+1$. Indeed, since $\omega_t \in [0,1)$, \lemref{lem:scrF-prop} guarantees that $\rho_{t+1} = \mathscr{F}_2(\omega_{t}) \in (0,\infty)$ and $\omega_{t+1} = \mathscr{F}_1(\rho_{t+1}) \in [0,1)$. Next, we verify that $(\serv{X}_{\star}, \serv{X}_{t+1}; \serv{A})$ form a scalar Gaussian channel with SNR $\omega_{t+1}$. From Definition~\ref{Def:OAMP_main}, we know that:
    \begin{align*}
    \serv{X}_{t+1} | (\serv{X}_{\star}; \serv{A}) \sim \gauss{\beta_{t+1} \serv{X}_{\star}}{\sigma_{t+1}^2},
    \end{align*}
    where:
    \begin{eqnarray*}
    \beta_{t+1}  & =& \frac{\E[\serv{X}_{\star} \cdot \dfbdnsr(\serv{X}_t; \serv{A} | \effsnrt{t}{})] \cdot \E[{\omd(\serv{\Lambda}_{\nu} \; ; \rhot{t+1}{})}]}{\sqrt{\effsnrt{t+1}{}}(1-\dmmse(\effsnrt{t+1})) } \cdot , \\ 
    \sigma_{t+1}^2 &=& \frac{ \{\E[\serv{X}_{\star} \cdot \dfbdnsr(\serv{X}_t; \serv{A} | \effsnrt{t}{})]\}^2  \cdot \Var[{\omd(\serv{\Lambda}_{\nu} \; ; \rhot{t+1}{})}]}{\effsnrt{t+1}{}(1-\dmmse(\effsnrt{t+1}))^2}  \\
    && + \frac{\{\E[\dfbdnsr^2(\serv{X}_t; \serv{A} | \effsnrt{t}{})] - (\E[\serv{X}_{\star} \cdot \dfbdnsr(\serv{X}_t; \serv{A} | \effsnrt{t}{})])^2 \}\cdot \E[{\omd^2(\serv{\Lambda} \; ; \rhot{t+1}{})}]}{\effsnrt{t+1}{}(1-\dmmse(\effsnrt{t+1}))^2}.
    \end{eqnarray*}
    In the above equations $\serv{\Lambda} \sim \mu$ and $\serv{\Lambda}_{\nu} \sim \nu$. We begin by simplifying $\beta_{t+1}$. By the induction hypothesis, $(\serv{X}_{\star}, \serv{X}_t; \serv{A})$ forms a Gaussian channel with SNR $\effsnrt{t}{}$. A convenient property of the DMMSE estimator for a scalar Gaussian channel is the following identity (see \lemref{lem:dmmse-scalar} in \appref{appendix:gauss-channel}):  
    \begin{align} \label{eq:Gauss-channel-convenient}
        \E[\serv{X}_{\star} \cdot \dfbdnsr(\serv{X}_t; \serv{A} | \effsnrt{t}{})] = \E[\dfbdnsr^2(\serv{X}_t; \serv{A} | \effsnrt{t}{})] =  1 - \dmmse(\effsnrt{t}{}).
    \end{align}
    To  compute $\E[{\omd(\serv{\Lambda}_{\nu} \; ; \rhot{t+1}{})}]$, we consider the Lebesgue decomposition of $\nu$ into the absolutely continuous part $\nu_{\parallel}$ and the singular part $\nu_{\perp}$:
    \begin{align} 
    \E[{\omd(\serv{\Lambda}_{\nu} \; ; \rhot{t+1}{})}] &\explain{\eqref{eq:optimal-OAMP-denoiser-2}}{=} 1 -  \left( \E \left[ \frac{ \phi(\serv{\Lambda})}{ \phi(\serv{\Lambda}) + \rho} \right] \right)^{-1} \cdot \E \left[\frac{\phi(\serv{\Lambda}_{\nu})}{\phi(\serv{\Lambda}_{\nu})+\rho_{t+1}} \right] \nonumber \\ & \explain{(a)}{=} 1 -  \left( \E \left[ \frac{ \phi(\serv{\Lambda})}{ \phi(\serv{\Lambda}) + \rho} \right] \right)^{-1} \cdot \int_{\R} \frac{\phi(\lambda)}{\phi(\lambda)+\rho_{t+1}}  \nu_{\parallel}(\diff \lambda)  \nonumber\\& \explain{(b)}{=} 1 - \frac{ \E \left[\frac{1}{\phi(\serv{\Lambda})+{\rho_{t+1}}} \right]}{ \E \left[ \frac{\phi(\serv{\Lambda})}{\phi(\serv{\Lambda}) + {\rho_{t+1}}} \right] } \explain{\eqref{eq:rho-omega-recall}}{=} \omega_{t+1}. \label{eq:convenient-formula-2}
    \end{align}
    In the above display, step (a) follows by recalling that $\phi(\lambda) = 0$ for $\nu_\perp$-a.e. $\lambda$ (Lemma~\ref{lem:RMT}, item (2)) and step (b) follows from recalling that the density of $\nu_{\parallel}$ is given by $\mu(\cdot)/\phi(\cdot)$ (Lemma~\ref{lem:RMT}, item (3)).
   Plugging in the identities \eqref{eq:Gauss-channel-convenient} and \eqref{eq:convenient-formula-2} into the expression for $\beta_{t+1}$, we obtain: \begin{equation}\label{Eqn:OAMP_opt_alpha_omega}
         \beta_{t+1}  = \frac{(1 - \dmmse(\effsnrt{t}{})) \cdot \omega_{t+1}}{\sqrt{\effsnrt{t+1}{}}(1-\dmmse(\effsnrt{t+1}))} \explain{}{=} \sqrt{\effsnrt{t+1}{}}.
    \end{equation}
    Similarly we can also simplify $\sigma_{t+1}^2$ by exploiting the identities \eqref{eq:Gauss-channel-convenient} and \eqref{eq:convenient-formula-2}:
    \begin{subequations}\label{Eqn:OAMP_opt_sigma2_omega}
    \begin{align}
        \sigma_{t+1}^2 & \explain{\eqref{eq:Gauss-channel-convenient}}{=} \frac{ \Var[{\omd(\serv{\Lambda}_{\nu} \; ; \rhot{t+1}{})}]  + \tfrac{\dmmse(\effsnrt{t+1})}{1-\dmmse(\effsnrt{t+1})} \cdot \E[{\omd^2(\serv{\Lambda} \; ; \rhot{t+1}{})}]}{\effsnrt{t+1}{}} \\
        & \explain{\eqref{eq:rho-omega-recall}}{=} \frac{\E[{\omd^2(\serv{\Lambda}_{\nu} \; ; \rhot{t+1}{})}] - (\E[{\omd(\serv{\Lambda}_{\nu} \; ; \rhot{t+1}{})}])^2  + \tfrac{1}{\rhot{t+1}} \cdot \E[{\omd^2(\serv{\Lambda} \; ; \rhot{t+1}{})}]}{\effsnrt{t+1}{}} \\
        & \explain{\eqref{eq:convenient-formula-2}}{=} \frac{\E[{\omd^2(\serv{\Lambda}_{\nu} \; ; \rhot{t+1}{})}] + \tfrac{1}{\rhot{t+1}} \cdot \E[{\omd^2(\serv{\Lambda} \; ; \rhot{t+1}{})}]  - \omega_{t+1}^2}{\effsnrt{t+1}{}}.
    \end{align}
    After a straightforward (but tedious) calculation analogous to the one used to derive \eqref{eq:convenient-formula-2}, we obtain the identity:
    \begin{align}
    \E[{\omd^2(\serv{\Lambda}_{\nu} \; ; \rhot{t+1}{})}] + \tfrac{1}{\rhot{t+1}} \cdot \E[{\omd^2(\serv{\Lambda} \; ; \rhot{t+1}{})}] & = \omega_{t+1},
    \end{align}
    which implies that $\sigma_{t+1}^2 = 1 - \omega_{t+1}$. 
    \end{subequations}
    Hence, 
    \begin{equation}\label{Eqn:OAMP_opt_Xt_omega_app}
    \serv{X}_{t+1} | (\serv{X}_{\star}; \serv{A}) \sim \gauss{\sqrt{\effsnrt{t+1}{}} \cdot \serv{X}_{\star}}{1- \effsnrt{t+1}{}},
    \end{equation}
    which verifies that $(\serv{X}_{\star}, \serv{X}_{t+1}; \serv{A})$ forms a Gaussian channel with SNR $\effsnrt{t+1}{}$. 

\paragraph*{Proof of Claim (2)} Recall that the sequences $\{\omega_t\}_{t\in\mathbb{N}}$ and $\{\rho_t\}_{t\in\mathbb{N}}$ are generated via the recursion:
\begin{subequations}\label{Eqn:OAMP_FP_proof_temp1}
\begin{align}
{\rho}_{t} &= \mathscr{F}_2(\omega_{t-1})\bydef\frac{1}{\dmmse({\omega}_{t-1})} - 1,\label{Eqn:OAMP_FP_proof_temp1_a}\\
\omega_t &=\mathscr{F}_1(\rho_{t})\bydef 1 -  \left( \E \left[ \frac{\phi(\serv{\Lambda})}{\phi(\serv{\Lambda}) + {\rho}_{t}} \right] \right)^{-1} \cdot \E \left[\frac{1}{\phi(\serv{\Lambda})+{\rho}_{t}} \right],\label{Eqn:OAMP_FP_proof_temp1_b}
\end{align}
\end{subequations}
where $\omega_{0}=0$.  We first prove the monotonicity of $\{\omega_t\}_{t\in\mathbb{N}}$ using an induction argument. Note that $\omega_t = \mathscr{F}_1\circ \mathscr{F}_2(\omega_{t-1})$. First, note that $\omega_1\ge\omega_0=0$ (recall we have already shown $\omega_t \geq 0$ for all $t \in \N$). Since the composite function $\mathscr{F}_1\circ \mathscr{F}_2$ is non-decreasing (\lemref{lem:scrF-prop}), it follows that $\mathscr{F}_1\circ \mathscr{F}_2(\omega_1)\ge \mathscr{F}_1\circ \mathscr{F}_2(\omega_0)$, namely, $\omega_2\ge\omega_1$. Repeating this argument proves that $\{\omega_t\}_{t\in\mathbb{N}}$ is non-decreasing. Moreover, $\{\omega_t\}_{t\in\mathbb{N}}$ is bounded sequence taking values in $[0,1)$ (proved in the first claim). Hence, $\{\omega_t\}_{t\in\mathbb{N}}$ converges to a limit point $\omega_{\star} \in [0,1]$. We eliminate the possibility that $\omega_{\star} = 1$ by observing that \lemref{lem:scrF-prop} guarantees that:
\begin{align*}
    \omega_{\star} = \lim_{t \rightarrow \infty} \omega_{t+1} = \lim_{t \rightarrow \infty}\mathscr{F}_1\circ \mathscr{F}_2(\omega_t) \leq  \lim_{\rho \rightarrow \infty}  \mathscr{F}_1(\rho)  < 1. 
\end{align*} \lemref{lem:scrF-prop} also guarantees the continuity of $\mathscr{F}_1\circ \mathscr{F}_2$ on $[0,1)$, and hence, the limit point $\omega_{\star}$ satisfies the equation $\omega_{\star} = \mathscr{F}_1\circ \mathscr{F}_2(\omega_{\star})$. 
Since the sequence $\{\rho_t\}_{t\in\mathbb{N}}$ is obtained by applying a continuous and non-decreasing function $\mathscr{F}_2$ to a non-decreasing and convergent sequence $\{\omega_t\}_{t\in\mathbb{N}}$, we conclude that $\{\rho_t\}_{t\in\mathbb{N}}$ is non-decreasing and converges to a limit $\rho_{\star} \explain{def}{=} \mathscr{F}_2(\omega_{\star}) \in (0,\infty)$. Observe that the limit points satisfy the equations: $\omega_{\star} = \mathscr{F}_1\circ \mathscr{F}_2(\omega_{\star}), \; \rho_{\star} = \mathscr{F}_2(\omega_{\star})$, which can be expressed as $\omega_{\star} = \mathscr{F}_1(\rho_{\star}), \; \rho_{\star} = \mathscr{F}_2(\omega_{\star})$ as claimed.  This concludes the proof of the proposition.
\end{proof}

\subsection{Proof of \lemref{lem:scrF-prop}} \label{app:scrF-prop-proof}
\begin{proof}[Proof of \lemref{lem:scrF-prop}] We prove the claims made about $\mathscr{F}_1$ and $\mathscr{F}_2$ below. 
\paragraph*{Properties of $\mathscr{F}_1$} Recall that:
\begin{align} \label{eq:F_1-def-recall}
\mathscr{F}_1(\rho) \explain{def}{=}   1 - \frac{ \E \left[\frac{1}{\phi(\serv{\Lambda})+{\rho}} \right]}{ \E \left[ \frac{\phi(\serv{\Lambda})}{\phi(\serv{\Lambda}) + {\rho}} \right]} \quad \forall \; \rho \; \in \; (0,\infty). 
\end{align}
The continuity of $\mathscr{F}_1(\rho)$ follows from the Dominated Convergence Theorem. To verify the monotonicity of $\mathscr{F}_1$, we calculate its derivative:
\[
\begin{split} \label{eq:scrF-1-recall}
\mathscr{F}_1^\prime(\rho) &=\frac{\E\left[\left(\frac{1}{\phi(\sfLambda)+\rho}\right)^2\right]\cdot\E\left[\frac{\phi(\sfLambda)}{\phi(\sfLambda)+\rho}\right]-\E\left[\frac{1}{\phi(\sfLambda)+\rho}\right]\cdot \E\left[\frac{\phi(\sfLambda)}{(\phi(\sfLambda)+\rho)^2}\right]}{\left(\E\left[\frac{\phi(\sfLambda)}{\phi(\sfLambda)+\rho}\right]\right)^2}\ge0,
\end{split}
\]
where the inequality follows from Chebyshev's association inequality \cite[Theorem 2.14]{Boucheron2013}, which states that for any two random variables $(X,Z)$ such that $Z \geq 0$ and any two non-decreasing functions $f,g$ we have:
\begin{align*}
\E[Z] \E[f(X) g(X) Z] \geq \E[f(X) Z] \cdot \E[g(X) Z].
\end{align*}
Indeed, applying this inequality with $X = \phi(\serv{\Lambda})$, $f(X) = X$, $g(X) = X + \rho$ and $Z = 1/(\phi(\serv{\Lambda}) + \rho)^2$ shows that the numerator of $\mathscr{F}_1^\prime(\rho)$ is non-negative. Finally, we check that $\mathscr{F}_1: (0,\infty) \mapsto [0, 1)$. Notice from \eqref{eq:F_1-def-recall} that for any $\rho \in (0,\infty)$,  $\mathscr{F}_1(\rho) < 1$. On the other hand, since $\mathscr{F}_1(\rho) \geq 0$ since, thanks to the monotonicity of $\mathscr{F}_1$ we have:
\begin{align*}
\mathscr{F}_1(\rho) \geq \mathscr{F}_1(0) = 1 - \E\left[ \frac{1}{\phi(\serv{\Lambda})} \right] \explain{(a)}{=} 1 - \nu_{\parallel}(\R) \geq 0.
\end{align*}
In the above display, the step (a) follows by recalling from Lemma~\ref{lem:RMT} (item (3)) that the density of $\nu_{\parallel}$, the absolutely continuous part of $\nu$ is given by $\mu(\cdot)/\phi(\cdot)$ and by observing that $\nu_{\parallel}(\R) \leq \nu(\R) = 1$. Last we compute:
\begin{align*}
\lim_{\rho \rightarrow \infty}  \mathscr{F}_1(\rho) & = 1 - \lim_{\rho \rightarrow \infty} \frac{ \E \left[\frac{1}{\phi(\serv{\Lambda})+{\rho}} \right]}{ \E \left[ \frac{\phi(\serv{\Lambda})}{\phi(\serv{\Lambda}) + {\rho}} \right]} =1 -  \lim_{\rho \rightarrow \infty} \frac{ \E \left[\frac{1}{\rho^{-1} \cdot \phi(\serv{\Lambda})+{1}} \right]}{ \E \left[ \frac{\phi(\serv{\Lambda})}{\rho^{-1} \cdot \phi(\serv{\Lambda}) + {1}} \right]} = 1 - \frac{1}{\E[\phi(\serv{\Lambda})]} < 1,
\end{align*}
as claimed.
\paragraph*{Properties of $\mathscr{F}_2$} Recall that:
\begin{align*}
\mathscr{F}_2(\omega) \explain{def}{=} \frac{1}{\dmmse({\omega})} - 1.
\end{align*}
 The continuity and monotonicity of the DMMSE function $\dmmse(\cdot)$ (see From Lemma \ref{lem:dmmse-monotonicity} in Appendix \ref{appendix:gauss-channel}) implies that $\mathscr{F}_2$  continuous and non-decreasing. Next, we verify that $\mathscr{F}_2: [0,1) \mapsto (0,\infty)$. Notice that for any $\omega \in [0,1)$, $\mathscr{F}_2(\omega) \geq \mathscr{F}_2(0) = (\dmmse(0))^{-1} -1 = (\mmse(0))^{-1} - 1 > 0$ (recall that $\mmse(0) \in (0,1)$ was one of the hypothesis assumed in Proposition~\ref{prop:optimal-OAMP-SE}). Moreover, for any $\omega < 1$, $\mathscr{F}_2(\omega) < \infty$ since:
 \begin{align*}
 \dmmse(\omega) \geq \mmse(\omega) \explain{(a)}{>} \mmse(1) = 0,
 \end{align*}
 where (a) follows from the fact that $\mmse(\cdot)$ is strictly decreasing if $\mmse(0) \in (0,1)$ (see \factref{fact:mmse-monotonic} in \appref{appendix:gauss-channel}).  Hence, $\mathscr{F}_2: [0,1) \mapsto (0, \infty)$. This concludes the proof of the lemma. 
\end{proof}
\section{Proof of the Optimality Result (Theorem~\ref{thm:optimality})}
\label{appendix:optimality-proof}

\subsection{Proof of \propref{prop:LOAMP-reduction}} \label{appendix:LOAMP-reduction}
We consider a given iterative algorithm (Definition~\ref{def:itr}) of the form:
\begin{align} \label{eq:itr-recall}
\iter{\vr}{t} & = \mfunc_t(\mY) \cdot \fnonlin_t(\iter{\vr}{1}, \dotsc, \vr_{t-1}; \va) +  \gnonlin_t(\iter{\vr}{1}, \dotsc, \vr_{t-1}; \va) \quad \forall \; t \in \N,
\end{align}
which returns the following estimator after $t$ iterations:
\begin{align} \label{eq:optimal-gfom-estimator}
\hat{\vr}_t & = \psi_t(\iter{\vr}{1}, \dotsc, \iter{\vr}{t}; \va),
\end{align}
Our goal will be to construct a {lifted OAMP algorithm} (\defref{def:LOAMP}), which can approximate the estimator above. Throughout the proof, we will refer to the given iterative algorithm \eqref{eq:itr-recall} as the \emph{target algorithm}. Our construction of the desired lifted OAMP algorithm proceeds in two steps.
\paragraph*{Step 1 (Divergence Removal)} We first implement the target algorithm \eqref{eq:itr-recall} using an \emph{intermediate OAMP algorithm} which takes the form:
\begin{align} \label{eq:div-free-algo}
\iter{\vv}{t,S} & = \dfnew{\mfunc}_{S}(\mY)  \cdot \dfnew{\Fnonlin}_t(\iter{\vv}{<t,\bdot}; \va) \quad \forall \;  t \in  \N, \; S \subset \N, \; 1 \leq |S| < \infty.
\end{align}
 In the above display,
\begin{enumerate}
\item The iterates generated by the algorithm are indexed by two indices, a time index $t$, which takes values in $\N$, and a set index $S$, which can be any finite and non-empty subset of $\N$. 
\item For any finite subset $S = \{s_1, s_2, \dotsc, s_k \}$ with sorted elements $s_1 < s_2 < \dotsb < s_k$, we define the matrix denoiser:
\begin{align*}
\mfunc_{S}(\lambda) & \explain{def}{=} \mfunc_{s_k}(\lambda) \cdot \mfunc_{s_{k-1}}(\lambda) \cdot \dotsb \cdot \mfunc_{s_1}(\lambda), \quad \dfnew{\mfunc}_S(\lambda) \explain{def}{=} \mfunc_{S}(\lambda) - \E_{\serv{\Lambda} \sim \mu} \left[ \mfunc_{S}(\serv{\Lambda})\right]  \quad \forall \; \lambda \; \in \; \R.
\end{align*}
\item The iterates are generated in the following sequence:
\begin{enumerate}
\item In step (1), we generate the iterate $\iter{\vv}{1,\{1\}}$. 
\item In step (2), we generate the new iterates $\iter{\vv}{1,\{1,2\}}$ and $\iter{\vv}{2,\{1\}}, \iter{\vv}{2,\{1,2\}}$.
\item At step $(t)$, we generate the new iterates:
\begin{align*}
\{\iter{\vv}{1,S}: S \subset [t], t \in S\},  \dotsc, \{\iter{\vv}{t-1,S}: S \subset [t], t \in S\}, \{\iter{\vv}{t,S}: S \subset [t]\}. 
\end{align*}
We use the notation $\iter{\vv}{ \leq t, \bdot}$ to denote the collection of iterates that have been generated at the end of $t$ steps:
\begin{align*}
\iter{\vv}{ \leq t, \bdot} & = \{\iter{\vv}{s,S} : s \leq t, \; S \subset [t]\}.
\end{align*}
The notation $\iter{\vv}{< t, \bdot}$ denotes the collection of iterates that have been generated before the $t$th step, which is the same as the collection $\iter{\vv}{\leq (t-1), \bdot}$.
\end{enumerate}
\item The functions $(\Fnonlin_t)_{t\in \N}$ are continuously differentiable, Lipschitz and satisfy the divergence-free requirement:
\begin{align*}
\E[\partial_{s,S} \dfnew{\Fnonlin}_t(\serv{V}_{<t, \bdot}; \serv{A})] & = 0 \quad \forall \; S \subset [t-1], \; s < t,
\end{align*}
where $(\serv{X}_{\star}, \{\serv{V}_{t,S}: t \in \N, S \subset \N\} ; \serv{A})$ are the state evolution random variables associated with the intermediate OAMP algorithm in \eqref{eq:div-free-algo}. 
\item Furthermore, the iterates generated by the intermediate OAMP algorithm in \eqref{eq:div-free-algo} can approximate the iterates of the target algorithm in \eqref{eq:itr-recall}. Specifically, for each $t \in \N$, there is a postprocessing function $\proc_{t}:\R^{t\cdot 2^t + \auxdim} \mapsto \R$, which is continuously differentiable and Lipschitz, which satisfies:
\begin{align} \label{eq:div-free-approx}
\plim_{\dim \rightarrow \infty} \frac{\|\iter{\vr}{t} - \proc_t(\iter{\vv}{\leq t,\bdot}, \va) \|^2}{\dim} & =  0.
\end{align}
\end{enumerate}
The construction of the intermediate OAMP algorithm with the properties stated above is presented in the proof of the following \lemref{lem:intermediate-OAMP}, whose proof is deferred to \sref{sec:div-removal}.
\begin{lemma}\label{lem:intermediate-OAMP} For any target algorithm of the form \eqref{eq:itr-recall}, there is an intermediate OAMP algorithm of the form \eqref{eq:div-free-algo}, whose iterates can approximate the iterates generated by the target algorithm in the sense of \eqref{eq:div-free-approx}.
\end{lemma}
\paragraph*{Step 2: Polynomial Approximation} Next, by approximating the functions $(\mfunc_{S})_{S \subset \N}$ by polynomials, we construct a lifted OAMP algorithm, whose iterates can approximate the iterates generated by the intermediate OAMP algorithm \eqref{eq:div-free-algo}. This construction is presented in the following \lemref{lem:poly-approx}, whose proof is deferred to \sref{sec:poly-approx-proof-sec}.

\begin{lemma} \label{lem:poly-approx-2} For each $\degree \in \N$, there is a degree-$\degree$ lifted AMP algorithm (\defref{def:LOAMP}):
\begin{align} \label{eq:LAMP-constructed}
\dup{\iter{\vw}{t,i}}{\degree} & = \left( \mY^i - \E_{\serv{\Lambda} \sim \mu}[\serv{\Lambda}^i] \cdot \mI_{\dim} \right) \cdot \dup{\dfnew{\Fnonlin}}{\degree}_{t}(\dup{\iter{\vw}{1,\bdot}}{\degree}, \dotsc, \dup{\iter{\vw}{t-1,\bdot}}{\degree}; \va) \quad t \in \N, \; i \in [\degree],
\end{align}
which can approximate the iterates produced by the intermediate OAMP algorithm \eqref{eq:div-free-algo} in the following sense: for any $\degree \in \N$, $t \in \N$, and any finite subset $S \subset \N$, there is a homogeneous linear postprocessing function $\dup{\hnonlin_{t,S}}{\degree}: \R^{\degree} \mapsto \R$ which has the property that the vector $\dup{\hnonlin}{\degree}_{t,S}(\dup{\iter{\vw}{t,\bdot}}{\degree})$ approximates the intermediate OAMP iterate $\iter{\vv}{t,S}$ in the sense:
\begin{align} \label{eq:approx}
\lim_{\degree \rightarrow \infty} \plimsup_{\dim \rightarrow \infty} \frac{\| \dup{\hnonlin_{t,S}}{\degree}( \dup{\iter{\vw}{t,\bdot}}{\degree}) - \iter{\vv}{t,S} \|^2}{\dim} & = 0
\end{align}
\end{lemma}
With these preliminary results, we can now present the proof of \propref{prop:LOAMP-reduction}. 
\begin{proof}[Proof of \propref{prop:LOAMP-reduction}] Throughout the proof, we shorthand an approximation result like \eqref{eq:div-free-approx} using the notation:
\begin{align} \label{eq:intermediate-OAMP-approx-shorthand}
\iter{\vr}{t} & \explain{$\dim\rightarrow \infty$}{\simeq} \proc_{t}(\iter{\vv}{\leq t, \bdot}; \va) \quad \forall \; t \in  \N.
\end{align}
Likewise, we shorthand an approximation result like \eqref{eq:approx} using:
\begin{align} \label{eq:poly-approx-shorthand}
\iter{\vv}{t,S}  \quad  \explain{$\dim, \degree \rightarrow \infty$}{\simeq} \dup{\iter{\vv}{t,S}}{\degree} \quad \forall \; t \in  \N, \; S \subset \N, \; 1 \leq |S| < \infty.
\end{align}
\lemref{lem:intermediate-OAMP} and \lemref{lem:poly-approx-2} together show that the lifted OAMP algorithm on \eqref{eq:LAMP-constructed} can approximate the iterates of the target algorithm in \eqref{eq:itr-recall}:
\begin{align} \label{eq:optimal-itr-to-LOAMP}
\iter{\vr}{t} \stackrel{(a)}{\explain{$\dim\rightarrow \infty$}{\simeq}} \proc_{t}(\iter{\vv}{1,\bdot}, \dotsc, \iter{\vv}{t,\bdot}; \va) \stackrel{(b)}{\explain{$\dim,\degree \rightarrow \infty$}{\simeq}} \proc_{t}(\{\dup{\hnonlin}{\degree}_{s,R}(\dup{\iter{\vw}{s, \bdot}}{\degree}): s \leq t, R \subset [t]\} ; \va) \quad \forall \; t \in \N.
\end{align}
In the above display, the approximation in step (a) follows from \eqref{eq:div-free-approx} and the approximation in step (b) follows by replacing each iterate $\iter{\vv}{s,R}, \; s \in [t], R \subset [t]$ of the intermediate OAMP algorithm by its approximation in \eqref{eq:approx}, and exploiting the fact that $\proc_t$ is a Lipschitz function. For convenience, we define the composite postprocessing functions $\{\dup{\proc}{\degree}_{t}: t \in \N, \degree \in \N\}$:
\begin{align*}
\dup{\proc}{\degree}_{t}(w_{\leq t, \bdot}; \aux) \explain{def}{=} \proc_{t}(\{\dup{\hnonlin}{\degree}_{s,R}({w}_{s, \bdot}): s \leq t, R \subset [t]\} ; \aux) \quad \forall \; w_{\leq t, \bdot} \in \R^{t \degree}, \; \aux \in \R^{\auxdim}.
\end{align*}
With this definition, \eqref{eq:optimal-itr-to-LOAMP} can be expressed as:
\begin{align*}
\iter{\vr}{t} \explain{$\dim,\degree \rightarrow \infty$}{\simeq} \dup{\proc}{\degree}_{t}(\dup{\iter{\vw}{\leq t, \bdot}}{\degree}; \va) \quad \forall \; t  \in \N.
\end{align*}
Plugging in the above approximation in \eqref{eq:optimal-gfom-estimator}, we have the following approximation for the estimator $\hat{\vr}_t$ returned by the target algorithm:
\begin{align} \label{eq:optimal-gfom-estimator-final-approx}
\hat{\vr}_t \quad \explain{$\dim,\degree \rightarrow \infty$}{\simeq}  \quad \dup{\tilde{\vw}_{t}}{\degree} \; \explain{def}{=} \; \psi_t(\dup{\proc}{\degree}_{1}(\dup{\iter{\vw}{1, \bdot}}{\degree}; \va), \dup{\proc}{\degree}_{2}(\dup{\iter{\vw}{\leq 2, \bdot}}{\degree}; \va), \dotsc, \dup{\proc}{\degree}_{t}(\dup{\iter{\vw}{\leq t, \bdot}}{\degree}; \va) ; \va).
\end{align}
Hence, we have constructed an estimator $\dup{\tilde{\vw}_t}{\degree}$ which can be computed by running a degree-$\degree$ lifted OAMP algorithm for $t$ steps, and approximates the estimator $\hat{\vr}_t$ returned by the target iterative algorithm after $t$ steps. This proves the claim of \propref{prop:LOAMP-reduction}. 
\end{proof}
We now present the proofs of \lemref{lem:intermediate-OAMP} and \lemref{lem:poly-approx-2}.

\subsubsection{Divergence Removal (Proof of \lemref{lem:intermediate-OAMP})} \label{sec:div-removal}
\begin{proof}[Proof of \lemref{lem:intermediate-OAMP}] Consider a target iterative algorithm of the form:
\begin{align} \label{eq:itr-intermediate-OAMP}
\iter{\vr}{t} & = \mfunc_t(\mY) \cdot \fnonlin_t(\iter{\vr}{1}, \dotsc, \iter{\vr}{t-1}; \va) +  \gnonlin_t(\iter{\vr}{1}, \dotsc, \iter{\vr}{t-1}; \va) \quad \forall \; t \in \N.
\end{align}
Our goal is to construct an intermediate OAMP algorithm of the form:
\begin{align}  \label{eq:intermediate-OAMP}
\iter{\vv}{t,S} & = \dfnew{\mfunc}_{S}(\mY)  \cdot \dfnew{\Fnonlin}_t(\iter{\vv}{<t,\bdot}; \va) \quad \forall \;  t \in  \N, \; S \subset \N, \; 1 \leq |S| < \infty.
\end{align}
For each $t \in \N$, we wish to design a postprocessing function $\proc_{t}:\R^{t\cdot 2^t + \auxdim} \mapsto \R$  which allows us to approximate the iterates of the target algorithm with the iterates of the intermediate OAMP algorithm in the sense:
\begin{align}
\plim_{\dim \rightarrow \infty} \frac{\|\iter{\vr}{t} - \proc_t(\iter{\vv}{\leq t,\bdot}, \va) \|^2}{\dim} \rightarrow 0 \quad \forall \; t \in \N.
\end{align}
We will shorthand approximation statements of the form \eqref{eq:div-free-approx} using the notation:
\begin{align} \label{eq:intermediate-OAMP-approx}
\iter{\vr}{t} & \explain{$\dim\rightarrow \infty$}{\simeq} \proc_{t}(\iter{\vv}{\leq t, \bdot}; \va)   \quad \forall \; t \in \N.
\end{align}
The construction proceeds by induction. Suppose that we have specified the first $t$ functions $(\dfnew{\Fnonlin}_s)_{s \leq t}$ and the post-processing functions $(\proc_{s})_{s \leq t}$ for $t$ iterations. We now extend our construction to the $(t+1)$th iteration. Recall the update rule for $\iter{\vr}{t+1}$:
\begin{align*}
&\iter{\vr}{t+1}  = \mfunc_{t+1}(\mY) \cdot \fnonlin_{t+1}(\iter{\vr}{\leq t}; \va) +  \gnonlin_{t+1}(\iter{\vr}{ \leq t}; \va) \\
& \explain{$\dim \rightarrow \infty$}{\simeq} \;   \mfunc_{t+1}(\mY)  \fnonlin_{t+1}(\proc_{1}(\iter{\vv}{\leq 1,\bdot}; \va), \dotsc, \proc_{t}(\iter{\vv}{\leq t,\bdot}; \va)) +  \gnonlin_{t+1}(\proc_{1}(\iter{\vv}{\leq 1,\bdot}; \va), \dotsc, \proc_{t}(\iter{\vv}{\leq t,\bdot}; \va)),
\end{align*}
where the approximation in the last step follows from the induction hypothesis $\iter{\vr}{s} \explain{$\dim \rightarrow \infty$}{\simeq} \quad \proc_{s}(\iter{\vv}{\leq s, \bdot}; \va)$ for any $s \leq t$, along with the fact that $\|\mfunc_{t}(\mY)\|_{\op}$ is uniformly bounded as $\dim \rightarrow \infty$ and the functions $\fnonlin_{t+1},\gnonlin_{t+1}$ are Lipschitz. We introduce the composite functions:
\begin{align*}
{\Fnonlin}_{t+1}(v_{\leq t, \bdot}; \aux) &\explain{def}{=} \fnonlin_{t+1}(\proc_{1}(v_{\leq 1,\bdot}; \aux), \proc_{2}(v_{\leq 2,\bdot}; \aux), \dotsc, \proc_{t}(v_{\leq t,\bdot}; \aux); \aux) \quad v_{\leq t,\bdot} \in \R^{t2^t}, \; \aux \in \R^{\auxdim}, \\
{\Gnonlin}_{t+1}(v_{\leq t, \bdot}; \aux) &\explain{def}{=} \gnonlin_{t+1}(\proc_{1}(v_{\leq 1,\bdot}; \aux), \proc_{2}(v_{\leq 2,\bdot}; \aux), \dotsc, \proc_{t}(v_{\leq t,\bdot}; \aux); \aux) \quad v_{\leq t,\bdot} \in \R^{t2^t}, \; \aux \in \R^{\auxdim}.
\end{align*}
With these definitions we have that,
\begin{align} \label{eq:given-algo-approx}
\iter{\vr}{t+1}  &  \explain{$\dim \rightarrow \infty$}{\simeq} \quad \mfunc_{t+1}(\mY) \cdot {\Fnonlin}_{t+1}(\iter{\vv}{\leq t,\bdot}; \va) +  {\Gnonlin}_{t+1}(\iter{\vv}{\leq t,\bdot}; \va).
\end{align}
The above expression suggests that we should use the function ${\Fnonlin}_{t+1}$ in the $t+1$ step of the intermediate OAMP algorithm. However, since $\Fnonlin_{t+1}$ is not divergence-free (cf. Definition~\ref{Def:OAMP_main}), we define:
 \begin{align*}
\dfnew{\Fnonlin}_{t+1}(v_{\leq t, \bdot}; \aux) \explain{def}{=} \Fnonlin_{t}(v_{ \leq t,\bdot}; \aux) - \sum_{s = 1}^{t} \sum_{R \subset [t]} b_{s,R} \cdot v_{s,R} \quad \forall  \;  v_{\leq t,\bdot}  \in \R^{t\cdot 2^{t}},  \; \aux \in \R^{\auxdim}.
\end{align*}
In the above display, we introduced the scalars:
\begin{align*}
{b}_{s,R} \explain{def}{=} \E\left[\partial_{s,R} \Fnonlin_{t+1}(\serv{V}_{\leq t,\bdot}; \serv{A}) \right] \quad \forall \; s \leq t, \; R \subset [t],
\end{align*}
where $(\serv{X}_{\star}, \{\serv{V}_{s,R}: s \leq t, R \subset [t]\} \; ; \serv{A})$ denote the state evolution random variables corresponding to the iterates of the intermediate OAMP algorithm that have already been constructed (by the induction hypothesis). Notice that by construction, $\dfnew{\Fnonlin}_{t+1}$ satisfies the divergence-free requirement:
\begin{align*}
\E[\partial_{s,S} \dfnew{\Fnonlin}_{t+1}(\serv{V}_{\leq t; \bdot}; \serv{A})] &  = 0 \quad \forall \; s \leq t, \; S \subset [t].
\end{align*}
We can now extend our construction of the intermediate OAMP algorithm to the $(t+1)$th step:
\begin{align} \label{eq:constructed-LAMP-new-iterate}
\iter{\vv}{t+1,S} & = \dfnew{\mfunc}_{S}(\mY) \cdot \dfnew{\Fnonlin}_{t+1}(\iter{\vv}{\leq t,\bdot}; \va)  \quad \forall \; S \subset \N \quad 1 \leq |S| < \infty. 
\end{align}
Next, we define the postprocessing function $H_{t+1}$. For convenience, we introduce the vectors $(\dfnew{\mF}_{\tau})_{\tau\in [t+1]}$:
\begin{align} \label{eq:boldF-def}
\dfnew{\mF}_{\tau} \explain{def}{=} \dfnew{\Fnonlin}_{\tau}(\iter{\vv}{< \tau,\bdot}; \va) \quad \forall \; \tau \in [t+1].
\end{align}
With these definitions, we can express the approximation in \eqref{eq:given-algo-approx} as:
\begin{align}
\iter{\vr}{t+1}  &  \explain{$\dim \rightarrow \infty$}{\simeq} \mfunc_{t+1}(\mY) \cdot {\Fnonlin}_{t+1}(\iter{\vv}{\leq t,\bdot}; \va) +  {\Gnonlin}_{t+1}(\iter{\vv}{\leq t,\bdot}; \va) \nonumber \\
& {\explain{$\dim \rightarrow \infty$}{\simeq}}  \mfunc_{t+1}(\mY) \cdot \dfnew{\mF}_{t+1} + \sum_{s=1}^t \sum_{S \subset [t]} {b}_{s,S} \cdot \mfunc_{t+1}(\mY) \cdot \iter{\vv}{s,S} +  {\Gnonlin}_{t+1}(\iter{\vv}{\leq t,\bdot}; \va). \label{eq:t+1-div-free-approx} \end{align}
For any $S \subset [t+1]$ we define:
\begin{align*}m_{S} \explain{def}{=} \E_{\serv{\Lambda} \sim \mu}[\mfunc_S(\Lambda)].
\end{align*}
With this definition, we can express the approximation in \eqref{eq:t+1-div-free-approx} as:
\begin{align}
&\iter{\vr}{t+1}   \explain{$\dim \rightarrow \infty$}{\simeq} \mfunc_{t+1}(\mY) \cdot \dfnew{\mF}_{t+1} + \sum_{s=1}^t \sum_{S \subset [t]} b_{s,S} \cdot \mfunc_{t+1}(\mY) \cdot \iter{\vv}{s,S} +  {\Gnonlin}_{t+1}(\iter{\vv}{\leq t,\bdot}; \va) \nonumber\\
& = \mfunc_{t+1}(\mY) \cdot \dfnew{\mF}_{t+1} + \sum_{s=1}^t \sum_{S \subset [t]} b_{s,S} \cdot \mfunc_{t+1}(\mY) \cdot (\mfunc_{S}(\mY) - {m}_{S} \mI_{\dim}) \cdot \dfnew{\mF}_{s} +  {\Gnonlin}_{t+1}(\iter{\vv}{\leq t,\bdot}; \va) \nonumber\\
& = \mfunc_{t+1}(\mY) \cdot \dfnew{\mF}_{t+1} + \sum_{s=1}^t \sum_{S \subset [t]} b_{s,S} \cdot (\mfunc_{\{t+1\} \cup S}(\mY) - m_S \mfunc_{t+1}(\mY)) \cdot \dfnew{\mF}_{s} +  {\Gnonlin}_{t+1}(\iter{\vv}{\leq t,\bdot}; \va)  \nonumber\\
& {\explain{\eqref{eq:intermediate-OAMP},\eqref{eq:constructed-LAMP-new-iterate}}{=}} \iter{\vv}{t+1,\{t+1\}}  + \left( \sum_{s=1}^t \sum_{S \subset [t]} b_{s,S}  \cdot (\iter{\vv}{s,S \cup \{t+1\}} - m_{S} \cdot \iter{\vv}{s,\{t+1\}}) \right) + {\Gnonlin}_{t+1}(\iter{\vv}{\leq t,\bdot}; \va)  \nonumber\\& \hspace{2cm} + m_{\{t+1\}} \cdot \dfnew{\mF}_{t+1} + \left(\sum_{s=1}^t \sum_{S \subset [t]} b_{s,S}  \cdot (m_{S \cup\{t+1\}} - m_{S} \cdot m_{\{t+1\}}) \cdot \dfnew{\mF}_s  \right) \label{eq:t+1-div-free-approx-final}. 
\end{align}
Recalling \eqref{eq:boldF-def},
\begin{align*}
&\iter{\vr}{t+1}  \explain{$\dim \rightarrow \infty$}{\simeq}  \iter{\vv}{t+1,\{t+1\}}  + \left( \sum_{s=1}^t \sum_{S \subset [t]} b_{s,S}  \cdot (\iter{\vv}{s,S \cup \{t+1\}} - m_{S} \cdot \iter{\vv}{s,\{t+1\}}) \right) + {\Gnonlin}_{t+1}(\iter{\vv}{\leq t,\bdot}; \va)  \nonumber\\& \hspace{0.1cm}+ m_{\{t+1\}} \cdot \dfnew{\Fnonlin}_{t+1}(\iter{\vv}{\leq t,\bdot}; \va) + \left(\sum_{s=1}^t \sum_{S \subset [t]} b_{s,S}  \cdot (m_{S \cup\{t+1\}} - m_{S} \cdot m_{\{t+1\}}) \cdot \dfnew{\Fnonlin}_s(\iter{\vv}{\leq s, \bdot}; \va)  \right).
\end{align*}
In light of the above expression, we define the postprocessing function for the $t+1$th iteration as: 
\begin{align*}
\proc_{t+1}(v_{\leq t+1,\bdot}; \aux) &\explain{def}{=} v_{t+1,\{t+1\}} +  \left( \sum_{s=1}^t \sum_{S \subset [t]} b_{s,S}  \cdot (v_{s,S \cup \{t+1\}} - m_{S} \cdot v_{s,\{t+1\}}) \right) \\ &  \hspace{1cm}+  {\Gnonlin}_{t+1}(v_{\leq t,\bdot}; \va) + m_{\{t+1\}}  \dfnew{\Fnonlin}_{t+1}(v_{\leq t,\bdot}; \aux) \\ & \hspace{2cm}+ \left(\sum_{s=1}^t \sum_{S \subset [t]} b_{s,S}   (m_{S \cup\{t+1\}} - m_{S}  m_{\{t+1\}})  \dfnew{\Fnonlin}_s(v_{\leq s, \bdot}; \aux)  \right).
\end{align*}
Hence, we can approximate the $t+1$th iterate of the target algorithm using the iterates of the intermediate OAMP algorithm:
\begin{align*}
\iter{\vr}{t+1} & \explain{$\dim\rightarrow \infty$}{\simeq} \proc_{t+1}(\iter{\vv}{\leq t+1, \bdot}; \va).
\end{align*}
This completes the construction of the intermediate OAMP algorithm with all the desired properties by induction. 
\end{proof}

\subsubsection{Polynomial Approximation (Proof of \lemref{lem:poly-approx-2})} \label{sec:poly-approx-proof-sec}
\begin{proof}[Proof of \lemref{lem:poly-approx-2}]
Recall the intermediate OAMP algorithm takes the form:
\begin{align}
\iter{\vv}{t,S} & = \dfnew{\mfunc}_{S}(\mY) \cdot \dfnew{\Fnonlin}_t(\iter{\vv}{<t,\bdot}; \va) \quad \forall \; t \in  \N, \; S \subset \N, \; 1 \leq |S| < \infty. 
\end{align}
To implement this algorithm using a degree-$\degree$ lifted OAMP algorithm, we begin by repeating the arguments used in the proof of \lemref{lem:poly-approx} to construct an OAMP algorithm of the form:
\begin{align} 
&\dup{\iter{\vv}{t,S}}{\degree} =  \nonumber \\ & \left( \dup{\mfunc_{S}}{\degree}(\mY) - \E_{\serv{\Lambda} \sim \mu}[\dup{\mfunc_{S}}{\degree}(\serv{\Lambda})]  \mI_{\dim} \right)  \left( \dfnew{\Fnonlin}_t(\dup{\iter{\vv}{<t,\bdot}}{\degree}; \va) - \sum_{s=1}^{t-1} \sum_{R \subset [t-1]} \E[\partial_{s,R} \dfnew{\Fnonlin}_t(\dup{\serv{V}_{<t,\bdot}}{\degree}; \serv{A})]  \dup{\iter{\vv}{s,R}}{\degree} \right) \nonumber \\ & \hspace{9cm} \forall \;  t \in  \N, \; S \subset \N, \; 1 \leq |S| < \infty,  \label{eq:low-degree-approx-div-free-algo-2}
\end{align}
which approximates the iterates generated by the intermediate OAMP algorithm in the sense that,
\begin{align} \label{eq:low-degree-appx-claim1}
\limsup_{\degree \rightarrow \infty} \plimsup_{\dim \rightarrow \infty} \frac{\|\dup{\iter{\vv}{t,S}}{\degree} - \iter{\vv}{t,S}\|^2}{\dim} & = 0\quad \forall \; t \in  \N, \; S \subset \N, \; 1 \leq |S| < \infty.
\end{align}
In the above display:
\begin{itemize}
\item The matrix denoiser $\dup{\mfunc_{S}}{\degree}: \R \mapsto \R$ is a degree-$\degree$ polynomial which approximates the matrix denoiser $\mfunc_S: \R \mapsto \R$ and is constructed using the Weierstrass approximation theorem.
\item $(\serv{X}_{\star}, \{\dup{\serv{V}}{\degree}_{t,S}: t \in \N, S \subset \N \}; \serv{A})$ denote the state evolution random variables associated with \eqref{eq:low-degree-approx-div-free-algo-2}.
\end{itemize}
We claim that there is a degree-$\degree$ lifted AMP algorithm (\defref{def:LOAMP}):
\begin{align} \label{eq:desired-LOAMP}
\dup{\iter{\vw}{t,i}}{\degree} & = \left( \mY^i - \E_{\serv{\Lambda} \sim \mu}[\serv{\Lambda}^i] \cdot \mI_{\dim} \right) \cdot \dup{\dfnew{\Fnonlin}}{\degree}_{t}(\dup{\iter{\vw}{<t,\bdot}}{\degree}; \va)  \quad \forall \;  t \in \N, \; i \in [\degree],
\end{align}
 along with  homogeneous linear postprocessing maps \begin{align*}
\dup{\hnonlin}{\degree}_{t,S}: \R^{\degree} \mapsto \R \quad t \in \N, \; S \subset \N, \; 1 \leq |S| < \infty,
\end{align*}
which can exactly recover the iterates generated by the degree-$\degree$ approximation to the intermediate OAMP algorithm \eqref{eq:low-degree-approx-div-free-algo-2}:
\begin{align} \label{eq:low-degree-appx-claim2}
\dup{\iter{\vv}{t,S}}{\degree} & = \dup{\hnonlin}{\degree}_{t,S}(\dup{\iter{\vw}{t,\bdot}}{\degree}) \quad \forall\;  t \in \N, \; S \subset \N, \; 1 \leq |S| < \infty.
\end{align}
Observe that \eqref{eq:low-degree-appx-claim1} and \eqref{eq:low-degree-appx-claim2} immediately imply \lemref{lem:poly-approx-2}. Indeed, for any $t \in \N$ and any finite subset $S \subset \N$:
\begin{align*}
\limsup_{\degree \rightarrow \infty} \plimsup_{\dim \rightarrow \infty} \frac{\|\dup{\hnonlin}{\degree}_{t,S}(\dup{\iter{\vw}{t,\bdot}}{\degree}) - \iter{\vv}{t,S}\|^2}{\dim} & \explain{\eqref{eq:low-degree-appx-claim2}}{=}  \limsup_{\degree \rightarrow \infty} \plimsup_{\dim \rightarrow \infty} \frac{\|\dup{\iter{\vv}{t,S}}{\degree} - \iter{\vv}{t,S}\|^2}{\dim}  \explain{\eqref{eq:low-degree-appx-claim1}}{=} 0,
\end{align*}
as desired. Hence, the remainder of the proof consists of constructing the lifted OAMP algorithm and proving \eqref{eq:low-degree-appx-claim2}. 

\paragraph*{Construction of the lifted OAMP Algorithm} We will construct the desired lifted OAMP algorithm \eqref{eq:desired-LOAMP} by induction. As our induction hypothesis, we will assume that we have specified the functions $\dup{\dfnew{\Fnonlin}}{\degree}_{1}, \dotsc, \dup{\dfnew{\Fnonlin}}{\degree}_{t-1}$, as well as the homogenous linear postprocessing maps:
\begin{align*}
\dup{\hnonlin}{\degree}_{s,S}: \R^{\degree} \mapsto \R \quad s < t, \; S \subset [t-1], \; S \geq 1,
\end{align*}
so that the iterates resulting lifted OAMP algorithm $\{\dup{\iter{\vw}{s,i}}{\degree}: s < t, \; i \in [\degree]\}$ can reconstruct the following iterates of the  degree-$\degree$ approximation to the intermediate OAMP algorithm:
\begin{align} \label{eq:induction-hypo-lowdegree-to-liftedOAMP}
    \dup{\iter{\vv}{s,S}}{\degree} & = \dup{\hnonlin}{\degree}_{s,S}(\dup{\iter{\vw}{s, \bdot}}{\degree}) \quad \forall \; s < t, \; S \subset [t-1], \; |S| > 1.
\end{align}
We now extend our construction to step $t$. To do so, we need to construct the function $\dup{\dfnew{\Fnonlin}}{\degree}_{t}$ used in the $t$th iteration of the lifted OAMP algorithm, as well as the linear postprocessing maps:
\begin{align*}
    \dup{\hnonlin}{\degree}_{t,S}: \R^{\degree} \mapsto \R \quad S \subset [t], \; S \geq 1, \\
     \dup{\hnonlin}{\degree}_{s,S}: \R^{\degree} \mapsto \R \quad S \subset [t], S \ni t,
\end{align*}
which ensure that:
\begin{align}
    \dup{\iter{\vv}{t,S}}{\degree} & = \dup{\hnonlin}{\degree}_{t,S}(\dup{\iter{\vw}{ t, \bdot}}{\degree}) \quad \forall\; S \subset [t], \; |S| > 1, \label{eq:implement-lowdegree-by-OAMP} \\
    \dup{\iter{\vv}{s,S}}{\degree} & = \dup{\hnonlin}{\degree}_{s,S}(\dup{\iter{\vw}{ s, \bdot}}{\degree}) \quad \forall \; s < t, \; S \subset [t], \; S \ni t. \label{eq:implement-OAMP-skip}
\end{align}
We will present the proof for \eqref{eq:implement-lowdegree-by-OAMP}, the proof for \eqref{eq:implement-OAMP-skip} is analogous. We begin by recalling the update rule for $\dup{\iter{\vv}{t,S}}{\degree}$:
\begin{align}
&\dup{\iter{\vv}{t,S}}{\degree}= \nonumber  \\ & \left( \dup{\mfunc_{S}}{\degree}(\mY) - \E_{\serv{\Lambda} \sim \mu}[\dup{\mfunc_{S}}{\degree}(\serv{\Lambda})]  \mI_{\dim} \right)  \underbrace{\left( \dfnew{\Fnonlin}_t(\dup{\iter{\vv}{<t,\bdot}}{\degree}; \va) - \sum_{s=1}^{t-1} \sum_{R \subset [t-1]} \E[\partial_{s,R} \dfnew{\Fnonlin}_t(\dup{\serv{V}_{<t,\bdot}}{\degree}; \serv{A})]  \dup{\iter{\vv}{s,R}}{\degree} \right)}_{(\diamondsuit)}  \label{eq:sub-here-temp}. 
\end{align}
Using the induction hypothesis in \eqref{eq:induction-hypo-lowdegree-to-liftedOAMP}, we can express the vector $(\diamondsuit)$ in the above display as:
\begin{align*}
    (\diamondsuit) & =  \dfnew{\Fnonlin}_t(\{\dup{\hnonlin}{\degree}_{s,R}(\dup{\iter{\vw}{ s,\bdot}}{\degree})\}; \va) - \sum_{s=1}^{t-1} \sum_{R \subset [t-1]} \E[\partial_{s,R} \dfnew{\Fnonlin}_t(\dup{\serv{V}_{<t,\bdot}}{\degree}; \serv{A})]   \dup{\hnonlin}{\degree}_{s,R}(\dup{\iter{\vw}{s,\bdot}}{\degree}).
\end{align*}In light of the above expression, we define $\Fnonlin_{t,\degree}$ as the composite function (for any $w_{<t, \bdot} \in \R^{(t-1)\degree}, \; \aux \in \R^{\auxdim}$)
\begin{align} \label{eq:dup-F}
\dup{\Fnonlin}{\degree}_{t}(w_{< t,\bdot}; \aux )  \explain{def}{=} \dfnew{\Fnonlin}_{t}(\{\dup{\hnonlin}{\degree}_{s,R}(w_{s, \bdot}): s < t, R \subset [t-1], |R| > 1\}; \aux),
\end{align}
and the divergence-free iterate denoiser:
\begin{align} \label{eq:dup-F-df}
\dup{\dfnew{\Fnonlin}}{\degree}_{t}(w_{< t,\bdot}; \aux ) & = \dup{\Fnonlin}{\degree}_{t}(w_{< t,\bdot}; \aux ) - \sum_{s=1}^{t-1}\sum_{i=1}^{\degree} \E[\partial_{s,i} \dup{\Fnonlin}{\degree}_{t}(\dup{\serv{W}}{\degree}_{<t,\bdot}; \serv{A})] \cdot w_{s,i},
\end{align}
where $(\serv{X}_{\star}, \{\serv{W}_{s,i}: s < t, i \in [\degree]\} ; \serv{A})$ are the state evolution random variables associated with the lifted OAMP algorithm \eqref{eq:desired-LOAMP} that have been constructed so far (as a part of the induction hypothesis). We now extend our construction of the desired lifted OAMP algorithm to step $t$. At step $t$, the lifted OAMP algorithm generates the iterates:
\begin{align} \label{eq:newly-constructed-liftedOAMP-iterates}
\dup{\iter{\vw}{t,i}}{\degree} & = \left( \mY^i - \E_{\serv{\Lambda} \sim \mu}[\serv{\Lambda}^i] \cdot \mI_{\dim} \right) \cdot \dup{\dfnew{\Fnonlin}}{\degree}_{t}(\dup{\iter{\vw}{<t,\bdot}}{\degree}; \va)  \quad i \in [\degree].
\end{align}
Introducing the shorthand notation  $\dup{\overline{\mfunc}_{S}}{\degree}(\mY) \explain{def}{=} \dup{\mfunc_{S}}{\degree}(\mY) - \E_{\serv{\Lambda} \sim \mu}[\dup{\mfunc_{S}}{\degree}(\serv{\Lambda})] \cdot \mI_{\dim}$ for convenience, we can express \eqref{eq:sub-here-temp} as:
\begin{align}
\dup{\iter{\vv}{t,S}}{\degree} & \explain{\eqref{eq:dup-F}}{=} \;  \dup{\overline{\mfunc}_{S}}{\degree}(\mY) \left( \dup{\Fnonlin}{\degree}_{t}(\dup{\iter{\vw}{<t,\bdot}}{\degree}; \va) - \sum_{s=1}^{t-1} \sum_{R \subset [t-1]} \E[\partial_{s,R} \dfnew{\Fnonlin}_t(\dup{\serv{V}_{<t,\bdot}}{\degree}; \serv{A})]   \dup{\hnonlin}{\degree}_{s,R}(\dup{\iter{\vw}{s,\bdot}}{\degree}) \right) \nonumber \\
& \explain{(b)}{=} \dup{\overline{\mfunc}_{S}}{\degree}(\mY)  \left( \dup{\Fnonlin}{\degree}_{t}(\dup{\iter{\vw}{<t,\bdot}}{\degree}; \va)  - \sum_{s=1}^{t-1} \sum_{i=1}^\degree \E[\partial_{s,i} \dup{\Fnonlin}{\degree}_{t}(\dup{\serv{W}}{\degree}_{<t,\bdot}; \serv{A})]  \dup{\iter{\vw}{s,i}}{\degree} \right) \nonumber \\
& \explain{\eqref{eq:dup-F-df}}{=} \left( \dup{\mfunc_{S}}{\degree}(\mY) - \E_{\serv{\Lambda} \sim \mu}[\dup{\mfunc_{S}}{\degree}(\serv{\Lambda})]  \mI_{\dim} \right)  \dup{\dfnew{\Fnonlin}}{\degree}_{t}(\dup{\iter{\vw}{<t,\bdot}}{\degree}; \va) \label{eq:sub-here-temp2}
\end{align}
where the step marked (b) follows by relating the derivatives of the function $\dup{\Fnonlin}{\degree}_{t}$ and the function $\dfnew{\Fnonlin}_{t}$ using the chain rule and by exploiting the linearity of postprocessing maps $\{h_{s,R}\}$. Recall that $\dup{\mfunc}{\degree}_{S}: \R \mapsto \R$ is a degree-$\degree$ polynomial, and so, it can be expressed as a linear combination of  the monomials $\{\lambda \mapsto \lambda^i : i \in [\degree]\}$:
\begin{align*}
\dup{\mfunc}{\degree}_{S}(\lambda) - \E_{\serv{\Lambda} \sim \mu}[\dup{\mfunc_{S}}{\degree}(\serv{\Lambda})] & = \sum_{i=1}^\degree \dup{\chi}{\degree}_{S}[i] \cdot (\lambda^i-\E_{\serv{\Lambda} \sim \mu}[\serv{\Lambda}^i]) \quad i \in [\degree].
\end{align*}
Using this representation, the formula for $\dup{\iter{\vv}{t,S}}{\degree}$ obtained in \eqref{eq:sub-here-temp2} can be expressed as a linear combination of the newly constructed lifted OAMP iterates in \eqref{eq:newly-constructed-liftedOAMP-iterates}:
\begin{align*}
\dup{\iter{\vv}{t,S}}{\degree} & = \sum_{i=1}^{\degree} \dup{\chi}{\degree}_{S}[i] \cdot \dup{\iter{\vw}{t,i}}{\degree}.
\end{align*}
Hence, we define the postprocessing function $\dup{\hnonlin}{\degree}_{t,S}$ as:
\begin{align*}
\dup{\hnonlin}{\degree}_{t,S}(w_{t,\bdot}) & = \sum_{i=1}^\degree \dup{\chi}{\degree}_{S}[i] \cdot w_{t,i} \quad \forall \;  w_{ t, \bdot} \in \R^{\degree}.
\end{align*}
which is linear and ensures that $\dup{\iter{\vv}{t,S}}{\degree} = \dup{\hnonlin}{\degree}_{t,S}(\dup{\iter{\vw}{t, \bdot}}{\degree})$, as desired. This completes the construction of the desired lifted OAMP algorithm by induction and concludes the proof of this lemma. 
\end{proof}

\subsection{Proof \propref{prop:greedy-optimality}} \label{appendix:optimal-LOAMP}
Consider any degree-$\degree$ lifted OAMP algorithm (\defref{def:LOAMP}) of the form:
\begin{align}
\iter{\vw}{t,i} & = \left( \mY^i - \E_{\serv{\Lambda \sim \nu}}[\serv{\Lambda}^i] \cdot \mI_{\dim} \right) \cdot \fnonlin_{t}\big(\iter{\vw}{1,\bdot}, \dotsc, \iter{\vw}{t-1,\bdot}; \va \big)  \quad i \in [\degree], \; t \in \N,
\end{align}
which returns the following estimator after $t$ iterations:
\begin{align*}
{\hat{\vw}_t} = \hnonlin_t(\iter{\vw}{1,\bdot}, \dotsc, \iter{\vw}{t,\bdot}; \va).
\end{align*}
Let $(\serv{X}_{\star}, \serv{W}_{1,\bdot}, \dotsc, \serv{W}_{t,\bdot};\serv{A})$ denote the state evolution random variables associated with the first $t$ iterations of the lifted OAMP algorithm, which form a Gaussian channel in the sense of Definition~\ref{def:gauss-channel}. We observe that, by \corref{cor:SE-LOAMP}:\begin{align} \label{eq:mmse-functional-motivation}
\plim_{\dim \rightarrow \infty} \frac{\|\vx_{\star} - {\hat{\vw}_t}\|^2}{\dim} & = \E[\{\serv{X}_{\star} - \hnonlin(\serv{W}_{1,\bdot}, \dotsc, \serv{W}_{t,\bdot}; \serv{A})\}^2] \nonumber \\
& \explain{(a)}{\geq} \min_{h \in L^2(\serv{W}_{\leq t,\bdot}; \serv{A})} \E[\{\serv{X}_{\star} - \hnonlin(\serv{W}_{1,\bdot}, \dotsc, \serv{W}_{t,\bdot}; \serv{A})\}^2] \nonumber \\
& \explain{def}{=} \bmmse(\serv{X}_{\star} | \serv{W}_{1,\bdot}, \dotsc, \serv{W}_{t,\bdot}; \serv{A}).
\end{align}
 In the above display, the inequality in step (a) follows by observing that since the function $h$ is a Lipschitz function (as required by \defref{def:LOAMP}),  and hence, $h \in L^2(\serv{W}_{\leq t,\bdot}; \serv{A})$. Furthermore, thanks to Assumption~\ref{assump:regularity} holds, the lower bound above is tight since we can choose $h$ as the MMSE estimator for the Gaussian channel $(\serv{X}_{\star}, \serv{W}_{1,\bdot}, \dotsc, \serv{W}_{t,\bdot};\serv{A})$.

In light of \eqref{eq:mmse-functional-motivation}, for a $t$ iteration lifted OAMP algorithm which uses functions $\fnonlin_{1:t}$ in its iterations, we define:
\begin{align*}
\scrM_t(f_1, \dotsc, f_t) \explain{def}{=} \bmmse(\serv{X}_{\star} | \serv{W}_{1,\bdot}, \dotsc, \serv{W}_{t,\bdot}; \serv{A}),
\end{align*}
where $(\serv{X}_{\star}, \serv{W}_{1,\bdot}, \dotsc, \serv{W}_{t,\bdot}; \serv{A})$ are the state evolution random variables associated with the lifted OAMP algorithm. Notice that the RHS of the above display depends on the functions $\fnonlin_{1:t}$ implicitly, since the distribution of the state evolution random variables is determined by $\fnonlin_{1:t}$ (recall \defref{def:LOAMP}).  
\paragraph*{A convenient formula for $\scrM_t(f_1, \dotsc, f_t)$} To determine the optimal lifted OAMP algorithm, we need to optimize the functional $\scrM_t$ with respect to $\fnonlin_{1:t}$.  To do so, we first derive a convenient formula for $\scrM_t(f_1, \dotsc, f_t) = \bmmse(\serv{X}_{\star} | \serv{W}_{1,\bdot}, \dotsc, \serv{W}_{t,\bdot}; \serv{A})$. \lemref{lem:scalarization-mmse} in \appref{appendix:gauss-channel} shows that the MMSE (and other relevant properties) of a multivariate Gaussian channel, such as $(\serv{X}_{\star}, \serv{W}_{1,\bdot}, \dotsc, \serv{W}_{t,\bdot}; \serv{A})$ can be expressed in terms of the MMSE of an \emph{effective scalar Gaussian channel}:
\begin{align} \label{eq:scalarization-recall}
(\serv{X}_{\star}, \serv{S} \explain{def}{=} \ip{\optlin}{\serv{W}_{\leq t,\bdot}} ; \serv{A}),
\end{align}
whose observation random variable $\serv{S}$ is a linear combination of the observation random variables of the original multivariate Gaussian channel. The weights of linear combination are given by $\optlin(\serv{X}_{\star} | \serv{W}_{1, \bdot}, \dotsc, \serv{W}_{t,\bdot}; \serv{A})$ (which we will sometimes shorthand as $\optlin$, when the multivariate Gaussian channel in question is clear from the context). The effective scalar Gaussian channel operates at SNR $\effsnr(\serv{X}_{\star} | \serv{W}_{1, \bdot}, \dotsc, \serv{W}_{t,\bdot}; \serv{A})$, and the MMSE of the original multivariate channel and the effective scalar channel are connected via the formula:
\begin{align*}
\bmmse(\serv{X}_{\star} | \serv{W}_{1,\bdot}, \dotsc, \serv{W}_{t,\bdot}; \serv{A}) = \bmmse(\serv{X}_{\star} | \serv{S}; \serv{A}) = \mmse \left(\effsnr(\serv{X}_{\star} | \serv{W}_{1,\bdot}, \dotsc, \serv{W}_{t,\bdot}; \serv{A}) \right),
\end{align*}
where for $\omega \in [0,1]$, $\mmse(\omega)$ is the MMSE of a scalar Gaussian channel at SNR $\omega$ (recall Definition~\ref{def:gauss-channel}). The following lemma provides a formula for the effective SNR $\effsnr(\serv{X}_{\star} | \serv{W}_{1,\bdot}, \dotsc, \serv{W}_{t,\bdot}; \serv{A})$. To state this formula, recall from \defref{def:LOAMP} that the joint distribution of the state evolution random variables $(\serv{X}_{\star}, \serv{W}_{\leq t,\bdot}; \serv{A})$ is given by:
\begin{subequations} \label{eq:liftedOAMP-SE}
\begin{align} 
(\serv{X}_{\star}; \serv{A}) &\sim \pi, \\ (\serv{W}_{1,\bdot}, \dotsc, \serv{W}_{t,\bdot}) | (\serv{X}_\star;\serv{A}) &\sim \gauss{\serv{X_\star} \cdot  \alpha \otimes q}{ \alpha \alpha^\top   \otimes (Q - q q^\top) + (\Sigma - \alpha \alpha^\top)  \otimes \Gamma}. \nonumber
\end{align}
In the above display, the entries of $q \in \R^{\degree}$, $Q \in \R^{\degree \times \degree}$, and $\Gamma \in \R^{\degree \times \degree}$ are given by:
\begin{align} 
q_i  &= \E_{\serv{\Lambda}_{\nu} \sim \nu}[{\serv{\Lambda}}_{\nu}^i] - \E_{\serv{\Lambda} \sim \mu}[{\serv{\Lambda}}^i],  \\  Q_{ij} 
&= \E_{\serv{\Lambda}_{\nu} \sim \nu}\left[ \left({\serv{\Lambda}}_{\nu}^i - \E_{\serv{\Lambda} \sim \mu}[{\serv{\Lambda}}^i] \right)  \left({\serv{\Lambda}}_{\nu}^j - \E_{\serv{\Lambda} \sim \mu}[{\serv{\Lambda}}^j] \right) \right], \; \Gamma_{ij} = \Cov_{\serv{\Lambda} \sim \mu}[\serv{\Lambda}^i ,\serv{\Lambda}^j].
\end{align}
the entries of $\alpha \in \R^t$ and $\Sigma \in \R^{t \times t}$ are given by:
\begin{align}
\alpha_{s} \explain{def}{=} \E[\serv{X}_{\star} \cdot {\fnonlin}_{s}(\serv{W}_{<s,\bdot}; \serv{A})], \quad \Sigma_{s \tau} \explain{def}{=} \E[{\fnonlin}_{s}(\serv{W}_{<s,\bdot}; \serv{A}) \cdot {\fnonlin}_{\tau}(\serv{W}_{<\tau,\bdot}; \serv{A})] \quad \forall \;  s,\tau \; \in \; [t],
\end{align}
 and $\otimes$ denotes the Kronecker (or tensor) product for matrices. 
\end{subequations}

\begin{lemma} \label{lemma:effsnr-LOAMP} Consider a lifted OAMP algorithm which uses functions $\fnonlin_1, \fnonlin_2, \dotsc, \fnonlin_t$ with $\fnonlin_{i}: \R^{(i-1)\cdot \degree + \auxdim} \mapsto \R$ in the first $t$ iterations. Let $(\serv{X}_{\star}, \serv{W}_{1,\bdot}, \dotsc, \serv{W}_{t,\bdot}; \serv{A})$ be the state evolution random variables associated with the algorithm. Then,
\begin{align*}
\effsnr(\serv{X}_{\star} | \serv{W}_{1, \bdot}, \dotsc, \serv{W}_{t,\bdot}; \serv{A}) & = \begin{cases} q \tran \left[ Q  + \frac{\kappa_t}{1-\kappa_t} \cdot \Gamma \right]^{\dagger} q  & :\kappa_t \in [0,1) \\ 0 & :\kappa_t = 1. \end{cases}
\end{align*}
Moreover, if $\kappa_t < 1$, then,
\begin{align*}
\optlin(\serv{X}_{\star} | \serv{W}_{1, \bdot}, \dotsc, \serv{W}_{t,\bdot}; \serv{A}) & = \frac{1}{\sqrt{\effsnr(\serv{X}_{\star} | \serv{W}_{\leq t, \bdot}; \serv{A}) }} \cdot (\Sigma^\dagger \alpha) \otimes [(1-\kappa_t) Q + \kappa_t \Gamma ]^\dagger q, 
\end{align*}
where the scalar $\kappa_t \in [0,1]$ is defined as\footnote{For a collection of functions $f_1, \dotsc, f_t$, $\mathrm{Span}(f_1, \dotsc, f_t)$ denotes the set of all functions that can be written as a linear combination of $f_{1},\dotsc, f_t$.}:
\begin{align*}
    \kappa_t \explain{def}{=} \min\bigg\{ \E\left[\{\serv{X}_{\star} - g(\serv{W}_{< t,\bdot}; \serv{A})\}^2\right] : g \in \mathrm{Span}({\fnonlin}_1, \dotsc, {\fnonlin}_t) \bigg\} .
\end{align*}
\end{lemma}
\begin{proof} The proof of this lemma is presented in \sref{sec:proof-mmse-functional}.
\end{proof}
Combining the effective SNR formula from the lemma above with \eqref{eq:scalarization-recall}, we obtain the following expression for $\scrM(\fnonlin_1, \dotsc, \fnonlin_t)$:
\begin{subequations}  
\label{eq:scrm}
\begin{align}
\scrM(\fnonlin_1, \dotsc, \fnonlin_t) & \explain{def}{=} \bmmse(\serv{X}_{\star} | \serv{W}_{1,\bdot}, \dotsc, \serv{W}_{t,\bdot}; \serv{A})  \\ &= \mmse(\effsnr(\serv{X}_{\star} | \serv{W}_{1,\bdot}, \dotsc, \serv{W}_{t,\bdot}; \serv{A})) = \scrm(\kappa_t),
\end{align}
where we introduced the function $\scrm:[0,1] \mapsto [0,1]$ defined as follows:
\begin{align}
\scrm(\kappa) \explain{def}{=} \begin{cases} \mmse \left( q \tran \left[ Q  + \frac{\kappa}{1-\kappa} \cdot \Gamma \right]^{\dagger} q \right) & :\kappa \in [0,1) \\ \mmse(0) & :\kappa = 1, \end{cases}
\end{align}
where $q, Q, \Gamma$ are as defined in \eqref{eq:liftedOAMP-SE}.
\end{subequations}

\paragraph*{Monotonicity of $\scrm(\cdot)$} The following monotonicity property of the function $\scrm$ introduced in \eqref{eq:scrm} will play an important role in the proof of \propref{prop:greedy-optimality}.

\begin{lemma}\label{lem:scrm-monotonic} The function $\scrm:[0,1] \mapsto [0,1]$ is non-decreasing. 
\end{lemma}
\begin{proof}
See \sref{sec:proof-scrm-monotonic}.
\end{proof}
\paragraph*{A Greedy Approach Minimizing $\scrM_t$} To derive the optimal OAMP algorithm, our goal is to solve the variational problem:
\begin{align} \label{eq:complex-variational-problem}
\min_{\fnonlin_1} \min_{\fnonlin_2} \dotsb \min_{\fnonlin_t} \scrM_t(\fnonlin_1, \dotsc, \fnonlin_t),
\end{align}
Recall that \propref{prop:greedy-optimality} claims that the iterate denoisers $(\fnonlin_t^{\star})_{t\in \N}$ that optimize \eqref{eq:complex-variational-problem} are defined recursively: once the optimal denoisers $\fnonlin^{\star}_{1}, \dotsc, \fnonlin_{t}^\star$ have been specified, the optimal denoiser for the $t+1$ iteration is the DMMSE estimator for the Gaussian channel  $(\serv{X}_\star, \serv{W}_{1,\bdot}, \dotsc,  \serv{W}_{t,\bdot} ; \serv{A})$, where $(\serv{X}_\star, \serv{W}_{1,\bdot}, \dotsc,  \serv{W}_{t,\bdot} ; \serv{A})$ denote the state evolution random variables corresponding to the first $t$ iterations. We will show that these choices for the iterate denoisers correspond to a \emph{greedy approach} to solving the optimization problem in \eqref{eq:complex-variational-problem}:
\begin{align*}
\fnonlin_1^\star \in \argmin_{\fnonlin_1} \scrM_1(\fnonlin_1), \  \fnonlin_2^{\star} \in \argmin_{\fnonlin_2} \scrM_2(\fnonlin_1^{\star},\fnonlin_2), \  \dotsc, \quad \fnonlin_t^{\star} \in \argmin_{\fnonlin_t} \scrM_t(\fnonlin_{1:t-1}^{\star}, \fnonlin_t).
\end{align*}
Proving this requires us to characterize the solution of the simpler (compared to \eqref{eq:complex-variational-problem}) variational problem:
\begin{align*}
\min_{\fnonlin_t} \scrM_t(\fnonlin_1, \dotsc, \fnonlin_t),
\end{align*}
which is done in the following lemma. 
\begin{lemma} \label{lem:greedy-lemma} Let:
\begin{enumerate}
    \item $\fnonlin_{1}, \dotsc, \fnonlin_{t-1}$ with $\fnonlin_i: \R^{(i-1) \times \degree + \auxdim} \mapsto \R$ be a fixed collection of iterate denoisers which satisfy the requirements imposed in the definition of lifted OAMP algorithms (\defref{def:LOAMP}) ;
    \item $(\serv{X}_{\star}, \serv{W}_{1; \bdot}, \dotsc, \serv{W}_{t-1,\bdot}; \serv{A})$ be the state evolution random variables associated with the first $t-1$ iterations of a lifted OAMP algorithm which uses the functions $\fnonlin_{1}, \dotsc, \fnonlin_{t-1}$ in the first $t-1$ iterations;
    \item $\fnonlin_t^{\sharp}$ denote the DMMSE estimator  for the Gaussian channel $(\serv{X}_{\star}, \serv{W}_{1; \bdot}, \dotsc, \serv{W}_{t-1,\bdot}; \serv{A})$.
\end{enumerate} 
Then, for any iterate denoiser $\fnonlin_t:\R^{(t-1) \times \degree + \auxdim} \mapsto \R$  which satisfies the requirements imposed in \defref{def:LOAMP}, we have:
\begin{align} \label{eq:greedy-LB}
\scrM_t(\fnonlin_1, \dotsc, \fnonlin_t)  \geq \scrM_t(\fnonlin_1, \dotsc, \fnonlin_t^{\sharp}) =  \scrm(\kappa_t^\sharp),
\end{align}
where $\kappa_{t}^{\sharp} \explain{def}{=} \dmmse \circ \mmse^{-1}(\scrM_{t-1}(\fnonlin_{1:t-1})))$.
\end{lemma}
\begin{proof}
The proof of this lemma is presented in \sref{sec:proof-greedy-lemma}.
\end{proof}

Using the above intermediate lemmas, we can now present the proof of \propref{prop:greedy-optimality}. 

\begin{proof}[Proof of \propref{prop:greedy-optimality}.] To prove the claim of \propref{prop:greedy-optimality}, we need to show that for any $t \in \N$ and for any choice of iterate denoisers $\fnonlin_{1:t}$ which satisfy the requirements imposed in \defref{def:LOAMP},
\begin{align} \label{eq:greedy-optimality-goal}
    \scrM_t(\fnonlin_{1}, \dotsc, \fnonlin_{t}) \geq \scrM_t(\fnonlin_{1}^\star, \dotsc, \fnonlin_t^\star),
\end{align}
where $\fnonlin_{1:t}^\star$ denote the functions used in the iterations of the optimal lifted OAMP algorithm, which are specified recursively as follows: after the functions $\fnonlin^{\star}_{1}, \dotsc, \fnonlin_{t-1}^\star$ have been specified, $\fnonlin_{t}^\star$ was defined as the DMMSE estimator for the Gaussian channel corresponding to the first $t-1$ iterations of the resulting lifted OAMP algorithm. By the \lemref{lem:greedy-lemma}, $\fnonlin_{1:t}^\star$ are precisely the functions derived from the greedy heuristic:
\begin{align*}
\fnonlin_1^\star \in \argmin_{\fnonlin_1} \scrM_1(\fnonlin_1), \  \fnonlin_2^{\star} \in \argmin_{\fnonlin_2} \scrM_2(\fnonlin_1^{\star},\fnonlin_2), \  \dotsc, \quad \fnonlin_t^{\star} \in \argmin_{\fnonlin_t} \scrM_t(\fnonlin_{1:t-1}^{\star}, \fnonlin_t),
\end{align*}
where the minimization is understood to be over iterate denoisers which satisfy the requirements imposed in \defref{def:LOAMP}. We will obtain the desired conclusion in \eqref{eq:greedy-optimality-goal} by induction. As our induction hypothesis, we assume that for any choice of the iterate denoisers $\fnonlin_{1}, \dotsc, \fnonlin_{t-1}$, we have:
\begin{align} \label{eq:greedy-optimality-induction}
    \scrM_{t-1}(\fnonlin_{1}, \dotsc, \fnonlin_{t-1}) \geq \scrM_{t-1}(\fnonlin_{1}^\star, \dotsc, \fnonlin_{t-1}^\star).
\end{align}
In order to obtain the desired conclusion \eqref{eq:greedy-optimality-goal}, we observe that:
\begin{align*}
 \scrM_t(\fnonlin_{1}, \dotsc, \fnonlin_{t}) & \explain{(a)}{\geq}   \scrm(\kappa_t^\sharp),  \text{ where } \kappa_{t}^{\sharp} \explain{def}{=} \dmmse \circ \mmse^{-1}(\scrM_{t-1}(\fnonlin_{1:t-1}))) \\
 &\explain{(b)}{\geq} \scrm(\kappa_t^\star),  \text{ where } \kappa_{t}^{\star} \explain{def}{=} \dmmse \circ \mmse^{-1}(\scrM_{t-1}(\fnonlin_{1:t-1}^\star))) \\
 & \explain{(c)}{=} \scrM_t(\fnonlin_{1}^\star, \dotsc, \fnonlin_{t}^\star).
\end{align*}
In the above display: 
\begin{itemize}
    \item Step (a) used the lower bound for $ \scrM_t(\fnonlin_{1}, \dotsc, \fnonlin_{t})$ provided in \lemref{lem:greedy-lemma}.
    \item In step (b), we observed by the induction hypothesis that $\scrM_{t-1}(\fnonlin_{1:t-1}^\star) \leq \scrM_{t-1}(\fnonlin_{1:t-1})$. Since $\mmse$ and $\dmmse$ are non-increasing functions (\lemref{fact:mmse-monotonic} and \lemref{lem:dmmse-monotonicity}), we conclude that $\kappa_t^\star \leq \kappa_t^\sharp$. Finally, since $\scrm$ is non-increasing (\lemref{lem:scrm-monotonic}), we concluded that $\scrm(\kappa_t^\sharp) \geq \scrm(\kappa_t^\star)$. 
    \item In step (c), we again appealed to \lemref{lem:greedy-lemma}:
    \begin{align*}
     \scrM_t(\fnonlin_{1}^\star, \dotsc, \fnonlin_{t}^\star) & = \min_{\fnonlin_t}  \scrM_t(\fnonlin_{1}^\star, \dotsc, \fnonlin_{t}) \\ &= \scrm(\kappa_t^\star), \text{ where } \kappa_{t}^{\star} \explain{def}{=} \dmmse \circ \mmse^{-1}(\scrM_{t-1}(\fnonlin_{1:t-1}^\star))).
    \end{align*}
\end{itemize} 
This concludes the proof of \propref{prop:greedy-optimality}.
\end{proof}

\subsubsection{Proof of \lemref{lemma:effsnr-LOAMP}} \label{sec:proof-mmse-functional}
\begin{proof}[Proof of \lemref{lemma:effsnr-LOAMP}] Recall from \defref{def:LOAMP} that the joint distribution of the state evolution random variables is given by:
\begin{align} \label{eq:SE-gauss-channel}
(\serv{X}_{\star}; \serv{A}) &\sim \pi, \\ (\serv{W}_{1,\bdot}, \dotsc, \serv{W}_{t,\bdot}) | (\serv{X}_\star;\serv{A}) &\sim \gauss{\serv{X_\star} \cdot  \alpha \otimes q}{ \alpha \alpha^\top   \otimes (Q - q q^\top) + (\Sigma - \alpha \alpha^\top)  \otimes \Gamma}.
\end{align}
In the above display, the entries of $q \in \R^{\degree}$, $Q \in \R^{\degree \times \degree}$, and $\Gamma \in \R^{\degree \times \degree}$ are given by:
\begin{subequations}  \label{eq:q-Q-Gamma-def}
   \begin{align}
q_i  &= \E_{\serv{\Lambda}_{\nu} \sim \nu}[{\serv{\Lambda}}_{\nu}^i] - \E_{\serv{\Lambda} \sim \mu}[{\serv{\Lambda}}^i],  \\  Q_{ij} 
&= \E_{\serv{\Lambda}_{\nu} \sim \nu}\left[ \left({\serv{\Lambda}}_{\nu}^i - \E_{\serv{\Lambda} \sim \mu}[{\serv{\Lambda}}^i] \right)  \left({\serv{\Lambda}}_{\nu}^j - \E_{\serv{\Lambda} \sim \mu}[{\serv{\Lambda}}^j] \right) \right], \; \Gamma_{ij} = \Cov_{\serv{\Lambda} \sim \mu}[\serv{\Lambda}^i ,\serv{\Lambda}^j].
\end{align} 
\end{subequations}
and the entries of $\alpha \in \R^t$ and $\Sigma \in \R^{t \times t}$ are given by:
\begin{align} \label{eq:alpha-Sigma-def}
\alpha_{s} \explain{def}{=} \E[\serv{X}_{\star} \cdot {\fnonlin}_{s}(\serv{W}_{<s,\bdot}; \serv{A})], \quad \Sigma_{s \tau} \explain{def}{=} \E[{\fnonlin}_{s}(\serv{W}_{<s,\bdot}; \serv{A}) \cdot {\fnonlin}_{\tau}(\serv{W}_{<\tau,\bdot}; \serv{A})] \quad \forall \;  s,\tau \; \in \; [t].
\end{align}
The proof of the claim is obtained by applying the effective SNR formula for general multivariate Gaussian channels stated as \lemref{lem:scalarization-mmse} in \appref{appendix:gauss-channel} to the Gaussian channel $(\serv{X}_{\star}; \serv{W}_{\leq t; \bdot}; \serv{A})$. According to this formula:
\begin{subequations}\label{eq:effective-snr-recall}
\begin{align} 
\effsnr(\serv{X}_{\star}| \serv{W}_{\leq t,\bdot}; \serv{A}) =  \begin{cases} \frac{\vartheta}{1+\vartheta} & \text{ if } \alpha \otimes q \in \mathrm{Range}(\alpha \alpha^\top   \otimes (Q - q q^\top) + (\Sigma - \alpha \alpha^\top)  \otimes \Gamma), \\
1 & \text{ if } \alpha \otimes q \notin \mathrm{Range}(\alpha \alpha^\top   \otimes (Q - q q^\top) + (\Sigma - \alpha \alpha^\top)  \otimes \Gamma),
\end{cases}
\end{align}
where:
\begin{align}
\vartheta \explain{def}{=}  \alpha \otimes q^\top ( \alpha \alpha^\top   \otimes (Q - q q^\top) + (\Sigma - \alpha \alpha^\top)  \otimes \Gamma)^{\dagger}  \alpha \otimes q.
\end{align}
\end{subequations}
The claim of the lemma is obtained by simplifying the above formula.
\paragraph*{Orthogonalization} To obtain the simplified formula stated in the lemma, it will be convenient to construct an orthonormal basis for $\mathrm{Span}({f}_{1:t})$. Let $r = \mathrm{Rank}(\Sigma)$ and  $\hat{\fnonlin}_{1}, \hat{\fnonlin}_{2}, \dotsc, \hat{\fnonlin}_r:\R^{(t-1)\degree} \mapsto \R$ denote the basis functions constructed by orthogonalizing the functions ${\fnonlin}_{1:t}$ with respect to the state evolution random variables $(\serv{W}_{<t,\bdot}; \serv{A})$, via the Gram-Schmidt procedure. These basis functions are orthonormal in the sense that:
\begin{align*}
\E[\hat{\fnonlin}_{s}(\serv{W}_{< t,\bdot}; \serv{A}) \hat{\fnonlin}_{\tau}(\serv{W}_{< t,\bdot}; \serv{A})] & = \begin{cases} 1 & : s= \tau \\ 0 & : \text{otherwise}. \end{cases}
\end{align*}
We can express the functions ${\fnonlin}_{1:t}$ as a linear combination of the basis functions:
\begin{align} \label{eq:ortho-decomp}
{\fnonlin}_{\tau} & = \sum_{s = 1}^r L_{\tau s} \hat{\fnonlin}_{s} \quad \forall \; \tau \; \in \; [t], \quad L_{\tau s} \explain{def}{=} \E[{\fnonlin}_{\tau}(\serv{W}_{<\tau,\bdot}; \serv{A}) \hat{\fnonlin}_{s}(\serv{W}_{< t,\bdot}; \serv{A})] \quad \forall \; s\ \in [r], \; \tau \in [t].
\end{align}
Let $L$ denote the $t \times r$ matrix whose entries are given by the coefficients $(L_{\tau s})_{\tau \in [t], s \in [r]}$ defined above. Using the decomposition in \eqref{eq:ortho-decomp}, along with the definitions of $\alpha, \Sigma$  in \eqref{eq:alpha-Sigma-def}, we obtain the following alternative expressions for $\alpha$ and $\Sigma$:
\begin{align} \label{eq:alpha-Sigma-alt}
\Sigma & = L L^\top, \quad \alpha = L \hat{\alpha}, \quad \hat{\alpha} \explain{def}{=} (\E[\serv{X}_{\star} \hat{\fnonlin}_1(\serv{W}_{< t, \bdot}; \serv{A})], \E[\serv{X}_{\star} \hat{\fnonlin}_2(\serv{W}_{< t, \bdot}; \serv{A})] \cdots, \E[\serv{X}_{\star} \hat{\fnonlin}_t(\serv{W}_{< t, \bdot}; \serv{A})])^\top.
\end{align}
Similarly, by exploiting the fact that $\hat{\fnonlin}_{1:t}$ form an orthonormal basis for $\mathrm{Span}({\fnonlin}_{1:t})$, we have the following alternate formula for $\kappa_t$:
\begin{align}
    \kappa_t &\explain{def}{=} \min\bigg\{ \E\left[\{\serv{X}_{\star} - g(\serv{W}_{< t,\bdot}; \serv{A})\}^2\right] : g \in \mathrm{Span}({\fnonlin}_1, \dotsc, {\fnonlin}_t) \bigg\} \nonumber \\
    & = 1 - \|\hat{\alpha}\|^2 \in [0,1]. \label{eq:kappa-alt}
\end{align}
These alternate expressions for $\alpha, \Sigma, \kappa_t$ will be useful in our subsequent calculations. 
\paragraph*{Useful Linear Algebraic Results} The proof relies on the following intermediate linear algebraic claims.
\begin{align}
&\text{For any }  \lambda \geq 0, \quad  q \in \mathrm{Range}(Q + \lambda \Gamma),  \label{eq:LA-claim1}\\
&\text{If } \kappa_t < 1,  \label{eq:LA-claim2} \\  &q^\top \left(Q + \frac{\kappa_t}{1-\kappa_t} \cdot \Gamma \right)^\dagger q \neq 1 \Leftrightarrow  \alpha \otimes q  \in \mathrm{Range}( \alpha \alpha^\top   \otimes (Q - q q^\top) + (\Sigma - \alpha \alpha^\top)  \otimes \Gamma). \nonumber
\end{align}
We defer the justification of these claims to the end of the proof.
\paragraph*{Corner Cases} We begin by verifying the claim of the lemma for two corner cases. Consider the situation where $\kappa_t = 1$. In this case, \eqref{eq:alpha-Sigma-alt} and \eqref{eq:kappa-alt} imply that $\alpha = 0$. As a result, using the formula for the effective SNR for a general multivariate Gaussian channel given in \lemref{lem:scalarization-mmse} (\appref{appendix:gauss-channel}), the effective SNR of the Gaussian channel in \eqref{eq:SE-gauss-channel} is $\effsnr = 0$,  as claimed. Next, we consider the case when:
\begin{align*}
 \kappa_t < 1,  & \quad q^\top \left(Q + \frac{\kappa_t}{1-\kappa_t} \cdot \Gamma \right)^\dagger q = 1.
\end{align*}
From \eqref{eq:LA-claim2}, we conclude that:
\begin{align*} 
\alpha \otimes q  \notin \mathrm{Range}( \alpha \alpha^\top   \otimes (Q - q q^\top) + (\Sigma - \alpha \alpha^\top)  \otimes \Gamma).
\end{align*}
Hence, from \eqref{eq:effective-snr-recall}, we conclude that the effective SNR of the Gaussian channel in \eqref{eq:SE-gauss-channel} is $\effsnr = 1$, as claimed. 
\paragraph*{Derivation of the claimed formulae} The remaining task is to prove the claim of the lemma when:
\begin{align*}
\kappa_t < 1,  & \quad q^\top \left(Q + \frac{\kappa_t}{1-\kappa_t} \cdot \Gamma \right)^\dagger q \neq  1,
\end{align*}
and we work under the above assumption for the remainder of the proof. From \eqref{eq:LA-claim2}, we conclude that:
\begin{align*} 
\alpha \otimes q  \in \mathrm{Range}( \alpha \alpha^\top   \otimes (Q - q q^\top) + (\Sigma - \alpha \alpha^\top)  \otimes \Gamma).
\end{align*}
From \eqref{eq:effective-snr-recall}, we have that:
\begin{align} \label{eq:snr-optlin-vartheta}
\effsnr(\serv{X}_{\star} | \serv{W}_{\leq t, \bdot}; \serv{A}) & =  \frac{\vartheta}{1+\vartheta}, \\ \optlin(\serv{X}_{\star} | \serv{W}_{\leq t, \bdot}; \serv{A}) &= \frac{1}{\sqrt{\vartheta + \vartheta^2}} \cdot  ( \alpha \alpha^\top   \otimes (Q - q q^\top) + (\Sigma - \alpha \alpha^\top)  \otimes \Gamma)^{\dagger} \cdot \alpha \otimes q
\end{align}
where:
\begin{align} \label{eq:vartheta}
\vartheta \explain{def}{=} \alpha \otimes q^\top ( \alpha \alpha^\top   \otimes (Q - q q^\top) + (\Sigma - \alpha \alpha^\top)  \otimes \Gamma)^{\dagger}  \alpha \otimes q.
\end{align}
To compute the pseudoinverse $( \alpha \alpha^\top   \otimes (Q - q q^\top) + (\Sigma - \alpha \alpha^\top)  \otimes \Gamma)^{\dagger}$, we define $\overline{\alpha}$ as the unit vector along $\hat{\alpha}$ and $P$ as the projector orthogonal to $\hat{\alpha}$:
\begin{align} \label{eq:alpha-bar-P}
\overline{\alpha} \explain{def}{=} \frac{\hat{\alpha}}{\|\hat{\alpha}\|}, \quad P \explain{def}{=} I - \overline{\alpha} \overline{\alpha}^\top.
\end{align}
We have,
\begin{align*}
&( \alpha \alpha^\top   \otimes (Q - q q^\top) + (\Sigma - \alpha \alpha^\top)  \otimes \Gamma)^{\dagger}  \\ &\hspace{1cm}\explain{\eqref{eq:alpha-Sigma-alt}}{=} (L^\dagger \otimes I)^{\top} \cdot [ \hat{\alpha} \hat{\alpha}^\top \otimes (Q- q q^\top - \Gamma ) + I \otimes \Gamma ]^\dagger \cdot (L^{\dagger} \otimes I) \\
& \hspace{1cm} \explain{\eqref{eq:kappa-alt},\eqref{eq:alpha-bar-P}}{=} (L^\dagger \otimes I)^{\top} \cdot [ \overline{\alpha} \overline{\alpha}^\top \otimes \{ (1-\kappa_t) \cdot (Q- q q^\top) + \kappa_t \cdot \Gamma \} + P \otimes \Gamma ]^\dagger \cdot (L^{\dagger} \otimes I) \\
& \hspace{1cm} = (L^\dagger \otimes I)^{\top} \cdot [ \overline{\alpha} \overline{\alpha}^\top \otimes \{ (1-\kappa_t) \cdot (Q- q q^\top) + \kappa_t \cdot \Gamma \}^{\dagger} + P \otimes \Gamma^\dagger] \cdot (L^{\dagger} \otimes I).
\end{align*}
The above formula for the pseudo-inverse yields: 
\begin{align} \label{eq:inspiration}
&( \alpha \alpha^\top   \otimes (Q - q q^\top) + (\Sigma - \alpha \alpha^\top)  \otimes \Gamma)^{\dagger} \cdot \alpha \otimes q \\ & \hspace{1cm}=   (L^{\dagger \top} \hat{\alpha})  \otimes   [ \{(1-\kappa_t) \cdot (Q- q q^\top) + \kappa_t \cdot \Gamma \}^{\dagger} \cdot q] \nonumber \\
& \hspace{1cm} \explain{(a)}{=} (1-\kappa_t)^{-1} \cdot  \left\{1 - q^\top \left(Q + \frac{\kappa_t}{1-\kappa_t} \cdot \Gamma \right)^\dagger q \right\}^{-1} \cdot (L^{\dagger \top} \hat{\alpha})  \otimes   \left(Q + \frac{\kappa_t}{1-\kappa_t} \cdot \Gamma \right)^\dagger q \nonumber\\
& \hspace{1cm} \explain{\eqref{eq:alpha-Sigma-alt}}{=} (1-\kappa_t)^{-1} \cdot  \left\{1 - q^\top \left(Q + \frac{\kappa_t}{1-\kappa_t} \cdot \Gamma \right)^\dagger q \right\}^{-1} \cdot (\Sigma^{\dagger} {\alpha})  \otimes   \left(Q + \frac{\kappa_t}{1-\kappa_t} \cdot \Gamma \right)^\dagger q, 
\end{align}
where we used the Sherman-Morrison formula in (a). Combining the computation above with the definition of $\vartheta$ in \eqref{eq:vartheta} yields the following formula for $\vartheta$:
\begin{align*}
\vartheta & \explain{\eqref{eq:kappa-alt}}{=} (\alpha^\top \Sigma^\dagger \alpha) \cdot (1-\kappa_t)^{-1} \cdot \left\{1 - q^\top \left(Q + \frac{\kappa_t}{1-\kappa_t} \cdot \Gamma \right)^\dagger q \right\}^{-1} \cdot q^\top \left(Q + \frac{\kappa_t}{1-\kappa_t} \cdot \Gamma \right)^\dagger q \\
& \explain{(a)}{=} \left\{1 - q^\top \left(Q + \frac{\kappa_t}{1-\kappa_t} \cdot \Gamma \right)^\dagger q \right\}^{-1} \cdot q^\top \left(Q + \frac{\kappa_t}{1-\kappa_t} \cdot \Gamma \right)^\dagger q
\end{align*}
where the equality marked (a) follows by observing that $ \alpha^\top \Sigma^\dagger \alpha = \|\hat{\alpha}\|^2 = 1-\kappa_t$. Plugging the expression of the previous display in \eqref{eq:snr-optlin-vartheta} yields:
\begin{align*}
\effsnr(\serv{X}_{\star} | \serv{W}_{\leq t, \bdot}; \serv{A}) &=  q^\top \left(Q + \frac{\kappa_t}{1-\kappa_t} \cdot \Gamma \right)^\dagger q, \\
\optlin(\serv{X}_{\star} | \serv{W}_{\leq t, \bdot}; \serv{A})& =   \frac{1}{\sqrt{\effsnr(\serv{X}_{\star} | \serv{W}_{\leq t, \bdot}; \serv{A})}} \cdot (\Sigma^{\dagger} {\alpha})  \otimes   \left((1-\kappa_t) \cdot Q + {\kappa_t} \cdot \Gamma \right)^\dagger q
\end{align*}
as claimed. To finish the proof, we need to prove the linear algebraic claims in \eqref{eq:LA-claim1} and \eqref{eq:LA-claim2}.
\paragraph*{Proof of claim \eqref{eq:LA-claim1}} Fix any $\lambda \geq 0$. To show that $q \in \mathrm{Range}(Q + \lambda \Gamma)$, it is sufficient to verify that $q$ is orthogonal to $\mathrm{Null}(Q + \lambda \Gamma)$. Towards this goal, we consider any $v \in \mathrm{Null}(Q + \lambda \Gamma)$ and aim to show that $\ip{v}{q} = 0$. Since $v$ lies in the null space of $Q + \lambda \Gamma$,
\begin{align*}
v^\top (Q + \lambda \Gamma) v & = 0 \implies v^\top Q v = 0,
\end{align*}
where the implication is obtained by observing that $Q \succeq 0$ and $\Gamma \succeq 0$ (cf. \eqref{eq:q-Q-Gamma-def}). Recalling the definition of $q$ from \eqref{eq:q-Q-Gamma-def}, we obtain:
\begin{align*}
|\ip{q}{v}|^2 & = \left|\E \bigg[\sum_{i=1}^\degree v_i \cdot (\serv{\Lambda}_{\nu}^i -\E[\serv{\Lambda}^i]) \bigg]\right|^2 \explain{(a)}{\leq}\E \left[\bigg|\sum_{i=1}^\degree v_i \cdot (\serv{\Lambda}_{\nu}^i -\E[\serv{\Lambda}^i]) \bigg|^2\right] \explain{\eqref{eq:q-Q-Gamma-def}}{=} v^\top Q v = 0.
\end{align*}
In the above display, $\serv{\Lambda} \sim \mu, \serv{\Lambda}_{\nu} \sim \nu$ and the inequality in (a) follows from the Cauchy–Schwarz inequality. Hence, $\ip{q}{v} = 0$ for any $v \in \mathrm{Null}(Q + \lambda \Gamma)$, which implies the claim in \eqref{eq:LA-claim1}. 
\paragraph*{Proof of claim \eqref{eq:LA-claim2}} 
For convenience, we introduce the definitions:
\begin{align*}
    S \explain{def}{=}  \alpha \alpha^\top   \otimes (Q - q q^\top) + (\Sigma - \alpha \alpha^\top)  \otimes \Gamma, \quad T \explain{def}{=} Q + \frac{\kappa_t}{1-\kappa_t} \cdot \Gamma.
\end{align*}
Assuming $\kappa_t < 1$, our goal is to show the equivalence:
\begin{align*}
q^\top T^\dagger q \neq 1 \Leftrightarrow  \alpha \otimes q  \in \mathrm{Range}(S).
\end{align*}
Consider the vector:
\begin{align*}
v \explain{def}{=} (L^{\dagger \top} \hat{\alpha})  \otimes   (T^{\dagger} q).
\end{align*}
A short computation reveals that:
\begin{align*}
S v & = (1-\kappa_t) \cdot (1-q^\top T^\dagger q) \cdot \alpha \otimes q, \quad \ip{q}{v} = (1-\kappa_t) \cdot q^\top T^\dagger q. 
\end{align*}
If $q^\top T^{\dagger} q \neq 1$, we have found a vector $u = (1-\kappa_t)^{-1} \cdot (1- q^\top T^\dagger q)^{-1} \cdot v$ such that $S u = \alpha \otimes q$, which certifies that $\alpha \otimes q \in \mathrm{Range}(S)$. On the other hand, if $q^\top T^{\dagger} q = 1$, then $S v = 0$ and we have found a vector $v \in \mathrm{Null}(S)$ such that $\ip{v}{q} = (1-\kappa_t) \neq 0$, which certifies that $\alpha \otimes q \notin \mathrm{Range}(S)$. This concludes the proof of this lemma. 
\end{proof}

\subsubsection{Proof of \lemref{lem:scrm-monotonic}} \label{sec:proof-scrm-monotonic}
\begin{proof}[Proof of \lemref{lem:scrm-monotonic}] Recall from \eqref{eq:scrm} that:
\begin{align}
\scrm(\kappa) \explain{def}{=} \begin{cases} \mmse \left( q \tran \left[ Q  + \frac{\kappa}{1-\kappa} \cdot \Gamma \right]^{\dagger} q \right) & :\kappa \in [0,1) \\ \mmse(0) & :\kappa = 1. \end{cases}
\end{align}
The expressions for the vector $q \in \R^{\degree}$ and the \emph{positive semi-definite} matrices $Q, \Gamma \in \R^{\degree \times \degree}$ appear in \eqref{eq:liftedOAMP-SE}, but will not be needed in the proof.  To show that $\scrm$ is a non-decreasing function, it suffices to verify that for any $0 \leq \kappa \leq \kappa^\prime < 1$, we have $\scrm(\kappa) \leq \scrm(\kappa^\prime) \leq \scrm(1)$. Indeed we have, 
\begin{align}
&0 \leq \kappa \leq \kappa^\prime < 1 \implies 0 \leq \frac{\kappa}{1-\kappa} \leq \frac{\kappa^{\prime}}{1-\kappa^{\prime}}  \explain{(a)}{\implies}  0 \preceq Q  + \frac{\kappa}{1-\kappa} \cdot \Gamma \preceq Q  + \frac{\kappa^\prime}{1-\kappa^\prime} \cdot \Gamma  \label{eq:Q-Gamma-psd-order} \\
&  \explain{(b)}{\implies} q^{\top}\left(Q  + \frac{\kappa}{1-\kappa} \cdot \Gamma\right)^{\dagger} q \geq q^{\top}\left(Q   + \frac{\kappa^\prime}{1-\kappa^\prime} \cdot \Gamma\right)^{\dagger} \geq 0 \nonumber \\
&  \explain{(c)}{\implies}\mmse \left( q \tran \left[ Q  + \frac{\kappa}{1-\kappa} \cdot \Gamma \right]^{\dagger} q \right) \leq \mmse \left( q \tran \left[ Q  + \frac{\kappa^{\prime}}{1-\kappa^{\prime}} \cdot \Gamma \right]^{\dagger} q \right) \leq \mmse(0), \nonumber \\
&\explain{(d)}{\implies}\scrm(\kappa) \leq \scrm(\kappa^\prime) \leq \scrm(1). \nonumber
\end{align}
In the above display, step (a) follows by observing that $Q,\Gamma$ are positive semidefinite ($\preceq$ denotes the standard Loewner partial ordering of symmetric matrices defined by the positive semidefinite cone), step (c) relies on the monotonicity of $\mmse(\cdot)$ (\factref{fact:mmse-monotonic}). Finally, step (b) follows from the Fenchel conjugate formula for convex quadratic forms (recall from \eqref{eq:LA-claim1} that $q \in \mathrm{Range}(Q + \lambda \Gamma)$ for any $\lambda \geq 0$):
\begin{align*}
&q^{\top}\left(Q  + \frac{\kappa}{1-\kappa} \cdot \Gamma\right)^{\dagger} q  = \sup_{\xi \in \R^{\degree}} 2\ip{q}{\xi} - \xi^\top \left( Q  + \frac{\kappa}{1-\kappa} \cdot \Gamma \right) \xi \\ &\hspace{2cm}\explain{\eqref{eq:Q-Gamma-psd-order}}{\geq} \sup_{\xi \in \R^{\degree}} 2\ip{q}{\xi} - \xi^\top \left( Q  + \frac{\kappa^\prime}{1-\kappa^\prime} \cdot \Gamma \right) \xi = q^{\top}\left(Q  + \frac{\kappa^\prime}{1-\kappa^\prime} \cdot \Gamma\right)^{\dagger} q.
\end{align*}
This verifies that $\scrm$ is non-decreasing. 
\end{proof}

\subsubsection{Proof of \lemref{lem:greedy-lemma}} \label{sec:proof-greedy-lemma}
\begin{proof}[Proof of \lemref{lem:greedy-lemma}] Given a collection of iterate denoisers $\fnonlin_{1}, \dotsc, \fnonlin_{t}$, which satisfy the requirements of \defref{def:LOAMP}, our goal is to show that:
\begin{align}
    \scrM_{t}(f_1, f_2, \dotsc, f_{t-1}, f_t) &\geq \scrm(\kappa_t^{\sharp}) \text{ and,} \label{eq:greedy-lemma-claim1}\\
    \scrM_{t}(f_1, f_2, \dotsc, f_{t-1}, f_t^{\sharp}) & =  \scrm(\kappa_t^{\sharp}) \label{eq:greedy-lemma-claim2},
\end{align}
where $\kappa_t^{\sharp} \explain{def}{=}  \dmmse \circ \mmse^{-1}(\scrM_{t-1}(\fnonlin_{1:t-1})))$ and $\fnonlin_t^{\sharp}$ is the DMMSE estimator for the Gaussian channel $(\serv{X}_{\star}, \serv{W}_{1,\bdot}, \dotsc, \serv{W}_{t-1,\bdot}; \serv{A})$ generated by the state evolution random variables associated with the lifted OAMP algorithm which uses the denoisers $\fnonlin_{1}, \dotsc, \fnonlin_{t-1}$ in the first $t-1$ iterations.  From \lemref{lemma:effsnr-LOAMP}, we know that:
\begin{subequations} \label{eq:mmse-functional-recall}
\begin{align}
\scrM_{t}(f_1, f_2, \dotsc, f_t) & = \scrm(\kappa_t)
\end{align}
where the scalar $\kappa_t \in [0,1]$ was defined as:
\begin{align}
    \kappa_t \explain{def}{=} \min\bigg\{ \E\left[\{\serv{X}_{\star} - g(\serv{W}_{< t,\bdot}; \serv{A})\}^2\right] : g \in \mathrm{Span}({\fnonlin}_1, \dotsc, {\fnonlin}_t) \bigg\}.
\end{align}
\end{subequations}
We consider the two claims of the lemma \eqref{eq:greedy-lemma-claim1} and \eqref{eq:greedy-lemma-claim2} one by one.
\paragraph*{Proof of Claim \eqref{eq:greedy-lemma-claim1}}  In light of the monotonicity of $\scrm$ (\lemref{lem:scrm-monotonic}), it suffices to show that
\begin{align} \label{eq:kappa_t-LB-goal}
\kappa_t &\explain{def}{=} \min\bigg\{ \E\left[\{\serv{X}_{\star} - g(\serv{W}_{< t,\bdot}; \serv{A})\}^2\right] : g \in \mathrm{Span}({\fnonlin}_1, \dotsc, {\fnonlin}_t) \bigg\} \geq \kappa_{t}^{\sharp} \\ &\explain{def}{=} \dmmse \circ \mmse^{-1}(\scrM_{t-1}(\fnonlin_{1:t-1}))).
\end{align} We can lower bound $\kappa_t$ as follows:
\begin{align}
&\kappa_t \explain{def}{=} \min_{g \in \mathrm{Span}({\fnonlin}_1, \dotsc, {\fnonlin}_t)} \E\left[\{\serv{X}_{\star} - g(\serv{W}_{< t,\bdot}; \serv{A})\}^2\right] \nonumber \\
& \explain{(a)}{\geq}  \min_{g \in \mathrm{Span}({\fnonlin}_1, \dotsc, {\fnonlin}_t)}  \E\left[\{\serv{X}_{\star} - g(\serv{W}_{< t,\bdot}; \serv{A})\}^2\right] \text{ s.t. } \E[\serv{Z}_{s,i} g(\serv{W}_{<t,\bdot}; \serv{A})] = 0 \; \forall \; s \in [t-1], \; i \in [\degree] \nonumber \\
& \explain{def}{=} \bdmmse(\serv{X}_{\star} | \serv{W}_{< t,\bdot}; \serv{A}). \label{eq: kappa_t-LB-prelim}
\end{align}
In the above display, the random variables $\serv{Z}_{<t,\bdot}$ represent the Gaussian noise random variables for the Gaussian channel $(\serv{X}_{\star}, \serv{W}_{<t,\bdot}; \serv{A})$. To obtain the inequality in step (a), we used the fact that any $g \in \mathrm{Span}(\fnonlin_1, \dotsc, {\fnonlin}_t)$ satisfies the divergence-free constraints:
\begin{align*}
\E[\serv{Z}_{s,i} \cdot g(\serv{W}_{<t,\bdot}; \serv{A})] = 0 \; \forall \; s \in [t-1], \; i \in [\degree].
\end{align*} 
Indeed, by Gaussian integration by parts, for any $i< t$, $s < t$ and $j \in [\degree]$, $$\E[\serv{Z}_{s,j} \cdot {\fnonlin}_{i}(\serv{W}_{<i;\bdot}; \serv{A})]$$ can be expressed as a linear combination of the expected derivatives $$\{\E[\partial_{\tau,\ell}{\fnonlin}_{i}(\serv{W}_{<i;\bdot}; \serv{A})]: \tau < i, \ell \in [\degree]\},$$ which are all identically zero thanks to the divergence-free requirement imposed on the iterate denoisers $\fnonlin_1, \dotsc, \fnonlin_{t-1}$ in \defref{def:LOAMP}. We now cast the lower bound on $\kappa_t$ obtained in \eqref{eq: kappa_t-LB-prelim} in the desired form \eqref{eq:kappa_t-LB-goal}. Using the DMMSE formula for general multivariate Gaussian channels (stated as \lemref{lem:scalarization-mmse} in Appendix \appref{appendix:gauss-channel}) we obtain:
\begin{align*}
\kappa_t \geq \bdmmse(\serv{X}_{\star} | \serv{W}_{< t,\bdot}; \serv{A}) = \dmmse(\effsnr),
\end{align*}
where $\effsnr$ is the effective SNR of the Gaussian channel $(\serv{X}_{\star}, \serv{W}_{<t,\bdot}; \serv{A})$. On the other hand, the MMSE formula for general multivariate Gaussian channels (stated as \lemref{lem:scalarization-mmse} in  \appref{appendix:gauss-channel}) yields:
\begin{align*}
\scrM_{t-1}(\fnonlin_1, \dotsc, \fnonlin_{t-1}) \explain{def}{=} \bmmse(\serv{X}_{\star} | \serv{W}_{<t,\bdot}; \serv{A}) = \mmse(\effsnr).
\end{align*}
Since $\mmse(\cdot)$ is a strictly decreasing function (\factref{fact:mmse-monotonic}), it has a well-defined inverse\footnote{Strictly speaking, \factref{fact:mmse-monotonic} guarantees that $\mmse:[0,1] \mapsto [0,1]$ is strictly increasing only when $\mmse(0) = \E[\Var[\serv{X}_{\star}| \serv{A}]] > 0$. In the corner case when $\mmse(0) = \E[\Var[\serv{X}_{\star}| \serv{A}]] = 0$, notice that $\dmmse(0) = \E[\Var[\serv{X}_{\star}| \serv{A}]= 0$. Since $\mmse$ and $\dmmse$ are non-negative, non-increasing functions (cf. \factref{fact:mmse-monotonic} and \lemref{lem:dmmse-monotonicity}), this means that $\mmse(\omega) = \dmmse(\omega) = 0$ for any $\omega \in [0,1]$. In this situation, we can define  $\mmse^{-1}$ arbitrarily. Irrespective of our convention, $\dmmse \circ \mmse^{-1}(\omega) = 0$, since $\dmmse$ is a constant function which takes the value $0$.} $\mmse^{-1}$, which can be used to relate the previous two displays:
\begin{align} \label{eq:kappa_t_star-DMMSE-relation}
\kappa_t \geq \dmmse(\effsnr) = \dmmse \circ \mmse^{-1}(\scrM_{t-1}(\fnonlin_1, \dotsc, \fnonlin_{t-1})) \explain{def}{=} \kappa_t^{\sharp},
\end{align}
which proves the claim made in \eqref{eq:kappa_t-LB-goal} and concludes the proof of the first part of the lemma. 
\paragraph*{Proof of Claim \eqref{eq:greedy-lemma-claim2}} In light of the formula in \eqref{eq:mmse-functional-recall}, it suffices to show that:
\begin{align*} 
 \min\bigg\{ \E\left[\{\serv{X}_{\star} - g(\serv{W}_{< t,\bdot}; \serv{A})\}^2\right] : g \in \mathrm{Span}({\fnonlin}_1, \dotsc,  {\fnonlin}_{t-1},  {\fnonlin}_t^{\sharp}) \bigg\} = \kappa_t^\sharp.
\end{align*}
From \eqref{eq:kappa_t_star-DMMSE-relation}, we already know that:
\begin{align*}
     \min\bigg\{ \E\left[\{\serv{X}_{\star} - g(\serv{W}_{< t,\bdot}; \serv{A})\}^2\right] : g \in \mathrm{Span}({\fnonlin}_1, \dotsc,  {\fnonlin}_{t-1},  {\fnonlin}_t^{\sharp}) \bigg\} \geq \kappa_t^\sharp
\end{align*}
Hence, we only need to show the upper bound:
\begin{align} \label{eq:kappa_t_hash}
     \min\bigg\{ \E\left[\{\serv{X}_{\star} - g(\serv{W}_{< t,\bdot}; \serv{A})\}^2\right] : g \in \mathrm{Span}({\fnonlin}_1, \dotsc,  {\fnonlin}_{t-1},  {\fnonlin}_t^{\sharp}) \bigg\} \leq \kappa_t^\sharp.
\end{align}
We upper bound the LHS in \eqref{eq:kappa_t_hash} by taking $g=\fnonlin_t^\sharp$ and use the DMMSE formula for a general multivariate Gaussian channel (see \lemref{lem:scalarization-mmse} in \appref{appendix:gauss-channel}) to obtain: 
\begin{align*}
 &   \min\bigg\{ \E\left[\{\serv{X}_{\star} - g(\serv{W}_{< t,\bdot}; \serv{A})\}^2\right] : g \in \mathrm{Span}({\fnonlin}_1, \dotsc,  {\fnonlin}_{t-1},  {\fnonlin}_t^{\sharp}) \bigg\} \\&
 \leq \E[(\serv{X}_{\star}- \fnonlin_t^{\sharp}(\serv{W}_{<t,\bdot}; \serv{A}))^2] \\ & \explain{(a)}{=} \bdmmse(\serv{X}_{\star}|\serv{W}_{<t,\bdot}; \serv{A}) \\& \explain{Lem. \ref{lem:scalarization-mmse}}{=}  \dmmse(\effsnr) \explain{\eqref{eq:kappa_t_star-DMMSE-relation}}{=}  \kappa_t^\sharp,
\end{align*}
where step (a) follows by recalling that $\fnonlin_t^{\sharp}$ was the DMMSE estimator for $(\serv{X}_{\star}, \serv{W}_{1,\bdot}, \dotsc, \serv{W}_{t-1, \bdot}; \serv{A})$. This concludes the proof of the lemma. 
\end{proof}

\subsection{Proof of \propref{prop:simplification}}\label{App:proof_prop_simplications}
We begin by reminding the reader that the optimal degree-$\degree$ lifted OAMP algorithm was given by:
\begin{align} \label{eq:optimal-LOAMP-general-form}
\iter{\vw}{t,i} & = \left( \mY^i - \E_{\serv{\Lambda} \sim \mu}[\serv{\Lambda}^i] \cdot \mI_{\dim} \right) \cdot \fnonlin^{\star}_{t}\big(\iter{\vw}{1,\bdot}, \dotsc, \iter{\vw}{t-1,\bdot}; \va \big)  \quad \forall \; i \in [\degree], \; t \in \N,
\end{align}
At the end of $t$ iterations, the optimal lifted OAMP algorithm estimates $\vx_{\star}$ by:
\begin{align*}
\dup{\hat{\vw}_t}{\degree} & = \hnonlin_t^{\star}(\iter{\vw}{1,\bdot}, \dotsc, \iter{\vw}{t,\bdot}; \va).
\end{align*}
\paragraph*{Description of the iterate denoisers and post-processing functions} The iterate denoisers functions $\fnonlin^{\star}_1, \fnonlin^{\star}_2, \dotsc$ and the post-processing functions $\hnonlin_1^\star, \hnonlin_2^{\star}, \dotsb $ were defined recursively: once $\fnonlin^{\star}_{1}, \dotsc, \fnonlin_{t}^\star$ have been specified, let $(\serv{X}_\star, \serv{W}_{1,\bdot}, \dotsc,  \serv{W}_{t,\bdot} ; \serv{A})$ denote the state evolution random variables corresponding to the first $t$ iterations of the resulting lifted OAMP algorithm. Then, the iterate denoiser for step $t+1$ is the DMMSE estimator for the Gaussian channel  $(\serv{X}_\star, \serv{W}_{1,\bdot}, \dotsc,  \serv{W}_{t,\bdot} ; \serv{A})$. \lemref{lem:scalarization-mmse} in \appref{appendix:gauss-channel} provides a general formula for the DMMSE estimator of a multivariate Gaussian channel. An application of this formula yields:
\begin{subequations} \label{eq:optimal-LOAMP-denoisers}    
\begin{align}
\fnonlin_{t+1}^{\star} (w ; \aux) & \explain{def}{=} \dfbdnsr\bigg(\ip{\optlintd{t}}{w} ; \aux \; \big| \;  \effsnrtd{t} \bigg) \quad \forall \; w \; \in \; \R^{t\degree}, \; \aux \; \in \; \R^{\auxdim}
\end{align}
where: \begin{align}
    {\optlintd{t}}\explain{def}{=} \optlin(\serv{X}_\star|\serv{W}_{1,\bdot}, \dotsc,  \serv{W}_{t,\bdot} ; \serv{A}),\quad \effsnrtd{t} \explain{def}{=}\effsnr(\serv{X}_\star|\serv{W}_{1,\bdot}, \dotsc,  \serv{W}_{t,\bdot}; \serv{A})
\end{align} 
denote the optimal linear combination and effective SNR of the channel $(\serv{X}_\star, \serv{W}_{1,\bdot}, \dotsc,  \serv{W}_{t,\bdot} ; \serv{A})$ (as defined in \lemref{lem:scalarization-mmse}).
Moreover, the estimator at the end of $t$ iterations is obtained by applying the MMSE estimator for the channel $(\serv{X}_\star, \serv{W}_{1,\bdot}, \dotsc,  \serv{W}_{t,\bdot} ; \serv{A})$ to the iterates:
\begin{align}
\dup{\iter{\hat{\vw}}{t}}{\degree} & = \hnonlin_{t}^{\star} (\iter{\vw}{\leq t, \bdot} ; \aux), \\
\hnonlin_{t}^{\star} (w ; \aux) & \explain{def}{=} \bdnsr\bigg(\ip{\optlintd{t}}{w} ; \aux \; \big| \;  \effsnrtd{t} \bigg) \quad \forall \; w \; \in \; \R^{t\degree}, \; \aux \; \in \; \R^{\auxdim}
\end{align}
\end{subequations}
\paragraph*{Dynamics of the Optimal Lifted OAMP Algorithm} By \corref{cor:SE-LOAMP}, the joint distribution of $(\serv{X}_{\star}, \serv{W}_{\leq t, \bdot} ; \serv{A})$ is given by: \begin{subequations} \label{eq:optimal-liftedOAMP-SE}
\begin{align} 
&(\serv{X}_{\star}; \serv{A}) \sim \pi, \\&(\serv{W}_{1,\bdot}, \dotsc, \serv{W}_{t,\bdot}) | (\serv{X}_\star;\serv{A}) \nonumber \\ & \hspace{2.2cm} \sim \gauss{\serv{X_\star} \cdot  \alpha_t \otimes \dup{q}{\degree}}{ \alpha_t \alpha_t^\top   \otimes (\dup{Q}{\degree} - \dup{q}{\degree} {\dup{q}{\degree}}^\top) + (\Sigma_t - \alpha_t \alpha^\top_t)  \otimes \dup{\Gamma}{\degree}}. \nonumber 
\end{align}
In the above display, the entries of $\dup{q}{\degree} \in \R^{\degree}$, $\dup{Q}{\degree} \in \R^{\degree \times \degree}$, and $\dup{\Gamma}{\degree} \in \R^{\degree \times \degree}$ are given by:
\begin{align}  \label{eq:optimal-LOAMP-denoisers-c}
\dup{q}{\degree}_i  &= \E_{\serv{\Lambda}_{\nu} \sim \nu}[{\serv{\Lambda}}_{\nu}^i] - \E[{\serv{\Lambda}}^i],  \;  \dup{Q}{\degree}_{ij} 
= \E_{\serv{\Lambda}_{\nu} \sim \nu}\left[ \left({\serv{\Lambda}}_{\nu}^i - \E_{\serv{\Lambda} \sim \mu}[{\serv{\Lambda}}^i] \right)  \left({\serv{\Lambda}}_{\nu}^j - \E_{\serv{\Lambda} \sim \mu}[{\serv{\Lambda}}^j] \right) \right], \\ \dup{\Gamma}{\degree}_{ij} &= \Cov_{\serv{\Lambda} \sim \mu}[\serv{\Lambda}^i ,\serv{\Lambda}^j]. \nonumber
\end{align}
the entries of $\alpha_t \in \R^t$ and $\Sigma_t \in \R^{t \times t}$ are given by:
\begin{align}
(\alpha_t)_{s} \explain{def}{=} \E[\serv{X}_{\star} \cdot {\fnonlin}^\star_{s}(\serv{W}_{<s,\bdot}; \serv{A})], \quad (\Sigma_t)_{s \tau} \explain{def}{=} \E[{\fnonlin}^\star_{s}(\serv{W}_{<s,\bdot}; \serv{A}) \cdot {\fnonlin}^\star_{\tau}(\serv{W}_{<\tau,\bdot}; \serv{A})] \quad \forall \;  s,\tau \; \in \; [t],
\end{align}
and $\otimes$ denotes the Kronecker (or tensor) product for matrices. Finally, the asymptotic MSE of the estimator $\dup{\hat{\vw}_t}{\degree}$ returned after $t$ iterations is given by:
\begin{align} \label{eq:optimal-LOAMP-mse}
\plim_{\dim \rightarrow \infty} \frac{\|\dup{\hat{\vw}_t}{\degree}-\vx_{\star}\|^2}{\dim} & = \E \left[ \left( \serv{X}_{\star} - \hnonlin_t^{\star}(\serv{W}_{\leq t, \bdot}; \serv{A}) \right)^2 \right] = \bmmse(\serv{X}_{\star}| \serv{W}_{\leq t, \bdot}; \serv{A}) \explain{Lem. \ref{lem:scalarization-mmse}}{=}\mmse(\effsnrtd{t}). 
\end{align}
\end{subequations}
The following lemma provides a formula for the effective SNR $\effsnrtd{t} \explain{def}{=}\effsnr(\serv{X}_\star|\serv{W}_{1:t,\bdot}; \serv{A})$ and the optimal linear combination ${\optlintd{t}}\explain{def}{=} \optlin(\serv{X}_\star|\serv{W}_{1,\bdot}, \dotsc,  \serv{W}_{t,\bdot} ; \serv{A})$ for the Gaussian channel $(\serv{X}_{\star}, \serv{W}_{1, \bdot}, \dotsc, \serv{W}_{t, \bdot}; \serv{A})$ at step $t$.  
\begin{lemma} \label{lem:optlin-LOAMP} Let $\effsnrtd{0} \explain{def}{=} 0$ and $\dmmsetd{0} \explain{def}{=} \mmse(0) = \dmmse(0) = \E[\Var[\serv{X}_{\star}| \serv{A}]]$. 
\begin{enumerate}
\item If $\dmmsetd{0} < 1$, then $\effsnrtd{t}$ and $\optlintd{t}$ admit the following recursive characterization:
\begin{align*}
\dmmsetd{t} & = \dmmse(\effsnrtd{t-1}), \\
\effsnrtd{t} & = {\dup{q}{\degree}}^\top \left[\dup{Q}{\degree} + \frac{\dmmsetd{t}}{1-\dmmsetd{t}} \cdot \dup{\Gamma}{\degree}   \right]^\dagger \dup{q}{\degree}, \\
\optlintd{t} & = {(\effsnrtd{t})^{-\frac{1}{2}}} \cdot (1-\dmmsetd{t})^{-1} \cdot  e_t \otimes \left[  \dup{Q}{\degree} + \frac{\dmmsetd{t}}{1-\dmmsetd{t}} \cdot  \dup{\Gamma}{\degree} \right]^{\dagger} \dup{q}{\degree},
\end{align*}
where $e_t = (0, 0, \dotsc, 1)^\top \in \R^t$ denotes the last standard basis vector.  Moreover, for each $t \in \N$, 
\begin{align*}
    \effsnrtd{t} > 0, \quad \dmmsetd{t} < 1 \quad \forall \; t \; \in \; \N.
\end{align*}
\item If $\dmmsetd{0} = 1$, then $\effsnrtd{t} = 0$ for all $t \in \N$.
\end{enumerate}
\end{lemma}
\begin{proof}
See \appref{sec:optlin-LOAMP}.
\end{proof}

\paragraph*{Eliminating Corner Cases} Next, we observe that the proof of \propref{prop:simplification} is quite simple in the corner cases $\mmse(0) = \E[\Var[\serv{X}_{\star} | \serv{A}]] = 1$ and $\mmse(0) = \E[\Var[\serv{X}_{\star} | \serv{A}]] = 0$. 
\begin{description}
\item[Case 1: $\mmse(0) = 1$.] In this case, combining  \eqref{eq:optimal-LOAMP-mse} and \lemref{lem:optlin-LOAMP} (Claim 2), we conclude that for any $t \in \N$,
\begin{align*}
\plim_{\dim \rightarrow \infty} \frac{\|\dup{\hat{\vw}_{t}}{\degree}-\vx_{\star}\|^2}{\dim} & = \mmse(\effsnrtd{t})=  \mmse(0) = 1.
\end{align*}
On the other hand, the optimal OAMP algorithm (recall \eqref{eq:optimal-OAMP} from Section~\ref{sec:OAMP_opt}) returns the estimator $\hat{\vx}_t = \vzero$ for any $t \in \N$. Hence,
\begin{align*}
\plim_{\dim \rightarrow \infty} \frac{\|\hat{\vx}_{t}-\vx_{\star}\|^2}{\dim} & = \plim_{\dim \rightarrow \infty} \frac{\|\vx_{\star}\|^2}{\dim} = \E[\serv{X}_{\star}^2] = 1 = \plim_{\dim \rightarrow \infty} \frac{\|\dup{\hat{\vw}_{t}}{\degree}-\vx_{\star}\|^2}{\dim},
\end{align*}
as claimed. 
\item[Case 2: $\mmse(0) = 0$.] In this situation, we observe that for any $t \in \N$,
\begin{align*}
0 \leq \plim_{\dim \rightarrow \infty} \frac{\|\dup{\hat{\vw}_{t}}{\degree}-\vx_{\star}\|^2}{\dim} & = \mmse(\effsnrtd{t}) \explain{\factref{fact:mmse-monotonic}}{\leq} \mmse(0) = 0 \implies \plim_{\dim \rightarrow \infty} \frac{\|\dup{\hat{\vw}_{t}}{\degree}-\vx_{\star}\|^2}{\dim}  = 0.
\end{align*}
On the other hand, the simplified OAMP algorithm returns the estimator $\hat{\vx}_t = \bdnsr(\va|0)$ for any $t \in \N$. Hence,
\begin{align*}
\plim_{\dim \rightarrow \infty} \frac{\|\hat{\vx}_{t}-\vx_{\star}\|^2}{\dim} &  = \E[(\serv{X}_{\star}-\bdnsr(\serv{A}|0))^2] = \mmse(0) = 0=  \plim_{\dim \rightarrow \infty} \frac{\|\dup{\hat{\vw}_{t}}{\degree}-\vx_{\star}\|^2}{\dim},
\end{align*}
as claimed. 
\end{description}
Hence, in the remainder of the proof, we will assume that:
\begin{align} \label{eq:eliminate-corner}
\mmse(0) = \dmmse(0)  \in (0,1).
\end{align}

\paragraph*{Simplifying the Optimal Lifted OAMP Algorithm}  Recall the estimator computed at the end of $t$ iterations of the optimal lifted OAMP algorithm is:
\begin{align*}
\hat{\vw}_t & = h_t^{\star}(\iter{\vw}{1, \bdot}, \dotsc, \iter{\vw}{t,\bdot}),
\end{align*}
where $h_t^{\star}$ is the MMSE estimator for the channel $(\serv{X}_{\star}, \serv{W}_{1,\bdot}, \dotsc, \serv{W}_{t,\bdot})$:
\begin{align*}
\hnonlin_{t}^{\star} (w ; \aux) & \explain{def}{=} \bdnsr\bigg(\ip{\optlintd{t}}{w} ; \aux| \effsnrtd{t}\bigg) \quad \forall \; w \; \in \; \R^{t\degree}, \; \aux \; \in \; \R^{\auxdim}.
\end{align*}
To compute this estimator, we only need to track the following vector $\dup{\vx_t}{\degree}$, which is a linear combination of the lifted OAMP iterates $\iter{\vw}{\leq t, \bdot}$ with weights given by $\optlintd{t}$:
\begin{align} \label{eq:suff-stat-LOAMP}
\dup{\vx_t}{\degree} &  \explain{def}{=} \sum_{s =1}^t \sum_{i = 1}^\degree (\optlintd{t})_{s,i} \cdot \iter{\vw}{s,i}.
\end{align}
Indeed, the estimator returned by the optimal lifted OAMP algorithm can be computed using $\dup{\vx_t}{\degree}$:
\begin{align*}
\dup{{\hat{\vw}}_{t}}{\degree} & = \bdnsr(\dup{\vx_t}{\degree}; \va | \effsnrtd{t}).
\end{align*}
In fact, $\dup{\vx_t}{\degree}$ can be tracked using a simpler iterative algorithm. To see this, we observe using the formula for $\optlintd{t}$ from \lemref{lem:optlin-LOAMP} (Claim 1) that we have:
\begin{align*}
\dup{\vx_t}{\degree} &  \explain{def}{=}  \sum_{i = 1}^\degree (\optlintd{t})_{t,i} \cdot \iter{\vw}{t,i} \\ &= \frac{1}{(1-\dmmsetd{t})\sqrt{\effsnrtd{t}}}\sum_{i=1}^{\degree} \left\{\left(  \dup{Q}{\degree} + \frac{\dmmsetd{t}}{1-\dmmsetd{t}} \cdot  \dup{\Gamma}{\degree} \right)^{\dagger} \dup{q}{\degree} \right\}_i  \cdot \iter{\vw}{t,i} \\
\end{align*}
Recall that the update equation for $\iter{\vw}{t,i}$:
\begin{align*}
\iter{\vw}{t,i} & = \left( \mY^i - \E_{\serv{\Lambda} \sim \mu}[\serv{\Lambda}^i] \cdot \mI_{\dim} \right) \cdot \fnonlin^{\star}_{t}\big(\iter{\vw}{1,\bdot}, \dotsc, \iter{\vw}{t-1,\bdot}; \va \big)  \quad \forall \; i \in [\degree], \; t \in \N,
\end{align*}
In light of the above equations, we introduce the degree-$\degree$ polynomial matrix denoising function:
\begin{align} \label{eq:matdnsr-D}
\omdd{\degree}(\lambda; \rho) \explain{def}{=} \sum_{i=1}^{\degree}\left\{\left(  \dup{Q}{\degree} + \frac{1}{\rho} \cdot  \dup{\Gamma}{\degree} \right)^{\dagger} \dup{q}{\degree} \right\}_i  \cdot (\lambda^i - \E_{\serv{\Lambda} \sim \mu}[\serv{\Lambda}^i]) \quad \forall \; \lambda \; \in \; \R, \; \rho \in (0,\infty],
\end{align}
as well as the sequence $(\rhotd{t})_{t \in \N}$:
\begin{align} \label{eq:rhotd-def}
\rhotd{t} \explain{def}{=} \frac{1}{\dmmsetd{t}} - 1, \quad \dmmsetd{t} = \dmmse(\effsnrtd{t-1}) \quad \forall \; t \; \in \; \N.
\end{align}
This leads to the following formula for $\dup{\vx_t}{\degree}$:
\begin{align*}
\dup{\vx_t}{\degree} & = \frac{1}{\sqrt{\effsnrtd{t}}} \left( \frac{1}{\rhotd{t}} + 1 \right) \cdot \omdd{\degree}(\mY; \rhotd{t}) \cdot  \fnonlin_{t}^{\star}(\iter{\vw}{1;\bdot}, \dotsc, \iter{\vw}{t-1;\bdot}; \va) \\
& \explain{\eqref{eq:optimal-LOAMP-denoisers},\eqref{eq:suff-stat-LOAMP}}{=} \frac{1}{\sqrt{\effsnrtd{t}}} \left( \frac{1}{\rhotd{t}} + 1 \right) \cdot \omdd{\degree}(\mY; \rhotd{t}) \cdot  \dfbdnsr(\dup{\iter{\vx}{t-1}}{\degree}; \va | \effsnrtd{t-1})
\end{align*}
In summary, we have obtained an OAMP algorithm (in the sense of Definition~\ref{Def:OAMP_main}):
\begin{align} \label{eq:simplified-algo-1}
\dup{\vx_t}{\degree} & = \frac{1}{\sqrt{\effsnrtd{t}}} \left( \frac{1}{\rhotd{t}} + 1 \right) \cdot \omdd{\degree}(\mY; \rhotd{t}) \cdot  \dfbdnsr(\dup{\iter{\vx}{t-1}}{\degree}; \va | \effsnrtd{t-1}) \quad \forall \; t \; \in \; \N,
\end{align}
which can reconstruct the estimator $\dup{\hat{\vw}_{t}}{\degree}$ computed by the optimal degree-$\degree$ lifted OAMP algorithm in \eqref{eq:optimal-LOAMP-general-form}:
\begin{align*}
\dup{\hat{\vw}_{t}}{\degree} & = \bdnsr(\dup{\vx_t}{\degree}; \va | \effsnrtd{t}) \quad \forall \; t \; \in \; \N.
\end{align*}
\paragraph*{Large-degree limit} Notice the similarity between the simplified version of the optimal degree-$\degree$ lifted OAMP algorithm in \eqref{eq:simplified-algo-1} and the optimal OAMP algorithm introduced in \eqref{eq:optimal-OAMP}. Our goal will be to show that the performance of the simplified OAMP algorithm in \eqref{eq:simplified-algo-1} converges to the performance of the optimal OAMP algorithm in \eqref{eq:optimal-OAMP} as $\degree \rightarrow \infty$. To do so, we will rely on the following lemma which shows that the matrix denoiser $\omdd{\degree}$ used in the simplified OAMP algorithm \eqref{eq:simplified-algo-1} can be viewed as the minimizer of a variational problem over degree-$\degree$ polynomials. By studying the same variational problem over $L^2(\mu+\nu)$, we obtain a candidate for the optimal matrix denoiser in the large degree limit, which turns out to be precisely the matrix denoiser \eqref{eq:optimal-OAMP-denoiser} used by the optimal OAMP algorithm in \eqref{eq:optimal-OAMP}.
\begin{lemma} \label{lem:large-degree-lim} Let $\scrL{\degree}$ and $\scrL{\infty}$ denote the following functions on the domain $(0,\infty)$:
\begin{align}
\scrL{\degree}(\rho) &\explain{def}{=} \inf_{\mfunc \in \polynom{\degree}} \E[|\mfunc(\serv{\Lambda}_{\nu})-1|^2] + \frac{1}{\rho} \cdot \E[\mfunc^2(\serv{\Lambda})] \quad \text{ subject to } \quad \E[\mfunc(\serv{\Lambda})] = 0,\label{eq:var-prob-D} \\
\scrL{\infty}(\rho) &\explain{def}{=} \inf_{\mfunc \in L^2(\mu + \nu)} \E[|\mfunc(\serv{\Lambda}_{\nu})-1|^2] + \frac{1}{\rho} \cdot \E[\mfunc^2(\serv{\Lambda})] \quad \text{ subject to } \quad \E[\mfunc(\serv{\Lambda})] = 0, \label{eq:var-prob-lim}
\end{align}
where $\serv{\Lambda} \sim \mu, \serv{\Lambda}_{\nu} \sim \nu$, $\polynom{\degree}$ denotes the set of all polynomial functions on $\R$ with degree at most $\degree$, and $L^2(\mu + \nu)$ denotes the set of all real valued functions on $\R$ which are square integrable with respect to $\mu + \nu$. Then, we have:
\begin{enumerate}
\item For any $\rho \in (0,\infty)$, the function $\omdd{\degree}(\cdot; \rho)$ defined in \eqref{eq:matdnsr-D} is the minimizer of the variational problem in \eqref{eq:var-prob-D} and,
\begin{align*}
\scrL{\degree}(\rho) & = 1 - q^\top_{\degree} \cdot \left(\dup{Q}{\degree} + \frac{1}{\rho} \cdot \dup{\Gamma}{\degree} \right)^{\dagger} \cdot \dup{q}{\degree} = 1 - \E[\omdd{\degree}(\serv{\Lambda}_{\nu}; \rho)].
\end{align*}
\item For any $\rho \in (0,\infty)$, the function: \begin{align*} 
     \omd(\lambda; \rho) & = 1 - \left( \E_{\serv{\Lambda} \sim \mu} \left[ \frac{ \phi(\serv{\Lambda})}{ \phi(\serv{\Lambda}) + \rho} \right] \right)^{-1} \cdot \frac{\phi(\lambda)}{\phi(\lambda)+\rho} \quad \forall \; \lambda \; \in \; \R, \; \rho \; \in \;  (0,\infty),
    \end{align*}
is a minimizer of the variational problem in \eqref{eq:var-prob-lim}.  In the above display, for any $\lambda \in \R$, $\phi(\lambda) \explain{def}{=} {(1- \pi \theta \hlb_{\mu}(\lambda))^2 + \pi^2 \theta^2  \mu^2(\lambda)}$, where $\hlb_{\mu}$ is the Hilbert transform of $\mu$ and $\theta$ is the SNR of the spiked matrix model $\mY = (\theta/\dim) \cdot \vx_{\star} \vx_{\star}^\top + \mW$.
\item For any $\degree \in \N \cup \{\infty\}$, the function  $\scrL{\degree}$ maps the interval $(0,\infty)$ to $(0,1]$. 
\item Let $(\rho_{\degree})_{\degree \in \N}$ be any non-decreasing sequence in $(0,\infty)$ which converges to $\rho \in (0,\infty)$ as $\degree \rightarrow \infty$. Then,  $\scrL{\degree}(\rho_{\degree}) \downarrow \scrL{\infty}(\rho)$ as $\degree \rightarrow \infty$. Moreover, $\omdd{\degree}(\cdot \; ; \rho_{\degree})$ converges to $\omd(\cdot \; ; \rho)$ in $L^2(\mu + \nu)$ as $\degree \rightarrow \infty$.
\item For any $\rho \in (0,\infty)$, the function $\omd( \cdot; \rho)$ satisfies $\E[\omd(\serv{\Lambda}_{\nu}; \rho)] = 1 - \scrL{\infty}(\rho)$.
\end{enumerate}
\end{lemma}
\begin{proof}
The proof of this lemma is provided in \appref{sec:large-degree-lim}.
\end{proof}
\paragraph*{Proof of \propref{prop:simplification}} We now have all the ingredients to complete the proof of \propref{prop:simplification}. Recall that the optimal estimator that can be computed using $t$ iterations of a degree-$\degree$ lifted OAMP algorithm is:
\begin{align*}
\dup{\hat{\vw}_{t}}{\degree} & = \bdnsr(\dup{\iter{\vx}{t}}{\degree}; \va | \effsnrtd{t}),
\end{align*}
where the iterates $\iter{\vx}{t}$ are computed using the update rule:
\begin{subequations} \label{eq:simplified-optimal-LOAMP}
\begin{align} \label{eq:OAMP-degree-D}
\dup{\vx_t}{\degree} & = \frac{1}{\sqrt{\effsnrtd{t}}} \left( \frac{1}{\rhotd{t}} + 1 \right) \cdot \omdd{\degree}(\mY; \rhotd{t}) \cdot  \dfbdnsr(\dup{\iter{\vx}{t-1}}{\degree}; \va | \effsnrtd{t-1}) \quad \forall \; t \; \in \; \N.
\end{align}
Furthermore, the limiting mean-squared error of the estimator $\dup{\hat{\vw}_t}{\degree}$ is given by:
\begin{align} \label{eq:MMSE-optimal-LOAMP}
\plim_{\dim \rightarrow \infty} \frac{\|\dup{\hat{\vw}_t}{\degree} - \vx_{\star}\|^2}{\dim} & = \mmse(\effsnrtd{t}).
\end{align}
In the above equations, the parameters $\{\effsnrtd{t}\}_{t \in \N}, \{\dmmsetd{t}\}_{t \in \N}, \{\rhotd{t}\}_{t \in \N}$ are updated according to the following recursion from \lemref{lem:optlin-LOAMP} (also recall the definition of $\rhotd{t}$ from \eqref{eq:rhotd-def}) initialized with $\effsnrtd{0} = 0$:
\begin{align} \label{eq:snr-dmmse-recursion-degree}
\dmmsetd{t} &= \dmmse(\effsnrtd{t-1}), \  \rhotd{t} = \frac{1}{\dmmsetd{t}} -1 , \\  \effsnrtd{t} & = (\dup{q}{\degree})^\top \left[\dup{Q}{\degree} + \frac{1}{\rhotd{t}}  \dup{\Gamma}{\degree}   \right]^\dagger \dup{q}{\degree} \explain{(a)}{=} 1- \scrL{\degree}(\rhotd{t}), \nonumber
\end{align}
\end{subequations}
where the equality marked (a) follows from \lemref{lem:large-degree-lim} (Claim (2)). To prove \propref{prop:simplification}, we need to show that in the large degree limit ($\degree \rightarrow \infty$), the performance of the estimator $\dup{\hat{\vw}_t}{\degree}$ is identical to the performance of the estimator $\hat{\vx}_{t}$:
\begin{align*}
\hat{\vx}_{t} & \explain{def}{=} \bdnsr(\iter{\vx}{t}; \va | \effsnrt{t}{}) \quad \forall \; t \; \in \; \N,
\end{align*}
returned by the optimal OAMP algorithm from \eqref{eq:optimal-OAMP}:
\begin{align} \label{eq:OAMP-degree-infinity}
{\vx}_{t} & = \frac{1}{\sqrt{{\omega}_{t}}}\left( 1+  \frac{1}{{\rho}_{t}}  \right) \cdot \omd(\mY; {\rho}_{t}) \cdot \dfbdnsr({\vx}_{t-1}; \va | {\omega}_{t-1}) \quad \forall \; t \; \in \; \N.
\end{align}
In the above equation, the parameters $\{\effsnrt{t}{}\}_{t \in \N}, \{\dmmset{t}\}_{t\in \N}, \{\rhot{t}{}\}_{t \in \N}$ are updated according to the following recursion initialized with $\effsnrt{0} = 0$: 
\begin{align}\label{eq:snr-dmmse-recursion-lim}
\dmmset{t}  = \dmmse(\effsnrt{t-1}), \quad \rhot{t} = \frac{1}{\dmmset{t}}-1, \quad \effsnrt{t}{} & =  1- \scrL{\infty}(\rhot{t}{}).
\end{align}
Let $(\serv{X}_{\star}, (\serv{X}_{t})_{t\in \N}; \serv{A})$ denote the state evolution random variables corresponding to the above algorithm. We will show by induction that for each $t \in \N$,
\begin{enumerate}
    \item for any $D \in \N$, the sequences 
the sequences $\{\effsnrtd{t}\}_{t \in \N}, \{\dmmsetd{t}\}_{t \in \N}, \{\rhotd{t}\}_{t \in \N}$ generated by \eqref{eq:snr-dmmse-recursion-degree} satisfy: \begin{align} \label{eq:finit-D-recursion-validity}
\effsnrtd{t} \in [0,1), \quad \dmmsetd{t} \in (0,1), \quad \rhotd{t} \in (0,\infty) \quad \forall \; t \; \in \; \N.
\end{align}
In particular, the update equation for $\rhotd{t}$ in \eqref{eq:snr-dmmse-recursion-degree} is well-defined (avoids division by zero). 
\item $\effsnrtd{t} \uparrow \effsnrt{t}{} \in [0,1), \; \dmmsetd {t}\downarrow \dmmset{t} \in (0,1)$, and $\rhotd{t} \uparrow \rhot{t}{}  \in (0,\infty)$ as $\degree \rightarrow \infty$.
\end{enumerate}
These claims immediately imply \propref{prop:simplification}:
\begin{align*}
\plim_{\dim \rightarrow \infty} \frac{\|\hat{\vx}_{t} - \vx_{\star}\|^2}{\dim} & \explain{Prop. \ref{prop:optimal-OAMP-SE}}{=} \mmse(\effsnrt{t}{}) \explain{(a)}{=} \lim_{\degree \rightarrow \infty} \mmse(\effsnrtd{t}) \explain{\eqref{eq:MMSE-optimal-LOAMP}}{=} \lim_{\degree \rightarrow \infty}\plim_{\dim \rightarrow \infty} \frac{\|\dup{\hat{\vw}_t}{\degree} - \vx_{\star}\|^2}{\dim},
\end{align*}
as claimed. In the above display step (a) follows from the claim $\effsnrtd{t} \uparrow \effsnrt{t}{}$ as $\degree \rightarrow \infty$. 
We now prove the three claims made above. As the induction hypothesis, we assume the claims hold for some $t \in \N$ and verify that they continue to hold at step $t+1$. 
\paragraph*{Proof of Claim (1)} We focus on proving \eqref{eq:finit-D-recursion-validity} holds at iteration $t+1$, assuming it holds at iteration $t$. 
The induction hypothesis $\effsnrtd{t} \geq 0$ and the monotonicity of $\dmmse(\cdot)$ (see \lemref{lem:dmmse-monotonicity} in \appref{appendix:gauss-channel}) imply that $\dmmsetd{t+1} = \dmmse(\effsnrtd{t}) \leq \dmmse(0) = \mmse(0) < 1$ (recall \eqref{eq:eliminate-corner}). On the other hand, since $\effsnrtd{t} < 1$, by the strict monotonicity of $\mmse$ (see \factref{fact:mmse-monotonic} in \appref{appendix:gauss-channel}), we conclude that $\dmmsetd{t+1} = \dmmse(\effsnrtd{t}) \geq \mmse(\effsnrtd{t})  > \mmse(1) = 0$. Hence, $\dmmsetd{t+1} \in (0,1)$. Since $\rhotd{t+1} = (\dmmsetd{t+1})^{-1} - 1$, we conclude that $\rhotd{t+1} \in (0,\infty)$. Finally, since $\effsnrt{t+1} = 1 - \scrL{\degree}(\rhotd{t+1})$ and \lemref{lem:large-degree-lim} guarantees that $\scrL{\degree}: (0,\infty) \mapsto (0,1]$, we conclude that $\effsnrtd{t} \in [0,1)$, as claimed in \eqref{eq:finit-D-recursion-validity}. This proves the first claim.
\paragraph*{Proof of Claim (2)} We recall from \eqref{eq:snr-dmmse-recursion-degree} that $ \dmmsetd{t+1} = \dmmse(\effsnrtd{t})$. By the induction hypothesis $\effsnrtd{t} \uparrow \effsnrt{t}{} \in [0,1)$. Since $\dmmse(\cdot)$ is a non-increasing {continuous} function (see \lemref{lem:dmmse-monotonicity} in \appref{appendix:gauss-channel}), we conclude that $\dmmsetd{t+1} \downarrow \dmmse(\effsnrt{t}{}) \explain{\eqref{eq:snr-dmmse-recursion-lim}}{=} \dmmset{t+1}{} \in (0,1)$ and $\rhotd{t+1} \explain{\eqref{eq:snr-dmmse-recursion-degree}}{=} (\dmmsetd{t+1})^{-1} - 1 \uparrow (\dmmset{t+1})^{-1} - 1 \explain{\eqref{eq:snr-dmmse-recursion-lim}}{=} \rhot{t+1} \in (0,\infty)$. Finally, recalling the update rule for $\effsnrtd{t+1}$ from \eqref{eq:snr-dmmse-recursion-degree} and using \lemref{lem:large-degree-lim} (Claim 4), we conclude that:  $\effsnrtd{t+1} = 1 - \scrL{\degree}(\rhotd{t+1}) \uparrow 1- \scrL{\infty}(\rhot{t+1}) \explain{\eqref{eq:snr-dmmse-recursion-lim}}{=} \effsnrt{t+1} \in [0,1)$. This proves the second claim. 

This concludes the proof of \propref{prop:simplification}.

\subsubsection{Proof of \lemref{lem:optlin-LOAMP}} \label{sec:optlin-LOAMP}
The proof of \lemref{lem:optlin-LOAMP} relies on the following intermediate result.

\begin{claim} \label{claim:banded-cov} For any $s,t \in \N$ with $s \leq t$:
\begin{align*}
\E[\fnonlin_s^{\star}(\serv{W}_{1,\bdot}, \dotsc, \serv{W}_{s-1,\bdot}; \serv{A}) \cdot \fnonlin_t^{\star}(\serv{W}_{1,\bdot}, \dotsc, \serv{W}_{t-1,\bdot}; \serv{A})] & = \E[\serv{X}_{\star} \cdot \fnonlin_s^{\star}(\serv{W}_{1,\bdot}, \dotsc, \serv{W}_{s-1,\bdot}; \serv{A})].
\end{align*}
\end{claim}
We will first prove \lemref{lem:optlin-LOAMP} assuming this claim. The proof of \clmref{claim:banded-cov} is provided at the end of this section. 
\begin{proof}[Proof of \lemref{lem:optlin-LOAMP}] Recall that for each $t \in \N$,
\begin{align*}
{\optlintd{t}} &\explain{def}{=} \optlin(\serv{X}_\star|\serv{W}_{1,\bdot}, \dotsc,  \serv{W}_{t,\bdot} ; \serv{A}), \quad 
\effsnrtd{t} \explain{def}{=}\effsnr(\serv{X}_\star|\serv{W}_{1,\bdot}, \dotsc,  \serv{W}_{t,\bdot}; \serv{A}),
\end{align*}
where $(\serv{X}_{\star}, (\serv{W}_{t,\bdot})_{t \in \N}; \serv{A})$ are the state evolution random variables associated with the optimal lifted OAMP algorithm. In addition, we will find it helpful to introduce the quantity:
\begin{align*}
\dmmsetd{t} &\explain{def}{=} \min\bigg\{ \E\left[\{\serv{X}_{\star} - g(\serv{W}_{< t,\bdot}; \serv{A})\}^2\right] : g \in \mathrm{Span}({\fnonlin}_1^{\star}, \dotsc, {\fnonlin}^{\star}_t) \bigg\}. 
\end{align*}
Since $\mathrm{Span}({\fnonlin}_1^{\star}) \subset \mathrm{Span}({\fnonlin}_1^{\star}, {\fnonlin}^{\star}_2) \subset \dotsb \mathrm{Span}({\fnonlin}_1^{\star}, \dotsc, {\fnonlin}^{\star}_t) \subset \dotsb$, $(\dmmsetd{t})_{t \in \N}$ is a non-increasing sequence. 
\paragraph*{Case 1: $\dmmsetd{0} < 1$} Notice that since $(\dmmsetd{t})_{t \in \N}$ is a non-increasing sequence and $\dmmsetd{0} < 1$, we conclude that $\dmmsetd{t}< 1$ for all $t \in \N$, which is one of the claims pertaining to this case made in the lemma. To  prove the other claims, we will show by induction on $t$ that:
\begin{subequations} \label{eq:effsnr-ind-hypo}
\begin{align}
    \dmmsetd{t} & = \dmmse(\effsnrtd{t-1}), \\
\effsnrtd{t} & = q^\top_{\degree} \left[\dup{Q}{\degree} + \frac{\dmmsetd{t}}{1-\dmmsetd{t}} \cdot \dup{\Gamma}{\degree}   \right]^\dagger \dup{q}{\degree}, \\
\effsnrtd{t} & > 0, \\
\optlintd{t} & = {(\effsnrtd{t})^{-\frac{1}{2}}} \cdot (1-\dmmsetd{t})^{-1} \cdot  e_1 \otimes \left[  \cdot \dup{Q}{\degree} + \frac{\dmmsetd{t}}{1-\dmmsetd{t}} \cdot  \dup{\Gamma}{\degree} \right]^{\dagger} \dup{q}{\degree}.
\end{align}
\end{subequations}
As our induction hypothesis we assume that \eqref{eq:effsnr-ind-hypo} holds, and show that \eqref{eq:effsnr-ind-hypo} also holds for $t+1$. Indeed, since $\fnonlin_{t+1}^{\star}$ is the DMMSE estimator for the Gaussian channel $(\serv{X}_{\star}, \serv{W}_{\leq t, \bdot}; \serv{A})$, by \lemref{lem:greedy-lemma},
\begin{align*}
\dmmsetd{t+1} &\explain{def}{=} \min\bigg\{ \E\left[\{\serv{X}_{\star} - g(\serv{W}_{\leq t,\bdot}; \serv{A})\}^2\right] : g \in \mathrm{Span}({\fnonlin}_1^{\star}, \dotsc, {\fnonlin}^{\star}_{t+1}) \bigg\} = \dmmse(\effsnrtd{t}).
\end{align*}
By appealing to \lemref{lemma:effsnr-LOAMP}, we conclude that:
\begin{align*}
\effsnrtd{t+1} & = q^\top_{\degree} \left[\dup{Q}{\degree} + \frac{\dmmsetd{t+1}}{1-\dmmsetd{t+1}} \cdot \dup{\Gamma}{\degree}   \right]^\dagger \dup{q}{\degree} \explain{}{>} 0  \\
\optlintd{t+1} & = {(\effsnrtd{t+1})^{-\frac{1}{2}}} \cdot (1-\dmmsetd{t+1})^{-1} \cdot  (\Sigma_{t+1}^\dagger \alpha_{t+1}) \otimes \left[  \dup{Q}{\degree} + \frac{\dmmsetd{t+1}}{1-\dmmsetd{t+1}} \cdot  \dup{\Gamma}{\degree} \right]^{\dagger} \dup{q}{\degree}.
\end{align*}
From \clmref{claim:banded-cov}, we know that:
\begin{align*}
(\alpha_{t+1})_s & = (\Sigma_{t+1})_{s,t+1} \quad \forall \; s \in \; [t+1]\implies \alpha_{t+1} = \Sigma_{t+1} e_{t+1} \implies \Sigma_{t+1}^\dagger \alpha_{t+1} = e_{t+1}.
\end{align*}
Hence,
\begin{align*}
\optlintd{t+1} & = {(\effsnrtd{t+1})^{-\frac{1}{2}}} \cdot (1-\dmmsetd{t+1})^{-1} \cdot  e_{t+1} \otimes \left[  \dup{Q}{\degree} + \frac{\dmmsetd{t+1}}{1-\dmmsetd{t+1}} \cdot  \dup{\Gamma}{\degree} \right]^{\dagger} \dup{q}{\degree},
\end{align*}
as desired. This proves the claim of the lemma when $\dmmsetd{0}<1$. 
\paragraph*{Case 2: $\dmmsetd{0} = 1$} In this situation, we will show by induction that for any $t\in \N$,
\begin{align*}
\dmmsetd{t} & = 1, \quad \effsnrtd{t}  = 0.
\end{align*}
We assume that the above claim holds at step $t$ as the induction hypothesis. As in the previous case, by \lemref{lem:greedy-lemma}, we have:
\begin{align*}
\dmmsetd{t+1} & = \dmmse(\effsnrtd{t}) = \dmmse(0) = \dmmsetd{0} = 1. 
\end{align*}
Appealing to \lemref{lemma:effsnr-LOAMP}, we conclude that $\effsnrtd{t+1}  = 0$, as desired.This proves the claim of the lemma.
\end{proof}
We now present the proof of \clmref{claim:banded-cov}.
\begin{proof}[Proof of \clmref{claim:banded-cov}]
By the Tower property, we have:
\begin{align} \label{eq:banded-tower}
\E[\serv{X}_{\star} \cdot \fnonlin_s^{\star}(\serv{W}_{1,\bdot}, \dotsc, \serv{W}_{s-1,\bdot}; \serv{A})] & = \E[\E[\serv{X}_{\star} | \serv{W}_{<t,\bdot},\serv{A}]\fnonlin_s^{\star}(\serv{W}_{1,\bdot}, \dotsc, \serv{W}_{s-1,\bdot}; \serv{A})].
\end{align}
Since $\E[\serv{X}_{\star} | \serv{W}_{<t,\bdot},\serv{A}]$ is the MMSE estimator for the channel $(\serv{X}_{\star}, \serv{W}_{<t,\bdot};\serv{A})$. Using the formula for the MMSE estimator for a general multivariate Gaussian channel (see \lemref{lem:scalarization-mmse} in \appref{appendix:gauss-channel}), we obtain:
\begin{align} \label{eq:banded-combine-1}
\E[\serv{X}_{\star} | \serv{W}_{<t,\bdot},\serv{A}] & = \bdnsr\left(\ip{\serv{W}_{<t,\bdot}}{\optlin}; \serv{A} |\effsnr \right),
\end{align}
where we use the shorthand notations $\effsnr$ and $\optlin$ denote the effective SNR and optimal linear combination\footnote{We suppress the dependence of these quantities on the iteration number $t$ and the degree $\degree$ of the lifted OAMP for notational convenience since it does not play a role in the proof.} for the Gaussian channel  $(\serv{X}_\star, \serv{W}_{1,\bdot}, \dotsc,  \serv{W}_{t-1,\bdot} ; \serv{A})$. Similarly, \lemref{lem:scalarization-mmse} gives the following formula for $\fnonlin_t^{\star}$, the DMMSE estimator for this channel:
\begin{align} \label{eq:banded-combine-2}
\fnonlin_{t}^{\star}(\serv{W}_{<t,\bdot};\serv{A})& = \dfbdnsr\left(\ip{\serv{W}_{<t,\bdot}}{\optlin}; \serv{A} |\effsnr \right).
\end{align}
Recall the formula for $\dfbdnsr$ (the DMMSE estimator in a scalar Gaussian channel) provided in Definition~\ref{def:gauss-channel} (cf. \eqref{eq:dmmse-scalar}):
\begin{align} \label{eq:banded-combine-3}
\dfbdnsr (x; \aux| \omega) &\explain{def}{=}  \begin{cases} 
 \left(1 - \sqrt{\omega}  \cdot \beta(\omega)   \right)^{-1} \cdot \left(\bdnsr(x;a | \omega) -\beta(\omega) \cdot x \right) & : \omega < 1 \\ \bdnsr(x; \aux| \omega) &: \omega = 1 \end{cases},
\end{align}
where:
\begin{align*}
\beta(\omega) \explain{def}{=} \frac{1}{\sqrt{1-\omega}} \cdot \E[\serv{Z}\bdnsr(\sqrt{\omega} \serv{X}_{\star} + \sqrt{1-\omega} \serv{Z};\serv{A} | \omega)] \quad \text{where } (\serv{X}_{\star}; \serv{A}) \sim \pi, \; \serv{Z} | \serv{X}_{\star}; \serv{A} \sim \gauss{0}{1}.
\end{align*}
Combining \eqref{eq:banded-combine-1}, \eqref{eq:banded-combine-2}, and \eqref{eq:banded-combine-3} yields the following formula relating the MMSE and DMMSE estimators for the Gaussian channel $(\serv{X}_\star, \serv{W}_{1,\bdot}, \dotsc,  \serv{W}_{t,\bdot} ; \serv{A})$:
\begin{align*}
\E[\serv{X}_{\star} | \serv{W}_{<t,\bdot},\serv{A}] & = \begin{cases} \left(1-\sqrt{\effsnr} \cdot \beta\left(\effsnr\right)\right) \cdot \fnonlin_t^{\star}(\serv{W}_{<t,\bdot};\serv{A}) +  \beta\left(\effsnr\right)\cdot \ip{\optlin}{\serv{W}_{<t,\bdot}} & : \effsnr < 1 \\ \fnonlin_t^{\star}(\serv{W}_{<t,\bdot};\serv{A}) & :  \effsnr = 1. \end{cases}
\end{align*}
We substitute the above expression in \eqref{eq:banded-tower}. If $\effsnr = 1$, the claim of the lemma is immediate. If $\effsnr<1$, we have:
\begin{eqnarray*}
\E[\serv{X}_{\star} \cdot \fnonlin_s^{\star}(\serv{W}_{<s,\bdot}; \serv{A})]  &= &\left(1-\sqrt{\effsnr} \cdot \beta\left(\effsnr\right)\right) \cdot \E[\fnonlin_s^{\star}(\serv{W}_{<s,\bdot}; \serv{A}) \cdot  \fnonlin_t^{\star}(\serv{W}_{<t,\bdot};\serv{A})] \\
&&+  \beta\left(\effsnr\right)\cdot \E[\ip{\optlin}{\serv{W}_{<t,\bdot}} \cdot \fnonlin_s^{\star}(\serv{W}_{<s,\bdot}; \serv{A})]. 
\end{eqnarray*}
Let $\serv{Z}_{1,\bdot}, \dotsc, \serv{Z}_{t-1,\bdot}$ denote the Gaussian noise random variables in the Gaussian channel $(\serv{X}_\star, \serv{W}_{1,\bdot}, \dotsc,  \serv{W}_{t,\bdot} ; \serv{A})$. Since $(\serv{X}_{\star}, \ip{\optlin}{\serv{W}_{<t,\bdot}}; \serv{A})$ forms a scalar Gaussian channel with SNR $\effsnr$:
\begin{align*}
\ip{\optlin}{\serv{W}_{<t,\bdot}}  & = \sqrt{\effsnr}  \cdot \serv{X}_{\star} + \sqrt{1-\effsnr} \cdot \serv{Z}_{\mathrm{eff}}, \quad \serv{Z}_{\mathrm{eff}} \explain{def}{=} \frac{\ip{\optlin}{\serv{Z}_{<t,\bdot}}}{\sqrt{1-\effsnr}}.
\end{align*}
Hence,
\begin{eqnarray*}
\E[\serv{X}_{\star} \cdot \fnonlin_s^{\star}(\serv{W}_{<s,\bdot}; \serv{A})] &=& (1-\sqrt{\effsnr} \cdot \beta(\effsnr)) \cdot \E[\fnonlin_s^{\star}(\serv{W}_{<s,\bdot}; \serv{A}) \cdot  \fnonlin_t^{\star}(\serv{W}_{<t,\bdot};\serv{A})] \nonumber\\ 
& & +  \beta(\effsnr)\cdot \E[(\sqrt{\effsnr}  \cdot \serv{X}_{\star} + \sqrt{1-\effsnr} \cdot \serv{Z}_{\mathrm{eff}}) \cdot \fnonlin_s^{\star}(\serv{W}_{<s,\bdot}; \serv{A})]  \nonumber\\
 &=& (1-\sqrt{\effsnr} \cdot \beta(\effsnr)) \cdot \E[\fnonlin_s^{\star}(\serv{W}_{<s,\bdot}; \serv{A}) \cdot  \fnonlin_t^{\star}(\serv{W}_{<t,\bdot};\serv{A})] \\
 &&+  \sqrt{\effsnr} \cdot \beta(\effsnr)\cdot \E[\serv{X}_{\star} \cdot \fnonlin_s^{\star}(\serv{W}_{<s,\bdot}; \serv{A})]. \label{eq:banded-penultimate}
\end{eqnarray*}
The final equality in the previous display follows by observing that:
\begin{align*}
\E[\serv{Z}_{\mathrm{eff}} \cdot \fnonlin_s^{\star}(\serv{W}_{<s,\bdot}; \serv{A})] &= \E[\E[\serv{Z}_{\mathrm{eff}}| \serv{X}_{\star}, \serv{Z}_{<s,\bdot}, \serv{A}] \cdot \fnonlin_s^{\star}(\serv{W}_{<s,\bdot}; \serv{A})] \\
&= \E[\E[\serv{Z}_{\mathrm{eff}}|\serv{Z}_{<s,\bdot}] \cdot \fnonlin_s^{\star}(\serv{W}_{<s,\bdot}; \serv{A})] = 0.
\end{align*}
This is because $\E[\E[\serv{Z}_{\mathrm{eff}}|\serv{Z}_{<s,\bdot}]]$ is a linear combination of $\serv{Z}_{<s,\bdot}$ (since $(\serv{Z}_{<s,\bdot}, \serv{Z}_{\mathrm{eff}})$ are jointly Gaussian) and the DMMSE estimator $\fnonlin_{s}^{\star}$ is uncorrelated with the channel noise $\serv{Z}_{<s,\bdot}$. Rearranging \eqref{eq:banded-penultimate} gives us:
\begin{align*}
\E[\serv{X}_{\star} \cdot \fnonlin_s^{\star}(\serv{W}_{<s,\bdot}; \serv{A})] &=\E[\fnonlin_s^{\star}(\serv{W}_{<s,\bdot}; \serv{A}) \cdot  \fnonlin_t^{\star}(\serv{W}_{<t,\bdot};\serv{A})],
\end{align*}
which is precisely the claim set out to prove. 
\end{proof}

\subsubsection{Proof of \lemref{lem:large-degree-lim}} \label{sec:large-degree-lim}
\begin{proof}[Proof of \lemref{lem:large-degree-lim}]
We consider each of the claims made in the lemma. 
\paragraph*{Proof of Claim (1)} To solve the variational problem:
\begin{align*}
\scrL{\degree}(\rho) &\explain{def}{=} \inf_{\mfunc \in \polynom{\degree}} \E[|\mfunc(\serv{\Lambda}_{\nu})-1|^2] + \frac{1}{\rho} \cdot \E[\mfunc^2(\serv{\Lambda})] \quad \text{ subject to } \quad \E[\mfunc(\serv{\Lambda})] = 0,
\end{align*}
we parameterize any polynomial $f$ of degree at most $\degree$ which satisfies $\E[\mfunc(\serv{\Lambda})] = 0$ as:
\begin{align*}
\mfunc(\lambda) & = \sum_{i=0}^{\degree} v_i \cdot (\lambda^i-\E_{\serv{\Lambda} \sim \mu}[\serv{\Lambda}^i)]) \quad \forall \; \lambda \in \R.
\end{align*}
Hence,
\begin{align} \label{eq:scrL-D-recall}
\scrL{\degree}(\rho) &= \min_{v \in \R^{\degree}}  \E\left| \sum_{i=1}^{\degree} v_i \cdot (\serv{\Lambda}^i_{\nu} - \E_{\serv{\Lambda} \sim \mu}[\serv{\Lambda}^i)]) - 1  \right|^2 + \frac{1}{\rho} \cdot \E\left|\sum_{i=1}^{\degree} v_i \cdot (\serv{\Lambda}^i - \E_{\serv{\Lambda} \sim \mu}[\serv{\Lambda}^i)])\right|^2.
\end{align}
Expanding the square norms and recalling that the entries of $\dup{q}{\degree} \in \R^{\degree}$, $\dup{Q}{\degree} \in \R^{\degree \times \degree}$, and $\dup{\Gamma}{\degree} \in \R^{\degree \times \degree}$ are given by:
\begin{align} 
\dup{q}{\degree}_i & = \E[{\serv{\Lambda}}_{\nu}^i] - \E[{\serv{\Lambda}}^i],  \\
\dup{Q}{\degree}_{ij} &= \E\left[ \left({\serv{\Lambda}}_{\nu}^i - \E[{\serv{\Lambda}}^i] \right) \cdot \left({\serv{\Lambda}}_{\nu}^j - \E[{\serv{\Lambda}}^j] \right) \right],\\\dup{\Gamma}{\degree}_{ij} &= \Cov[\serv{\Lambda}^i ,\serv{\Lambda}^j] \quad \text{ where } \serv{\Lambda} \sim \mu, \; \serv{\Lambda}_{\nu} \sim \nu,
\end{align}
we obtain the following formula for $\scrL{\degree}(\rho)$:
\begin{align*}
\scrL{\degree}(\rho) &= \min_{v \in \R^{\degree}} 1 +  v^\top \dup{Q}{\degree} v  - 2 \ip{\dup{q}{\degree}}{v} +  \frac{ v^\top \dup{\Gamma}{\degree} v }{\rho}.
\end{align*}
The optimizer of the above quadratic form is:
\begin{align*}
v & = \left(\dup{Q}{\degree} + \frac{1}{\rho} \cdot  \dup{\Gamma}{\degree} \right)^{\dagger} \cdot \dup{q}{\degree}.
\end{align*}
Hence, a minimizer of \eqref{eq:scrL-D-recall} is:
\begin{align*}
\mfunc(\lambda) & = \sum_{i=1}^{\degree} \left\{ \left(\dup{Q}{\degree} + \frac{1}{\rho} \cdot \dup{\Gamma}{\degree} \right)^{\dagger} \cdot \dup{q}{\degree} \right\}_i \cdot (\lambda^i - \E_{\serv{\Lambda}\sim \mu}[\serv{\Lambda}^i]) \explain{\eqref{eq:matdnsr-D}}{=} \omdd{\degree}(\lambda; \rho) \quad \forall \; \lambda \in \R,
\end{align*}
and:
\begin{align*}
\scrL{\degree}(\rho) & = 1 -   {\dup{q}{\degree}}^\top \left(\dup{Q}{\degree} + \frac{1}{\rho} \cdot  \Gamma_{\degree} \right)^{\dagger}  \dup{q}{\degree} = 1 - \ip{\dup{q}{\degree}}{v} = 1 - \E_{\serv{\Lambda}_{\nu} \sim \nu}[\omdd{\degree}(\serv{\Lambda}_{\nu}; \rho)],
\end{align*}
as claimed. 
\paragraph*{Proof of Claim (2)} Our goal is to solve the optimization problem:
\begin{align} \label{eq:scrL-limit-repara}
 \scrL{\infty}(\rho) &\explain{def}{=} \inf_{\mfunc \in L^2(\mu + \nu)} \E[|\mfunc(\serv{\Lambda}_{\nu})-1|^2] + \frac{1}{\rho} \cdot \E[\mfunc^2(\serv{\Lambda})] \quad \text{ subject to } \quad \E[\mfunc(\serv{\Lambda})] = 0.
\end{align} This is exactly the problem considered in \eqref{Eqn:local_opt_Psi_MSE}, where we showed that the minimizer of \eqref{eq:scrL-limit-repara} is:
\begin{align} \label{eq:opt-denoiser-recall}
\mfunc(\lambda) & =  1 - \left( \E_{\serv{\Lambda} \sim \mu} \left[ \frac{ \phi(\serv{\Lambda})}{ \phi(\serv{\Lambda}) + \rho} \right] \right)^{-1} \cdot \frac{\phi(\lambda)}{\phi(\lambda)+\rho} \explain{def}{=} \omd(\lambda; \rho) \quad \forall \; \lambda \; \in \; \R, 
\end{align}
as claimed.
\paragraph*{Proof of Claim (3)} We need to show that for any $D \in \N \cup\{\infty\}$, the function $\scrL{\degree}$ maps the interval $(0,\infty)$ to $(0,1]$. We will show this claim for $\scrL{\infty}$ and the exact same argument works for $\scrL{\degree}$ for any $\degree \in \N$. The definition of $\scrL{\infty}(\rho)$ (cf. \eqref{eq:scrL-limit-repara}) implies that $\scrL{\infty}(\rho) \geq 0$. Taking $\mfunc = 0$ in \eqref{eq:scrL-limit-repara} shows that $\scrL{\infty}(\rho) \leq 1$. To prove Claim (3), we need to verify that if $\rho \in  (0,\infty)$ then, $\scrL{\infty}(\rho) > 0$. We prove this by contradiction. If $\scrL{\infty}(\rho) = 0$, then the function $\omd(\cdot ; \rho)$, which minimizes the objective in \eqref{eq:scrL-limit-repara} satisfies $\E[|\omd(\serv{\Lambda}_{\nu}; \rho)-1|^2]  + \rho^{-1} \cdot \E[|\omd(\serv{\Lambda}; \rho)|^2] = 0$. This means that:
\begin{align*}
\nu(\{\lambda \in \R: \omd(\lambda; \rho)  = 1\}) = 1, \quad \mu(\{\lambda \in \R: \omd(\lambda; \rho)  = 0\}) = 1.
\end{align*}
In particular $\mu,\nu$ are mutually singular. However, this contradicts Lemma~\ref{lem:RMT} which tells us that $\nu_{\parallel}$, the absolutely continuous part of $\nu$ with respect to the Lebesgue measure has density $\frac{\diff \nu_{\parallel}}{\diff \lambda}(\lambda) = \frac{\mu(\lambda)}{\phi(\lambda)}$, and hence $\nu_{\parallel} \ll \mu$. 

\paragraph*{Proof of Claim (4)} Consider a non-decreasing sequence $(\rho_{\degree})_{\degree \in \N}$ which converges to $\rho \in (0,\infty)$ as $\degree \rightarrow \infty$. We begin by observing that $\scrL{\degree}(\rho_{\degree})$ is a non-increasing sequence:
\begin{align*}
\scrL{\degree+1}(\rho_{\degree+1}) &\explain{\eqref{eq:scrL-D-recall}}{=}   \inf_{\mfunc \in \polynom{\degree+1}} \E[|\mfunc(\serv{\Lambda}_{\nu})-1|^2] + \frac{1}{\rho_{\degree+1}} \cdot \E[\mfunc^2(\serv{\Lambda})] \  \text{ subject to } \  \E[\mfunc(\serv{\Lambda})] = 0 \\
& \explain{(a)}{\leq}\inf_{\mfunc \in \polynom{\degree+1}} \E[|\mfunc(\serv{\Lambda}_{\nu})-1|^2] + \frac{1}{\rho_{\degree}} \cdot \E[\mfunc^2(\serv{\Lambda})] \  \text{ subject to } \  \E[\mfunc(\serv{\Lambda})] = 0\\
& \explain{(b)}{\leq} \inf_{\mfunc \in \polynom{\degree}} \E[|\mfunc(\serv{\Lambda}_{\nu})-1|^2] + \frac{1}{\rho_{\degree}} \cdot \E[\mfunc^2(\serv{\Lambda})] \  \text{ subject to } \  \E[\mfunc(\serv{\Lambda})] = 0 \explain{\eqref{eq:scrL-D-recall}}{=}\scrL{\degree}(\rho_{\degree}). 
\end{align*}
In the above display, step (a) follows by observing that $\rho_{\degree} \leq \rho_{\degree+1}$. Step (b) relies on the observation that $\polynom{\degree} \subset \polynom{\degree+1}$ and the infimum over larger sets is smaller. Hence $(\scrL{\degree}(\rho_{\degree}))_{\degree \in \N}$ is a non-increasing sequence. To show that $\scrL{\degree}(\rho_{\degree}) \downarrow \scrL{\infty}(\rho)$, we begin by observing that by repeating the argument used in the previous display, we can also obtain:
\begin{align*}
\scrL{\degree}(\rho_{\degree}) \geq \scrL{\infty}(\rho). 
\end{align*}
Hence, we have obtained a lower bound for $\scrL{\degree}({\rho_{\degree}})$ in terms of $\scrL{\infty}({\rho})$. To show that in fact $\scrL{\degree}(\rho_{\degree}) \downarrow \scrL{\infty}(\rho)$ we will also need an upper bound. Towards this goal, we consider a sequence of functions $\{\hat{\mfunc}_{\degree}\}_{\degree \in \N}$ indexed by $\degree \in \N$ such that $\hat{\mfunc}_{\degree}: \R \mapsto \R$ is a degree-$\degree$ polynomial and:
\begin{equation} \label{eq:poly-approx}
\begin{split}
& \lim_{\degree \rightarrow \infty} \| \hat{\mfunc}_{\degree} -  {\omd}(\cdot; \rho)\|_{\mu + \nu}^2  = 0 \\
 \Leftrightarrow  & \lim_{\degree \rightarrow \infty} \left(\E_{\serv{\Lambda} \sim \mu}|\hat{\mfunc}_{\degree}(\serv{\Lambda}) -  {\omd}(
\serv{\Lambda}; \rho)|^2 + \E_{\serv{\Lambda}_{\nu} \sim \nu}|\hat{\mfunc}_{\degree}(\serv{\Lambda}_{\nu}) -  {\omd}(
\serv{\Lambda}_{\nu}); \rho)|^2  \right)= 0.
\end{split}
\end{equation}
The existence of these polynomial approximations follows by the fact that polynomials form a dense subset of $L^2(\mu + \nu)$ since $\mu + \nu$ is a compactly supported measure \citep[Corollary 14.24, Definition 14.1]{schmudgen2017moment}. We can use the function $\lambda \mapsto \hat{\mfunc}_{\degree}(\lambda) - \E[\hat{\mfunc}_{\degree}(\serv{\Lambda})]$ as a candidate minimizer to upper bound $\scrL{\degree}(\rho_{\degree})$:
\begin{align*}
\scrL{\infty}(\rho) &  \leq \scrL{\degree}(\rho_{\degree}) \leq \E\left[\left|\hat{\mfunc}_{\degree}(\serv{\Lambda}_{\nu})-1 - \E[\hat{\mfunc}_{\degree}(\serv{\Lambda})] \right|^2\right] + \frac{1}{\rho_{\degree}} \cdot \E\left[ \left|\hat{\mfunc}_{\degree}(\serv{\Lambda})-\E[\hat{\mfunc}_{\degree}(\serv{\Lambda})]\right|^2 \right]
\end{align*}
We let $\degree \rightarrow \infty$ in the above display and exploit \eqref{eq:poly-approx} and $\rho_{\degree} \rightarrow \rho \in (0,\infty)$ to conclude that:
\begin{align*}
\lim_{\degree \rightarrow \infty} \scrL{\degree}(\rho_{\degree}) &= \E[|\omd(\serv{\Lambda}_{\nu}; \rho)-1|^2] + \frac{1}{\rho} \cdot \E[\omd^2(\serv{\Lambda};\rho)] = \scrL{\infty}(\rho),
\end{align*}
as claimed. Finally, we show that:
\begin{align*}
   \lim_{\degree \rightarrow \infty} \E_{\serv{\Lambda} \sim \mu}|\omdd{\degree}(\serv{\Lambda};\rho_{\degree}) -  {\omd}(
\serv{\Lambda}; \rho)|^2 & = 0, \quad \lim_{\degree \rightarrow \infty}  \E_{\serv{\Lambda}_{\nu} \sim \nu}|\omdd{\degree}(\serv{\Lambda}_{\nu}) -  {\omd}(
\serv{\Lambda}_{\nu}); \rho_{\degree})|^2  = 0
\end{align*}
To do so, we will exploit the strong convexity of the objective function:
\begin{align} \label{eq:scrO}
\mathscr{O}(\mfunc) \explain{def}{=} \E[|\mfunc(\serv{\Lambda}_{\nu})-1|^2] + \frac{1}{\rho} \cdot \E[\mfunc^2(\serv{\Lambda})].
\end{align}
Notice that $\mathscr{O}(\mfunc) -  \E[|\mfunc(\serv{\Lambda}_{\nu})-1|^2]$ is convex and hence for any $\lambda \in (0,1)$ and any $\mfunc,\hat{\mfunc} \in L^2(\mu + \nu)$,
\begin{align*}
&\mathscr{O}(\lambda \mfunc + (1-\lambda) \hat{\mfunc}) -  \E[|\lambda \mfunc(\serv{\Lambda}_{\nu}) + (1-\lambda) \hat{\mfunc}(\serv{\Lambda}_{\nu})-1|^2] \\ &  \leq \lambda \mathscr{O}(\mfunc) + (1-\lambda) \mathscr{O}(\hat{\mfunc}) - \lambda \E[|\mfunc(\serv{\Lambda}_{\nu})-1|^2] - (1-\lambda) \E[|\hat{\mfunc}(\serv{\Lambda}_{\nu})-1|^2].
\end{align*}
Rearranging yields:
\begin{align*}
\E[|{\mfunc}(\serv{\Lambda}_{\nu})-\hat{\mfunc}(\serv{\Lambda}_{\nu})|^2]\leq  \frac{ \lambda \mathscr{O}(\mfunc) + (1-\lambda) \mathscr{O}(\hat{\mfunc}) - \mathscr{O}(\lambda \mfunc + (1-\lambda) \mfunc)}{\lambda (1-\lambda)} .
\end{align*}
We instantiate the above inequality with $\lambda = 1/2$, $\mfunc = \omd(\cdot \; ; \rho)$ and $\hat{\mfunc} = \omdd{\degree}(\cdot \; ; \rho_{\degree})$. For notational convenience, we define the ``midpoint'' between these two matrix denoisers as $\dup{\Phi}{\degree} \explain{def}{=} (\omd(\cdot \; ; \rho) + \omdd{\degree}(\cdot \; ; \rho_{\degree}))/2$:
\begin{align}
    \E[|{\mfunc}(\serv{\Lambda}_{\nu})-\hat{\mfunc}(\serv{\Lambda}_{\nu})|^2] &\leq  2 \cdot \left[{ \mathscr{O}(\omd(\cdot \; ; \rho)) + \mathscr{O}( \omdd{\degree}(\cdot \; ; \rho_{\degree})) - 2\mathscr{O}(\dup{\Phi}{\degree})} \right]\label{eq:str-cvx}
\end{align}Recalling that $\omd(\cdot \; ; \rho)$ is the minimizer of the variational problem that defines $\scrL{\infty}(\rho)$ (claim 2 of the lemma):
\begin{align} \label{eq:str-cvx-1}
\mathscr{O}( \omd(\cdot \; ; \rho)) & \explain{\eqref{eq:scrO}}{=} \E[|\omd(\serv{\Lambda}_{\nu})-1|^2] + \frac{1}{\rho} \cdot \E[\omd^2(\serv{\Lambda})] = \scrL{\infty}(\rho).
\end{align}
Moreover, since $\dup{\Phi}{\degree}$ is a feasible point for the optimization problem that defines $\scrL{\infty}(\rho)$ (recall \eqref{eq:scrL-limit-repara}), we conclude that:
\begin{align} \label{eq:str-cvx-2}
\mathscr{O}(\dup{\Phi}{\degree}) & \explain{\eqref{eq:scrO}}{=} \E[|\dup{\Phi}{\degree}(\serv{\Lambda}_{\nu})-1|^2] + \frac{1}{\rho} \cdot \E[\dup{\Phi}{\degree}(\serv{\Lambda})^2]  \explain{\eqref{eq:scrL-limit-repara}}{\geq} \scrL{\infty}(\rho).
\end{align}
Recall that $\rho_{\degree} \uparrow \rho$ and that $\omdd{\degree}(\cdot \; ; \rho_{\degree})$ is the minimizer of the variational problem that defines $\scrL{\degree}(\rho_{\degree})$. Hence,
\begin{align} 
\mathscr{O}( \omdd{\degree}(\cdot \; ; \rho_{\degree}))   &\explain{\eqref{eq:scrO}}{=} \E[|\omdd{\degree}(\serv{\Lambda}_{\nu};\rho_{\degree})-1|^2] + \frac{1}{\rho} \cdot \E[\omdd{\degree}(\serv{\Lambda}; \rho_{\degree})^2]  \nonumber\\
& \leq\E[|\omdd{\degree}(\serv{\Lambda}_{\nu};\rho_{\degree})-1|^2] + \frac{1}{\rho_{\degree}} \cdot \E[\omdd{\degree}(\serv{\Lambda}; \rho_{\degree})^2] = \scrL{\degree}(\rho_{\degree}). \label{eq:str-cvx-3}
\end{align}
Plugging in \eqref{eq:str-cvx-1}, \eqref{eq:str-cvx-2}, and \eqref{eq:str-cvx-3} into \eqref{eq:str-cvx}, we obtain:
\begin{align*}
 \E[|{\mfunc}(\serv{\Lambda}_{\nu})-\hat{\mfunc}(\serv{\Lambda}_{\nu})|^2] & \leq 2 \cdot (\scrL{\degree}(\rho_{\degree}) - \scrL{\infty}(\rho)).
\end{align*}
We have already shown that $\scrL{\degree}(\rho_{\degree}) \rightarrow  \scrL{\infty}(\rho)$ as $\degree \rightarrow \infty$. Hence, \begin{align}\lim_{D \rightarrow \infty}  \E[|{\mfunc}(\serv{\Lambda}_{\nu})-\hat{\mfunc}(\serv{\Lambda}_{\nu})|^2] = 0. \label{eq:L2nu-convergence}\end{align} An analogous argument yields $\lim_{D \rightarrow \infty} \E[|{\mfunc}(\serv{\Lambda})-\hat{\mfunc}(\serv{\Lambda})|^2]= 0$, which completes the proof of Claim (4) in the statement of the lemma. 
\paragraph*{Proof of Claim (5)} Recall that from the first claim of the lemma:
\begin{align*}
1 - \E_{\serv{\Lambda} \sim \nu}[\omdd{\degree}(\serv{\Lambda}; \rho)] = \scrL{\degree}(\rho).
\end{align*}
Taking $\degree \rightarrow \infty$ and using the fact that $\scrL{\degree}(\rho) \rightarrow \scrL{\infty}(\rho)$ and \eqref{eq:L2nu-convergence} yields the desired conclusion. This concludes the proof of the lemma.
\end{proof}

\section{Equivalence of State Evolution and Replica Fixed Point Equations (Proposition~\ref{Lem:SE2_quartic})}\label{App:replica}

This appendix is devoted to proving Proposition~\ref{Lem:SE2_quartic}. We begin with some preliminary results about the trace ensemble. The proof of Proposition~\ref{Lem:SE2_quartic} will be presented in Appendix \ref{App:replica_quartic}.

\subsection{Preliminary Results on the Trace Ensemble}\label{App:trace_ensemble_pre} We remind the reader that under the trace ensemble, the noise matrix $\mW$ is drawn from the probability density function:
\begin{align} \label{eq:trace-ensemble-supp}
p(\mW) & \propto \exp\left( - \frac{\dim}{2}\sum_{i=1}^\dim V(\lambda_i(\mW)) \right) \text{ where $\lambda_{1:\dim}(\mW)$ denote the eigenvalues of $\mW$,}
\end{align}
and $V : \R \mapsto \R$ denotes the potential function, which is a functional parameter for the noise model. In addition, we recall the definition of the function $\phi:\R \mapsto \R$ (introduced in Lemma~\ref{lem:RMT}), which plays an important role in this paper:
\begin{align} \label{eq:replica-phi}
\phi(\lambda) & \explain{def}{=} {(1- \pi \theta \hlb_{\mu}(\lambda))^2 + \pi^2 \theta^2  \mu^2(\lambda)} \quad \lambda \; \in \; \R.
\end{align}
In the above display, $\hlb_\mu:\R \mapsto \R$ is the Hilbert transform of $\mu$, the limiting spectral measure of the noise matrix $\mW$. Lemma \ref{Lem:phi} below summarizes some useful properties of the trace ensemble.

\begin{lemma}[The function $\phi(\lambda)$ for trace ensemble]\label{Lem:phi} Suppose that the noise matrix $\mW$ is drawn from the trace ensemble defined in \eqref{eq:trace-ensemble-supp}. Assume that  the potential $V:\mathbb{R}\mapsto\mathbb{R}$ is a degree-$k$ polynomial for some $k\ge2$. Let $\mu$ denote the limiting spectral measure of $\mW$ and let $\phi:\mathbb{R}\mapsto\mathbb{R}$ be the function  from \eqref{eq:replica-phi}. Then:
\begin{enumerate}
\item The spectral measure $\mu$ is absolutely continuous with density function:
\begin{equation}\label{Eqn:trace_ensemble_density}
\mu(\lambda)=\frac{1}{2\pi}\sqrt{\big(4Q(\lambda)-(V'(\lambda))^2\big)_+},\quad\forall\lambda\in\mathbb{R},
\end{equation}
where $(x)_+=\max\{x,0\}$ and $Q(\lambda)$ is given by the Cauchy principal value integral:
\begin{equation}\label{Eqn:Q_def}
Q(\lambda)=\int\frac{V'(\lambda)-V'(t)}{\lambda-t}\mu(t)\mathrm{d}t.
\end{equation}
The support of $\mu$ is a finite union of finite intervals, and the density function $\mu(\lambda)$ is Holder continuous.
\item Define ${\phi}_{\text{poly}}:\mathbb{R}\mapsto\mathbb{R}$ as
\begin{equation}\label{Eqn:phi_poly_second}
\begin{split}
{\phi}_{\text{poly}}(\lambda)\bydef 1-\theta\cdot V'(\lambda)+\theta^2\cdot Q(\lambda),\quad\forall \; \lambda \; \in \; \mathbb{R}.
\end{split}
\end{equation}
Then, ${\phi}_{\text{poly}}$ is a degree $k-1$ polynomial and
\begin{align}
\phi(\lambda)={\phi}_{\text{poly}}(\lambda) \quad\forall \; \lambda \; \in \; \text{Supp}(\mu), \label{Eqn:phi_phi_poly_inside} \\
\phi(\lambda)={\phi}_{\text{poly}}(\lambda) \quad\forall \; \lambda \; \in \; \text{Supp}(\nu) \label{Eqn:phi_phi_poly_inside-2}
\end{align}
\end{enumerate}

\end{lemma}

\begin{proof} We consider the two claims one by one. 
\begin{enumerate}
\item The expression of the density function of $\mu$ can be found in, e.g., \citep[Theorem 11.2.1]{pastur2011eigenvalue}. Since $V(x)$ is a degree-$k$ polynomial, $V'(\lambda)$ is a degree $k-1$ polynomial and $Q(\lambda)$ defined in \eqref{Eqn:Q_def} is a degree $k-2$ polynomial. Hence, the support $\text{Supp}(\mu)=\left\{\lambda|4Q(\lambda)-(V'(\lambda))^2\ge0\right\}$ must be a union of finite intervals \citep[Eq.~(11.2.14)]{pastur2011eigenvalue}. Furthermore, $\mu(\lambda)$ is a composition of Holder continuous functions and is therefore Holder continuous. 
\item To prove the second claim, we note that the Hilbert transform of $\mu(\lambda)$ is related to the potential function via \citep[Theorem 11.2.1]{pastur2011eigenvalue}:
\begin{equation}\label{Eqn:trace_ensemble_Hilbert}
\pi\cdot\mathscr{H}_{\mu}(\lambda)=\frac{1}{2}V'(\lambda),\quad \forall \lambda\in\text{Supp}(\mu).
\end{equation}
Substituting \eqref{Eqn:trace_ensemble_density} and \eqref{Eqn:trace_ensemble_Hilbert} into \eqref{eq:replica-phi}, we get
\begin{equation}\label{Eqn:phi_second}
\begin{split}
\phi(\lambda) &\bydef \left[1-\theta\pi\mathscr{H}_{\mu}(\lambda)\right]^2+\pi^2\theta^2\mu^2\left(\lambda\right) \\ &=1-\theta\cdot V'(\lambda)+\theta^2\cdot Q(\lambda) \bydef \phi_{\text{poly}}(\lambda) \quad\forall\lambda\in\text{Supp}(\mu).
\end{split}
\end{equation}
This verifies \eqref{Eqn:phi_phi_poly_inside}. To show that \eqref{Eqn:phi_phi_poly_inside-2}, we consider the Lebesgue decomposition of $\nu = \nu_{\parallel} + \nu_{\perp}$ into the absolutely continuous part $\nu_{\parallel}$ and the singular part $\nu_{\perp}$.  Recall from Lemma~\ref{lem:RMT} that $\nu_{\parallel}$ has the same support as $\mu$. Hence, we only need to show that $\phi = \phi_{\text{poly}}$ on $\text{Supp}(\nu_{\perp})$. Lemma~\ref{lem:RMT} shows that the measure $\nu_\perp$ concentrates on the set $S\bydef\{x\in\mathbb{R}:\phi(x)=0\}$. It is therefore sufficient to prove $\phi_{\text{poly}}(\lambda)=\phi(\lambda)$ for $\lambda\in S$. From \eqref{eq:replica-phi}, we have
\[
\pi\mathscr{H}_\mu(\lambda)=\frac{1}{\theta}\quad\text{and}\quad \mu(\lambda)=0,\quad\forall \lambda\in S.
\]
On the other hand, the Stieltjes transform of $\mu$ for the trace ensemble with a polynomial potential $V$ satisfies \citep[Theorem 11.2.1]{pastur2011eigenvalue}
\[
\stl_{\mu}^2(z)-V'(z)\stl_{\mu}(z)+Q(z)=0,\quad\forall z\in \mathbb{C}\backslash \mathbb{R}, 
\]
where the function $Q$ is defined in \eqref{Eqn:Q_def}. By the Sohotsky-Plemelj formula \citep[Section 2.1]{pastur2011eigenvalue},
\begin{align*}
\lim_{\epsilon \downarrow 0}\stl_{\mu}(\lambda + \I \epsilon)= \pi \hlb_{\mu}(\lambda) -\I \cdot\pi\mu(\lambda),\quad\forall \lambda\in\mathbb{R}.
\end{align*}
Combining the above three equations leads to the following equation:
\[
\frac{1}{\theta^2}-V'(\lambda)\cdot \frac{1}{\theta} +Q(\lambda)=0,\quad \forall \lambda\in S.
\]
Therefore,
\[
\phi_{\text{poly}}(\lambda)\bydef1-\theta\cdot V'(\lambda)+\theta^2\cdot Q(\lambda)=0,\quad\forall \lambda\in S.
\]
Hence, $\phi_{\text{poly}}(\lambda)=\phi(\lambda)$ for $\lambda\in S$, which completes the proof of \eqref{Eqn:phi_phi_poly_inside-2}.
To check that ${\phi}_{\text{poly}}$ is a degree $k-1$ polynomial, we note that $V'(\lambda)$ is a degree $k-1$ polynomial and $Q(\lambda)$ in \eqref{Eqn:trace_ensemble_density} is a degree $k-2$ polynomial.  
\end{enumerate}
\end{proof}

Lemma \ref{Lem:phi} shows that $\phi(\lambda)$ agrees with a polynomial $\phi_{\text{poly}}(\lambda)$ for $\lambda \in\text{Supp}(\mu)$, but they are generally different outside of the support. The following proposition shows that we can replace $\phi(\cdot)$ by the polynomial ${\phi}_{\text{poly}}(\cdot)$ in the optimal OAMP algorithm with no performance loss asymptotically. 

\begin{proposition}[A variant of optimal OAMP]\label{Lem:optimal_OAMP_variant}
Assume that $\bm{W}$ is drawn from the trace ensemble with a polynomial potential. Consider a variant of the optimal OAMP algorithm from \eqref{Eqn:psi_bar_normalization} where $\phi(\cdot)$ in the matrix denoiser is replaced by $\phi_{\text{poly}}(\cdot)$. Then, Proposition~\ref{prop:optimal-OAMP-SE} still holds for this OAMP algorithm. 
\end{proposition}

\begin{proof}
First of all, the variant of the optimal OAMP algorithm satisfies the required trace-free condition: $\mathbb{E}\left[\Psi_{\text{poly}}(\sfLambda;\rho)\right]=0$, where
\[
\Psi_{\text{poly}}(\lambda; \rho)\bydef 1- \left( \E \left[ \frac{ \phi_{\text{poly}}(\serv{\Lambda})}{ \phi_{\text{poly}}(\serv{\Lambda}) + \rho} \right] \right)^{-1} \cdot \frac{\phi_{\text{poly}}(\lambda)}{\phi_{\text{poly}}(\lambda)+\rho}, \quad \forall \; \lambda \; \in \; \R, \; \rho \; \in \;  (0,\infty).
\]
Therefore, the state evolution framework applies to this variant of optimal OAMP.

{To prove that the state evolution characterization claimed in Proposition~\ref{prop:optimal-OAMP-SE} holds for this new variant of the optimal OAMP algorithm it suffices to verify that:
\begin{subequations} \label{eq:variant-quant}
\begin{align}
\mathbb{E}\big[{\Psi_{\text{poly}}}(\sfLambda_\nu;\rho)] &=\mathbb{E}\big[{\Psi}_\star(\sfLambda_\nu)], \quad \mathbb{E}\big[{\Psi^2_{\text{poly}}}(\sfLambda_\nu ;\rho)] =\mathbb{E}\big[{\Psi}^2_\star(\sfLambda_\nu ;\rho)] \label{eq:variant-quant-1}\\
\mathbb{E}\big[{\Psi_{\text{poly}}}(\sfLambda ;\rho )], &=\mathbb{E}\big[{\Psi}_\star(\sfLambda ;\rho)], \quad \mathbb{E}\big[{\Psi^2_{\text{poly}}}(\sfLambda ;\rho)] =\mathbb{E}\big[{\Psi}^2_\star(\sfLambda ;\rho)]. \label{eq:variant-quant-2}
\end{align}
\end{subequations}
This is because, if $(\serv{X}_{\star}, (\serv{X}_t)_{t \in \N}; \serv{A})$ represent the state evolution random variables associated with the new variant of the optimal OAMP algorithm, then the joint distribution of $(\serv{X}_{\star}, \serv{X}_t; \serv{A})$ is determined only by the quantities in \eqref{eq:variant-quant} (cf. Definition~\ref{Def:OAMP_main}).} Moreover,  the claims in \eqref{eq:variant-quant} are immediate from item 3 of Lemma \ref{Lem:phi}, which shows that $\phi, \phi_{\text{poly}}$ agree on $\text{Supp}(\mu)$ and $\text{Supp}(\nu)$, and hence, $\Psi_\star = \Psi_{\text{poly}}$ on $\text{Supp}(\mu)$ and $\text{Supp}(\nu)$.
\end{proof}

We next define a function closely related to ${\phi}_{\text{poly}}$, which recovers the matrix processing function used in the Bayes-optimal AMP (BAMP) algorithm proposed in \citet{barbier2023fundamental}.

\begin{definition}[Definition of $J(\lambda)$]\label{Def:phi_J}
Let $k\ge2$ and $V:\mathbb{R}\mapsto\mathbb{R}$ be a degree $k$ polynomial. Let the degree $k-1$ polynomial ${\phi}_{\text{poly}}$ in \eqref{Eqn:phi_poly_second} be represented as
\BS\label{Eqn:phi_J_defs}
\begin{equation}\label{Eqn:phi_polynomial}
{\phi}_{\text{poly}}(\lambda)=C_\phi+\sum_{i=1}^{k-1} C_{\phi,i}\lambda^i,\quad\forall\lambda\in\mathbb{R}.
\end{equation}
We define the function $J:\mathbb{R}\mapsto\mathbb{R}$ as
\begin{equation}\label{Eqn:J_def}
J(\lambda)\bydef-\sum_{i=1}^{k-1} C_{\phi,i}\lambda^i,\quad\forall\lambda\in\mathbb{R}.
\end{equation}
\ES
\end{definition}

\paragraph*{Examples of $\phi_{\text{poly}}(\lambda)$ and $J(\lambda)$} We now provide explicit formulas for the functions $\phi_{\text{poly}}(\lambda)$ and $J(\lambda)$ for three polynomial potential functions considered in \citet{barbier2023fundamental}. 

\begin{enumerate}
\item \emph{Quadratic potential}:
\BS
\begin{equation}\label{Eqn:quadratic}
\begin{split}
V(\lambda) &=\frac{\lambda^2}{2}.
\end{split}
\end{equation}
This corresponds to the Gaussian orthogonal ensemble (GOE). The density function and its Hilbert transform are $\mu(\lambda)=\frac{1}{2\pi}\sqrt{4-\lambda^2}$ and $\mathscr{H}_\mu(\lambda)=V'(\lambda)/2\pi$, $\forall\lambda\in[-2,2]$. From \eqref{Eqn:phi_poly_second} and \eqref{Eqn:J_def}, we have
\begin{align}
\phi_{\text{poly}}(\lambda) &= 1+\theta^2-\theta\cdot\lambda,\\
J(\lambda) &= \theta\cdot\lambda.
\end{align}
\ES

\item \emph{Quartic potential}:
\BS
\begin{equation}\label{Eqn:quartic}
V(\lambda)=\frac{\gamma \lambda^2}{2}+\frac{\kappa \lambda^4}{4},
\end{equation}
where $\gamma\in[0,1]$ is a parameter and $\kappa$ is given by: $$\kappa=\kappa(\gamma)= \frac{8-9\gamma +\sqrt{64-144\gamma+108\gamma^2-27\gamma^3}}{27}\in \left[0,\frac{16}{27}\right].$$ This choice of $\kappa$ ensures that the limiting eigenvalue distribution $\mu$ has unit variance. Its density function is given by $\mu(\lambda)=(\gamma+2a^2\kappa + \kappa \lambda^2)\sqrt{4a^2-\lambda^2}/(2\pi)$, $\forall\lambda\in[-2a,2a]$, where $a^2:=\big(\sqrt{\gamma^2+12\kappa}-\gamma\big)/(6\kappa)$. The functions $\phi_{\text{poly}}$ and $J$ read \citep{potters2020first}
\begin{align}
\phi_{\text{poly}}(\lambda) &= \theta^2a^2(\gamma+2a^2\kappa)^2 + 1 -\theta\left(\gamma\lambda-\theta\kappa\lambda^2+\kappa\lambda^3\right),\label{Eqn:phi_quartic}\\
J(\lambda) &=\theta\left(\gamma\lambda-\theta\kappa\lambda^2+\kappa\lambda^3\right).\label{Eqn:J_quartic}
\end{align}
\ES

\item \emph{Purely sestic potential}:
\BS
\begin{equation}\label{Eqn:sestic}
V(\lambda)=\frac{\xi \lambda^6}{6},
\end{equation}
where $\xi=\frac{27}{80}$. Again, this value of $\xi$ ensures that the spectral measure $\mu$ has unit variance. Its density function is given by $\mu(\lambda)=( 6b^4\xi + 2b^2\xi\lambda^2 + \xi\lambda^4 )\sqrt{4b^2-\lambda^2}/(2\pi)$, $\forall \lambda\in[-2b,2b]$, where $b^2:=2/3$. The functions $\phi_{\text{poly}}$ and $J$ are given by \citep{barbier2023fundamental}
\begin{align}
\phi_{\text{poly}}(\lambda) &= \frac{27}{50}\theta^2 + 1 -\theta\left(-\theta\xi\lambda^2-\theta\xi\lambda^4+\xi\lambda^5\right),\\
J(\lambda) &= \theta\left(-\theta\xi\lambda^2-\theta\xi\lambda^4+\xi\lambda^5\right).
\end{align}
\ES
\end{enumerate}

Fig.~\ref{Fig:phi} plots the functions $\phi(\cdot)$ and $\phi_{\text{poly}}(\cdot)$ corresponding to the quartic potential. In this setup, the measure $\mu$ is supported on the interval $[-2a,2a]$ with $2a\approx1.7$, and it can be shown that $\nu$ has a point mass (corresponding to the outlying eigenvalue of $\bm{Y}$) at $\lambda_o\approx 2.3$. We see from Fig.~\ref{Fig:phi} that $\phi(\cdot)$ and $\phi_{\text{poly}}(\cdot)$ coincide inside the support of $\mu$ but differ outside. Moreover, the curves of $\phi(\cdot)$ and $\phi_{\text{poly}}(\cdot)$ intersect precisely at the location of the outlying eigenvalue: $\phi(\lambda_o)=\phi_{\text{poly}}(\lambda_o)=0$; see the proof of Proposition \ref{Lem:optimal_OAMP_variant}. 

\begin{figure}[htbp]
\begin{center}
\includegraphics[width=0.5\linewidth]{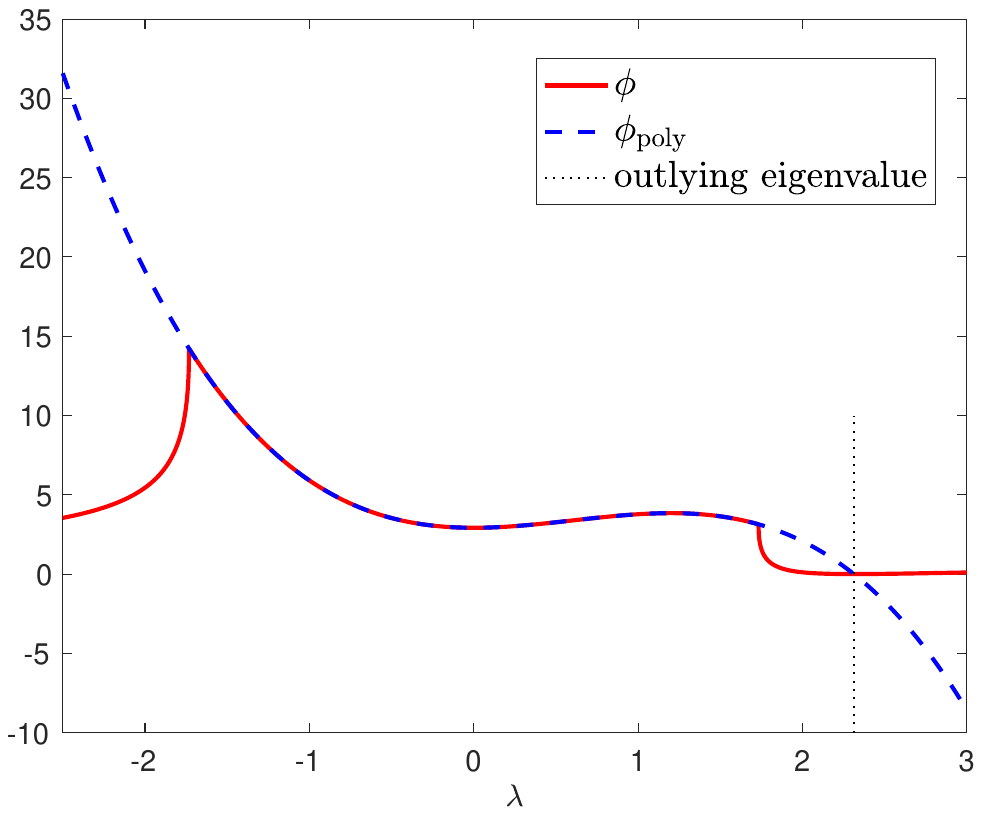}
\end{center}
\caption{Plot of $\phi$ and $\phi_{\text{poly}}$ for the quartic ensemble with $\gamma=0$. $\theta=1.8$. $\text{Supp}(\mu)=[-2a,2a]$ with $2a\approx1.7$. The singular part of the measure $\nu$ contains a single point mass at $\lambda_o\approx2.3$.}
\label{Fig:phi}
\end{figure} 

\subsection{Proof of Proposition~\ref{Lem:SE2_quartic}} \label{App:replica_quartic}

We will show that any solution $(\rho,\omega)$ to the state evolution (SE) fixed point equation:
\begin{align} \label{eq:rho-omega-supp}
    \omega = 1- \left( \E \left[ \frac{\phi(\serv{\Lambda})}{\phi(\serv{\Lambda}) + {\rho}} \right] \right)^{-1} \cdot {\mathbb{E}\left[\frac{1}{\rho+\phi(\sfLambda)}\right]}, \quad \rho=\frac{1}{\dmmse(\omega)}-1, \quad \serv{\Lambda} \sim \mu.
\end{align}
is also a solution to the replica fixed point equations:
\begin{subequations}\label{Eqn:SE_FP_quartic}
\begin{align}
&m = 1-\mmse(\omega),\label{Eqn:SE_FP_quartic_0}\\
&\mathbb{E}[\mathsf{H}] = 1-m,\label{Eqn:SE_FP_quartic_1}\\
&\chi = \mathbb{E}[\sfLambda \mathsf{Q H}],\label{Eqn:SE_FP_quartic_2}\\
&\hat{m}\equiv \frac{\omega}{1-\omega} = \kappa\theta^2 \left(\frac{m}{1-m}\E[\sfLambda^2\mathsf{H}]+\frac{\chi}{1-m}\E[\sfLambda \mathsf{H}]+\E[\sfLambda^2 \mathsf{QH}] \right)+ \gamma\theta^2 m \label{Eqn:SE_FP_quartic_3},
\end{align}
where $m\in\mathbb{R}$, $\chi\in\mathbb{R}$ are two intermediate variables, $\serv{\Lambda} \sim \mu$, and the random variables $\mathsf{Q} = \mathsf{Q}(\sfLambda,m,\chi)$, $\mathsf{H}=\mathsf{H}(\sfLambda,\rho)$ are defined by
\begin{align}
\mathsf{Q} &\bydef \kappa\theta^2 m\sfLambda^2 + \kappa\theta^2\chi\sfLambda
- \frac{\kappa\theta^2}{1-m} \E\left[ m\sfLambda^2\mathsf{H} + \chi\sfLambda \mathsf{H}\right] + \frac{m}{1-m} ,
\label{eq:replica_fp_JY3_3}
\\
\mathsf{H} &\bydef \left( \rho + \theta^2a^2(\gamma+2a^2\kappa)^2 + 1 -\theta\left(\gamma\serv{\Lambda}-\theta\kappa\serv{\Lambda}^2+\kappa\serv{\Lambda}^3\right) \right)^{-1}.
\label{eq:replica_fp_JY3_4}
\end{align}
\end{subequations}
{The proof for the other direction, namely, each solution to the replica fixed point equation is also a solution to the SE equations is similar and thus omitted.}

We will introduce a series of manipulations of the SE equations \eqref{eq:rho-omega-supp} that eventually lead to the replica equations \eqref{Eqn:SE_FP_quartic}, for the \textit{quartic potential} case. Our proof consists of the following steps:
\begin{enumerate}
\item In Lemma \ref{lemma:alt_fp_opt_se}, we rewrite \eqref{eq:rho-omega-supp} as \eqref{Eqn:OAMP_FP2}. Note that unlike the original SE equation \eqref{eq:rho-omega-supp} which only involves the measure $\mu$, the new formulation \eqref{Eqn:OAMP_FP2} involves both $\mu$ and $\nu$. The introduction of the new formulation \eqref{Eqn:OAMP_FP2} is somewhat ad hoc, and is only for the purpose of deriving the replica equations as presented in \cite{barbier2023fundamental}.
\item We recognize that the term in \eqref{Eqn:OAMP_FP2_a} that involves $\nu$ is a resolvent operator for $J(\sfLambda_\nu)$. We will show that the expectation of this term is the limit of $\frac{1}{N}\bm{x}_\star^\UT\left(\rho\bm{I}_N+C_\phi\bm{I}_N-J(\bm{Y})\right)^{-1}\bm{x}_\star$ as $N\to\infty$. In the quartic potential case, the function $J(\cdot)$ is a degree three polynomial; see \eqref{Eqn:J_quartic}. It is possible to invoke the matrix inversion lemma to express the term into a resolvent in terms of $J(\bm{W})$ plus a low-rank perturbation term, and subsequently calculate the asymptotics of each term involved. This step is summarized in Lemma \ref{Lem:aux13}.
\item Using Lemma \ref{Lem:aux13} and some algebraic manipulations, we finally arrive at the  replica equations we aim to prove in Proposition~\ref{Lem:SE2_quartic}. The derivations are postponed to Section \ref{App:quartic_replica_details}.
\end{enumerate}

\subsubsection{Auxiliary results}

To compare with the replica calculations in \cite{barbier2023fundamental}, we rewrite the fixed point equation \eqref{eq:rho-omega-supp} into an equivalent form in the following lemma.

\begin{lemma}[Alternative fixed point equations for SE]\label{lemma:alt_fp_opt_se}
Assume that $\E\Var[\serv{X}_{\star} | \serv{A}] \in (0,1)$. Let $(\omega,\rho)$ be a solution to the state evolution fixed point equations \eqref{eq:rho-omega-supp}. Then, $(\omega,\rho)$ satisfies the following equations:
\begin{subequations}\label{Eqn:OAMP_FP2}
\begin{align}
&\frac{\omega}{1-\omega} = \left(1-\mathbb{E}\left[ \frac{1}{\rho+\phi(\sfLambda)}\right]\right)\cdot \left(\left(\mathbb{E}\left[   \frac{1}{\rho+\phi(\sfLambda)}\right]\right)^{-1}-\left(\mathbb{E}\left[   \frac{1}{\rho+\phi(\sfLambda_\nu)}\right]\right)^{-1}\right),\label{Eqn:OAMP_FP2_a}\\
&\mathbb{E}\left[ \frac{1}{\rho+\phi(\sfLambda)}\right]= \mmse(\omega),\label{Eqn:OAMP_FP2_b}
\end{align}
\end{subequations}
where $\sfLambda\sim\mu$ and $\sfLambda_\nu\sim \nu$. 
\end{lemma}

\begin{proof}
We prove the two equations seperately.
\paragraph*{Proof of \eqref{Eqn:OAMP_FP2_a}} We rewrite the first equation of \eqref{eq:rho-omega-supp} as
\BS
\begin{align}
\frac{\omega}{1-\omega} &= \left(\mathbb{E}\left[\frac{1}{\rho+\phi(\sfLambda)}\right]\right)^{-1}-(1+\rho)\label{Eqn:OAMP_FP3_a}\\
&=\left(1-\mathbb{E}\left[ \frac{1}{\rho+\phi(\sfLambda)}\right]\right)\cdot \left(\left(\mathbb{E}\left[   \frac{1}{\rho+\phi(\sfLambda)}\right]\right)^{-1}-\rho\cdot\left(1-\mathbb{E}\left[ \frac{1}{\rho+\phi(\sfLambda)}\right]\right)^{-1}\right). \label{Eqn:OAMP_FP3_b}
\end{align}
\ES
Comparing \eqref{Eqn:OAMP_FP2_a} and \eqref{Eqn:OAMP_FP3_a}, we see that it suffices to prove the following
\begin{equation}\label{Eqn:opt_g_key_property}
\mathbb{E}\left[\frac{1}{\rho+\phi(\sfLambda_\nu)}\right]=\frac{1}{\rho}\cdot\left(1-\mathbb{E}\left[\frac{1}{\rho+\phi(\sfLambda)}\right]\right),\quad \forall \rho>0.
\end{equation}
To prove \eqref{Eqn:opt_g_key_property}, we note that
\[
\begin{split}
\mathbb{E}\left[\frac{1}{\rho+\phi(\sfLambda_\nu)}\right] &=\int_{\mathbb{R}} \frac{1}{\rho+\phi(\lambda)}\nu(\mathrm{d}\lambda)\\
&= \int_{\mathbb{R}} \frac{1}{\rho+\phi(\lambda)}\nu_{\perp}(\mathrm{d}\lambda)+\int_{\mathbb{R}} \frac{1}{\rho+\phi(\lambda)}\nu_{\parallel}(\mathrm{d}\lambda)\\
&\overset{(a)}{=} \frac{1}{\rho}\cdot\nu_{\perp}(\mathbb{R})+\int_{\mathbb{R}} \frac{1}{\rho+\phi(\lambda)}\nu_{\parallel}(\mathrm{d}\lambda)\\
&= \frac{1}{\rho}\cdot\left(1-\nu_{\parallel}(\mathbb{R})\right)+\int_{\mathbb{R}} \frac{1}{\rho+\phi(\lambda)}\nu_{\parallel}(\mathrm{d}\lambda)\\
&\overset{(b)}{=} \frac{1}{\rho}\cdot\left(1-\int_{\mathbb{R}} \frac{1}{\phi(\lambda)}\mu(\mathrm{d}\lambda)\right)+\int_{\mathbb{R}} \frac{1}{\rho+\phi(\lambda)}\cdot\frac{1}{\phi(\lambda)}\mu(\mathrm{d}\lambda)\\
&=\frac{1}{\rho}\cdot\left(1-\mathbb{E}\left[\frac{1}{\rho+\phi(\sfLambda)}\right]\right),\quad \sfLambda\sim\mu,
\end{split}
\]
where step (a) and (b) are due to items 2 and 3 of Lemma~\ref{lem:RMT}, respectively.

\paragraph*{Proof of \eqref{Eqn:OAMP_FP2_b}} 

From the second equation of \eqref{eq:rho-omega-supp}, we have
\[
\rho =\frac{1}{\dmmse(\omega)} -1\overset{(a)}{=} \frac{1}{\mmse(\omega)}-\frac{1}{1-\omega},
\]
where step (a) follows from Lemma \ref{lem:dmmse-scalar}. Hence,
\begin{equation}\label{Eqn:OAMP_SE_FP_app_temp1}
\left(\rho+\frac{1}{1-\omega}\right)^{-1}=\mmse(\omega).
\end{equation}
Now recall the first equation of \eqref{eq:rho-omega-supp}:
\[
\begin{split}
\omega&= 1 -  \left( \E \left[ \frac{\phi(\serv{\Lambda})}{\phi(\serv{\Lambda}) + {\rho}} \right] \right)^{-1} \cdot \E \left[\frac{1}{\phi(\serv{\Lambda})+{\rho}} \right].
\end{split}
\]
By simple algebra, we have
\begin{equation}\label{Eqn:OAMP_SE_FP_app_temp2}
\left(\rho+\frac{1}{1-\omega}\right)^{-1}=\mathbb{E}\left[\frac{1}{\phi(\sfLambda)+\rho}\right].
\end{equation}
Combining \eqref{Eqn:OAMP_SE_FP_app_temp1} and \eqref{Eqn:OAMP_SE_FP_app_temp2} leads to \eqref{Eqn:OAMP_FP2_b}.

\end{proof}


Note that \eqref{Eqn:OAMP_FP2_a} involves the term $\mathbb{E}\left[ \frac{1}{\rho+\phi(\sfLambda_\nu)}\right]$, which depends on the function $\phi(\cdot)$. To facilitate further analysis, we rewrite this term using the function $\phi_{\text{poly}}(\lambda)\bydef C_\phi - J(\lambda)$ defined in Appendix \ref{App:trace_ensemble_pre}; see Lemma \ref{Lem:phi} and Definition \ref{Def:phi_J}. Lemma \ref{Lem:phi_J_R} below summarizes this result. Its proof is similar to the proof of Proposition \ref{Lem:optimal_OAMP_variant} and omitted.

\begin{lemma}\label{Lem:phi_J_R}
Let $\mu$ be the limiting spectral measure associated with the quartic potential. Let $C_\phi$ and $J(\cdot)$ be defined as in Definition \eqref{Def:phi_J}.
Then, the following holds for any $\rho>0$:
\begin{equation}
\mathbb{E}\left[\frac{1}{\rho+\phi(\sfLambda_\nu)}\right] = \mathbb{E}\left[\frac{1}{\rho+C_\phi-J(\sfLambda_\nu)}\right],\quad \sfLambda_\nu\sim\nu.
\end{equation}
\end{lemma}


The following lemma reformulates the term $\mathbb{E}\left[\frac{1}{\rho+C_\phi-J(\sfLambda_\nu)}\right]$, which involves the measure $\sfLambda_\nu\sim\nu$, into a (complicated) expression that only depends on $\sfLambda\sim\mu$. 

\begin{lemma}\label{Lem:aux13}
Consider the quartic potential for which the function $J(\lambda)$ is given in \eqref{Eqn:J_quartic}. Then, the following identity holds for any $\rho>0$:
\begin{align}\label{Eqn:g_hat_quartic_result}
&\mathbb{E}\left[ \frac{1}{\rho+C_\phi-J(\sfLambda_\nu)}\right]= -\frac{1}{(\kappa\theta^2)^2 d_4 + \gamma\theta^2 +\frac{3\kappa\theta^2 d_2 + (\kappa\theta^2)^2 ( 2d_1 d_3-3d_2^2) + (\kappa\theta^2)^3 ( d_2^3 + d_0 d_3^2 - 2d_1 d_2 d_3) - 1}{d_0+\kappa\theta^2 d_1^2-\kappa\theta^2 d_0 d_2}},
\end{align}
where $\sfLambda_\nu\sim\nu$, and $\kappa$ and $\gamma$ are parameters for the quartic potential function and $\theta$ is the parameter for the observation model, and
\begin{equation}\label{Eqn:di_def_app}
d_{i} \bydef \E\left[\frac{\sfLambda^{i}}{ \rho+\phi(\sfLambda)}\right],\quad\forall i\in\mathbb{N}\cup\{0\}.
\end{equation}
\end{lemma}

\begin{proof}
Recall that the probability measure $\nu$ is the limit of the empirical measure $\nu_N\bydef \frac{1}{\dim} \sum_{i=1}^\dim  \ip{\vu_i(\mY)}{\vx_{\star}}^2 \cdot \delta_{\lambda_i(\mY)} $, where $\{\bm{u}_i(\bm{Y})\}_{i\in[N]}$ and $\{\lambda_i(\bm{Y})\}_{i\in[N]}$ are the eigenvectors and eigenvalues of $\bm{Y}$. By the convergence of $\nu_N$ (Lemma~\ref{lem:RMT}), we have
\begin{align}\label{Eqn:xT_g_x}
 \underset{N\to\infty}{\text{plim}}\  \frac{1}{N} \vx_\star^{\UT} \left(c_\ast\bm{I}_N-J(\bm{Y})\right)^{-1} \vx_\star=\mathbb{E}\left[ \frac{1}{c_\ast-J(\sfLambda_\nu)}\right] .
\end{align}
where $c_\ast\explain{def}{=}\rho+C_\phi$. Recall that $J(\cdot)$ is a degree three polynomial for quartic potential:
\begin{equation}
J(\lambda)=\theta\left(\gamma\lambda-\theta\kappa\lambda^2+\kappa\lambda^3\right).
\end{equation}
Based on the above definition of $J(\cdot)$, we can decompose $J(\mY)$ as $J(\bm{W})$ plus a low-rank perturbation term:
\begin{eqnarray*}
J(\mY)& =& J\left(\frac{\theta\bm{x}_\star\bm{x}_\star^\UT}{N}+\bm{W}\right) \\
&=&\gamma\theta\mW-\kappa\theta^2\mW^2+\kappa\theta\mW^3 + \frac{\gamma\theta^2}{N}\vx_\star\vx_\star^\UT + \frac{1}{N}\kappa\theta^3\vx_\star^\UT\mW\vx_\star\cdot \frac{\bm{x}_\star\bm{x}_\star^\UT}{N}\\
&& + \frac{1}{N}\kappa\theta^2\left(\mW^2\vx_\star\vx_\star^\UT+\vx\vx^\UT\mW^2+\mW\vx_\star\vx_\star^\UT\mW\right) \\
&=&J(\mW) + \frac{1}{N}\vx_\star\vx_\star^\UT\left(\gamma\theta^2\mI_N+\kappa\theta^2\mW\right) + \frac{1}{N}\kappa\theta^2\mW\vx_\star\vx_\star^\UT\mW\\
&& + \frac{1}{N}\kappa\theta^2\mW^2\vx_\star\vx_\star^\UT+ \frac{1}{N}\kappa\theta^3\vx_\star^\UT\mW\vx_\star\cdot \frac{\bm{x}_\star\bm{x}_\star^\UT}{N}  \\
&=&J(\mW) + \frac{1}{N}\mP\mQ^\UT,
\end{eqnarray*}
where
\begin{align}
\mP &\explain{def}{=} \left[\vx_\star,\mW\vx_\star,\mW^2\vx_\star\right]\in\mathbb{R}^{N\times 3}, \label{Eqn:P_matrix_def}\\ 
\mQ &\explain{def}{=} \left[\left(\left(\gamma\theta^2+\kappa\theta^3\frac{\vx_\star^\UT\mW\vx_\star}{N}\right)\bm{I}_N+\kappa\theta^2\mW^2\right)\vx_\star,\kappa\theta^2\mW\vx_\star,\kappa\theta^2\vx_\star\right]\in\mathbb{R}^{N\times 3}. \label{Eqn:Q_matrix_def}
\end{align}
Hence,
\begin{align*}
\left(c_\ast\mI-J(\bm{Y})\right)^{-1} &=\left(c_\ast\mI-J(\mW)-\frac{1}{N}\mP\mQ^\UT\right)^{-1} \\
&\overset{(a)}{=} \left(\bm{A}-\frac{1}{N}\mP\mQ^\UT\right)^{-1} \\
&\overset{(b)}{=}\bm{A}^{-1}+\frac{1}{N}\bm{A}^{-1}\mP\left( \mI_3 -\frac{\mQ^\UT \bm{A}^{-1} \mP }{N}\right)^{-1}\mQ^\UT \bm{A}^{-1},
\end{align*}
where the matrix $\bm{A}$ in step (a) denotes
\begin{equation}\label{Eqn:A_matrix_def}
\bm{A}\explain{def}{=}c_\ast\mI-J(\mW),
\end{equation}
and step (b) is from the Morrison–Woodbury formula. Using the above identity, we obtain the following representation of the LHS of \eqref{Eqn:xT_g_x}:
\begin{align}\label{Eqn:app_replica_temp}
\frac{1}{N} \vx^{\UT}\left(c_\ast\mI-J(\bm{Y})\right)^{-1}  \vx&= \frac{\bm{x}_\star^\UT\bm{A}^{-1}\bm{x}_\star}{N} + 
\frac{1}{N^2} \vx_\star^\UT \bm{A}^{-1}\mP\left( \mI_3 - \frac{\mQ^\UT \bm{A}^{-1} \mP }{N}\right)^{-1}\mQ^\UT \bm{A}^{-1}\vx_\star.
\end{align}
Therefore, to calculate the limit of $\frac{1}{N} \vx^{\UT}\left(c_\ast\mI-J(\bm{Y})\right)^{-1}$ as $N\to\infty$, it suffices to derive the limit of the terms $\frac{1}{N} \bm{x}_\star^\UT\bm{A}^{-1}\bm{x}_\star$, $\frac{1}{N}\vx_\star^\UT \bm{A}^{-1}\mP\in\mathbb{R}^{1\times 3}$, $\frac{1}{N}\mQ^\UT \bm{A}^{-1} \mP\in\mathbb{R}^{3\times 3}$ and $\frac{1}{N}\mQ^\UT \bm{A}^{-1}\vx_\star\in\mathbb{R}^{3\times 1}$, which essentially involves computing the limit of $\frac{1}{N}\vx_\star^\UT\bm{A}^{-1}\mW^{i} \vx_\star$, for $i=1,2$; cf.~\eqref{Eqn:P_matrix_def} and \eqref{Eqn:Q_matrix_def}. For this purpose, we can apply standard results on concentration of quadratic forms of rotationally invariant matrices (see \factref{fact:qf} in \appref{appendix:misc}) to obtain
\begin{equation}\label{Eqn:App_A_inv_W}
\begin{split}
\underset{N\to\infty}{\text{plim}}\  \frac{1}{N}\vx_\star^\UT\bm{A}^{-1}\mW^{i} \vx_\star &=\underset{N\to\infty}{\text{plim}}\  \frac{1}{N}\vx_\star^\UT\left(c_\ast\mI-J(\mW)\right)^{-1}\mW^{i} \vx_\star,\quad \forall i\in\mathbb{N}\cup\{0\}\\
&= \E\left[\frac{\sfLambda^i }{c_\ast-J(\sfLambda)}\right]\\
&\explain{(a)}{=}\E\left[\frac{\sfLambda^i }{\rho+C_\phi-J(\sfLambda)}\right]\\
 &\explain{def}{=}d_i,
\end{split}
\end{equation}
where $\sfLambda\sim\mu$ denotes the limiting eigenvalue distribution of $\bm{W}$, and step (a) is from the definition of $c_\ast\explain{def}{=}\rho+C_\phi$.  Similarly,
\begin{equation}\label{Eqn:App_A_inv_W2}
\begin{split}
\underset{N\to\infty}{\text{plim}}\  \frac{1}{N} \bm{x}_\star^\UT \bm{W}\bm{x}_\star& = \mathbb{E}[\sfLambda]=0.
\end{split}
\end{equation}
Using \eqref{Eqn:App_A_inv_W} and \eqref{Eqn:App_A_inv_W2}, and recalling the definitions of the matrices $\bm{P}$ and $\bm{Q}$ in \eqref{Eqn:P_matrix_def} and \eqref{Eqn:Q_matrix_def}, we arrive at
\begin{equation}\label{Eqn:App_P_Q_temp}
\begin{split}
\underset{N\to\infty}{\text{plim}}\ \frac{1}{N} \vx_\star^\UT \bm{A}^{-1}\mP 
&= [d_0,d_1,d_2], \\
\underset{N\to\infty}{\text{plim}}\  \frac{1}{N} \vx_\star^\UT \bm{A}^{-1}\mQ 
&= [\gamma\theta^2d_0+\kappa\theta^2d_2,\kappa\theta^2d_1,\kappa\theta^2d_2], \\
\underset{N\to\infty}{\text{plim}}\ \frac{1}{N}\mQ^\UT \bm{A}^{-1}\mP 
&=\begin{bmatrix}
\kappa\theta^2 d_2 + \gamma\theta^2 d_0 & \kappa\theta^2 d_3 + \gamma\theta^2 d_2  & \kappa\gamma^2 d_4 + \gamma\theta^2 d_2 \\
\kappa\theta^2 d_1 & \kappa\theta^2 d_2 & \kappa\theta^2 d_3 \\
\kappa\theta^2 d_0 & \kappa\theta^2 d_1 & \kappa\theta^2 d_2
\end{bmatrix}.
\end{split}
\end{equation}
Combining \eqref{Eqn:app_replica_temp}, \eqref{Eqn:App_A_inv_W} and \eqref{Eqn:App_P_Q_temp}, and after lengthy but straightforward calculations, we obtain
\begin{equation}\label{Eqn:resolvent_di_final}
\begin{split}
 &\underset{N\to\infty}{\text{plim}}\  \frac{1}{N} \vx^{\UT}_\star\left(c_\ast\mI-J(\bm{Y})\right)^{-1}\vx_\star\\
 &=-\frac{1}{(\kappa\theta^2)^2 d_4 + \gamma\theta^2 + \frac{3\kappa\theta^2 d_2 + (\kappa\theta^2)^2(2d_1 d_3-3d_2^2) + (\kappa\theta^2)^3(d_2^3+d_0 d_3^2-2d_1 d_2 d_3) - 1}{d_0+\kappa\theta^2 d_1^2 - \kappa\theta^2 d_0 d_2}}.
 \end{split}
\end{equation}
Combining \eqref{Eqn:resolvent_di_final} and \eqref{Eqn:xT_g_x} yields the desired identity in \eqref{Eqn:g_hat_quartic_result}.
\end{proof}

\subsubsection{Proof of Proposition~\ref{Lem:SE2_quartic}}\label{App:quartic_replica_details}

In what follows, we prove that any solution $(\rho,\omega)$ to the state evolution (SE) fixed point equation \eqref{eq:rho-omega-supp}, with $\omega \in (0,1), \rho \in (0,\infty)$, is a solution to the replica fixed point equation of \eqref{Eqn:SE_FP_quartic}. This is essentially achieved by combining Lemma \ref{lemma:alt_fp_opt_se} and Lemma \ref{Lem:aux13}. More specifically, the replica fixed point equations in \eqref{Eqn:SE_FP_quartic}, which we aim to prove, can be obtained by replacing the term $\mathbb{E}\left[ \frac{1}{\rho+\phi(\sfLambda_\nu)}\right]$ in \eqref{Eqn:OAMP_FP2_a} by the alternative representation of $\mathbb{E}\left[ \frac{1}{\rho+C_\phi-J(\sfLambda_\nu)}\right]$ in \eqref{Eqn:g_hat_quartic_result}, which is valid because of Lemma \ref{Lem:phi_J_R}, together with additional algebraic manipulations.

Before we present the detailed proof, we first introduce a few definitions. Let
\begin{align}\label{Eqn:chi_def_app}
\chi &\equiv\chi(\rho)\explain{def}{=} \frac{(1-d_0)(d_1+\kappa\theta^2 d_0 d_3-\kappa\theta^2 d_1 d_2)}{d_0-\kappa\theta^2 d_0 d_2+\kappa\theta^2 d_1^2},
\end{align}
where 
\[
\begin{split}
d_i &\explain{def}{=}\mathbb{E}[\mathsf{H}\sfLambda^i]\quad \forall i\in\mathbb{N}\cup\{0\},\\
\mathsf{H}&\bydef\mathsf{H}(\sfLambda,\rho) = \frac{1}{\rho+\phi_{\text{poly}}(\sfLambda) }= \frac{1}{\rho+C_\phi-J(\sfLambda) }\\
&\explain{\eqref{Eqn:phi_quartic}}{=}\frac{1}{\rho + \theta^2a^2(\gamma+2a^2\kappa)^2 + 1 -\theta\left(\gamma\serv{\Lambda}-\theta\kappa\serv{\Lambda}^2+\kappa\serv{\Lambda}^3\right) }.
\end{split}
\]
The rationale for the above definition of $\chi$ will be clear from subsequent derivations. The notations $\chi(\rho)$ and $ \mathsf{H}(\sfLambda,\rho) $ above highlight the fact that $\chi$ and $\mathsf{H}$ both depend on the variable $\rho$, which is defined to be a solution to the SE fixed point equation. We further define
\begin{align}
m&\equiv m(\omega)\bydef 1-\mmse(\omega),\label{Eqn:m_mmse_app}\\
\mathsf{Q} &\equiv\mathsf{Q}(\sfLambda,\mathsf{H},\rho,\omega)\bydef \kappa\theta^2 m\sfLambda^2 + \kappa\theta^2\chi\sfLambda
- \frac{\kappa\theta^2}{1-m} \E\left[ m\sfLambda^2\mathsf{H} + \chi\sfLambda \mathsf{H}\right] + \frac{m}{1-m} .\label{Eqn:Q_def_app}
\end{align}
Again, the notation $m(\omega)$ and $\mathsf{Q}(\sfLambda,\mathsf{H},\rho,\omega)$ indicate that $m$ depends on $\omega$, and $\mathsf{Q}$ depends on the random variables $\sfLambda$ and $\mathsf{H}$ (which is a function of $\rho$) and the variables $(\rho,\omega)$.

To finish the proof, it remains to verify that the variables $\chi$, $m$, $\mathsf{H}$ and $\mathsf{Q}$ as defined above satisfy the replica fixed point equations \eqref{Eqn:SE_FP_quartic}.

\paragraph*{Proof of \eqref{Eqn:SE_FP_quartic_1}} By the definition of $\mathsf{H}$:
\[
\begin{split}
\mathbb{E}[\mathsf{H}] &\explain{\eqref{eq:replica_fp_JY3_4}}{=} \mathbb{E}\left[\frac{1}{\rho+C_\phi-J(\sfLambda) }\right]\\
&\explain{\eqref{Eqn:phi_phi_poly_inside}}{=}\mathbb{E}\left[\frac{1}{\rho+\phi(\sfLambda)}\right]\\
&\explain{\eqref{Eqn:OAMP_FP2_b}}{=}\mmse(\omega)\\
&\explain{\eqref{Eqn:m_mmse_app}}{=}1-m.
\end{split}
\]
This proves \eqref{Eqn:SE_FP_quartic_1}.

\paragraph*{Proof of \eqref{Eqn:SE_FP_quartic_2}} We first verify equation \eqref{Eqn:SE_FP_quartic_1}: $\chi=\mathbb{E}[\sfLambda\mathsf{QH}]$. To this end, we will substitute the definitions of $\mathsf{Q}$ and $\chi$ into $\mathbb{E}[\sfLambda\mathsf{QH}]$ and then show that it is equal to $\chi$:
\begin{equation}\label{Eqn:E_LQH_app}
\begin{split}
&\mathbb{E}[\sfLambda\mathsf{QH}] \explain{(a)}{=} \E\left[\sfLambda \left(\kappa\theta^2 m\sfLambda^2 + \kappa\theta^2\chi\sfLambda - \frac{\kappa\theta^2}{1-m}\E\left[ m\sfLambda^2\mathsf{H} + \chi\sfLambda \mathsf{H}\right] + \frac{m}{1-m} \right)\mathsf{H}\right] \\
&= \kappa\theta^2 \left( m\E[\sfLambda^3 \mathsf{H}] + \chi\E[\sfLambda^2 \mathsf{H}] - \frac{1}{1-m}\E[\sfLambda \mathsf{H}]\left( m\E[\sfLambda^2 \mathsf{H}] + \chi\E[\sfLambda \mathsf{H}]\right) \right) + \frac{m}{1-m}\E[\sfLambda \mathsf{H}] \\
&\explain{(b)}{=} \kappa\theta^2 \left( md_3+\chi d_2-\frac{1}{1-m}d_1(md_2+\chi d_1) \right) + \frac{m}{1-m}d_1 \\
&\explain{(c)}{=} \kappa\theta^2 \left( (1-d_0)d_3 + \chi d_2-\frac{1}{d_0}(1-d_0)d_1 d_2 + \chi \frac{1}{d_0}d_1^2 \right)+\frac{1-d_0}{d_0}d_1\\
&\explain{(d)}{=} \frac{(1-d_0)(d_1+\kappa\theta^2 d_0 d_3-\kappa\theta^2 d_1 d_2)}{d_0-\kappa\theta^2 d_0 d_2+\kappa\theta^2 d_1^2}\\
&\explain{(e)}{=}\chi,
\end{split}
\end{equation}
where step (a) is from the definition of $\mathsf{Q}$ in \eqref{Eqn:Q_def_app}; step (b) is from the definition of the shorthand $d_i$: $d_i=\mathbb{E}[\sfLambda^i\mathsf{H}]$; step (c) is due to the identity $d_0=1-m$, which further follows from definition $d_0=\mathbb{E}[\mathsf{H}]$ and the equation $\mathbb{E}[\mathsf{H}]=1-m$ (see \eqref{Eqn:SE_FP_quartic_1}); step (d) follows from the definition of $\chi$ in \eqref{Eqn:chi_def_app} and straightforward calculations; step (e) is from the definition of $\chi$ in \eqref{Eqn:chi_def_app}. This completes the proof of \eqref{Eqn:SE_FP_quartic_2}.

\paragraph*{Proof of \eqref{Eqn:SE_FP_quartic_3}} The term $\E[\sfLambda^2 \mathsf{QH}]$ appears on the RHS of \eqref{Eqn:SE_FP_quartic_3}. We first rewrite it as an explicit function of $(d_0,d_1,d_2,d_3)$ and $\chi$:
\begin{equation}\label{Eqn:E_L2QH_app}
\begin{split}
\E[\sfLambda^2 \mathsf{QH}]&=\E\left[\sfLambda^2 \left(\kappa\theta^2 m\sfLambda^2+\kappa\theta^2\chi\sfLambda -\frac{\kappa \theta^2}{1-m}\E\left[ m\sfLambda^2\mathsf{H} + \chi\sfLambda \mathsf{H}\right] + \frac{m}{1-m}\right)\mathsf{H}\right] \\
&\explain{(a)}{=}\kappa\theta^2 m d_4+\kappa\theta^2 d_3\chi - \frac{\kappa\theta^2}{1-m}\left(md_2^2+d_1 d_2\chi\right) + \frac{m}{1-m}d_2 \\
&\explain{(b)}{=}(1-d_0)\left(\kappa\theta^2d_4+\frac{\kappa\theta^2 d_3\chi}{1-d_0} - \frac{\kappa\theta^2}{d_0}\left(d_2^2+\frac{\chi}{1-d_0}d_1 d_2\right) + \frac{1}{d_0}d_2\right),
\end{split}
\end{equation}
where step (a) is from the definition of $\mathsf{Q}$ in \eqref{Eqn:Q_def_app}; step (b) uses the identity $d_0=1-m$, similar to step (c) of \eqref{Eqn:E_LQH_app}. Using \eqref{Eqn:E_L2QH_app}, we can continue to rewrite the RHS of \eqref{Eqn:SE_FP_quartic_3} as a function of $(d_0,d_1,d_2,d_3)$ and $\chi$:
\begin{equation}\label{Eqn:beta_RHS_0}
\begin{split}
&\text{RHS of }\eqref{Eqn:SE_FP_quartic_3}\\
&= \kappa\theta^2 \left(\frac{m}{1-m}\E[\sfLambda^2\mathsf{H}]+\frac{\chi}{1-m}\E[\sfLambda \mathsf{H}]+\E[\sfLambda^2 \mathsf{QH}] \right)+ \gamma\theta^2 m\\
&= \kappa\theta^2 \left(\frac{1}{d_0}(1-d_0)d_2+\frac{1}{d_0}d_1\chi+\E[\sfLambda^2 \mathsf{QH}]  \right)+(1-d_0)\gamma\theta^2 \\
&=(1-d_0)\left((\kappa\theta^2)^2d_4+\gamma\theta^2+\frac{\kappa\theta^2}{d_0(1-d_0)}\left(d_1+\kappa\theta^2(d_0d_3-d_1d_2)\right)\chi+\frac{\kappa\theta^2}{d_0}(2d_2-\kappa\theta^2d_2^2)\right).
\end{split}
\end{equation}
Note that $\chi$ is itself a function of $(d_0,d_1,d_2,d_3)$; see \eqref{Eqn:chi_def_app}. Substituting \eqref{Eqn:chi_def_app} into \eqref{Eqn:beta_RHS_0} eliminates the variable $\chi$. After some algebra we finally obtain
\begin{equation}
\begin{split}
&\text{RHS of }\eqref{Eqn:SE_FP_quartic_3}\\
&\explain{(a)}{=} (1-d_0)\bigg((\kappa\theta^2)^2 d_4 + \gamma\theta^2  \bigg. \\ & \hspace{2cm} \bigg. + \frac{3\kappa\theta^2 d_2 + (\kappa\theta^2)^2 ( 2d_1 d_3-3d_2^2) + (\kappa\theta^2)^3 ( d_2^3 + d_0 d_3^2 - 2d_1 d_2 d_3) - 1}{d_0+\kappa\theta^2 d_1^2-\kappa\theta^2 d_0 d_2}+\frac{1}{d_0}\bigg)\\
&\explain{(b)}{=}(1-d_0)\left(-\left(\mathbb{E}\left[ \frac{1}{\rho+C_\phi-J(\sfLambda_\nu)}\right]\right)^{-1}+\frac{1}{d_0}\right)\\
&\explain{(c)}{=}\left(1-\E\left[\frac{1}{ \rho+\phi(\sfLambda)}\right]\right)\left\{-\left(\mathbb{E}\left[ \frac{1}{\rho+C_\phi-J(\sfLambda_\nu)}\right]\right)^{-1}+\left(\E\left[\frac{1}{ \rho+\phi(\sfLambda)}\right]\right)^{-1}\right\}\\
&\explain{(d)}{=}\frac{\omega}{1-\omega},
\end{split}
\end{equation}
where step (a) is a consequence of \eqref{Eqn:chi_def_app} and \eqref{Eqn:beta_RHS_0}; step (b) is from Lemma \ref{Lem:aux13}; step (c) is from the definition of $d_0$ (see \eqref{Eqn:di_def_app}); and step (d) is due to \eqref{Eqn:OAMP_FP2} (Lemma \ref{lemma:alt_fp_opt_se}). This proves \eqref{Eqn:SE_FP_quartic_3}.

To conclude, we have verified that the stationary point of the state evolution of OAMP $(\rho,\omega)$ satisfies equations \eqref{Eqn:SE_FP_quartic_0}-\eqref{Eqn:SE_FP_quartic_3}. The proof for the other direction, namely, each fixed point of \eqref{Eqn:SE_FP_quartic} such that $\rho>0$ and $\omega\in(0,1)$ also satisfies the SE equations \eqref{eq:rho-omega-supp} is similar and thus omitted.

\section{Numerical Results}\label{App:numerical_additional}

{This appendix provides the details for the numerical results presented in Section~\ref{Sec:simulations}. It also includes additional numerical results not shown in the main paper.

\subsection{Data Processing Details for Fig.~\ref{Fig:Universality2}}

Both the 1000G and Hapmap3 datasets undergo the preprocessing steps, producing a $2054\times 2054$ symmetric matrix for 1000G and a $1397\times 1397$ symmetric matrix for Hapmap3. We denote this dataset by $\bm{Y}_{\text{full}}$ (which is a ${2504\times 100000}$ matrix for the 1000G dataset and a $1397\times 100000$ matrix for the Hapmap3 dataset). The processing steps for both datasets are the same, except for the fact that the dimensions of the two datasets are different. \textit{In what follows, we shall focus on the 1000G dataset only.} 

In our experiments, we randomly select 3000 SNPs and we denote this subsampled dataset $\bm{Y}_{\text{sampled}}\in\mathbb{R}^{2504\times 3000}$ matrix. We generate a symmetric noise matrix by applying the following additional processing steps:
\begin{enumerate}
\item We estimate the ``ground truth'' signal components $\bm{X}\in\mathbb{R}^{2504\times 2504}$ by forming the covariance matrice using all SNPs from the original datasets. Specifically, the ground truth signal component is defined by $\bm{X}\bydef \bm{Y}_{\text{full}}\bm{Y}_{\text{full}}^\UT/100000$.
\item We generate a noise matrix from the covariance matrix of the sub-sampled data with the ground truth signal removed, namely, $\bm{W}_1\bydef\bm{Y}_{\text{sampled}}\bm{Y}_{\text{sampled}}^\UT/3000-\bm{X}$.
\item We then remove the top 8 principal components (PCs) of $\bm{W}_1$ (corresponding to the outlying eigenvalues) and let this matrix be $\bm{W}_2$. The final noise matrix is obtained by and further centering and scaling the matrix: $\bm{W}\bydef  \sqrt{2504/\|\bm{W}_3\|_F^2}\cdot\bm{W}_3$, where $\bm{W}_3 \bydef  (\bm{W}_2 - \frac{\text{Tr}(\bm{W}_2)}{2504}\mI)$. The idea is the keep the eigenvalues of $\bm{W}_2$ unchanged, but standarize the eigenvalues.
\end{enumerate}
The final observation matrix for which the OAMP algorithm is applied to is given by $\bm{Y}=({\theta}/{N}) \cdot \bm{x}_\star\bm{x}_\star^\UT+\bm{W}$, where $\theta=3.5$ and $\bm{W}$ is generated as described above, and $\bm{x}_\star$ is randomly sampled from an i.i.d. prior {$0.1\mathcal{N}(0,10)+0.9\delta_0$}. The histogram of the observation matrix with noise matrix generated from the 1000G dataset is plotted in Fig.~\ref{Fig:hist_G1000}.

 \begin{figure}[htbp]
\begin{center}
\includegraphics[width=0.45\linewidth]{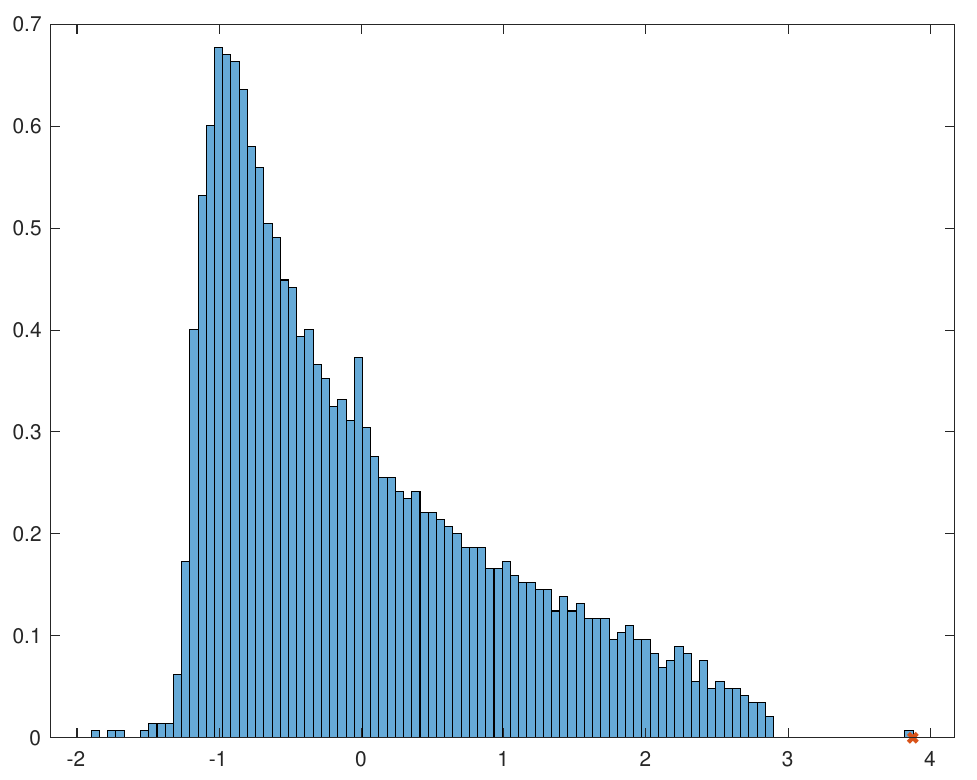}
\end{center}
\caption{Histogram of the eigenvalues of the observation matrix with noise matrix generated from the 1000G dataset.}
\label{Fig:hist_G1000}
\end{figure} 
}

\subsection{Additional Numerical Results}
We now provide some additional numerical results not shown in the main paper. Fig.~\ref{Fig:overlap_vs_snr} compares the following AMP algorithms:
\begin{itemize}
\item \textbf{OAMP:} The optimal OAMP algorithm proposed in the current paper.
\item \textbf{BAMP:} The Bayes-AMP algorithm proposed in \citep{barbier2023fundamental}.
\item \textbf{AMP-AP:} The AMP alternating posteriors algorithm proposed in \citep{barbier2023fundamental}.
\item \textbf{AMP:} The AMP algorithm (for rotationally-invariant model) proposed in \cite{fan2022approximate}. 
\end{itemize}
We remark that both the AMP-AP and AMP algorithms in the above list use a multivariate MMSE signal denoiser (which depends on all previous iterates), while OAMP and BAMP use the single-step MMSE signal denoiser (which only depends on the current iterate). Fig.~\ref{Fig:overlap_vs_snr} also includes the state evolution predictions for the performances of the above AMP algorithms, and the conjectured Bayes-optimal performance calculated via the replica method \citep{barbier2023fundamental}. 

Fig.~\ref{Fig:overlap_vs_snr} shows that the state evolution results for all the AMP algorithms well match the simulations. Moreover, the optimal OAMP algorithm indeed achieves the best performance among the above AMP algorithms. {We notice that the performance of the AMP-AP algorithm is very close to that of OAMP, suggesting that AMP-AP might enjoy similar optimality properties as the OAMP algorithm proposed in our paper. However, showing this seems challenging since the state evolution of AMP-AP is significantly more complicated.}

\begin{figure}[htbp]
\begin{center}
\subfloat{
\includegraphics[width=0.33\linewidth]{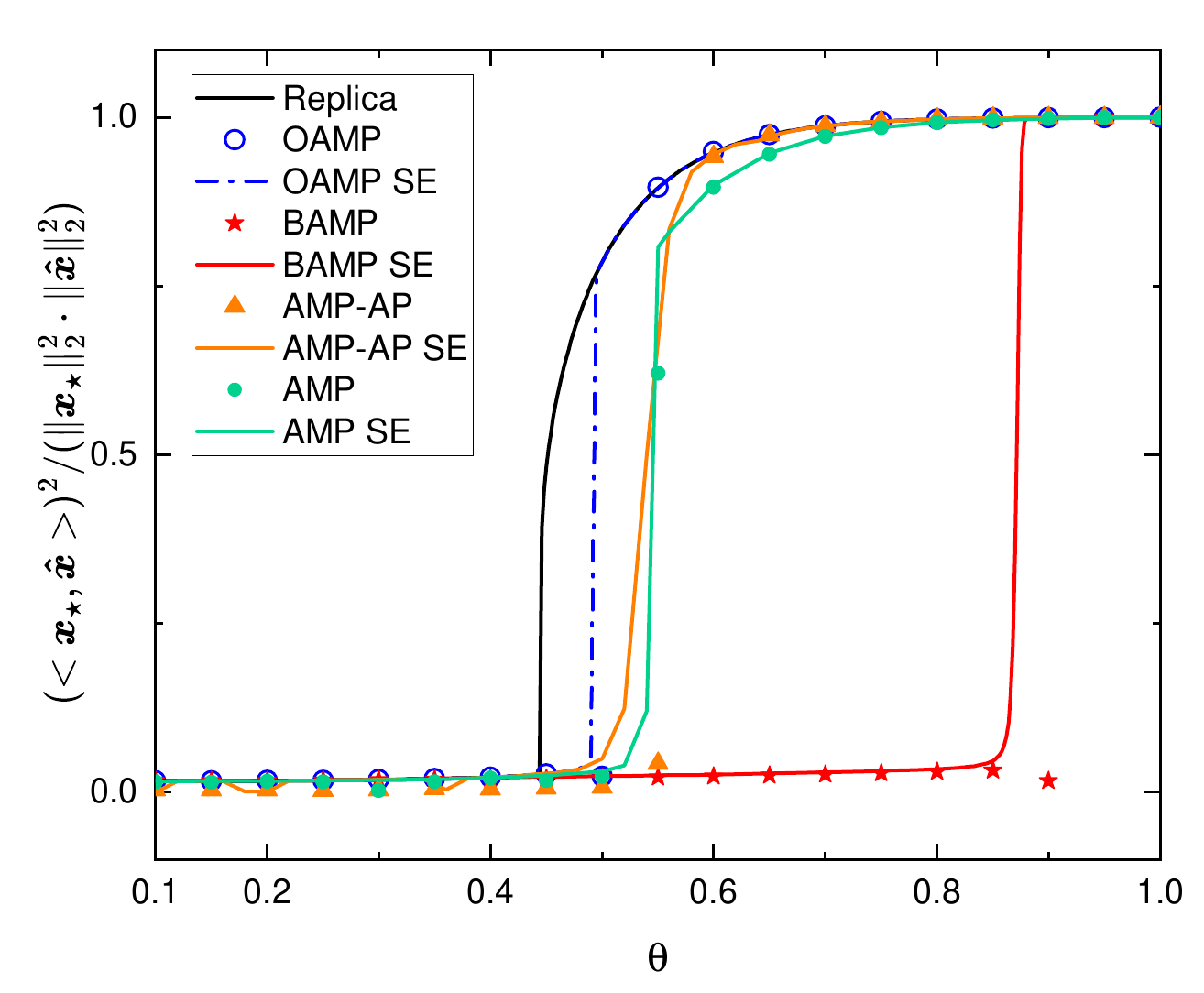}}
\subfloat{
\includegraphics[width=0.33\linewidth]{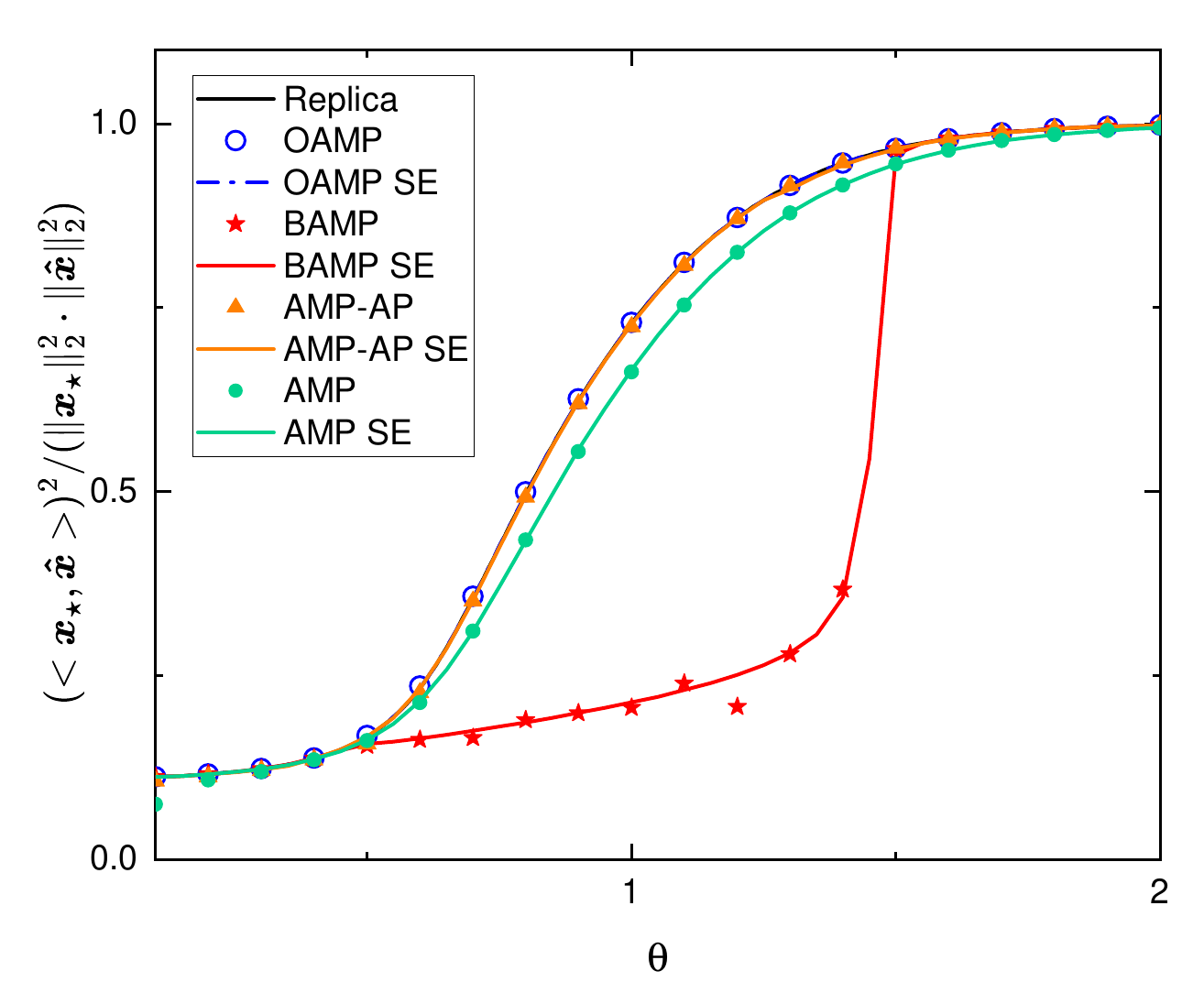}}
\subfloat{
\includegraphics[width=0.33\linewidth]{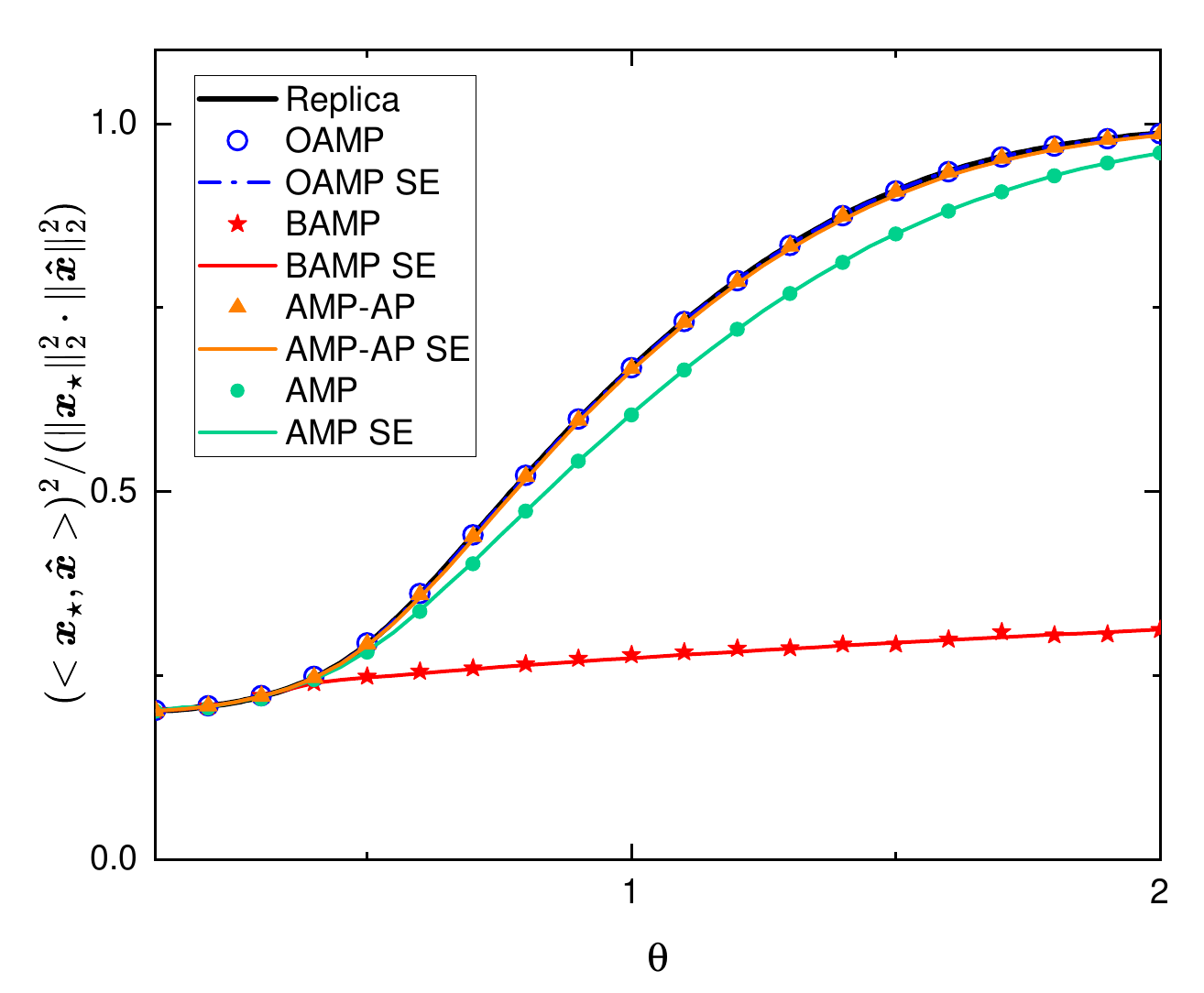} }
\end{center}
\caption{Comparison of the performances of the Bayes-optimal estimator, the optimal OAMP algorithm, the BAMP algorithm and AMP-AP algorithms in \citep{barbier2023fundamental}. The detailed experimental setups are as follows. \textbf{(Left)} Noise spectrum: quartic potential with $\gamma=0$. Signal prior: two points prior with $\epsilon=\frac{1}{8}$. \textbf{(Center)} Noise spectrum: quartic potential with $\gamma=0$. Signal prior: two Points prior with $\epsilon=\frac{1}{3}$. \textbf{(Right)} Noise spectrum: purely sestic potential. Signal prior: two Points prior with $\epsilon=\frac{1}{\sqrt{5}}$. The signal dimension is $2\times 10^{6}$. All algorithms use random initialization.}
\label{Fig:overlap_vs_snr}
\end{figure} 

Fig.~\ref{Fig:MSE_vs_Iteration} compares the dynamics of the AMP algorithms shown in Fig.~\ref{Fig:overlap_vs_snr}. We use the same noise model as before and sample the signal from an i.i.d. two-point prior with $\epsilon=0.5$ and $\theta=1.8$. In the experiments, all the AMP algorithms use random initialization. We see that both OAMP and AMP-AP converge to the conjectured Bayes-optimal performance as predicted by the replica method. In contrast, the AMP algorithm of \citet{fan2022approximate} is strictly inferior than the replica prediction, as has already been observed in \citep{barbier2023fundamental}. We notice that the performance of BAMP degrades after a few iteration. Notice that the performance of BAMP is well predicted by its asymptotic state evolution. Hence, the observed behavior of BAMP is not caused by finite-$\dim$ effects or numerical stability issues and seems to be an inherent limitation of the algorithm.

\begin{figure}[htbp]
\begin{center}
\includegraphics[width=0.55\linewidth]{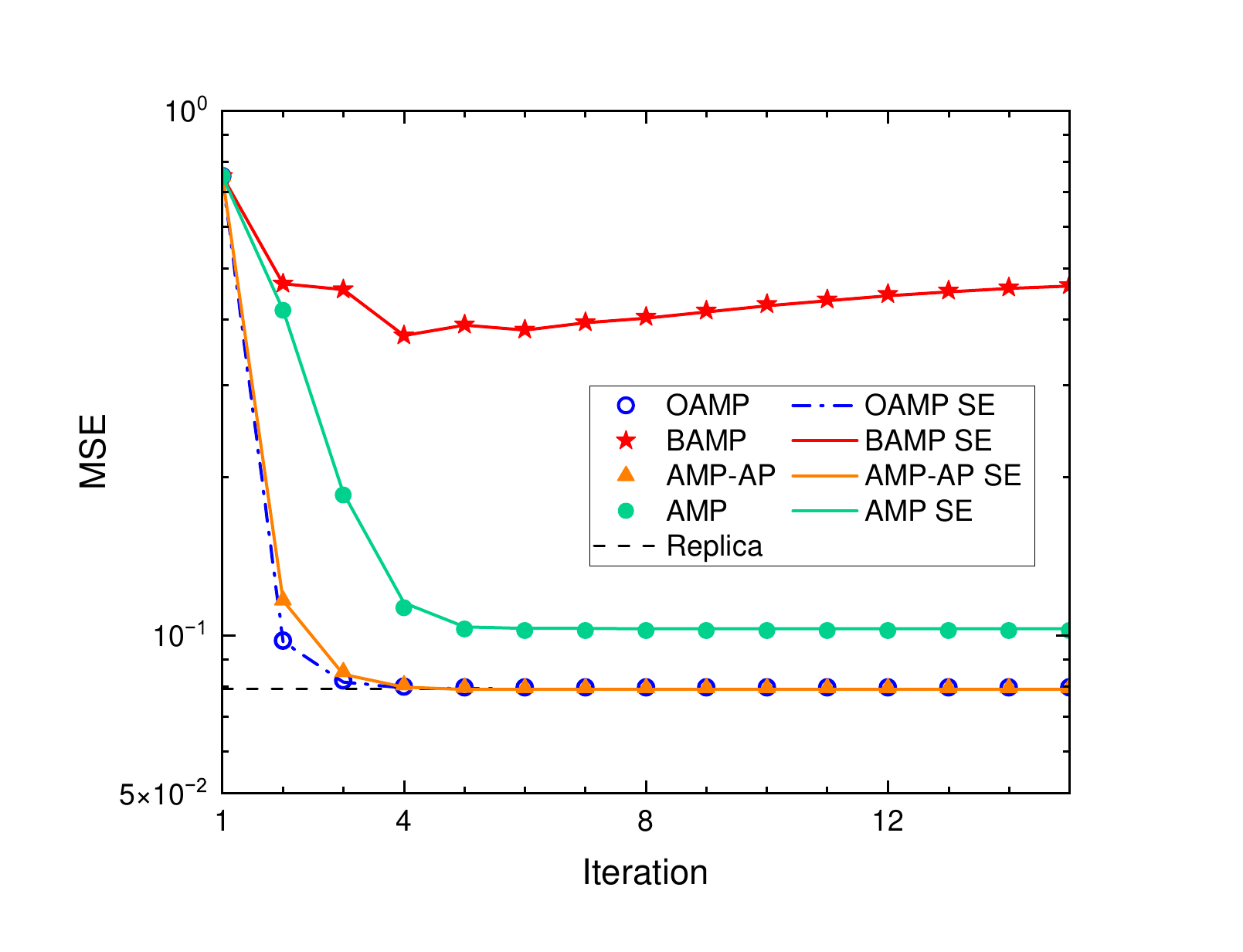}
\end{center}
\caption{MSE performance of various AMP algorithms as a function of the iteration number. The replica conjecture for the Bayes-optimal performance is included for reference.}
\label{Fig:MSE_vs_Iteration}
\end{figure}

\end{document}